\documentclass[12 pt]{book}

\usepackage[toc,page]{appendix}
\usepackage{etex}
\usepackage[shortlabels]{enumitem}
\usepackage{amsmath}
\usepackage{amsxtra}
\usepackage{amscd}
\usepackage{amsthm}
\usepackage{amsfonts}
\usepackage{amssymb}
\usepackage{eucal}
\usepackage[all]{xy}
\usepackage{graphicx}
\usepackage{tikz-cd}
\usepackage{mathrsfs}
\usepackage{mathpazo}
\usepackage{euler}
\usepackage[colorinlistoftodos, textsize=tiny]{todonotes}
\usepackage{morefloats}
\usepackage{mathdots}
\usepackage{pdfpages}
\usepackage{thm-restate}
\usepackage[utf8]{inputenc}
\usepackage{epigraph}
\usepackage{csquotes}
\usepackage{url}

\graphicspath{ {chapters/images/} }

\RequirePackage{color}
\definecolor{myred}{rgb}{0.75,0,0}
\definecolor{mygreen}{rgb}{0,0.5,0}
\definecolor{myblue}{rgb}{0,0,0.65}

\usepackage{tikz}
\usetikzlibrary{matrix,arrows,decorations.pathmorphing}

\theoremstyle{plain}
\newtheorem{theorem}{Theorem}[section]
\newtheorem{proposition}[theorem]{Proposition}
\newtheorem{lemma}[theorem]{Lemma}
\newtheorem{corollary}[theorem]{Corollary}
\theoremstyle{definition}
\newtheorem{definition}[theorem]{Definition}
\newtheorem{remark}[theorem]{Remark}
\newtheorem{example}[theorem]{Example}
\newtheorem{exercise}[theorem]{Exercise}

\newtheorem{question}[theorem]{Question}
 
\newtheorem{warn}[theorem]{Warning}

\theoremstyle{remark}

\numberwithin{equation}{section}
  
\newcommand\nc{\newcommand}
\nc\on{\operatorname}
\nc\renc{\renewcommand}

\newcommand\ssec{\subsection}
\newcommand\sssec{\subsubsection}

\newcommand\bq{{\mathbb Q}}
\newcommand\bp{{\mathbb P}}

\newcommand\bz{{\mathbb Z}}
\newcommand\ba{{\mathbb A}}

\newcommand\fp{{\mathfrak p}}
\newcommand\fq{{\mathfrak q}}
\newcommand\fm{{\mathfrak m}}
\newcommand\so{{\mathscr O}}

\newcommand\scc{\mathscr C}

\newcommand\sce{\mathscr E}
\newcommand\scf{\mathscr F}
\newcommand\scg{\mathscr G}
\newcommand\sch{\mathscr H}
\newcommand\sci{\mathscr I}

\newcommand\scl{\mathscr L}
\newcommand\scm{\mathscr M}

\newcommand\sco{\mathscr O}

\newcommand\scq{\mathscr Q}
\newcommand\scs{\mathscr Q}

\newcommand\scu{\mathscr U}
\newcommand\scv{\mathscr V}

\newcommand\scx{\mathscr X}
\newcommand\scy{\mathscr Y}
\newcommand\scz{\mathscr Z}

\newcommand \ra{\rightarrow}
\newcommand{\id}{\mathrm{id}}

\newcommand \spec{\text{Spec }}
\newcommand \proj{\text{Proj }}

\newcommand \stor{\textit{Tor}}

\newcommand \mg{{\mathscr M_g}}
\newcommand \minhilb[2]{\sch^{\text{scroll}}_{#1, #2}}
\newcommand \hilb[1]{\sch_{#1}}
\newcommand \ringhilb[2]{\sch_{#1, #2}}
\newcommand \uhilb[1]{\scv_{#1}}
\newcommand \ringuhilb[2]{\scv_{#1, #2}}
\newcommand \ci[3]{\sch^{ci}_{#1, #2, #3}}
\newcommand \scroll[1]{\text{S}_{#1}}
\newcommand \minhilbsmooth[2]{\sch^{\text{scroll}}_{#1, #2, \text{sm}}}
\newcommand \minhilbsing[2]{\sch^{scroll}_{#1, #2, sing}}
\newcommand \broken[2]{\sch^{broken}_{#1, #2}}
\newcommand \brokengeneral[2]{\sch^{broken, \circ}_{#1, #2}}
\newcommand \gbundle [2]{\mathbb G (#1, #2)}

\newcommand \flag[2]{\sch(#1 \subset #2)}
\makeatletter
\usepackage{hyperref}
\newcommand{\customlabel}[2]{%
   \protected@write \@auxout {}{\string \newlabel {#1}{{#2}{\thepage}{#2}{#1}{}} }%
   \hypertarget{#1}{#2}
}

\DeclareMathOperator\Hilb{Hilb}
\DeclareMathOperator\rk{rk}
\DeclareMathOperator\di{div}
\DeclareMathOperator\pic{Pic}

\DeclareMathOperator\codim{codim}

\DeclareMathOperator\supp{Supp}

\DeclareMathOperator\spn{Span}
\DeclareMathOperator\im{im}
\DeclareMathOperator\End{End}
\DeclareMathOperator\sym{Sym}
\DeclareMathOperator\pgl{PGL}
\DeclareMathOperator\blow{Bl}
\DeclareMathOperator\ch{char}
\newcommand\bbf{{\mathbb F}}
\newcommand\bk{{\Bbbk}}

\def\listtodoname{List of Todos}
\def\listoftodos{\@starttoc{tdo}\listtodoname}

\title{A Thesis of Minimal Degree}
\author{Aaron Landesman}

\begin{document}

\includepdf[pages={1}]{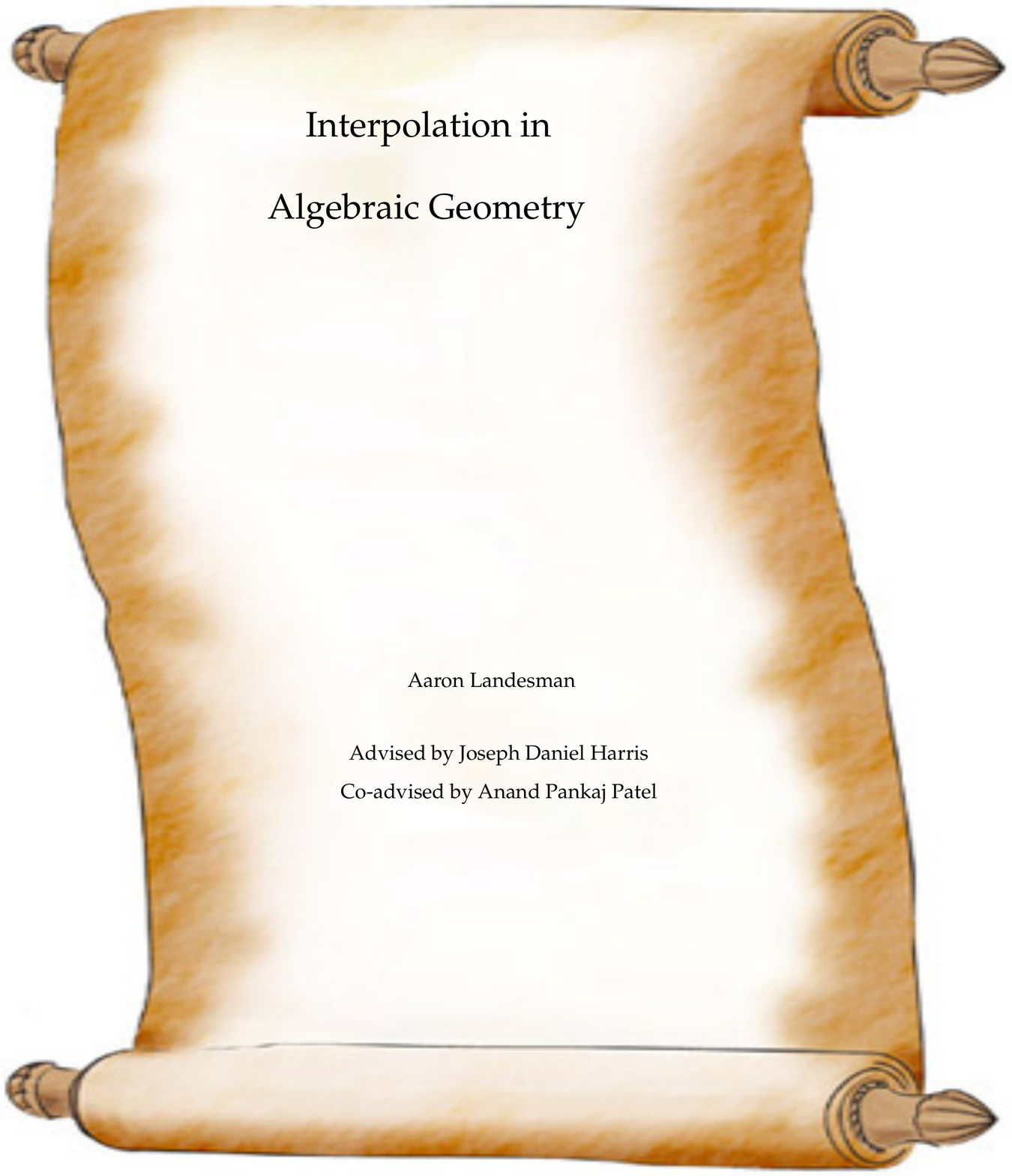}

\tableofcontents

\chapter*{Abstract}
\addcontentsline{toc}{chapter}{Abstract}

This thesis is an expanded version of the two papers
\cite{landesman:interpolation-of-varieties-of-minimal-degree}
and
\cite{landesmanP:interpolation-problems-del-pezz-surfaces}.
In this thesis, we discuss interpolation of projective varieties through points.
It is well known that one can find
a rational normal curve in $\mathbb P^n$ 
through $n+3$ general points. More recently, it
was shown that one can always find nonspecial curves
through the expected number of general points and linear
spaces.
We consider the generalization of this question
to varieties of all dimensions and explain why
smooth varieties of minimal degree satisfy interpolation.
We continue to develop the theory of interpolation,
giving twenty-two equivalent formulations of interpolation.
We also classify when Castelnuovo curves satisfy weak interpolation.
In the appendix, cowritten 
with Anand Patel, we prove that del Pezzo surfaces
satisfy weak interpolation.
Our techniques for proving interpolation include
deformation theory, degeneration and specialization, and association.

\chapter*{Acknowledgements}
\addcontentsline{toc}{chapter}{Acknowledgements}

I would like to start by thanking my thesis advisor, Joe Harris,
for all the time he has invested in teaching me and answering my questions.
He taught me how to ``feel'' the geometry, a viewpoint so often
glossed over in this increasingly technical world.
His endless pool of questions has inspired 
the work in this thesis. In particular,
I would like to thank Joe Harris for arranging weekly meetings with me in my
junior spring, even when I hadn't yet chosen him as my thesis adviser.

Second, I thank my co-advisor, Anand Patel,
for the countless hours he spent helping me think
through my ideas, filling in details, and
always tailoring his explanations so that I could
best understand them.

I thank all my teachers from mathematics,
computer science, and other fields, from Harvard, summer programs, and from high school.
In particular, I thank Professors Dennis Gaitsgory, Joe Harris,
Vic Reiner, and David Zureick-Brown for advising me at Harvard, Minnesota, and Emory.
In particular, I thank David Zureick-Brown for spurring my interest in moduli spaces,
which ultimately led to the genesis of the problems I explored in this thesis.
I also thank Joe Harris, Noam Elkies, Benedict Gross, Mike Hopkins, Mark Kisin, Jacob Lurie,
and David Zureick-Brown for meeting with me to help me choose the topic of my senior thesis.

I thank many more people for especially helpful conversations and
guidance. I thank Brian Conrad for guiding me to Joshua Greene's thesis on 
the construction of the Hilbert scheme, which was the starting point of work on my thesis.
I also thank Izzet Coskun for pointing out many methods to streamline several
arguments.
I thank 
J\'anos Koll\'ar for the suggestion
that the normal bundle to the $2$-Veronese may fail to satisfy interpolation
in characteristic $2$.
I thank Brian Conrad and Ravi Vakil for helping me resolve several technical
details. I thank Eric Larson 
for numerous meetings and for explaining
the subtleties of interpolation.
I also thank 
Levent Alpoge,
Atanas Atanasov,
Francesco Cavazzani, 
Atticus Christensen,
Aise Johan de Jong,
Anand Deopurkar,
Ashwin Deopurkar,
Phillip Engel, 
Changho Han,
Brendan Hassett,
Allen Knutsen,
Carl Lian,
John Lesieutre,
Alex Perry,
Geoffrey Smith,
Hunter Spink,
Jason Starr,
Adrian Zahariuc,
Yifei Zhao, and
Yihang Zhu for helpful conversations.
I thank Brendan Hassett and Rahul Pandharipande
for explaining why the $2$-Veronese surface satisfies interpolation via email correspondence with Francesco Cavazzani.
For helpful comments on drafts of my thesis, I thank Mboyo Esole, Peter Landesman, and Eric Larson.

I deeply thank my parents, Peter and Susan Landesman, for their love and support.

I thank anyone I may have missed above, including
all my friends and family for helping me learn and develop.

\newpage
 
\chapter{Introduction}
\label{section:introduction}

\section{Ancient history}

The question of interpolation is one of the most classical
questions in algebraic geometry. The origins of interpolation date back to the
very beginning of geometry, starting in Chapter 1 of Book 1 of Euclid's
``The Elements,'' written in Alexandria, Ptolemaic Egypt, around $300$ BCE.
After opening with a series of $23$ definitions, Euclid states the following $5$ postulates,
which are the basis for Euclidean geometry:

\begin{center}
\includegraphics[scale=.3]{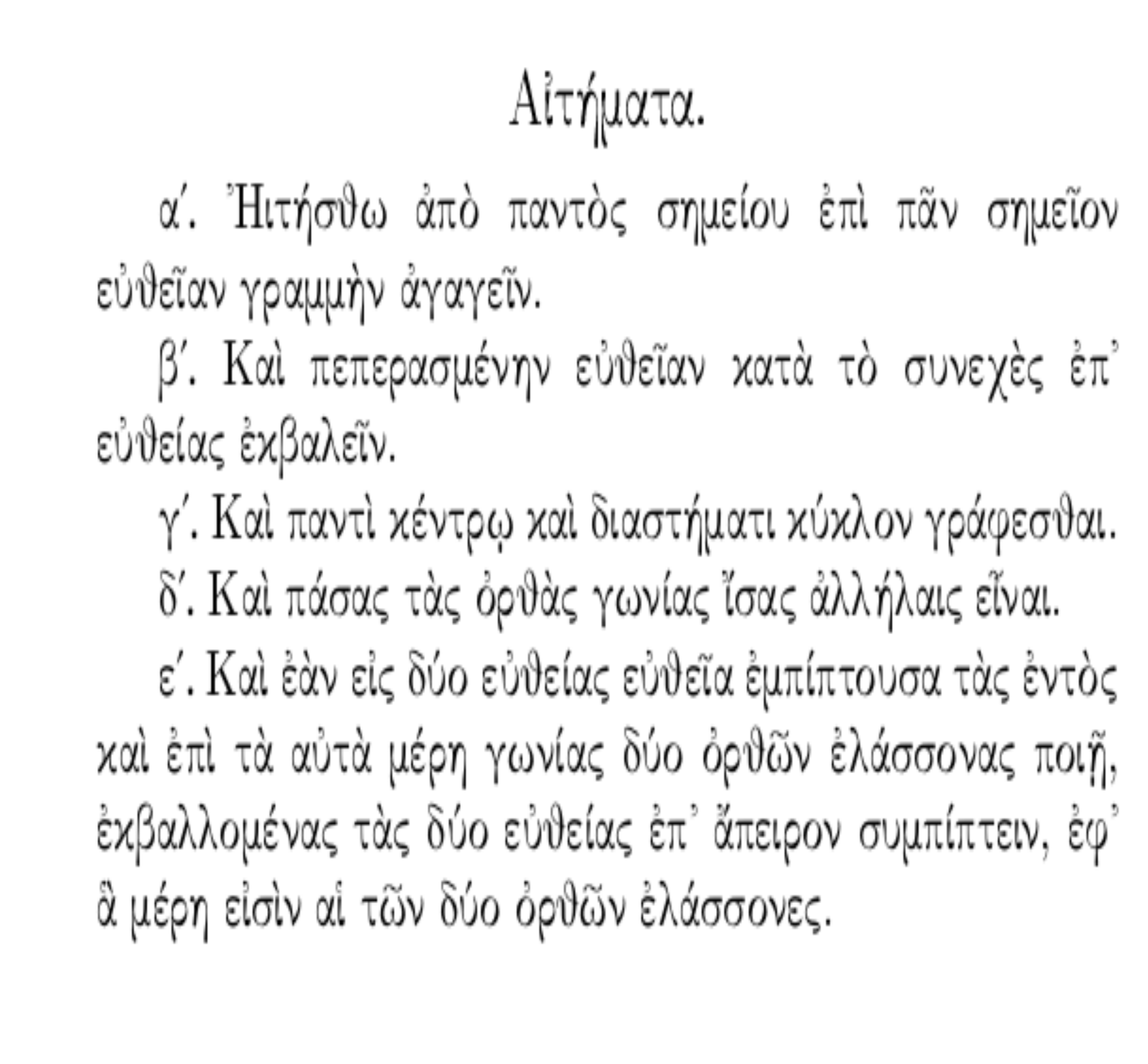}
\end{center}

We concentrate on the first postulate,
which is translated literally
in \cite[p.\ 7]{fitzpatrick:euclids-elements-of-geometry}
as 
\begin{displayquote}
Let it have been postulated to draw
a straight-line from any point to any point.
\end{displayquote}
Or, perhaps in its more well known form,
one might translate Euclid's first postulate as
\begin{displayquote}
Through any two points there passes a line.
\end{displayquote}
This is the first instance of interpolation,
which, loosely speaking, is a souped-up game of connect the dots.
That is, we are asking if there is some type of algebraic object
passing through a given collection of general points.

For example, a natural generalization of the fact that a line
passes through two points is that if we specify a general
set of three points in space, we can always find a plane
passing through them.

To generalize this in another direction, given a triangle, one
can always find a circumscribed circle. This is the same as saying
that one can interpolate a circle through any three points in the real plane:
If we start with three points, we can draw the triangle whose vertices
are those three points.
Then, the circle circumscribed about the triangle
will pass through those three points.

\begin{figure}
	\centering
	\includegraphics[scale=.8]{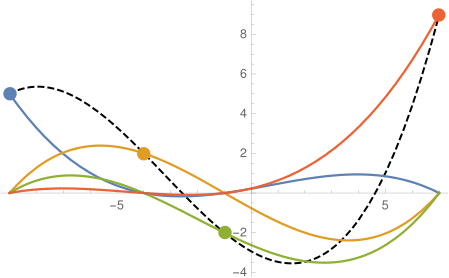}
	\caption{A visualization of Lagrangian interpolation}
	\label{figure:lagrange-interpolation}
\end{figure}

Fast forwarding to the eighteenth century, a new question
of interpolation was brought to the forefront of mathematics:
that of Lagrangian interpolation.
This question asks whether we can find a polynomial in the plane
passing through a given collection of points.
More precisely:
\begin{question}
	\label{question:lagrangian-interpolation}
	Given $n$ points $(x_1, y_1), \ldots, (x_n, y_n)$
	in the $xy$-plane, does there exist some polynomial
	$p(x)$
	of degree at most $n-1$
	so that $p(x_i) = y_i$ for $1 \leq i \leq n$?
\end{question}
Although typically associated with Lagrange,
this question was first answered affirmatively
in 1779 by Edward Warring
\cite{wikipedia:lagrange-interpolation}.
It is also an immediate consequence of one
of Euler's works, published in 1783
\cite{wikipedia:lagrange-interpolation}.
Nevertheless, Lagrange published this result in
1795, and it is usually attributed to him
\cite{wikipedia:lagrange-interpolation}.

Lagrangian interpolation has proven essential not only in mathematics,
but throughout numerous other scientific fields.
To name a few applications, Lagrangian interpolation has led to 
the creation of shapes of letters in typography,
protection of secret information, like
bank accounts and missile codes,
and developments in algorithmic computer science
\cite{wikipedia:polynomial-interpolation}.

\section{Description of interpolation}

In simple terms, an interpolation problem involves two pieces of data:
\begin{enumerate}
	\item a class $\sch$ of varieties in projective space (e.g. ``rational normal curves'') often specified by a component of a Hilbert scheme,
	\item a collection of (usually linear) incidence conditions (e.g. ``passing through five fixed points and incident to a fixed $2$-plane'').
\end{enumerate}
The problem is then to determine whether there exists a variety $[X] \in \sch$ meeting a general choice of conditions of the specified type. 

In a related vein, we can also ask about the number
of these algebraic objects passing through the specified collection
of points. For example, one might be interested in knowing not
only whether there exists a line passing through two points,
but also how many lines pass through two points.
Of course, there is only 1 such line.
In this thesis, we will mostly be concerned
with whether there exists some algebraic object passing through
a specified number of points, for the simple reason that counting
the number proves much harder.
Nevertheless, when possible, we will also address the questions
of how many such objects pass through a specified number of points.

The first nontrivial case of interpolation in higher dimensional projective space is that rational
normal curves satisfy interpolation.
This means that through any 
$n + 3$ points in $\bp^n$, there is a rational normal curve
passing through them, see \autoref{sssec:castelnuovos-lemma}.
Interpolation of higher genus curves in projective space is extensively
studied in 
\cite{stevens:on-the-number-of-points-determining-a-canonical-curve}, 
\cite[Chapter 13]{stevens:deformations-of-singularities},
\cite{atanasovLY:interpolation-for-normal-bundles-of-general-curves}, and 
\cite{larson:interpolation-for-restricted-tangent-bundles-of-general-curves}. We review interpolation for rational normal curves and results of interpolation for higher genus curves in \autoref{section:lay-of-land}.

\section{Main results}

Surprisingly, despite being such a  fundamental problem,
very little is known about interpolation
of higher dimensional varieties in projective space.  To our knowledge, the work of Coble in \cite{coble:associated-sets-of-points}, of
Coskun in \cite{coskun:degenerations-of-surface-scrolls}, and of Eisenbud and Popescu in
\cite[Theorem 4.5]{eisenbudP:the-projective-geometry-of-the-gale-transform} 
are the only places where a higher dimensional interpolation problem is addressed.

In this thesis, we study interpolation problems for higher dimensional
varieties,
particularly those of minimal and almost minimal degree.

\begin{restatable}{theorem}{minimalInterpolation}
	\label{theorem:interpolation-minimal-surfaces}
	Smooth varieties of minimal degree satisfy interpolation.
\end{restatable}

\begin{remark}
	\label{remark:}
	Parts of ~\autoref{theorem:interpolation-minimal-surfaces}
	have been previously established. For example,
	the dimension $1$ case is that there is a rational curve through $n + 3$ points in $\bp^n$.
	The Veronese surface was shown to satisfy interpolation
	in \cite[Theorem 19]{coble:associated-sets-of-points},
	see ~\autoref{theorem:counting-veronese-interpolation}	
	for a more detailed description of this proof.
	It was already established that $2$-dimensional
	scrolls satisfy interpolation in
	Coskun's thesis ~\cite[Example, p.\ 2]{coskun:degenerations-of-surface-scrolls}, and furthermore,
	Coskun gives a method for computing the number of scrolls
	meeting a specified collection of general linear spaces.
	Finally, weak interpolation was established for
	scrolls of degree $d$ and dimension $k$ with
	$d \geq 2k - 1$ in \cite[Theorem 4.5]{eisenbudP:the-projective-geometry-of-the-gale-transform}.

Although the three works cited above all prove bits and pieces of
\autoref{theorem:interpolation-minimal-surfaces}, the remaining
unproven cases were some of the trickiest to deal with in
the proof we present.
While our methods are similar to those of \cite{coskun:degenerations-of-surface-scrolls},
they differ drastically from those of 
\cite[Theorem 4.5]{eisenbudP:the-projective-geometry-of-the-gale-transform}.
\end{remark}

Following our proof that varieties of minimal degree satisfy interpolation,
we look at surfaces of almost minimal degree. If such a surface
is smooth and linearly normal, it is a del Pezzo surface.
This leads to the following result, proven in the appendix.

\begin{restatable}{theorem}{main}
	\label{theorem:main}
All del Pezzo surfaces over a field of characteristic $0$ satisfy weak interpolation.
\end{restatable}

We also prove several additional results.
For example, we characterize which Castelnuovo curves satisfy interpolation in
\autoref{theorem:castelnuovo-interpolation}.

\section{Relevance of interpolation}
\label{subsection:relevance-of-interpolation}
Before detailing what is currently known about interpolation,
we pause to describe several ways in which interpolation
arises in algebraic geometry.

First, interpolation arises naturally when studying families of varieties.  As an example, we consider the problem of producing {\sl moving curves} in the moduli space of genus $g$ curves, 
$\mg$. 
Suppose we know, for example, that canonical (or multi-canonical) curves
satisfy interpolation through a collection of points
and linear spaces. Then, after imposing the correct number of incidence conditions, one obtains a moving curve in $\mg$.
Indeed, as one varies the incidence conditions, these
curves sweep out a dense open set in $\mg$, and hence
determine a moving curve.
One long-standing open problem in this area  is that of determining the least upper bound for
the slope $\delta/\lambda$ of a moving
curve in $\mg$.
In low genera, moving
curves constructed via interpolation realize the least upper bounds. Establishing interpolation is a necessary first step
in the construction of such moving curves. For a more in depth discussion of slopes, see \cite[Section 3.3]{chenFM:effective-divisors-on-moduli-spaces-of-curves-and-abelian-varieties}.
This application is also outlined in the second
and third
paragraphs of \cite{atanasov:interpolation-and-vector-bundles-on-curves}.

Interpolation can be used to construct explicit degenerations of
 varieties in projective space. By specializing the incidence conditions, the varieties interpolating through them may be forced to degenerate as well.  This could potentially shed light on the boundary
of a component of a Hilbert scheme. 

We next provide an application of interpolation to the problems in Gromov-Witten theory.
Gromov-Witten theory can be used to count the number of curves satisfying incidence or tangency conditions.
Techniques in interpolation can also be used to count this number,
and we now explain how interpolation techniques can sometimes
lead to solutions where Gromov-Witten Theory fails.
When the Kontsevich space is irreducible and of the correct dimension 
one can employ Gromov-Witten theory without too much
difficulty to count the number of varieties meeting a certain
number of general points. 
In more complicated cases, one needs a virtual fundamental class,
and then needs to find the contributions of this
virtual fundamental class from nonprincipal components
and subtract the contributions from these components.
However, arguments in interpolation can very often be used
to count the number of varieties containing a general set of points,
as is done for surface scrolls in \cite[Results, p.\ 2]{coskun:degenerations-of-surface-scrolls}.
Coskun's technique also allows one to efficiently compute
Gromov-Witten invariants for curves in $\mathbb G(1,n)$.
Although there was a prior algorithm to compute this using
Gromov-Witten theory, Coskun notes that his method
is exponentially faster. The
standard algorithm, when run on Harvard's MECCAH cluster
``took over four weeks to determine the cubic invariants
of $\mathbb G(1,5)$. The algorithm we prove here allows
us to compute some of these invariants by hand'' \cite[p.\ 2]{coskun:degenerations-of-surface-scrolls}.
 
 Interpolation also distinguishes components of Hilbert schemes. For a typical example of this phenomenon,
consider the Hilbert scheme of twisted cubics in $\bp^3$.
This connected component of the Hilbert scheme has two
irreducible components.
One of these components has general member which is a smooth
rational normal curve in $\bp^3$ and is $12$ dimensional.
The other component
has general member corresponding to the union of a plane
cubic and a point in $\bp^3$, which is $15$ dimensional.
While the first component parameterizing smooth
rational normal curves satisfies interpolation
through $6$ points, there won't 
be a single member of the second component 
passing through
$5$ general points,
even though this second component has a larger dimension than the first.

\section{Interpolation: a lay of the land}\label{section:lay-of-land}

In this section we survey what is known about interpolation so far.
We start off with interpolation for rational normal curves, move on to describing what is known about
higher genus curves, and conclude with what is known about higher dimensional varieties.

\subsection{Interpolation for rational normal curves}
\label{sssec:castelnuovos-lemma}

Through $r+3$ general points in $\bp^{r}$ there exists a unique rational normal curve $\bp^{1} \subset \bp^{r}$.  A dimension count provides evidence for existence: the (main component of the) Hilbert scheme of rational normal curves is $r^{2}+2r-3 = (r+3)(r-1)$ dimensional, and the requirement of passing through a point imposes $r-1$ conditions on rational normal curves. Therefore we {\sl expect} finitely many rational normal curves through $r+3$ general points.  

``Counting constants'' as above  only provides a plausibility argument for existence of rational curves interpolating through the required points -- it is not a proof. To illustrate this, we give an example where interpolation is not satisfied, even though the dimension count says otherwise.

\begin{example}
	\label{example:}
	A parameter count suggests there should be a genus $4$ canonical curve through $12$ general points in $\bp^{3}$. However, such a canonical
	curve is a complete intersection of a quadric and a cubic.
	Since
	a quadric is determined by $9$ general points,
	the curve,
	which lies on the quadric, cannot contain $12$ general points.
	In other words, genus $4$ canonical curves do not satisfy interpolation.
\end{example}

There are many proofs of interpolation for rational normal curves.  One proof proceeds by directly constructing a rational normal curve using explicit equations.  Another approach is via a degeneration argument, as in \autoref{example:rational-normal-curve-interpolation}. One can also use {\sl association}, also known as the Gale transform, (see \cite{eisenbudP:the-projective-geometry-of-the-gale-transform}) to deduce the lemma.  A purely synthetic proof also exists, as found in \cite[Proposition 2.4.4]{pereiraP:an-invitation-to-web-geometry}.

\subsection{Higher genus curves} One way to generalize interpolation for rational normal curves is to consider higher genus curves in projective space. For many reasons it is simpler to consider curves embedded via nonspecial linear systems.  By far, the most comprehensive theorem involving interpolation for nonspecial curves in projective space is the following recent result of Atanasov--Larson--Yang: 

\begin{theorem}[Theorem 1.3, \cite{atanasovLY:interpolation-for-normal-bundles-of-general-curves}]\label{thm:NaskoEricDavid}
Strong interpolation holds for the main component of the Hilbert scheme parameterizing nonspecial curves of degree $d$ and genus $g$ in projective space $\bp^{r}$, unless
\begin{align*}
	(d,g,r) \in \left\{ (5,2,3), (6,2,4), (7,2,5) \right\}.
\end{align*}
\end{theorem}

It is also shown in 
\cite[p.\ 108]{stevens:deformations-of-singularities}
(which combines the work in ~\cite{stevens:on-the-number-of-points-determining-a-canonical-curve},
dealing with the canonical curves of genus not equal to $8$ and
~\cite[Proposition, p.\ 3715]{stevens:on-the-computation-of-versal-deformations},
dealing with canonical curves of genus 8)
that canonical curves of genus at least $3$ fail to satisfy weak interpolation if and only if their genus is $4$ or $6$.

Adding on to the work of Stevens ~\cite{stevens:on-the-number-of-points-determining-a-canonical-curve}
and \autoref{thm:NaskoEricDavid},
we are able to give an complete description of which
Castelnuovo curves satisfy weak interpolation.
This shows that canonical curves approximately
``form the boundary'' between Castelnuovo curves satisfying
interpolation and 
Castelnuovo curves not satisfying interpolation.
See \autoref{section:castelnuovo-curves}
for a definition and discussion of Castelnuovo curves.

\begin{restatable}{theorem}{castelnuovo}
	\label{theorem:castelnuovo-interpolation}
	Let $\bk$ be an algebraically closed field of characteristic $0$.
	Castelnuovo curves of degree $d$ and genus $g$ in $\bp^r_\bk$
	satisfy weak interpolation if and only if
	$d \leq 2r$ and
	\begin{align*}
		(d,g,r) \notin \left\{ (5,2,3), (6,2,4), (7,2,5), (6,4,3), (10,6,5) \right\}.
	\end{align*}
	Further, a Castelnuovo curve of degree $d$ and genus $g$
	in $\bp^r_\bk$ of degree not equal
	to $2r$ satisfies interpolation if and only if
	$d < 2r$ and 
	\begin{align*}
		(d,g,r) \notin \left\{ (5,2,3), (6,2,4), (7,2,5) \right\}.
	\end{align*}
\end{restatable}

\subsection{Higher dimensional varieties: varieties of minimal degree and del Pezzo surfaces}

In this thesis, we establish interpolation for all varieties of minimal degree, see \autoref{theorem:interpolation-minimal-surfaces}, and for smooth linearly normal surfaces
of almost minimal degree, i.e., del Pezzo surfaces, see \autoref{theorem:main}. Recall
that a variety of dimension $k$ and degree $d$ in $\bp^n$
is of minimal degree
if it is not contained in a hyperplane and
$d  = n + 1 - k$.
Further, by \autoref{exercise:reason-for-minimal-degree},
any nondegenerate variety in $\bp^n$ of dimension $k$ 
cannot have degree less than $n + 1 - k$.
By \cite[Theorem 1]{eisenbudH:on-varieties-of-minimal-degree}, an irreducible variety is of minimal degree if and only if
it is a degree 2 hypersurface, the $2$-Veronese in $\bp^5$
or a rational normal scroll.
A variety is of almost minimal degree if its degree satisfies $d = n + 2 - k$.
That is, if its degree is one more than minimal.

\subsection{Approaches to interpolation}
There are at least three approaches to solving interpolation problems.

The first approach is to {\bf directly construct} a variety $[Y] \in \sch$ meeting the specified constructions.  This method is quite ad hoc: For one, we would need ways of constructing varieties in projective space.  Our ability to do so is very limited and always involves special features of the variety. For examples of this approach, see
\autoref{proposition:segre-lines-interpolation}, \autoref{proposition:segre-plane}
as well as
\autoref{sec:degree-5}, \autoref{sec:degree-6}, and \autoref{sec:degree-8-type-0}.

The second standard approach is via {\bf specialization
and degeneration}. In this approach,
we specialize the points to a configuration
for which it is easy to see there is an isolated point
of $\sch$ containing such a configuration.
Often, although not always, the isolated point of $\sch$
corresponds to
a singular variety.
Finding singular varieties may often be easier than
finding smooth ones, particularly if those singular
varieties have multiple components, because we may
be able to separately interpolate each of the components
through two complementary subsets of the points.
See \autoref{example:rational-normal-curve-interpolation}
for an example.
The proof of \autoref{theorem:interpolation-minimal-surfaces}
also involves many examples of degeneration arguments.

The third approach is via {\bf association}, also known as the Gale transform. See
~\autoref{section:association} for more details on
what this means. The general picture
is that association determines a natural way of
identifying a set of $t$ points in $\bp^a$ with a collection of $t$ points
in $\bp^b$, up to the action of $PGL_{a+1}(\bk)$ on $\bp^a$
and $PGL_{b+1}(\bk)$ on $\bp^b$. 
Then, if one can find a certain variety through the $t$ points
in $\bp^a$, one may be able to use association to find
the desired variety through the $t$ associated points in
$\bp^b$. For an example of this approach, see
\autoref{ssec:2-veronese-interpolation} and \autoref{sec:degree-7-8-9}.

\section{Overview}
\label{subsection:overview}

We now present an overview of this thesis.
First, in \autoref{section:background}, we list our conventions, notation, and some
additional fairly well known results which will be used frequently throughout the thesis.
We also include a discussion on Hilbert schemes.
In \autoref{section:interpolation-in-general},
we formally define interpolation and give $22$ equivalent formulations of interpolation
in \autoref{theorem:equivalent-conditions-of-interpolation}.
In \autoref{section:basics-of-scrolls}, we define scrolls and show four
descriptions of scrolls are equivalent.
Then, in \autoref{section:preliminaries-on-varieties-of-minimal-degree},
we detail how to show that smooth varieties of minimal degree correspond to smooth points
of the Hilbert scheme.
After, in \autoref{section:degenerations-of-varieties-of-minimal-degree},
we focus on a particular degeneration of a scroll into the union of a scroll of 
one lower degree and a plane, showing that this is indeed a degeneration
of a smooth scroll and examining how the locus of such degenerate scrolls
lies in the Hilbert scheme.
Using the understanding of varieties of minimal degree from
\autoref{section:preliminaries-on-varieties-of-minimal-degree}
and \autoref{section:degenerations-of-varieties-of-minimal-degree},
we prove that varieties of minimal degree satisfy interpolation
in \autoref{section:interpolation-of-varieties-of-minimal-degree}.
In \autoref{section:castelnuovo-curves}, we classify whether
Castelnuovo curves satisfy interpolation.
We then present several interesting open questions regarding interpolation
in \autoref{section:further-questions}.
Finally, in \autoref{section:appendix}, we
prove that del Pezzo surfaces satisfy weak interpolation.

\chapter{Background and notation}
\label{section:background}

Here, we include general preliminaries, background on the Hilbert scheme, and idiosyncratic notation.

\section{General preliminaries}
\label{subsection:background-preliminaries}

In this section, we briefly recall some fairly standard notation which will be used in this thesis.
As a general rule, our conventions follow those given in \cite{vakil:foundations-of-algebraic-geometry}.

Unless otherwise stated, we work over an algebraically closed field $\bk$ of arbitrary characteristic.
In particular, we do not restrict $\bk$ be of characteristic $0$ except in \autoref{section:appendix}.

We now selectively recall a few commonly used conventions, which we will use frequently.

\begin{itemize}
	\item A variety over a field $\bk$ is a separated, reduced scheme of finite type over $\spec \bk$. In particular,
		we do not assume a variety is irreducible.
	\item If $\iota: X \ra \bp^n$ is a map to projective space, we denote
		$\sco_X(m) := \iota^* \sco_{\bp^n}(m)$, where $\sco_{\bp^n}(m)$
		is the the invertible sheaf whose global sections are degree $m$ polynomials, as defined in \cite[Section 14.1]{vakil:foundations-of-algebraic-geometry}.
		In more generality, if $\scf$ is a sheaf on $X$, then $\scf(m) := \scf \otimes \sco_X(m)$.
	\item If $\pi: X \hookrightarrow Y$ is a closed subscheme, we let $\sci_{X/Y}$ denote the ideal sheaf of $X$ in $Y$,
		which is by definition the kernel of the map $\sco_Y \ra \pi^* \sco_X$. 
		We often write $\sci_X$ for $\sci_{X/Y}$ when reference to $Y$ is clear.
	\item We say a general point of $X$ satisfies a property $P$ if there is a dense open subset
		$U \subset X$ so that every point in $U$ satisfies property $P$.
	\item If $\sce$ is a sheaf over a scheme $X$,
and $x$ is a point of $X$, we use $\sce|_x$ to denote the fiber of $\sce$ at $x$ and $\sce_x$ to denote
the stalk of $\sce$ at $x$.
\item We use $\sce^\vee$ to denote the dual of $\sce$.
\item Throughout, we take the ``Grothendieck'' convention that 
$\bp \sce := \proj \sym^\bullet \sce$ (crucially, we take $\sce$ instead of $\sce^\vee$).
\item For $x \in X$, we let $\kappa(x)$ denote the fraction field of $x$.
\item If $X$ is a variety, we use $p_X$ to denote the Hilbert polynomial
of $X$ and $h_X$ to denote the Hilbert function of $X$.
\item We define $h^i(X, \scf) := \dim H^i(X, \scf)$ for $i \in \bz_{\geq 0}$.
 \item To ease notation, we notate a sequence $(a, \ldots, a)$ with $a$ repeated $b$ times as $(a^b)$. So, for example,
we would notate $(1,1,1,2,3,3)$ as $(1^3, 2, 3^2)$.
\item For notational convenience, if we have a map of schemes $f: X \ra Y$ defined over $\bk$
	and a map of fields $g:\spec L \ra \spec \bk$, let $f_L$ denote the base change of $f$ by $g$
	and let $Y_L$ denote the base change of $Y$ by $g$
\end{itemize}

The following classical terminology is used pervasively throughout this thesis 
\begin{itemize}
	\item A {\bf quadric, cubic, quartic, etc.,} refers to an equation or hypersurface in some projective space of degree $2,3,4$, etc.
	\item A {\bf rational normal curve} refers to an embedding of $\bp^1$ by the linear system $\sco_{\bp^1}(d)$ into $\bp^d$.
	\item A {\bf conic} or plane conic refers to a degree 2 rational normal curve or, equivalently, a quadric in $\bp^2$.
	\item A {\bf twisted cubic} refers to a degree 3 rational normal curve.
\end{itemize}

Next, we recall some standard results from algebraic geometry which will be central
to the remainder of this text.

First, because varieties of minimal degree which are not the $2$-Veronese or a
quadric surface are projective bundles over $\bp^1$,
we will need to understand the structure of these projective bundles.
The structure turns out to be as simple as possible.
It is often attributed to Grothendieck, as it is a special case of one of his theorems,
although it was independently proven several times, as detailed in the discussion following
\cite[Theorem 18.5.6]{vakil:foundations-of-algebraic-geometry}.

\begin{theorem}[\protect{\cite[Theorem 18.5.6]{vakil:foundations-of-algebraic-geometry}}]
	\label{theorem:locally-free-sheaves-on-p1}
	If $\sce$ is a rank $r$ invertible sheaf on $\bp^1$ then
	\begin{align*}
		\sce \cong \sco_{\bp^1}(a_1) \oplus \cdots \oplus \sco_{\bp^1}(a_r)
	\end{align*}
	for a unique nondecreasing sequence of integers $a_1, \ldots, a_r$.
\end{theorem}

Next, we include a simple lemma on Hilbert polynomials, which
will come in handy at several later points.

\begin{lemma}
\label{lemma:hilbert-polynomial-intersection}
Suppose $X, Y$ are closed subschemes of $\bp^n$.
Then $X \cup Y$ has Hilbert polynomial given by the inclusion-exclusion formula
\begin{align*}
	p_{X \cup Y} = p_X + p_Y - p_{X \cap Y}.
\end{align*}
\end{lemma}
\begin{proof}
Say $X \cup Y \subset \bp^n$.
Observe that we have an exact sequence of sheaves
on $\bp^n$
\begin{equation}
	\nonumber
	\begin{tikzcd}
		0 \ar {r} &  \sco_{X \cup Y} \ar {r} & \sco_X \oplus \sco_Y \ar {r} & \sco_{X \cap Y} \ar {r} & 0.
	\end{tikzcd}\end{equation}
	Here, $\sco_Z(m)$ refers to the invertible sheaf on $Z$ as described in
	the notation at the beginning of
	\autoref{subsection:background-preliminaries}.
To see this sequence is exact, we can check it on the level of rings.
Let $X \cup Y = \spec R$, let $I$ correspond to the ideal sheaf of $X$ on $R$, and let $J$ correspond to the ideal sheaf of $Y$ on $R$. 
It is an elementary diagram chase to verify that
\begin{equation}
	\nonumber
	\begin{tikzcd}
		0 \ar {r} &  R/I \cap J \ar {r} & R/I \oplus R/J \ar {r} & R/(I+J) \ar {r} & 0 
	\end{tikzcd}\end{equation}
is exact.

Next, by Serre vanishing \cite[Theorem 18.1.4(ii)]{vakil:foundations-of-algebraic-geometry}, we have $H^1(\sco_{X \cup Y}(m)) = 0$ for $m \gg 0$.
In particular, for $m \gg 0$, we obtain that
\begin{align*}
	h_{X \cup Y}(m) = h_X(m) + h_Y(m) - h_{X \cap Y}(m),
\end{align*}
where $h_Z$ is the Hilbert function of $Z$. Since $h_Z(m) = p_Z(m)$ for $m \gg 0$,
by \cite[Theorem 18.6.1]{vakil:foundations-of-algebraic-geometry}
we have
\begin{align*}
	p_{X \cup Y} = p_X + p_Y - p_{X \cap Y}.
\end{align*}
\end{proof}

\section{Background on Hilbert schemes}
\epigraph{Having thus refreshed ourselves in the oasis of a proof, we now turn again into the desert of definitions.
}{\textit{Br{\"o}cker J{\"a}nich \cite[p.\ 25]{brocker:introduction-to-differential-topology}}}

Next, we briefly recall the definition and key properties of the Hilbert scheme and its variants.
Our personal favorite reference for the construction of the Hilbert scheme is
the wonderfully detailed Harvard senior thesis
\cite{greene:the-existence-of-hilbert-schemes}.
Some other references include \cite[Chapter 5]{FantechiGIK:fundamentalAlgebraicGeometry},
\cite[Chapter 4]{sernesi:deformations-of-algebraic-schemes}, and
\cite[Proposition-Definition 6.8]{Eisenbud:3264-&-all-that}.

In particular, here, we collect definitions of the Hilbert scheme, the universal family,
the Fano scheme, and the flag Hilbert scheme.
We are mostly following
\cite{greene:the-existence-of-hilbert-schemes},
although we follow
\cite[Chapter 4]{sernesi:deformations-of-algebraic-schemes}
for the description of the flag Hilbert scheme.
These are all defined as schemes representing certain functors.

\begin{definition}
	\label{definition:}
	Let $S$ be a locally Noetherian scheme and $X$ a locally projective $S$ scheme,
	meaning there is an open cover $\left\{ S_i \right\}_{i \in I}$ of $S$ so that,
	for each $i$, the map
	$X \times_S S_i \ra S_i$ is projective.
	Let $\scc_S$ be the category of locally Noetherian $S$-schemes
	and let Set denote the category of sets.
	Define the {\bf Hilbert functor} by
	\begin{align*}
		\Hilb(X/S) : \scc_S & \rightarrow \text{ Set }\\
		Z & \mapsto \left\{ \text{closed immersions } Y \hookrightarrow Z \times_S X \text{ flat over }Z \right\}
	\end{align*}
\end{definition}
\begin{definition}
	\label{definition:hilbert-scheme}
	If the Hilbert functor is representable, meaning that there is a natural isomorphism
	\begin{align*}
		 Hom_S(\bullet, \sch(X/S)) \xrightarrow \eta \Hilb(X/S)
	\end{align*}
	for some $S$-scheme $\sch(X/S)$, then we define
	$\sch(X/S)$, to be the {\bf Hilbert scheme} associated to $X$ over $S$.
	
	Further, suppose there exists a closed subscheme $V$ with
	\begin{equation}
		\nonumber
		\begin{tikzcd} 
			V \ar {rd}{\pi} \ar {r} & X \times_S \sch(X/S) \ar {d} \\
			 & \sch(X/S)
		\end{tikzcd}\end{equation}
	so that $\pi$ is flat.
	Also, assume that the natural isomorphism $\eta$ sends a map
	$g : Z \rightarrow \sch(X/S)$ to the subscheme $W \hookrightarrow Z \times X$,
	defined as the fiber product
	\begin{equation}
		\nonumber
		\begin{tikzcd} 
			W \ar {r} \ar {d} & Z \times_S X \ar {d}{g \times \id} \\
			V \ar {r} & \sch(X/S) \times_S X.
		\end{tikzcd}\end{equation}
	Then we define $V$ to be {\bf the universal family} of closed subschemes of $X$ over $S$.	
\end{definition}

\begin{definition}
	\label{definition:}
Let $S$ be a locally Noetherian scheme and $X$ a locally projective $S$ scheme.
Fix a very ample invertible sheaf $\sco_X(1)$ on $X$. Let $P \in \bq[t]$ be a polynomial,
and recall that we are writing $\scc_S$ for the category of locally Noetherian $S$-schemes.
	Define the functor
	\begin{align*}
		\Hilb^P(X/S) : \scc_S \rightarrow &\text{ Set }\\
		Z \mapsto &\left\{ \text{closed immersions } Y \hookrightarrow Z \times_S X \text{ flat over }Z \right. \\
		& \left. \text{ so that for all }z \in Z, Y \times_Z \spec \kappa(z) \text{ has Hilbert} \right. \\
		& \left. \text{ polynomial $P$ over $\spec \kappa(z)$ with respect to the } \right. \\
	& \left. \text{ pullback of $\sco_X(1)$ to $\spec \kappa(z)$}\right\}
	\end{align*}
	When $\Hilb^P(X/S)$ is representable, we define $\sch^P(X/S)$ to be the scheme representing it.
\end{definition}

\begin{theorem}[\protect{\cite[Subsection 1.2]{greene:the-existence-of-hilbert-schemes}}]
	\label{theorem:}
	Let $S$ be a locally Noetherian scheme, let $X$ be a locally projective $S$ scheme, let $P \in \bq[t]$ be a polynomial,
	and let $\sco_X(1)$ be a very ample invertible sheaf on $X$ used to define $\sch^P(X/S)$. 
	Then, the Hilbert scheme $\sch(X/S)$, the universal family of $X$ over $S$, and $\sch^P(X/S)$ exist
	and are locally Noetherian $S$-schemes.
	Further, if $X \rightarrow S$ is projective, then 
	$\sch^P(X/S)$ is projective.
\end{theorem}

\begin{theorem}[Hartshorne \protect{\cite[Corollary 5.9]{hartshorne:connectedness-of-the-hilbert-scheme}}]
	\label{theorem:connected-components-of-hilbert-scheme}
	Under the identification of $\sch^P(X/S)$ as a subscheme of $\sch(X/S)$ by viewing $\Hilb^P(X/S)$
	as a subfunctor of $\Hilb(X/S)$, the schemes $\sch^P(X/S)$ are precisely the connected components of
	$\sch(X/S)$.
\end{theorem}

\begin{definition}
	\label{definition:fano-scheme}
	In the case that $X \subset \bp^n_S$ is a projective $S$-scheme, with $S$ a field,
	we define the {\bf Fano Scheme}
	of $k$-planes in $X$ to be the connected component of the Hilbert scheme $\sch^P(X/S)$
	where $P$ is the Hilbert polynomial of a $k$-plane in $X$.
\end{definition}
\begin{remark}
	\label{remark:}
	In fact, every closed point of the Fano scheme is a $k$-plane,
	as is shown in \cite[Theorem 28.3.4]{vakil:foundations-of-algebraic-geometry}.
\end{remark}

\begin{definition}
	\label{definition:}
	Let $S$ be a locally Noetherian scheme, let $X$ be a locally projective $S$ scheme, 
	let $\mathbf P = (P_1(t), \ldots, P_m(t)) \in \bq[t]^m$ be an
	$m$-tuple of polynomials,
	and let $\sco_X(1)$ be a very ample invertible sheaf on $X$ used to define $\sch^P(X/S)$. 
	Recall $\scc_S$ denotes the category of locally Noetherian $S$-schemes.
	Define the {\bf flag Hilbert functor}
	\begin{align*}
		\Hilb^{\mathbf P(t)}(X/S) : \scc_S  \rightarrow & \text{ Set }\\
		Z \mapsto &\left\{ \text{closed immersions } X_1 \subset \cdots \subset X_m \hookrightarrow Z \times_S X \right.\\
		& \left. \text{ so that each $X_i$ is flat over $Z$} \right. \\
		& \left. \text{ and has Hilbert polynomial $P_i$}. \right\}
	\end{align*}	
	If the flag Hilbert functor is representable, we call the scheme representing it the {\bf flag Hilbert scheme}
	and, analogously to \autoref{definition:hilbert-scheme}, we define a universal family for the
	flag Hilbert scheme as a collection of closed subschemes
	$W_1 \subset \cdots \subset W_m \xrightarrow \iota X \times \Hilb^{\mathbf P(t)}(X/S)$
	so that the natural isomorphism between the functor of maps to the flag Hilbert scheme and the flag Hilbert functor
	is given by sending $g: Z \rightarrow \Hilb^{\mathbf P(t)}(X/S)$ to
	the collection of $m$ schemes $X_1, \ldots, X_m$, where
	$X_i = (g \times \id)^{-1}(W_i)$.
\end{definition}

\begin{theorem}[\protect{\cite[Theorem 4.5.1]{sernesi:deformations-of-algebraic-schemes}}]
	\label{theorem:}
	The flag Hilbert functor is representable by a projective scheme.	
\end{theorem}

\section{Idiosyncratic notation}

We conclude the background with a collection of some notation idiosyncratic to this thesis.
Much of this notation relates to scrolls, which are defined and discussed in
\autoref{section:basics-of-scrolls}.

\begin{itemize}
	\item When dealing with a scroll $X$ in projective space, we use $d$ to refer to its degree, $k$ to refer to its dimension,
		and $n = d + k-1$ to refer to the dimension of the ambient projective space, $X \subset \bp^n$.
	\item If $X \subset \bp^n$ is a variety so that $[X] \in \sch(\bp^n/ \spec \bk)$ lies in a unique irreducible
		component of the Hilbert scheme, we define $\hilb X$ to be that irreducible component
		and $\uhilb X$ to be the universal family over that component. See
		\autoref{definition:Hilbert-scheme-component} for more details.
	\item If $X$ is a smooth scroll of degree $d$ and dimension $k$, we use $\minhilb d k := \hilb X$.
	\item We use $\scroll {a_1, \ldots, a_k}$ to refer to a smooth scroll of type $a_1, \ldots, a_k$.
	\item Let $X$ be a smooth scroll of degree $d$ and dimension $k$.
		We use $\minhilbsing d k$ to refer to the image under $\pi$ of the singular locus of the map $\pi:\uhilb X \ra \hilb X$.
		That is, $\minhilbsing d k$ consists of points in $\minhilb d k$ which correspond to singular scrolls.
		Similarly, we define $\minhilbsmooth d k$ to be the complement of
		$\minhilbsing d k$ in $\minhilb d k$.
	\item We use $\brokengeneral d k$ to denote the locus of points in $\minhilb d k$ corresponding to
		varieties which are the union of a $k$-plane and a degree $d-1$, dimension $k$ scroll,
		intersecting in a $(k-1)$-plane which is a ruling plane of the degree $d-1$ scroll.
		We define $\broken d k$ to the be closure of $\brokengeneral d k$ in the
		Hilbert scheme.
		See \autoref{definition:least-broken} for more detail.

\end{itemize}

\chapter{Interpolation in general}
\label{section:interpolation-in-general}

\epigraph{Mathematics is the art of giving the same name to different things.
}{\textit{Henri Poincar{\'e} \cite{poincare:the-future-of-mathematics}}}

In this chapter, we present $22$ notions of interpolation and
prove they are all equivalent under mild hypotheses in
 ~\autoref{theorem:equivalent-conditions-of-interpolation}.
The most classical definition of interpolation about
a variety passing through points, or meeting a collection
of planes is
\autoref{interpolation-naive} in
~\autoref{theorem:equivalent-conditions-of-interpolation}.

\section[Definition of interpolation]{Definition and equivalent characterizations of interpolation}

We now lay out the key definitions of interpolation.
First, we describe a more formal way of expressing interpolation
in  ~\autoref{definition:interpolation}.
This comes in two flavors: interpolation and pointed interpolation.
The latter also keeps track of the points at which the planes
meet the given variety. Then, we give a cohomological
definition in ~\autoref{definition:vector-bundle-interpolation}.

\begin{definition}
	\label{definition:Hilbert-scheme-component}
	Let $X \subset \bp^n$ be projective scheme with a fixed embedding into projective space which lies on a unique irreducible component
	of the Hilbert scheme.
	Define $\hilb X$ to be the irreducible component of the Hilbert scheme on which $[X]$ lies, taken with reduced scheme structure.
	If $\sch$ is the Hilbert scheme of closed subschemes of $\bp^n$ over $\spec \bk$ and $\scv$
	is the universal family over $\sch$, then define
	define $\uhilb X$ to be the universal family over $\hilb X$, defined as the
	fiber product
	\begin{equation}
	\nonumber
		\nonumber
		\begin{tikzcd} 
		  \uhilb X \ar{r} \ar{d} & \scv \ar{d} \\
		  \hilb X \ar{r} & \sch.
		\end{tikzcd}\end{equation}
\end{definition}

\begin{definition}
	\label{definition:lambda-constraints}
	Given an integral subscheme of the Hilbert scheme $U$
	parameterizing subschemes of $\bp^n$ of dimension $k$,
	we consider
	sequences
\begin{align*}
	\lambda := \left( \lambda_1, \ldots, \lambda_m \right)
\end{align*}
satisfying the following conditions:
\begin{enumerate}
	\item $\lambda$ is a weakly decreasing sequence.
		That is, $\lambda_1 \geq \lambda_2 \geq \cdots \geq \lambda_m$,
	\item for all $1 \leq i \leq m$, we have $0 \leq \lambda_i \leq n - k$,
	\item and
\begin{align*}
	\sum_{i=1}^m \lambda_i \leq \dim U.
\end{align*}
\end{enumerate}
\end{definition}

\begin{definition}
	\label{definition:interpolation}	
	Let $U$ be an integral subscheme of the Hilbert scheme
	parameterizing subschemes of $\bp^n$ of dimension $k$
	and let $\scv(U)$ denote the universal family over $U$.
	Let $\lambda$ be as in ~\autoref{definition:lambda-constraints}, let $p \in \bp^n$ be a point, and let
	$\Lambda_i$ be a plane of dimension $n - k - \lambda_i$ for $1 \leq i \leq m$.
Define
\begin{align*}
	\Psi_\lambda := \left( \uhilb {\Lambda_1} \times_{\bp^n} \scv(U) \right) \times_{U} \cdots \times_{U} \left( \uhilb {\Lambda_m} \times_{\bp^n} \scv(U) \right).
\end{align*}
Then, since $\hilb {\Lambda_i} \cong Gr(n- k - \lambda_i + 1, n+1)$,
we obtain that $\Psi_\lambda$ is a closed subscheme of
$U \times \prod_{i=1}^m Gr(n - k - \lambda_i + 1, n + 1) \times (\bp^n)^m $
(see \autoref{remark:interpolation-as-incidence-correspondence} for one viewpoint as to what this
inclusion is).
Define $\Phi_\lambda$ to be the image of the composition
\begin{equation}
	\nonumber
	\begin{tikzcd} 
	  \Psi_\lambda \ar{r} & U \times \prod_{i=1}^m Gr(n-k - \lambda_i + 1, n + 1) \times (\bp^n)^m \ar{d} \\
		 & U \times \prod_{i=1}^m Gr(n - k - \lambda_i + 1, n + 1).
	\end{tikzcd}\end{equation}
We have natural projections
\begin{equation}
	\nonumber
	\begin{tikzcd}
		\qquad & \Phi_\lambda \ar {ld}{\pi_1} \ar {rd}{\pi_2} & \\
		U  &&  \prod_{i=1}^m Gr(n- k - \lambda_i +1, n+1)
	 \end{tikzcd}\end{equation}
and
\begin{equation}
	\nonumber
	\begin{tikzcd}
		\qquad & \Psi_\lambda \ar {ld}{\eta_1} \ar {rd}{\eta_2} & \\
		U  &&  \prod_{i=1}^m Gr(n - k - \lambda_i +1, n+1).
	 \end{tikzcd}\end{equation}

Define $q$ and $r$ so that $\dim U = q \cdot (n - k) + r$ with $0 \leq r < n - k$. Then,
$U$ satisfies
\begin{enumerate}
	\item {\bf $\lambda$-interpolation}
if the projection map $\pi_2$ is surjective,
	\item {\bf weak interpolation} if $U$ satisfies $((n - k)^q)$-interpolation,
	\item {\bf interpolation} if $U$ satisfies $((n- k)^q, r)$-interpolation,
	\item {\bf strong interpolation}
if $U$ satisfies $\lambda$-interpolation for all $\lambda$
as in \autoref{definition:lambda-constraints}.
\end{enumerate}

We define  $\lambda$-pointed interpolation, weak pointed interpolation, pointed interpolation, strong pointed interpolation similarly. More precisely, we say that $U$ satisfies
\begin{enumerate}
	\item {\bf $\lambda$-pointed interpolation} if $\eta_2$
is surjective,
\item {\bf weak pointed interpolation} if $U$ satisfies
	$\left( (n- k)^q \right)$-pointed interpolation,
\item {\bf pointed interpolation} if $U$ satisfies
$\left( (n- k)^q, r \right)$-pointed interpolation,
\item {\bf strong pointed interpolation} if $X$ satisfies
$\lambda$-pointed interpolation for all $\lambda$
as in \autoref{definition:lambda-constraints}.
\end{enumerate}

If $X \subset \bp^n$ lies on a unique irreducible
component of the Hilbert scheme $\hilb X$, we say $X$ satisfies $\lambda$-interpolation (and all variants as
above) if $\hilb X$ satisfies $\lambda$-interpolation.
When $\lambda$ is clear from context, we often refer to
$\Psi_\lambda$ and $\Phi_\lambda$ as $\Psi$ and $\Phi$.
\end{definition}

\begin{remark}
	\label{remark:interpolation-as-incidence-correspondence}
	Those who prefer incidence correspondences to fiber products
	may appreciate the following description of the reduced schemes $\Phi_\lambda$ and $\Psi_\lambda$.
	
	The closed points of $\Phi_\lambda$, without describing a scheme structure, can be written as
\begin{align*}
	\left\{ \left( [Y], \Lambda_1, \ldots, \Lambda_m \right) \subset U \times \prod_{i=1}^m Gr(n - k - \lambda_i +1, n+1): \Lambda_i \cap Y \neq \emptyset \right\}.
\end{align*}
Similarly, the closed points of $\Psi_\lambda$, without describing a scheme structure, can be written as
\begin{align*}
	&\left\{ \left( [Y], \Lambda_1, \ldots, \Lambda_m , p_1, \ldots, p_m \right)  \right. \\ 
& \qquad \subset U \times \prod_{i=1}^m Gr(n-k - \lambda_i +1, n+1) \times (\bp^n)^m: \\
& \qquad \left. p_i \in \Lambda_i, p_i \in Y \right\}.
\end{align*}
	Note that although this definition as an incidence correspondence may be less opaque,
	for many of the later proofs, it will greatly help to work with the definition of
	$\Phi_\lambda$ and $\Psi_\lambda$ as fiber products. Further, the definition given in terms of
	fiber products yields a natural scheme structure, which may well not be
	the reduced one.
\end{remark}

\begin{remark}
	\label{remark:}
	Although it is not anywhere in the literature, Joe
	Harris likes to say a variety is ``flexible''
	if it satisfies the notion of interpolation
	defined in \autoref{definition:interpolation},
	because the variety can be thought to flexibly bend
	so as to meet
	the linear spaces $\Lambda_i$.
\end{remark}

\begin{definition}[Interpolation of locally free sheaves, see Definition 3.1 and 3.3 of ~\cite{atanasov:interpolation-and-vector-bundles-on-curves}]
	\label{definition:vector-bundle-interpolation}
	Let $\lambda$ be as in \autoref{definition:lambda-constraints} and let $E$ be a locally free sheaf on a scheme $X$
	with $H^1(X, E) = 0$.
	Choose points $p_1, \ldots, p_m$ on $X$ and
	vector subspaces $V_i \subset E|_{p_i}$ for $1 \leq i \leq m$
	with $\dim V_i = \lambda_i$.
	Then, define $E'$ as the kernel of the natural quotient
	\begin{equation}
		\label{equation:sequence-cohomological-interpolation-definition}
		\begin{tikzcd}
			0 \ar{r} & E' \ar{r} & E \ar{r} & \oplus_{i=1}^m E|_{p_i}/V_i \ar{r} & 0.
		\end{tikzcd}\end{equation}
	We say $E$ satisfies {\bf $\lambda$-interpolation}
	if there exist points $p_1, \ldots, p_n$ as above
	and subspaces $V_i \subset E|_{p_i}$ as above
	so that
	\begin{align*}
		h^0(E) - h^0(E') = \sum_{i=1}^m \lambda_i.
	\end{align*}

	Write $h^0(E) = q \cdot \rk E + r$ with $0 \leq r < \rk E$.
	We say $E$ satisfies
	\begin{enumerate}
		\item {\bf weak interpolation} if it satisfies
			$\left( (\rk E)^q \right)$
			interpolation,
		\item {\bf interpolation} if it satisfies
			$\left( (\rk E)^q, r \right)$ interpolation,
		\item {\bf strong interpolation} if it satisfies
			$\lambda$-interpolation for all
			admissible $\lambda$ as in 
			\autoref{definition:lambda-constraints}.
	\end{enumerate}
\end{definition}

\begin{remark}
	\label{remark:}
	See ~\cite[Section 4]{atanasovLY:interpolation-for-normal-bundles-of-general-curves} for further useful properties of interpolation.
	While some of the discussion there is specific
	to curves, much of it generalizes immediately to higher
	dimensional varieties. 
\end{remark}

We now come to the main result of the chapter. 
Because it has so many moving parts, after stating
it, we postpone its proof until
\autoref{ssec:proof-of-equivalent-conditions-of-interpolation},
once we have developed the tools necessary to prove it.

Perhaps the most nontrivial consequence of \autoref{theorem:equivalent-conditions-of-interpolation}
is that it implies the equivalence of interpolation and strong interpolation
for $\hilb X$ when $X$ is a smooth projective scheme with $H^1(X, N_{X/\bp^n}) = 0$,
over an algebraically
closed field of characteristic $0$.

\begin{theorem}
	\label{theorem:equivalent-conditions-of-interpolation}
	For the remainder of this theorem, assume $X \subset \bp^n$ is an integral projective scheme
	lying on a unique irreducible component of the Hilbert scheme.
	Write $\dim \hilb X = q \cdot \codim X + r$ with $0 \leq r < \codim X$.
	The following are equivalent:
	\begin{enumerate}
		\item[\customlabel{interpolation-definition}{(1)}]$\hilb X$ satisfies interpolation.
		\item[\customlabel{interpolation-pointed}{(2)}] $\hilb X$ satisfies pointed interpolation.
		\item[\customlabel{interpolation-dominant}{(3)}] The map $\pi_2$ given in 
			\autoref{definition:interpolation} 
			for $\lambda = (\left(\codim X \right)^q, r)$
			is dominant.
		\item[\customlabel{interpolation-finite}{(4)}] The map $\pi_2$ given in 
			\autoref{definition:interpolation} 
			for $\lambda = (\left(\codim X \right)^q, r)$
			is
			generically finite.
		\item[\customlabel{interpolation-isolated}{(5)}] The scheme $\Phi_\lambda$ defined in 
			\autoref{definition:interpolation}
			for $\lambda = (\left(\codim X \right)^q, r)$
			has a closed point $x$ which is isolated in
			its fiber $\pi_2^{-1}(\pi_2(x))$.
	\item[\customlabel{interpolation-pointed-dominant}{(6)}] The map $\eta_2$ given in 
			\autoref{definition:interpolation} 
			for $\lambda = (\left(\codim X \right)^q, r)$
			is dominant.
		\item[\customlabel{interpolation-pointed-finite}{(7)}] The map $\eta_2$ given in 
			\autoref{definition:interpolation} 
			for $\lambda = (\left(\codim X \right)^q, r)$
			is
			generically finite.
		\item[\customlabel{interpolation-pointed-isolated}{(8)}] The scheme $\Psi_\lambda$ defined in 
			\autoref{definition:interpolation}
			for $\lambda = (\left(\codim X \right)^q, r)$
			has a closed point $x$ which is isolated in
			its fiber $\eta_2^{-1}(\eta_2(x))$.
		\item[\customlabel{interpolation-naive}{(9)}] Given any set of $q$ points in $\bp^n$ and an $(\codim X - r)$-dimensional
			plane $\Lambda \subset \bp^n$, there exists
			an element $[Y] \in \hilb X$ so that $Y$ contains those points and
			meets $\Lambda$.
		\item[\customlabel{interpolation-sweep}{(10)}] Given a set of $q$ points in $\bp^n$,
			the subscheme of $\bp^n$ swept out by varieties
			of $\hilb X$ containing those points is
			$\dim X + r$ dimensional.
	\end{enumerate}
	Secondly, the following statements are equivalent:
	\begin{enumerate}[(i)]
		\item[\customlabel{strong-definition}{(i)}] $\hilb X$ satisfies strong interpolation.
		\item[\customlabel{strong-equality}{(ii)}] $\hilb X$ satisfies $\lambda$-interpolation for all $\lambda$ with $\sum_{i=1}^m \lambda_i = \dim \hilb X$.

		\item[\customlabel{strong-pointed}{(iii)}] $\hilb X$ satisfies strong pointed interpolation.
		\item[\customlabel{strong-pointed-equality}{(iv)}] $\hilb X$ satisfies $\lambda$-pointed interpolation for all $\lambda$ with $\sum_{i=1}^m \lambda_i = \dim \hilb X$.
		\item[\customlabel{strong-naive}{(v)}] Given any collection of planes $\Lambda_1, \ldots, \Lambda_m$ satisfying the conditions given in 
			\autoref{definition:lambda-constraints},
			there is some $[Y] \in \hilb X$ meeting
			all of $\Lambda_1, \ldots, \Lambda_m$.
		\item[\customlabel{strong-naive-equality}{(vi)}] Given any collection of planes $\Lambda_1, \ldots, \Lambda_m$ satisfying the conditions given in 
			\autoref{definition:lambda-constraints},
			with $\sum_{i=1}^m \lambda_i = \dim \hilb X$,
			there is some $[Y] \in \hilb X$ meeting
			all of $\Lambda_1, \ldots, \Lambda_m$.
	\end{enumerate}
Also, \ref{strong-definition}-\ref{strong-naive-equality}
imply
\ref{interpolation-definition}-\ref{interpolation-sweep}.
Thirdly, 
further assume $H^1(X, N_X) = 0$.
Then, the following properties are equivalent:
	\begin{enumerate}[(a)]
		\item[\customlabel{cohomological-definition}{(a)}] The sheaf $N_{X/\bp^n}$ satisfies interpolation.
		\item[\customlabel{cohomological-restatement}{(b)}] There is a subsheaf $E' \rightarrow N_{X/\bp^n}$ 
whose cokernel is supported at $q+1$ points if $r > 0$ and $q$ points if $r = 0$,
so that the scheme theoretic support at $q$ of these points has
dimension equal to $\rk N_{X/\bp^n}$
and $H^0(X, E') = H^1(X, E') = 0$.
		\item[\customlabel{cohomological-strong}{(c)}] The sheaf $N_{X/\bp^n}$ satisfies strong interpolation.
		\item[\customlabel{cohomological-sections}{(d)}] For every $d \geq 1$, there exist points
			$p_1, \ldots, p_d \in X$ so that
			\begin{align*}
				\dim H^0(X, N_{X/\bp^n} \otimes \sci_{p_1, \ldots, p_d}) = \max\left\{ 0, h^0(X, N_{X/\bp^n}) - dn \right\}
			\end{align*}
			(cf. ~\cite[Definition 4.1]{atanasovLY:interpolation-for-normal-bundles-of-general-curves}).
		\item[\customlabel{cohomological-vanish}{(e)}] For every $d \geq 1$, a general collection of points
			$p_1, \ldots, p_d$ in $X$ satisfies either
			\begin{align*}
				h^0(X, N_{X/\bp^n} \otimes \sci_{p_1, \ldots, p_d}) && \text{ or } && h^1(X, N_{X/\bp^n}\otimes \sci_{p_1, \ldots, p_d}) = 0.
			\end{align*}
			(cf. ~\cite[Proposition 4.5]{atanasovLY:interpolation-for-normal-bundles-of-general-curves}).
		\item[\customlabel{cohomological-boundary}{(f)}] 			A general set of $q$ points
			$p_1, \ldots, p_q$
			satisfy $h^1(X, N_{X/\bp^n} \otimes \sci_{p_1, \ldots, p_q}) = 0$
			and a general set of $q+1$ points
			$q_1, \ldots, q_{q+1}$ satisfy
			$h^0(X, N_{X/\bp^n} \otimes \sci_{p_1, \ldots, p_q}) = 0$
			(cf. ~\cite[Proposition 4.6]{atanasovLY:interpolation-for-normal-bundles-of-general-curves}).
	\end{enumerate}
Additionally, 
retaining the assumptions that $H^1(X, N_X) = 0$ 
	and $X$ is a local complete intersection, 
	and further assuming $X$ is generically smooth,	
	the equivalent conditions
\ref{cohomological-definition}-\ref{cohomological-boundary} imply the equivalent conditions
\ref{interpolation-definition}-\ref{interpolation-sweep} and the equivalent conditions \ref{cohomological-definition}-\ref{cohomological-boundary} imply the equivalent conditions \ref{strong-definition}-\ref{strong-naive-equality}.

Finally, still 
retaining the assumptions that $H^1(X, N_X) = 0$ 
	and that $X$ is a local complete intersection, in the case that $\bk$ has characteristic $0$,
all statements \ref{interpolation-definition}-\ref{interpolation-sweep}, \ref{strong-definition}-\ref{strong-naive-equality}, \ref{cohomological-definition}-\ref{cohomological-boundary} are equivalent.
\end{theorem}

\begin{remark}
	\label{remark:}
	Note that the equivalence of all conditions above
	holding in characteristic $0$ does not hold in
	characteristic $2$.
	That is, 
  \ref{interpolation-pointed}
    does not imply
  \ref{cohomological-definition}
    in characteristic $2$.

	An example of a component of the Hilbert scheme which satisfies
        \ref{interpolation-pointed}
	but not
        \ref{cohomological-definition}
	is provided by the irreducible
	component of the Hilbert scheme whose general member
	is a $2$-Veronese surface, as shown in
	\autoref{corollary:failure-of-2-veronese-vector-bundle-interpolation}.
\end{remark}

\begin{remark}[Weakening hypotheses of \autoref{theorem:equivalent-conditions-of-interpolation}]
	\label{remark:weakening-interpolation-conditions}
	Heuristically, the condition that 
	$H^1(X, N_{X/\bp^n}) = 0$ and $X$
	is a local complete intersection
	is not too much of an imposition,
	because it gives a relatively easy way of checking
	that $[X]$ is a smooth point of the Hilbert scheme,
	and in general, it is difficult to check whether
	$[X]$ is a smooth point of the Hilbert scheme
	when $H^1(X, N_{X/\bp^n}) \neq 0$.

	Next, we make a comment about the hypothesis
	that $X$ be integral. This comes in two parts:
	the irreducibility and the reducedness of $X$.
	
	First,
	the assumption that $X$ is irreducible can be
	done away with, by introducing some further technical
	baggage. 
	We will want to say a certain Hilbert scheme satisfies
	interpolation if we can find an element
	of that scheme meeting a general ``expected number''
	of points, and meeting one more plane of the
	``expected dimension.''
	We can also rephrase this
	in terms of a projection map from an incidence correspondence
	being surjective.
	If we take the naive generalization, we will
	have troubles when attempting to prove that the incidence
	correspondence is irreducible. To get
	around this issue, we can define a certain symmetric power
	of the Hilbert scheme, together with a choice of component.
	Informally, this will correspond to determining
	the distribution of the number of points among each component.
	Then, the resulting Hilbert scheme will satisfy interpolation
	if this map from a symmetric power of the Hilbert scheme with
	a choice of component is surjective.

	On the other hand, we do not see a way to loosen
	the reducedness hypothesis. 
	Reducedness is useful for knowing the equivalence
	of conditions \ref{interpolation-definition} and \ref{interpolation-pointed}, whose proof crucially
	depends on the general member of $\hilb X$ being
	irreducible. In order to prove that a general member
	of $\hilb X$ is irreducible, we need a result like
	upper semicontinuity of the number of irreducible components.
	This holds when every fiber is reduced, using \autoref{proposition:upper-semicontinuous-number-of-components}, but 
	may be false when every fiber is
	is nonreduced.
\end{remark}

\begin{remark}
	\label{remark:interpolation-strategy}
	In general, to show a certain variety satisfies interpolation there will be two steps.
	The harder step will be to show there exists some variety in the Hilbert scheme passing through the given set of points
	and linear spaces.
	However, there will typically also be an easier step, in which we have to check that if there is one such variety then
	there are only finitely many. 
	We will often set this problem up in the following fashion. We will have some sort of incidence correspondence $X$
	with projections
	\begin{equation}
		\nonumber
		\begin{tikzcd}
			\qquad & X \ar {ld}{\pi_1} \ar {rd}{\pi_2} & \\
			Y && Z.
		\end{tikzcd}\end{equation}
	Typically, $Y$ will be a component of the Hilbert scheme, $X$ will be some sort of incidence correspondence,
	and $Z$ will be some product of Grassmannians or other parameter space, as in \autoref{definition:interpolation}.
	For example, if we are asking that a twisted cubic contain six general points, then $Y$ will be the irreducible
	component of the Hilbert scheme whose general member is a smooth twisted cubic,
	$X$ will be the fiber product of six copies of the universal family over $Y$, and $Z$ will be $(\bp^3)^6$.
	We consider the following two approaches to showing there are only finitely many elements in a general fiber of $\pi_2$.
	\begin{enumerate}
		\item[\customlabel{technique-dimension}{technique-1}] Show that $\dim X = \dim Z$
		\item[\customlabel{technique-conditions}{technique-2}] Assume $\pi_2$ is generically finite.
			``Count conditions on objects of $Y$''
			by showing that,
			for a general point $p$ in $Z$, $\pi_2^{-1}(p)$ is a finite set of points of $Y$.
	\end{enumerate}
	Both of these methods imply that $Y$ satisfies interpolation, assuming $\pi_2$ is generically finite. In the case of \ref{technique-dimension}, if we know $X$ and $Z$ are of the same dimension, and are generically finite, then $\pi_2$ must be dominant.
	The case of \ref{technique-conditions} means that for a general point $p \in Z$, corresponding to some conditions which are being imposed on the Hilbert scheme $Y$, there
	is only a finite number of elements of the Hilbert scheme satisfying those conditions. This is the reason that we call this property
	interpolation: we can ``interpolate'' varieties in the Hilbert scheme through a given number of points. Assuming that $\pi_2$ is generically
	finite and dominant, this is equivalent
	to the given definition of interpolation using \autoref{lemma:linear-condition-counting}, because
	generic finiteness of $\pi_2$ implies that a general fiber of $\pi_2$ is a finite set of points, and so must be mapped
	generically finitely under $\pi_1$.

	While the approach of \ref{technique-dimension} may seem more straightforward, it has the serious
	disadvantage of being rather opaque, since it will often appear quite
	mysterious as to how one guessed the correct number of conditions. Returning to our example of degree three rational normal curves:
	How did we come up with the condition that we should pass the curve through 6 points? When we count the dimensions of $X$ and $Y$,
	we find that they are both equal to $18$, but it would be quite annoying to have to set up these incidence correspondences and
	count the dimension of everything, each time we wanted to find the correct number of points or other conditions.

	Fortunately, \ref{technique-conditions} yields a more efficient method of counting the finding the correct number of conditions to impose.
	This method can be described in more generality, but for clarity of exposition, we
	we now just explain how to use \ref{technique-conditions} to find the number of point conditions
	to impose on $X$ for weak interpolation. That is, we explain
	an efficient way to find the number $r$, defined in \autoref{definition:interpolation} by
	$\dim \hilb X = r \codim X + q$, where $X$ is a variety in projective space.
	
	The idea is that, under the assumption that the second projection is dominant, it will follow that a codimension $\codim X$
	locus of the Hilbert scheme will pass through a single fixed point in projective space, where $X$ is a projective variety.
	Then, the codimension of varieties passing through $t$ points is just $t\cdot \codim X$.
	We then want to choose the correct number of points so that this codimension as close as possible to the dimension of the Hilbert scheme,
	without going over to achieve weak interpolation.
	
	We are keeping this description intentionally vague as we will use this technique in a wide variety of situations.
	However, the general technique is quite well illustrated by examining the particular example of twisted cubics, as is done
	in \autoref{example:twisted-cubic-interpolation-dimension-count}.
\end{remark}
	\begin{example}
		\label{example:twisted-cubic-interpolation-dimension-count}
		Let's see \ref{technique-conditions} in action, in the case that $Y$ is the irreducible component of the Hilbert scheme
		whose general member is a smooth twisted cubic.	
		First, $Y$ is $12$ dimensional because 
		for $[C] \in Y$ a smooth twisted cubic,
		$H^0(C, N_{C/\bp^3}) = 12, H^1(C, N_{C/\bp^3}) = 0$.
		To start, we will find the dimensional of varieties corresponding to points of $Y$, passing through a single fixed point.
		To do this, we take 
	$Y$ to be the Hilbert scheme, $X$ to be the universal family over the Hilbert scheme, and $Z$ to be $\bp^3$. Note that $\bp^3$,
	is three dimensional, so the preimage of a general point in $\bp^3$ under $\pi_2$ will be a codimension $3$ subset of $Z$.
	Then, the relative dimension of $\pi_1$ is $1$, as each fiber is a curve. So, by 
	\autoref{lemma:linear-condition-counting},
	the codimension of $\pi_1(\pi_2^{-1}(p))$ for a general point $p$ will be
	$3 - 1 = 2$, in the case that $\pi_1$ is generically finite and that $\pi_2$ is dominant. 
	Heuristically, we say that a point ``imposes 2 conditions.'' Therefore, $6$ points ``impose
	$2 \cdot 6 = 12$ conditions.'' We also say that the ``expected dimension'' of curves passing through $6$ points is $0$.
	The reason for the word ``expected`` it may not be the correct dimension when the map $\pi_2$ is not dominant.

	A crucial detail of this argument is that if $\pi_2$ is dominant, the hypothesis of ~\autoref{lemma:linear-condition-counting}
	that $\pi_1$ be generically finite when restricted to a fiber of $\pi_2$ is automatically satisfied. The reason for this is that the preimage under $\pi_2$
	of a general point will be supported at a finite set of points. Since the image of a finite set of points is always a finite
	set of points, $\pi_1$, restricted to such a fiber, will indeed be generically finite.

	Just to spell things out in a bit more detail, let us
	now explain why we can multiply the number of conditions by $t$
	when we ask that the curve pass through $t$ points.
	Take $X_t, Y_t, Z_t, \pi_{1,t}, \pi_{2,t}$ to be the corresponding schemes
	and maps for $t$ points. For the case of interpolating
	through $t$ points, we will still take $Y_t$ to be the irreducible component of the Hilbert scheme, but we will take
	$X_t$ to be the $t$-fold fiber product of the universal family over the Hilbert scheme and $Z_t = (\bp^3)^t$. That is,
	we take $Y_t := \uhilb X \times_{\hilb X}\cdots \times_{\hilb X}\uhilb X$, with $\uhilb X$ appearing
	$t$ times.
	The relative dimension of $\pi_{1,t}$ is then $t$ times the relative dimension of $\pi_{1,1}: \uhilb X \ra \hilb X$,
	while the dimension of $Z_t$ is $t$ times the dimension of $Z_1$ for one point.
	So, by \autoref{lemma:linear-condition-counting}, assuming $\pi_{2,t}$ is dominant and the restriction of $\pi_{1,t}$ to a general fiber
	of $\pi_{2,t}$ is generically finite,
	for a general $p \in Z_t$, $\dim \pi_{2,t}^{-1}(\pi_{1,t}(p))$
	will be $t$ times $\dim \pi_{2,1}^{-1}(\pi_{1,1}(q))$ for a general point $q \in Z_1$.
	\end{example}

	\section[Tools for irreducibility]{Tools for irreducibility of incidence correspondences}
\label{ssec:tools-irreducibility}

A key ingredient for establishing the equivalence of conditions
\ref{interpolation-definition}-\ref{interpolation-sweep}
is the irreducibility of the incidence correspondences $\Phi, \Psi$ of 
\autoref{definition:interpolation}.
This is important to establish that the map
$\pi_2$ of \autoref{definition:interpolation} is surjective if and only if it is dominant if and only
if it is generically finite if and only if it has a point
in an isolated fiber.
Our goal for this section is to prove 
\autoref{lemma:irreducible-and-dimension}.

We start with a general Lemma relating irreducibility of fibers
to irreducibility of the source.

\begin{lemma}
	\label{lemma:irreducible-base-and-fibers}
	Suppose $\pi:X \ra Y$ is a map of separated finite type schemes over $\spec \bk$ with $Y$ irreducible. If all fibers of $\pi$ are irreducible of dimension $\delta$
	then $X$ is irreducible
	and $\dim X = \dim Y + \delta$.
	Further, if $\pi$ is flat, $Y$ is irreducible, and there is a nonempty open set $U \subset Y$ so that $\pi^{-1}(u)$ is irreducible of dimension $\delta$ for all closed points $u \in U$, then
	$X$ is irreducible of dimension $\dim X + \dim Y + \delta$.
\end{lemma}
\begin{proof}
	First, we show that if the map $\pi$ has irreducible fibers and $Y$ is irreducible then $X$ is as well.
	Indeed, since irreducibility is a topological property, we may give $X$ and $Y$ the reduced scheme structures $X_{red}$ and $Y_{red}$.
	Then, the fact that $X$ is irreducible is precisely \cite[Exercise 11.4.C]{vakil:foundations-of-algebraic-geometry}.
	The statement on dimension holds due to generic flatness: we can find an open set $U$ over which the map is flat by
	\cite[Exercise 24.5.N]{vakil:foundations-of-algebraic-geometry}. We know that $\dim \pi^{-1}(U) = \dim U + \delta$
	by \cite[Proposition 24.5.5]{vakil:foundations-of-algebraic-geometry}. This gives a lower bound for the dimension of $X$.
	This is also an upper bound by \cite[Exercise 11.4.A]{vakil:foundations-of-algebraic-geometry}.

	Next, we show that when $\pi$ is flat, $Y$ is irreducible, and the fibers are generically irreducible,
	then $X$ is irreducible. First, observe that $\pi^{-1}(U)$ is irreducible by the first part of this lemma.
	Now, suppose $X$ has two components $A$ and $B$. Since $X$ is a scheme of finite type over a field, both $A$ and $B$ must have
	a closed points not contained in the other. 
	So, up to permutation of $A$ and $B$, we may assume that $\pi^{-1}(U) \subset A$. Now, let $V : = B \setminus A$ be a nonempty open set.
	By \cite[Exercise 24.5.G]{vakil:foundations-of-algebraic-geometry}, the map $\pi$ is open. However, $\pi(V) \subset Y \setminus U$.
	Since $Y$ is irreducible, $\pi(V)$ can only be open in $Y$ if it is empty, a contradiction. Therefore, $X$ must be irreducible.
	The statement on dimension follows from \cite[Proposition 24.5.5]{vakil:foundations-of-algebraic-geometry}.
\end{proof}

Using \autoref{lemma:irreducible-base-and-fibers},
we can now prove in \autoref{proposition:irreducible-incidence} irreducibility of $\Phi, \Psi$ from \autoref{definition:interpolation}.

\begin{proposition}
	\label{proposition:irreducible-incidence}
	Let $X \subset \bp^n$ be a variety so that the general member
	of $\hilb X$ corresponds to an irreducible variety.
	Let $\lambda, \Phi_\lambda, \Psi_\lambda$ be as in \autoref{definition:interpolation}.
	Then, $\dim \Phi_\lambda = \dim \Psi_\lambda$ and both $\Phi_\lambda$ and $\Psi_\lambda$ are irreducible.
\end{proposition}
\begin{proof}
We will prove this in the case that $m = 1$ as
the general case is completely analogous. We write
$\lambda := \lambda_1, \Lambda := \Lambda_1, p := p_1, \Phi := \Phi_\lambda, \Psi := \Psi_\lambda$ for notational
convenience.
Observe that we have a commutative diagram of natural projections
\begin{equation}
	\nonumber
	\begin{tikzcd} 
		\qquad & \Psi \ar {ld}{\mu_1} \ar {rd}{\mu_2} & \\
		\uhilb X \ar {rd}{\mu_3} & & \Phi \ar{ld}{\mu_4} \\
		\qquad & \hilb X &
	\end{tikzcd}\end{equation}
Note that $\Psi = \uhilb \Lambda \times_{\bp^n} \uhilb X$ and $\mu_1$ is simply the second projection map.
Observe that since the map $\mu_2$ is surjective, once
we know $\Psi$ is irreducible, $\Phi$ will be too.

Next, we check that
$\dim \Psi = \dim \Phi$.
Note that if we take the point $(Y, \Lambda)$ in $\Phi$ chosen
so that $\Lambda$ meets $Y$ at finitely many points, the fiber
of $\mu_2$ over that point is necessarily $0$ dimensional.
By upper semicontinuity of fiber dimension for proper maps, there is an open
set of $\Psi$ on which the fiber is $0$ dimensional,
implying they map is generically finite, so $\dim \Phi = \dim \Psi$.

To complete the proof, we need only show $\Psi$ is irreducible.
Note that the map $\mu_3$ is flat. 
The assumption that the general member of $\hilb X$ is irreducible
precisely says that the general fiber of $\mu_3$ is irreducible.
So, by
\autoref{lemma:irreducible-base-and-fibers},
$\uhilb X$ is irreducible.
If we knew that $\Psi$ were a Grassmannian bundle over $\uhilb X$, we would then
obtain that $\Psi$ is also irreducible. That
$\Psi$ is a Grassmannian bundle over $\uhilb X$ follows from \autoref{lemma:grassmannian-bundle-over-universal-family}.
\end{proof}
\begin{lemma}
	\label{lemma:grassmannian-bundle-over-universal-family}
	Let $\Psi_{\lambda_1}$ be as in \autoref{definition:interpolation} for the single element
	partition $\lambda_1$. 	
	Then, the projection map
	\begin{align*}
		 \Psi & \rightarrow \uhilb X \\
		 (X, \Lambda, p) & \mapsto (X, p)
	\end{align*}
realizes $\Psi$ as a Grassmannian bundle over $\uhilb X$.
\end{lemma}
\begin{proof}
Note that we have a fiber square
\begin{equation}
	\nonumber
	\begin{tikzcd} 
		\Psi \ar {r} \ar {d} &  \uhilb \Lambda \ar {d} \\
		\uhilb X \ar {r} & \bp^n.
	\end{tikzcd}\end{equation}
To show the left vertical map is a Grassmannian bundle, it suffices to show the right vertical map is a Grassmannian bundle.
This follows from \autoref{exercise:grassmannian-bundle-planes}.
\begin{exercise}
	\label{exercise:grassmannian-bundle-planes}
		The map
	\begin{align*}
		\left\{ \left( \Lambda, p \right) \subset Gr(\lambda, n+1) \times \bp^n: p \in \Lambda \right\} &\ra \bp^n \\
		(\Lambda, p) & \mapsto p
	\end{align*}
	realizes the source as a Grassmannian bundle over $\bp^n$.
	{\it Hint:} Show that over each standard affine open
	chart in $\bp^n$, this bundle is locally trivial.
	Possibly do this by describing an open covering of the
	Grassmannian on each such open set.
\end{exercise}
\end{proof}

We can almost apply \autoref{proposition:irreducible-incidence}
to prove \autoref{lemma:irreducible-and-dimension}.
However, in order to apply \autoref{proposition:irreducible-incidence}, we will have to know that the general member of $\hilb X$
is irreducible only from the information that $X$ itself is integral.
This is why we need the following \autoref{proposition:upper-semicontinuous-number-of-components}.

\begin{proposition}
	\label{proposition:upper-semicontinuous-number-of-components}
	Let $f: X \ra Y$ be a flat proper map of finite type schemes over an algebraically closed field
	so that the fibers over the closed points of $Y$ are geometrically reduced. Then, the number of irreducible components of the geometric fiber of a point in $Y$
	is upper semicontinuous on $Y$.
\end{proposition}
This proof is that outlined in nfdc23's comments in ~\cite{MO:why-is-the-number-of-irreducible-components-upper-semicontinuous-in-nice-situations}.

\begin{proof}
	First, we reduce to the case that $Y$ is a discrete valuation ring.
	By \cite[Corollaire 9.7.9]{EGAIV.3}, the set of points in $Y$ with
	a given number $n\in \bz$ of geometrically irreducible components
	is a constructible subset of $Y$. 
	This implies that the set of points in $Y$ with at least $n$ geometrically irreducible
	components is constructible as follows.
	The set of points in $Y$ with $j$ geometrically irreducible components with $j < n$ is 
	constructible. Therefore the set of points with less than $n$ geometrically
	irreducible components is constructible, as it is a finite union of constructible
	sets. Finally, the set of points in $Y$ with at least $n$ geometrically irreducible
	components is constructible as it is the complement of a constructible set. Then, in order to show a constructible set is closed, it suffices
	to show it is closed under specialization by
	\cite[Exercise 7.4.C(a)]{vakil:foundations-of-algebraic-geometry}.

	Next, suppose $\eta \in Y, x \in \overline \eta$ are two points in $Y$.
	In order to show upper semicontinuity, it suffices to show that
	the number of geometrically irreducible components in the fiber over $x$ is
	at least as big as the number of geometrically irreducible components in
	the fiber over $\eta$.
	There exists a discrete valuation ring $R$ so that there is a map
	$\phi:\spec R \ra Y$ sending the generic point of $R$ to $\eta$ and the
	closed point to $x$ by \cite[Lemma 27.5.10]{stacks-project}.
	Because the fiber over a point $\spec \bk \ra \spec R$ of $\spec R \times_Y X \ra \spec R$
	is isomorphic to the corresponding point $\spec \bk \ra \spec R \ra X$ of $X \ra Y$,
	in order to show the fiber over $x$ has at least as many geometrically irreducible
	components as the fiber over $\eta$, it suffices to show the same statement for
	the preimage of $x,\eta$ in $\spec R$. 
	
	Hence, we have reduced to the case $Y = \spec R$ for
	$R$ a discrete valuation ring. We now
	demonstrate the proposition in this case.
	First, by \cite[Th\'eor\`eme 12.2.4(v)]{EGAIV.3}, the set of points in $Y$ with geometrically reduced
fiber is open. By assumption the fiber over the closed point of $Y$ is geometrically reduced,
and so the fiber over the generic point of $Y$ is also geometrically reduced.
In particular, both these fibers have no embedded points, as they are reduced.

Then, by \cite[Th\'eor\`eme 12.2.4(ix)]{EGAIV.3}, since the fibers of $f$ over both points of $Y$
are reduced and have no embedded points, we obtain that the total multiplicity,
as defined in \cite[D\'efinition 4.7.4]{EGAIV.2},
is upper semicontinuous. Since the scheme is reduced,
the total multiplicity is equal to the number of irreducible components.
In other words, the number of irreducible components over 
the generic point of $Y$ is at most the number of irreducible components over the closed point of $Y$.

Therefore, the number of irreducible components is upper semicontinuous on the target.
\end{proof}

\begin{lemma}
	\label{lemma:irreducible-and-dimension}
	Suppose $X$ is an integral scheme. Then, $\Phi_\lambda, \Psi_\lambda$
	as defined in \autoref{definition:interpolation} are irreducible and $\dim \Phi_\lambda= \dim \Psi_\lambda$.
\end{lemma}
\begin{proof}
	To see this, note that by the assumption that $X$ is reduced, the map
$\uhilb X \rightarrow \hilb X$ has general member which is reduced by
\cite[Th\'eor\`eme 12.2.4(v)]{EGAIV.3}.
Therefore, applying 
\autoref{proposition:upper-semicontinuous-number-of-components},
the general point of $\hilb X$ has preimage in $\uhilb X$ which
is integral.
So, applying \autoref{proposition:irreducible-incidence},
we conclude that the incidence correspondences $\Phi_\lambda$ and $\Psi_\lambda$ given
in \autoref{definition:interpolation} 
are irreducible of the same 
dimension.
\end{proof}

\section[Tools for dimension counting]{Tools for showing equality of dimensions of the source and target}
\label{ssec:tools-dimensions}

In this section we develop some more technical tools for proving
\autoref{theorem:equivalent-conditions-of-interpolation}.
Our goal for this section is to prove \autoref{lemma:interpolation-dimension}.
Before embarking on this task, we start with a simple tool for proving the equivalence of
\ref{interpolation-dominant} and \ref{interpolation-isolated}.

\begin{lemma}
	\label{lemma:isolated-fiber-implies-dominant}
	Let $\pi: X \ra Y$ be a proper map of locally Noetherian schemes of the same pure dimension. If there is some point $x \in X$ which is isolated in its fiber, then $\dim (\im \pi) = \dim Y$.
\end{lemma}
\begin{proof}
By Zariski's Main Theorem in Grothendieck's form 
	\cite[Theorem 29.6.1(a)]{vakil:foundations-of-algebraic-geometry}
	there is a nonempty open subscheme $X_0 \subset X$ so that all closed point
	of $X_0$ are isolated in their fibers.
	Now, since $X$ is irreducible, $X_0$ is dense. Then, consider the map $\pi|_{X_0} : X_0 \ra Y$.
	It suffices to show this map is dominant. Suppose not. Then, there is a map
	$X_0 \ra Y_0 \subset Y$ where $\dim Y_0 < \dim Y$. This implies $\dim X_0 = \dim Y_0$
	by \cite[Exercise 11.4.A]{vakil:foundations-of-algebraic-geometry}, since if if $p \in X$ and $q:=\pi(p)$ we obtain
	\begin{align*}
		\codim_{X_0} p \leq \codim_{Y_0} q + \codim_{\pi^{-1}(q)}p \leq \dim Y_0 + 0 = \dim Y_0
	\end{align*}
	Since the dimension of $X_0$ is the supremum over all points $p$ of $\codim_{X_0} p$
	we have that $\dim X_0 \leq \dim Y_0$, which is a contradiction because
	we would then obtain
	\begin{align*}
		\dim X = \dim X_0 \leq \dim Y_0 < \dim Y = \dim X.
	\end{align*}
\end{proof}

In order to accomplish our goal of proving \autoref{lemma:definition-and-pointed-equivalence-with-lambda-equality}, we will need an efficient way
of counting the number of conditions imposed a point of the Hilbert
scheme, so that we can apply \autoref{technique-conditions}.
This is established in \autoref{lemma:linear-condition-counting},
and to prove that, we will first need \autoref{lemma:preimage-dimension}.

\begin{lemma}
	\label{lemma:preimage-dimension}
	Suppose $\pi: X \rightarrow Z$ is a dominant map of irreducible schemes of finite type over a field.
	For a general closed point $p$ of $Z$, we have that $\pi^{-1}(p) \subset X$
	is a closed subscheme of dimension $\dim X - \dim Z$.
\end{lemma}
\begin{proof}
First, by ~\cite[Easy Exercise 24.5.N]{vakil:foundations-of-algebraic-geometry},
there is an open set $V \subset Z$ on which $\pi|_{\pi^{-1}(V)}$ is flat.
Now, consider the map $\pi^{-1}(V) \ra V$. It suffices to show that for any $p \in V$, we have $\dim \pi^{-1}(p) = \dim X - \dim Z$.
So, we have reduced to the problem to the case that the map is flat.
The result then follows from ~\cite[Exercise 24.5.J]{vakil:foundations-of-algebraic-geometry}.
\end{proof}

\begin{lemma}
	\label{lemma:linear-condition-counting}
	Suppose we have two maps
	\begin{equation}
		\nonumber
		\begin{tikzcd}
			\qquad & X \ar {ld}{\pi_1} \ar {rd}{\pi_2} & \\
			Y && \bp^n
		\end{tikzcd}\end{equation}
	where $X, Y$ are both irreducible schemes of finite type over a field,
	and $\pi_1, \pi_2$ are dominant maps.
	Suppose further that for a general closed point $p \in \bp^n$, $\pi_1|_{\pi_2^{-1}(p)}$ is generically finite.
	Then, for a general $t$-plane $\Lambda \subset \bp^n$, 
	(meaning a general point of the Grassmannian of $t$-planes)	
	we have
	\begin{align*}
		\codim_Y \pi_1(\pi_2^{-1}(\Lambda)) = (n-t) - \left( \dim X - \dim Y \right).
	\end{align*}
	That is, the codimension of $\pi_1(\pi_2^{-1}(\Lambda))$ is the codimension of $\Lambda$, minus the relative dimension of $\pi_1$.	
\end{lemma}
\begin{proof}
	Let $p \subset \bp^n$ be a general point. By \autoref{lemma:preimage-dimension}, we know $\codim \pi_2^{-1}(p) = n$.
	Therefore, since a general $t$-plane will contain a general point,
	we also obtain that the preimage of a general $t$-plane will
	$\Lambda$ will satisfy $\codim \pi_2^{-1}(\Lambda) = \codim \Lambda$.
	Then, since $\pi_1|_{\pi_2^{-1}(p)}$ is generically finite,
	for a general point p, we have 	
	$\dim \pi_2^{-1}(p) = \dim \pi_1 (\pi_2^{-1}(p))$.
	Hence, for a general plane $\Lambda$, we obtain
	$\dim \pi_2^{-1}(\Lambda) = \dim \pi_1 (\pi_2^{-1}(\Lambda))$.
	Finally, since the dimension of $Y$ is equal to the sum of the dimension and codimension of $\pi_1 (\pi_2^{-1}(\Lambda))$,
	we obtain 
	\begin{align*}
		\codim_Y \pi_1(\pi_2^{-1}(\Lambda)) &= \dim Y - \left( \dim X - (n-t) \right) \\
		&= \codim \Lambda - \left( \dim X - \dim Y \right).
	\end{align*}
\end{proof}

\begin{lemma}
	\label{lemma:interpolation-dimension}
	With notation as in \autoref{definition:interpolation}, 
	if $\sum_{i=1}^m \lambda_i = \dim \hilb X$, we have $\dim \Phi = \dim \prod_{i=1}^m Gr(\codim X - \lambda_i +1, n+1)$.
	In particular, the source and target of the map $\pi_2$ have the same dimension.
\end{lemma}
\begin{proof}
	This is purely a dimension counting argument, and so we will use
	~\autoref{technique-conditions}.
	Note that we have $m$ independent conditions, one for each
	$1 \leq i \leq m$. Or more precisely,
	
	$\Phi$ is a fiber product
	\begin{align*}
		\Phi = \Phi_{(\lambda_1)}\times_{\hilb X} \Phi_{(\lambda_2)} \times_{\hilb X} \cdots \times_{\hilb X} \Phi_{(\lambda_m)}.
	\end{align*}
	of incidence correspondences, $\Phi_{(\lambda_i)}$ as defined in \autoref{definition:interpolation}
	with natural projections
	\begin{equation}
		\nonumber
		\begin{tikzcd}
			\qquad & \Phi_{(\lambda_i)} \ar {ld}{\pi^i_1} \ar {rd}{\pi^i_2} & \\
			 \hilb X && \bp^n. 
		 \end{tikzcd}\end{equation}
	Therefore, the sum of the codimensions for each
	such condition is the total number of conditions in $\hilb X$.
	In particular, the fiber of $\pi_2$ over a given collection
	of planes $\Lambda_1, \ldots, \Lambda_m$ is the product of the
	preimages of $\pi^i_2$ for all $1 \leq i \leq m$.
	Using \autoref{lemma:linear-condition-counting}, for each $i$, we have 
	\begin{align*}
		\codim_{\hilb X} (\pi_2^i)^{-1}(\Lambda_i) &= (n- (\codim X - \lambda_i)) - \dim X \\
		&= \dim X + \lambda_i - \dim X \\
		&= \lambda_i,
	\end{align*}
	assuming the projections $\pi_1^i$ are generically finite.
	Therefore,
	a fiber of $\pi_2$ has codimension $\sum_{i=1}^m \lambda_i = \dim \hilb X$,
	by assumption from \autoref{definition:lambda-constraints}
	if the map $\pi_1$ is generically finite.
	So, if $\pi_1$ were generically finite, then the image
	of a fiber of $\pi_2$, would be zero dimensional.
	This implies that 
	a fiber of $\pi_2$ is zero dimensional, and so the source
	and target of $\pi_2$ have the same dimension.
\end{proof}

\section[Interpolation of locally free sheaves]{Equivalent formulations of interpolation of locally free sheaves}
\label{ssec:vector-bundle-interpolation-properties}

The goal of this section is to prove
\autoref{lemma:equivalent-vector-bundle-interpolation},
which
gives generalizations to higher dimensional varieties
of the
equivalent formulations for interpolation of locally free sheaves, as detailed in
\cite[Section 4]{atanasovLY:interpolation-for-normal-bundles-of-general-curves} 
and \cite[Section 3]{atanasov:interpolation-and-vector-bundles-on-curves}.
We omit much of the proofs, since they are nearly
identical to those given in \cite[Section 4]{atanasovLY:interpolation-for-normal-bundles-of-general-curves}, mutatis mutandis. However, we restate
these generalizations here for clarity.

In order to state a generalization of Proposition \cite[Proposition 4.23]{atanasovLY:interpolation-for-normal-bundles-of-general-curves},
we will need to give some definitions.

\begin{definition}[Definition 4.21 of \cite{atanasovLY:interpolation-for-normal-bundles-of-general-curves}]
	\label{definition:}
Let $V$ be a vector space and let
$\left\{ W_b \subset V : b \in B \right\}$ be a collection
of subspaces indexed by a set $B$. Call the collection
$\left\{ W_b : b \in B \right\}$ {\bf linearly general}
if for each subspace $W \subset V$ there is
some $b \in B$ so that $W_b$ intersects $W$ transversely.
\end{definition}

Recall that for $p_1, \ldots, p_d \in X$ closed points, $\sci_{p_1, \ldots, p_d}$ denotes the ideal sheaf of
$\left\{ p_1 \right\} \cup \cdots \cup \left\{ p_d \right\}$ in $X$.

\begin{definition}
	\label{definition:coherent-sheaf-interpolation}
	We will say a coherent sheaf $G$ (not necessarily a locally free sheaf)
	on a variety $X$
	satisfies property $(*)$ if
	for all $d \geq 1$, for a general collection of points 
	$p_1, \ldots, p_d$ in $X$, we have $H^0(X, G \otimes \sci_{p_1, \ldots, p_d}) = 0$ or $H^1(X, G \otimes \sci_{p_1, \ldots, p_d}) = 0$.
\end{definition}

\begin{lemma}
	\label{lemma:4-23-generalization-from-ALY}
	Suppose $E$ is a locally free sheaf over a variety
	$X$ and $p$ is a smooth point in $X$.
	Suppose we have a collection of subsheaves
	$\left\{ G_b \subset E : b \in B \right\}$ indexed
	by a set $B$ and $F \subset E$ a subsheaf
	so that
	\begin{enumerate}[(a)]
		\item $F|_p \subset G_b|_p$ for all $b \in B$
		\item $\left\{ G_b/F|_p : b \in B \right\}$
			is linearly general in $E/F|_p$.
	\end{enumerate}
	If $E$ satisfies $(*)$ and
	the kernel of the restriction map
	$E \rightarrow E/F|_p$ satisfies $(*)$ of \autoref{definition:coherent-sheaf-interpolation},
	then there is some $b \in B$ so that
	the kernel of the restriction map $E \rightarrow E/G_b|_p$ satisfies $(*)$ of \autoref{definition:coherent-sheaf-interpolation}.
\end{lemma}
\begin{proof}
	The exact same proof given in \cite[Proposition 4.23]{atanasovLY:interpolation-for-normal-bundles-of-general-curves}
	goes through, except with
	one minor issue:
	We need to check that if we start with a sequence of
	sheaves
	\begin{equation}
		\nonumber
		\begin{tikzcd}
			0 \ar {r} & F \ar {r} & E \ar {r} & A \ar {r} & 0 
		\end{tikzcd}\end{equation}
	where $A$ has zero dimensional support, then for a general
	collection of points $p_1, \ldots, p_d$ the twisted
	sequence
	\begin{equation}
		\nonumber
		\begin{tikzcd}
			0 \ar {r} & F \otimes \sci_{p_1, \ldots, p_d} \ar {r} & E \otimes \sci_{p_1, \ldots, p_d}\ar {r} & A \otimes \sci_{p_1, \ldots, p_d}\ar {r} & 0 
		\end{tikzcd}\end{equation}
	remains exact.
	Since the points are general, we may assume $p_1, \ldots, p_d$
	are all distinct and do not intersect $\supp A$.
	To check this is exact, we only need verify that
	$\stor^1(\sci_{p_1, \ldots, p_d}, A) = 0$.

	Indeed, since $Hom$ commutes with localization,
	\cite[Exercise 1.6.G]{vakil:foundations-of-algebraic-geometry},
	$\stor^1$ does as well, and so it suffices to check
	\begin{align*}
		(\stor^1(\sci_{p_1, \ldots, p_d}, A))_{\fm_{\overline p}} = \stor^1\left( (\sci_{p_1, \ldots, p_d})_{\fm_{\overline p}}, A_{\fm_{\overline p}} \right) = 0,
	\end{align*}
	localized at maximal ideals $\fm_{\overline p}$ as $\overline p$
	ranges over all closed points $\overline p$ of
	$X$. In the case that $\overline p \notin \left\{ p_1, \ldots, p_d \right\}$, we have that $(\sci_{p_1, \ldots, p_d})_{\fm_{\overline p}}$
	is locally free, and so 
	$\stor^1\left( (\sci_{p_1, \ldots, p_d})_{\fm_{\overline p}}, A_{\fm_{\overline p}} \right) = 0$.
	In the case that $\overline p \in \left\{ p_1, \ldots, p_d \right\}$,
	using the assumption that $p_1, \ldots, p_d$ does not intersect
	$\supp A$, we obtain that $A_{\fm_{\overline p}} = 0$, and so again
	$\stor^1((\sci_{p_1, \ldots, p_d})_{\fm_{\overline p}}, A_{\fm_{\overline p}} )= 0$.
\end{proof}

\begin{corollary}
	\label{corollary:equivalent-vector-bundle-interpolation-with-only-points}
	Let $E$ be a locally free sheaf on a variety $X$
	satisfying $(*)$ of \autoref{definition:coherent-sheaf-interpolation},
	let $p \in X$ be a general smooth point, and let
	$0 \leq t \leq \rk E$.
	There exists a vector subspace $V \subset E|_p$
	with $\dim V = t$ so that the subsheaf
	$G \rightarrow E$ with corresponding exact sequence
	\begin{equation}
		\nonumber
		\begin{tikzcd}
			0 \ar {r} & G \ar {r} & E \ar {r} & E|_p/V \ar {r} & 0 
		\end{tikzcd}\end{equation}
	is a sheaf
satisfying $(*)$ of \autoref{definition:coherent-sheaf-interpolation}.
\end{corollary}
\begin{proof}
	In the statement of \autoref{lemma:4-23-generalization-from-ALY},	
	take $B := Gr(t, \rk E), F := E \otimes \sci_p,
	G_{[V]} := E|_p/V$.
	Here, $[V] \in Gr(t,\rk E)$ is the point corresponding to the
	vector subspace $V \subset E|_p$. Next, we wish to apply 
	\autoref{lemma:4-23-generalization-from-ALY}.
	From the definition of $(*)$, since $E$ satisfies $(*)$, we obtain that 
	$F = E \otimes \sci_p$ also satisfies $(*)$.
Then, since conditions $(a)$ and $(b)$ of
\autoref{lemma:4-23-generalization-from-ALY} are trivially satisfied, $G$ satisfies $(*)$.
\end{proof}

\begin{lemma}[A generalization of Propositions 4.5 and 4.6 and 4.23 of \cite{atanasovLY:interpolation-for-normal-bundles-of-general-curves}]
	\label{lemma:equivalent-vector-bundle-interpolation}
	Let $E$ be a locally free sheaf on $X$ with $H^1(X, E) = 0$ and $H^0(X, E) = q \cdot \rk E + r$.
	The following statements are equivalent:
	\begin{enumerate}
		\item The locally free sheaf $E$ satisfies interpolation.
		\item There is a subsheaf $E' \rightarrow E$ 
whose cokernel is supported at $q+1$ points if $r > 0$ and at $q$ points if
$r = 0$,
so that the scheme theoretic support at $q$ of these points has
dimension equal to $\rk E$
and $H^0(X, E') = H^1(X, E') = 0$.
		\item For every $d \geq 1$
			and points $p_1, \ldots, p_d \in X$,
			we have
			\begin{align*}
				h^0(X, E \otimes \sci_{p_1, \ldots, p_d}) = \max\left\{ 0, h^0(X, E) - dn \right\}
			\end{align*}
		\item For every $d \geq 1$, a general collection of points
			$p_1, \ldots, p_d$ satisfies either
			\begin{align*}
				h^0(X, E) && \text{ or } && h^1(X, E) = 0.
			\end{align*}
	\item A general set of $q$ points
			$p_1, \ldots, p_q$
			satisfy $h^0(X, E \otimes \sci_{p_1, \ldots, p_q}) = 0$
			and a general set of $q+1$ points
			$p_1, \ldots, p_{q+1}$ satisfy
			$h^1(X, E \otimes \sci_{p_1, \ldots, p_q}) = 0$.
	\end{enumerate}
\end{lemma}
\begin{proof}
First, we show that $(1)$ and $(2)$ are equivalent.
Their equivalence follows almost immediately from their definition.
The only slight difference
is that we must check that the sheaf $E'$ from 
\autoref{definition:vector-bundle-interpolation} has $H^1(X, E') = 0$, which follows
from the sequence on cohomology associated to the exact sequence
\eqref{equation:sequence-cohomological-interpolation-definition}.

	The equivalence of $(3)$ and $(4)$ is an immediate
	generalization of the statement and proof of \cite[Proposition 4.5]{atanasovLY:interpolation-for-normal-bundles-of-general-curves}.

	The equivalence of $(3)$ and $(5)$ is an immediate generalization of the statement and proof of \cite[Proposition 4.6]{atanasovLY:interpolation-for-normal-bundles-of-general-curves}.
Note also that the last part of \cite[Proposition 4.5]{atanasovLY:interpolation-for-normal-bundles-of-general-curves} regarding
Euler characteristics does not hold for higher dimensional varieties
because it may be that $H^i(X, N_{X/\bp^n}) \neq 0$ for $i > 1$.

	To complete the proof, we will show
	$(4)$ implies $(1)$ and $(2)$ implies $(5)$. 
	For notational convenience,
	for the remainder of this proof, we shall deal with the case that
	$r \neq 0$, as the case $r = 0$ is completely analogous.
	Take $p_1, \ldots, p_q, p_{q+1}$
	to be $q + 1$ general points in $X$.
	
	First, we show that $(4)$ implies $(1)$ by $q+1$ applications of
	\autoref{corollary:equivalent-vector-bundle-interpolation-with-only-points}.
	Let $G_1$ be the sheaf which is the kernel of
	$E \ra E|_{p_1}$.
	By \autoref{corollary:equivalent-vector-bundle-interpolation-with-only-points}
	$G_1$ satisfies $(*)$.
	For $2 \leq i \leq q$, let
	$G_i$ be the kernel of $G_{i-1} \ra E|_{p_i}$.
	Then, $G_i$ satisfies $(*)$ 
	using \autoref{corollary:equivalent-vector-bundle-interpolation-with-only-points}
	and the inductive assumption that $G_{i-1}$ satisfies $(*)$.
	Finally, by 
	\autoref{corollary:equivalent-vector-bundle-interpolation-with-only-points}
	there exists $V \subset E|_{p_{q+1}}$ with $\dim V = \codim X - r$ so that
	$G_{q+1} := \ker(G_q \ra E|_{p_{q+1}}/V)$ satisfies $(*)$.
	Then, we obtain an exact sequence
	\begin{equation}
		\label{equation:sequence-for-points-to-linear-spaces}
		\begin{tikzcd}
			0 \ar {r} &  G_{q+1} \ar {r} & E \ar {r} & E|_{q+1}/V \oplus \left(\bigoplus_{i=1}^q E|_{p_i} \right)\ar {r} & 0 
		\end{tikzcd}\end{equation}
	as in \eqref{equation:sequence-cohomological-interpolation-definition}
	so that either $H^1(X, G_{q+1}) = 0$ or $H^0(X, G_{q+1}) = 0$.
	But then, by the associated long exact sequence to \eqref{equation:sequence-for-points-to-linear-spaces}, we have
	$H^1(X, G_{q+1}) = 0$ if and only if $H^0(X, G_{q+1}) = 0$,
	implying that $H^1(X, G_{q+1}) = 0$ and so $E$ satisfies interpolation.

	Finally, we show, $(2)$ implies $(5)$ from a fairly straightforward exact sequence.
Suppose that $E$ satisfies interpolation
with corresponding subsheaf $E'$ 
and points $p_1, \ldots, p_q,p_{q+1}$,
as in 
\autoref{definition:vector-bundle-interpolation}.
Assuming that $E$ satisfies interpolation, 
so that $H^0(E') = H^1(E') = 0$, we see
we see
that $E$ satisfies $(5)$, by considering
cohomology associated to the exact sequence beginning with
$E' \rightarrow E \otimes \sci_{p_1, \ldots, p_q}$
and the exact sequence beginning with
$E \otimes \sci_{p_1, \ldots, p_{q+1}} \rightarrow E'$.
\end{proof}

\section{Additional tools}
\label{ssec:tools-additional}

In this section, we introduce a couple more tools to prove
\autoref{theorem:equivalent-conditions-of-interpolation}, one
easy and one more difficult.

We start with an extremely elementary lemma, useful
for establishing the equivalence between pointed interpolation and interpolation.

\begin{lemma}
	\label{lemma:definition-and-pointed-equivalence-with-lambda-equality}
	Suppose $\sum_{i=1}^m \lambda_i = \dim \hilb X$. Then,
	$\lambda$-interpolation is equivalent to $\lambda$-pointed interpolation.
\end{lemma}
\begin{proof}
	The map
$\eta_2$ factors as
\begin{equation}
	\nonumber
	\begin{tikzcd}
		\Psi \ar {rr}{\tau} \ar {rd}{\eta_2} && \Phi \ar {ld}{\pi_2} \\
		 &  \prod_{i=1}^m Gr(\codim X - \lambda_i +1, n+1).
& 
	 \end{tikzcd}\end{equation}
Since $\tau$ is surjective, we have that $\pi_2$ is surjective
if and only if $\eta_2$ is surjective, and so
$\lambda$-interpolation is equivalent to $\lambda$-pointed interpolation.
\end{proof}

In the remainder of this section, we prove a result from
deformation theory 
crucial in establishing
the equivalence between interpolation of
a locally free sheaf and interpolation of a Hilbert scheme.
This is the crux of the proof of
the equivalence of the distinct groups of conditions in
\autoref{theorem:equivalent-conditions-of-interpolation},

\begin{proposition}
	\label{proposition:tangent-space-to-psi}
	Let $\Psi, \eta_2$ be as in \autoref{definition:interpolation}
	and let 
	\begin{align*}
	  \overline p := (X, \Lambda_1, \ldots, \Lambda_m, p_1, \ldots, p_m) \in \Psi
	\end{align*}
	be a closed point of $\Psi$,
	so that $\Lambda_i$ meets $X$ quasi-transversely
	and so that the $p_i$ are distinct smooth points of $X$. Choose subspaces $V_i \subset N_{X/\bp^n}|_{p_i}$ where
$V_i$ is the image of the composition
\begin{equation}
	\nonumber
	\begin{tikzcd}
	  N_{p_i/\Lambda_i} \ar{r} & N_{p_i/\bp^n} \ar{r} & N_{X/\bp^n}|_{p_i}. 
	\end{tikzcd}\end{equation}
For any closed point $\overline q$  of $\Psi$, let \[d\eta_2|_{\overline{q}}: T_{\overline{q}} \Psi \ra T_{\eta_2(\overline q)}\prod_{i=1}^m Gr(\codim X - \lambda_i +1, n+1)\] be the
induced map on tangent spaces. Then, $d{\eta_2}$ is surjective if and only if
if and only if the map
\begin{equation}
\nonumber
	\begin{tikzcd}
		H^0(X, N_{X/\bp^n}) \ar{r}{\tau} & H^0(X, \oplus_{i=1}^m N_{X/\bp^n}|_{p_i}/ V_i)
	\end{tikzcd}\end{equation}
is surjective.
\end{proposition}
\begin{proof}
To set things up properly, we will need some definitions.
Recall that $\uhilb {\Lambda_i}$ is the universal family over the Hilbert scheme of $\dim \Lambda_i$ planes in $\bp^n$.
That is, it is the universal family over $Gr(\codim X - \lambda_i + 1, n+1)$.
Next, take $\sch$ to be the connected component of the Hilbert
scheme containing $X$ and let $\scv$ be the universal family over $\sch$.
Next, define the scheme
\begin{align*}
	\scf &:= (\sch \times_{\bp^n} \uhilb {\Lambda_1}) \times_\scv \cdots \times_\scv (\sch \times_{\bp^n} \uhilb {\Lambda_m}) \\
	& \cong (\scu \times_\sch \cdots \times_\sch \scu) \times_{(\bp^n)^m} (\uhilb {\Lambda_1} \times \cdots \times \uhilb {\Lambda_m}),
\end{align*}
where there are $m$ copies of $\scu$ in the first parenthesized expression on the second line.

Note that here $\scf$ is not necessarily the same as $\Psi$ because we need not have $\sch = \hilb X$:
The former is the connected component of the Hilbert scheme containing $X$ while $\hilb X$ is the
irreducible component of the Hilbert scheme containing $X$.
However, we will later explain why the tangent spaces of these two schemes are identical, which
is enough for our purposes.

Now, under our assumption that $p_1, \ldots, p_n$ are distinct, we have a diagram

\begin{equation}
\label{equation:three-by-three-deformation-theory}
\hspace*{-2cm}
\includegraphics{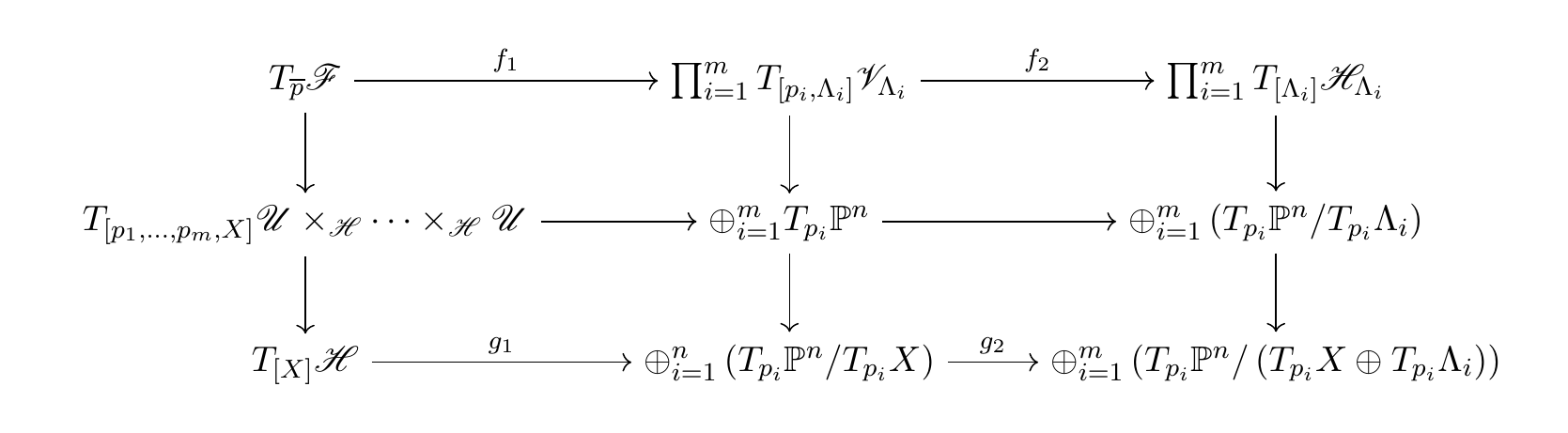}
\end{equation}
\noindent	
in which every square is a fiber square.

First, let us justify why the four small squares of \eqref{equation:three-by-three-deformation-theory} are fiber squares.
The lower right hand square of \eqref{equation:three-by-three-deformation-theory} is a fiber square by elementary linear algebra and the assumption that
$\Lambda_i$ meet $X$ quasi-transversely.
The upper right square of \eqref{equation:three-by-three-deformation-theory} is a fiber square for each $i$
by \cite[Remark 4.5.4(ii)]{sernesi:deformations-of-algebraic-schemes},
as the universal family over the Hilbert scheme is precisely the Hilbert flag scheme of points inside that Hilbert scheme.
Next, the lower left hand square of \eqref{equation:three-by-three-deformation-theory} is a fiber
square because when the points $p_1, \ldots, p_n$ are distinct, the tangent space to this
$n$-fold fiber product of universal families over the Hilbert scheme is the same as the tangent space of
to the Hilbert flag scheme of degree $n$ schemes inside schemes with the same Hilbert polynomials as $X$.
Then, the fiber square follows from \cite[Remark 4.5.4(ii)]{sernesi:deformations-of-algebraic-schemes} for this flag Hilbert scheme.
Finally, the upper left square of \eqref{equation:three-by-three-deformation-theory} is a fiber square
because $\scf$ is defined as a fiber product of
$(\scu \times_\sch \cdots \times_\sch \scu)$ and $(\uhilb {\Lambda_1} \times \cdots \times \uhilb {\Lambda_m})$,
and the fiber product of the tangent spaces is the tangent space of the fiber product.

Now, observe that the composition $f_2 \circ f_1$ is precisely the map on tangent spaces $d\eta_2|_{\overline{p}}$.
To make this identification, we need to know that we can naturally identify
$T_{\overline p} \scf \cong T_{\overline p}\Psi$.
However, the assumption that $H^1(X, N_{X/\bp^n}) = 0$
and $X$ is a locally complete intersection means that $[X]$ is a smooth point of the Hilbert scheme.
Because the fiber over $[X]$ of the projection $\Psi \ra \hilb X$ is smooth, it follows that $\Psi$ is smooth at $\overline p$.
For the same reason, it follows that $\scf$ is smooth at the corresponding point $\overline p$.
Therefore, both $\scf$ and $\Psi$ are smooth on some open neighborhood $U$ containing
$\overline p$.
Now, since both $\Psi$ and $\scf$ are defined in terms of fiber products, which agree on some open neighborhood $V$
contained in $U$, it follows that on $V$ we have an isomorphism $\scf|_V \cong \Psi|_V$, and in particular their
tangent spaces are isomorphic. So, we can identify $f_2 \circ f_1$ with $d\eta_2|_{\overline{p}}$.

Since all four subsquares of \eqref{equation:three-by-three-deformation-theory}
are fiber squares, the full square \eqref{equation:three-by-three-deformation-theory} is a fiber square,
and hence $f_2 \circ f_1$ is an isomorphism if and only if $g_2 \circ g_1$ is an isomorphism.

To complete the proof, we only need identify the map $g_2 \circ g_1$ with $\tau$.
But this follows from the identifications
\begin{align*}
	T_{[X]} \sch &\cong H^0(X, N_{X/\bp^n}) \\
	T_{\bp_i} \bp^n/T_{p_i} X &\cong H^0(X, N_{X/\bp^n}|_{p_i}) \\
	T_{\bp_i} \bp^n/\left(T_{p_i} X \oplus T_{p_i} Y  \right) &\cong H^0(X, N_{X/\bp^n}|_{p_i}/V_i).
\end{align*}
The first isomorphism follows from \cite[Theorem 1.1(b)]{Hartshorne:deformation}.
The second isomorphism holds because the normal exact sequence
\begin{equation}
\nonumber
	\begin{tikzcd}
	  0 \ar{r} &  T_{p_i}X \ar{r} & T_{p_i} \bp^n \ar{r} & N_{X/\bp^n}|_{p_i} \ar{r} & 0 
	\end{tikzcd}\end{equation}
is exact on global sections, as all sheaves are supported at $p_i$.
The third isomorphism holds because
$\left( T_{p_i} \bp^n / \left( T_{p_i} X \oplus T_{p_i} \Lambda_i \right) \right)$ can be viewed as the
quotient of $T_{p_i} \bp^n$ first by $T_{p_i} X$ and then by the image of $T_{p_i} \Lambda_i$ in that quotient.
However, $T_{p_i} \bp^n / T_{p_i} X \cong N_{X/\bp^n}|_{p_i}$, and then
$V_i$ is by definition the image of $T_{p_i} \Lambda_i$ in $N_{X/\bp^n}|_{p_i}$.
\end{proof}

\section{Proof of Theorem ~\getrefnumber{theorem:equivalent-conditions-of-interpolation}}
\label{ssec:proof-of-equivalent-conditions-of-interpolation}

\begin{proof}[Proof of \autoref{theorem:equivalent-conditions-of-interpolation}]

The structure of proof is as follows:
\begin{enumerate}
	\item Show equivalence of conditions \ref{interpolation-definition}-\ref{interpolation-sweep}
	\item Show equivalence of conditions \ref{strong-definition}-\ref{strong-naive-equality}
	\item Show equivalence of conditions \ref{cohomological-definition}-\ref{cohomological-boundary}
	\item Demonstrate the implications that \ref{cohomological-definition}-\ref{cohomological-boundary} imply \ref{interpolation-definition}-\ref{interpolation-sweep},  \ref{cohomological-definition}-\ref{cohomological-boundary} imply
\ref{strong-definition}-\ref{strong-naive-equality}, and
\ref{strong-definition}-\ref{strong-naive-equality}
imply
\ref{interpolation-definition}-\ref{interpolation-sweep}, in all characteristics. Further, all statements are equivalent in characteristic $0$.

\end{enumerate}
\ssec{Equivalence of conditions \ref{interpolation-definition}-\ref{interpolation-sweep}}

First, \ref{interpolation-definition} and
\ref{interpolation-pointed} are equivalent 
by \autoref{lemma:definition-and-pointed-equivalence-with-lambda-equality}
applied to $\lambda = ((\codim X)^q, r)$.

Next, note that a proper map of irreducible schemes of the same dimension is
surjective if and only if it is dominant if and only if it is generically
finite if and only if there is some point isolated in its fiber.
The first three equivalences are immediate, the last follows from
\autoref{lemma:isolated-fiber-implies-dominant}.
Since $\dim \prod_{i=1}^m Gr(\codim X - \lambda_i +1, n+1) = \dim \Phi$,
by \autoref{lemma:interpolation-dimension},
we have that \ref{interpolation-definition}, \ref{interpolation-dominant},
\ref{interpolation-finite}, \ref{interpolation-isolated} are equivalent.

Next, since $\dim \Phi = \dim \Psi$, and $\Psi$ is irreducible, 
by Lemma
\ref{lemma:irreducible-and-dimension},
we have
$\dim \prod_{i=1}^m Gr(\codim X - \lambda_i +1, n+1) = \dim \Psi$.
So, by reasoning analogous to that of the previous paragraph, we obtain that \ref{interpolation-pointed}, \ref{interpolation-pointed-dominant}, \ref{interpolation-pointed-finite}, and \ref{interpolation-pointed-isolated} are equivalent.

Next, \ref{interpolation-definition} is equivalent to
\ref{interpolation-naive} because surjectivity of a proper map
of varieties is equivalent to surjectivity on closed points of the varieties.
Since the fibers of the map $\pi_2$ precisely consists of those
elements of $\hilb X$ meeting a specified collection of $q$ points
and a plane $\Lambda$, being surjective is equivalent to there
being some element of $\hilb X$ passing through these $q$ points and meeting $\Lambda$.

Finally, \ref{interpolation-naive} is equivalent to \ref{interpolation-sweep}
because the condition that the variety swept out by the elements of $\hilb X$
containing $q$ points
meet a general plane $\Lambda$ of dimension $\codim X - r$ is equivalent
to the variety swept out by the elements of $\hilb X$ being
$\dim X + r$ dimensional.
This is just using the fact that a variety of dimension $d$ in $\bp^n$
meets a general plane of dimension $d'$ if and only if $d + d' \geq n$.
But, of course, the dimension swept out by the elements of $\hilb X$
containing $q$ general points is at most $\dim X + r$ dimensional, because
there is at most an $r$ dimensional space of varieties in $\hilb X$
containing $r$ general points, using \autoref{lemma:preimage-dimension}.

This shows the equivalence of properties \ref{interpolation-definition}
through \ref{interpolation-sweep}.

\ssec{Equivalence of conditions \ref{strong-definition}-\ref{strong-naive-equality}
}

By \autoref{lemma:definition-and-pointed-equivalence-with-lambda-equality},
for all $\lambda$ with $\sum_{i=1}^m \lambda_i = \dim \hilb X$, $\lambda$-interpolation is equivalent
to $\lambda$-pointed interpolation. This establishes the equivalence
of \ref{strong-equality} and \ref{strong-pointed-equality}
and the equivalence of \ref{strong-definition} and \ref{strong-pointed}.

Next, \ref{strong-definition} is equivalent to \ref{strong-naive},
because the map $\pi_2$ contains a point corresponding to a
collection of planes $\Lambda_1, \ldots, \Lambda_m$ in its image
if and only if there is some element of the Hilbert
schemes meeting those planes. Similarly,
\ref{strong-equality} is equivalent to \ref{strong-naive-equality}.

To complete these equivalences, we only need show \ref{strong-naive} is equivalent to \ref{strong-naive-equality}. Clearly \ref{strong-naive} implies \ref{strong-naive-equality}. For the reverse implication, observe that if we start with a collection
of planes $\Lambda_1, \ldots, \Lambda_s$ with $\Lambda_i \in Gr(\codim X - \lambda_i  + 1, n+1)$, so that $\sum_{i=1}^s \lambda_i < \dim \hilb X$, we can extend the sequence $\lambda$ to a sequence $\mu = \left( \mu_1, \ldots, \mu_m \right)$
for $m > s$,
with
$0 \leq \mu_i \leq \codim X$, 
$\mu_i = \lambda_i$ for $i \leq s$, and $\sum_{i=1}^m \mu_i = \dim \hilb X$. Then, if some element of $\hilb X$ meets planes $\Lambda_1, \ldots, \Lambda_m$ corresponding to the sequence $\mu$, it certainly
also meets $\Lambda_1, \ldots, \Lambda_s$. Hence, \ref{strong-naive-equality} implies \ref{strong-naive}.

\ssec{Equivalence of conditions \ref{cohomological-definition}-\ref{cohomological-boundary}}

The equivalence of \ref{cohomological-definition},
\ref{cohomological-restatement}, \ref{cohomological-sections},
\ref{cohomological-vanish}
\ref{cohomological-boundary} is immediate from
\autoref{lemma:equivalent-vector-bundle-interpolation}, taking $E := N_{X/\bp^n}$.

Perhaps also
most surprising, part of these equivalences is
the equivalence of \ref{cohomological-definition}
and \ref{cohomological-strong}. This is an immediate
generalization of \cite[Theorem 8.1]{atanasov:interpolation-and-vector-bundles-on-curves} to higher dimensional varieties.
The proof is almost the verbatim the same, replacing
curves with arbitrary varieties.
Note that the key ingredient in the proof of
\cite[Theorem 8.1]{atanasov:interpolation-and-vector-bundles-on-curves}
is
\cite[Proposition 8.3]{atanasov:interpolation-and-vector-bundles-on-curves}, which is just an elementary linear algebraic fact.

\ssec{Implications among all conditions}

By definition \ref{strong-definition} implies \ref{interpolation-definition}.

To complete the proof, we only need to show \ref{cohomological-definition}
implies \ref{strong-pointed} and \ref{interpolation-pointed} (in all characteristics)
and that the reverse implications hold true in characteristic $0$.

For this, choose $\lambda$ with $\sum_{i=1}^m \lambda_i = \dim \hilb X$. We will show that $\lambda$-interpolation of $N_{X/\bp^n}$ implies $\lambda$-pointed
interpolation in all characteristics, and the reverse
implication holds in characteristic $0$.
It suffices to prove this, as this will yield the desired implications.
For example, this implies the relation between \ref{interpolation-pointed} and
\ref{cohomological-definition}, by taking $\lambda = \left( (\codim X)^q, r \right)$.

To see this statement about $\lambda$-pointed interpolation and $\lambda$-interpolation of $N_{X/\bp^n}$,
let $\overline p := (Y, \Lambda_1, \ldots, \Lambda_m, p_1, \ldots, p_m), V_i, \tau$ be as in
\autoref{proposition:tangent-space-to-psi}.

By \autoref{proposition:tangent-space-to-psi}, we have that the map
$d\eta_2|_{\overline{p}}$ is surjective if and only if
the corresponding map $\tau$
is surjective. But this latter map is precisely that from ~\eqref{equation:sequence-cohomological-interpolation-definition}
in the definition of interpolation for vector bundles,
taking $E := N_{X/\bp^n}$.

So, to complete the proof, it suffices to show that if $d\eta_2|_{\overline{p}}$
is surjective, then $\eta_2$ is surjective, and the converse holds in characteristic $0$.

But now we have reduced this to a general statement about varieties.
Note that $\eta_2$ is a map between two varieties of the same dimension, 
by \autoref{lemma:interpolation-dimension} and that $\overline p$
is a smooth point of $\Psi$ by assumption.
So, it suffices to show that a map between two proper varieties of the same
dimension is surjective if it is surjective on tangent spaces,
and that the converse holds in characteristic $0$.
For the forward implication,
if the map is surjective on tangent spaces,
the map is smooth of
relative dimension $0$ at $\overline p$ by
\cite[Exercise 25.2.F(b)]{vakil:foundations-of-algebraic-geometry}.
But, this means that $\overline p$ is isolated in its fiber, and so by \autoref{lemma:isolated-fiber-implies-dominant},
we obtain that $\eta_2$ is surjective.

To complete the proof, we only need to show that if $\eta_2$ is surjective and $\bk$ has characteristic
$0$, then there is a point at which $d\eta_2|_{\overline{p}}$ is surjective.
That is, we only need to show there is a point at which $\eta_2$ is smooth.
But, this follows by generic smoothness, which crucially uses the characteristic $0$ hypothesis!
\end{proof}

\section{Complete intersections}

\begin{definition}
	\label{definition:}
	Define $\ci k d n$ to be the closure in the 
	Hilbert scheme of the locus of complete intersections
	of $k$ polynomials of degree $d$ in $\bp^n$.
\end{definition}

\begin{warn}
	\label{warning:}
	If $[X] \in \ci k d n$ is a general complete intersection, it is not
	necessarily the case that $\hilb X = \ci k d n$. 
	In the case they are not equal, we are applying a slight variant
	of the interpolation problem, where we generalize the question
	from an irreducible component
	of the Hilbert scheme satisfying interpolation to an arbitrary integral subscheme of the
	Hilbert scheme satisfying interpolation.
\end{warn}

\begin{lemma}
	\label{lemma:balanced-complete-intersection}
	Let $k, d, n$ be positive integers.
	Then, $\ci k d n$ satisfies interpolation.
	In particular, any Hilbert scheme of hypersurfaces
	$\ci 1 d n$
	satisfies interpolation.
	Furthermore, interpolation
	is equivalent to meeting $\binom{d+n}{d}-k$
	general points in $\bp^n$.
\end{lemma}
\begin{proof}
	First, observe that $\dim \ci k d n = k(\binom{d+ n}{d} - k)$
	because a point of $\ci k d n$ corresponds to the
	variety cut out by the intersection of all degree $d$ polynomials
	in a k dimensional subspace of $H^0(\bp^n, \sco_{\bp^n}(d))$.
	In other words, there is a birational map between
	the locus of complete intersections and $G(k, H^0(\bp^n, \sco_{\bp^n}(d)))$,
	which is $k(\binom{d+n}{d}-k)$ dimensional.
	So, to show $\ci k d n$ satisfies interpolation, it suffices
	to show there exists such a complete intersection through
	$\binom{d + n}{d} - k$ general points.
	First, since points impose independent conditions on
	degree $d$ hypersurfaces in $\bp^n$, there will indeed be a $k$
	dimensional subspace of $H^0(\bp^n, \sco_{\bp^n}(d))$
	passing through the any collection
	of $\binom{d+n}{d} - k$ points.

	It remains to verify that if the points are chosen
	generally, then the intersection of degree
	$d$ hypersurfaces in the subspace passing through
	the points is a complete intersection.
	To see this, note that the map $\pi_2$
	from \autoref{definition:interpolation}
	is a generically finite map between varieties of
	the same dimension. In particular,
	the element of	
	$G(k, H^0(\bp^n, \sco_{\bp^n}(d)))$
	through a general collection of
	$\binom{d+n}{d}-k$ points
	will be general in
	$G(k, H^0(\bp^n, \sco_{\bp^n}(d)))$.
	Then, since a general element of 
	$G(k, H^0(\bp^n, \sco_{\bp^n}(d)))$
	corresponds to a complete intersection,
	there will indeed be a complete intersection
	passing through a general collection of
	$\binom{d+n}{d}-k$ points.
\end{proof}

\chapter{Basics of scrolls}
\label{section:basics-of-scrolls}

\epigraph{Rational normal scrolls \dots occur throughout projective and algebraic geometry, and the student will never regret the investment of time studying them.}{\textit{Miles Reid \cite[p.\ 19]{kollar:complex-algebraic-geometry}}}

\section{The definition of scrolls}

After defining scrolls,
we give an alternate construction of a scroll as the planes
joining several rational normal curves.
We start by describing this construction
in the case
of the smooth degree $3$ surface scroll in $\bp^4$. We then generalize this
construction to scrolls of all dimensions and degrees. Finally, we
explain the equivalence of various geometric descriptions of scrolls.
A good reference for the equivalent geometric descriptions of scrolls is
~\cite[Section 1]{eisenbudH:on-varieties-of-minimal-degree}.
Another useful reference is ~\cite[Section 9.1.1]{Eisenbud:3264-&-all-that}.

\begin{figure}
\centering
\includegraphics[scale=.3]{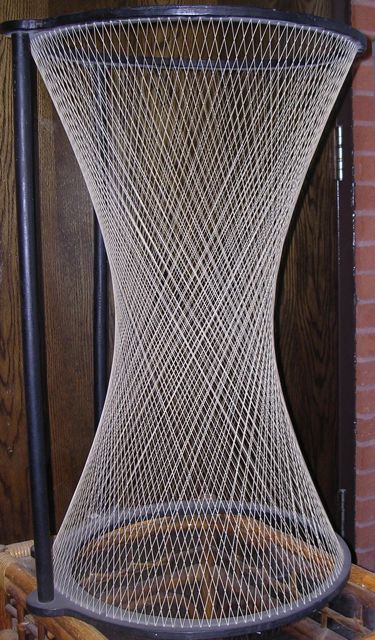}  
\includegraphics[scale=.06]{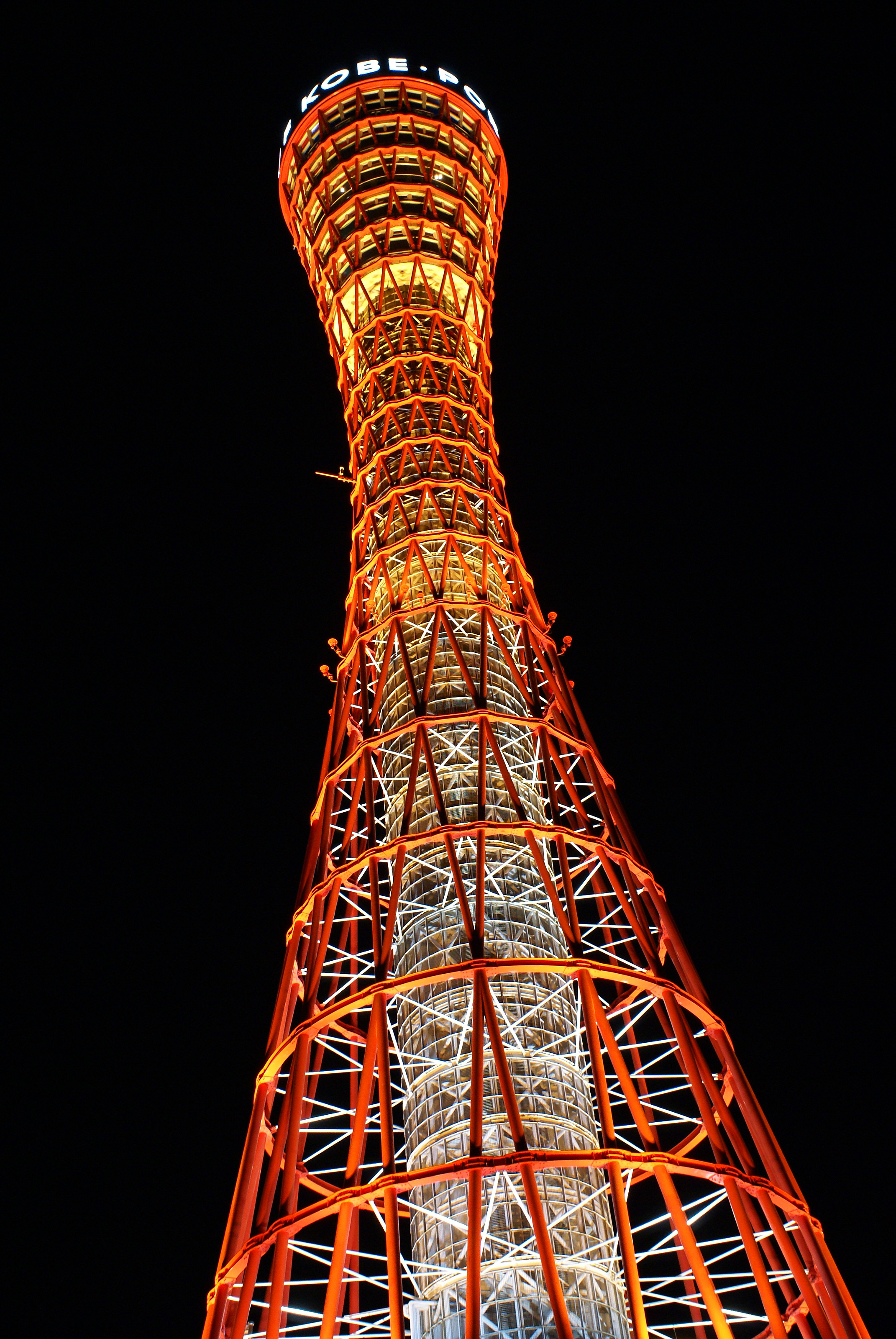}
\includegraphics[scale=.6]{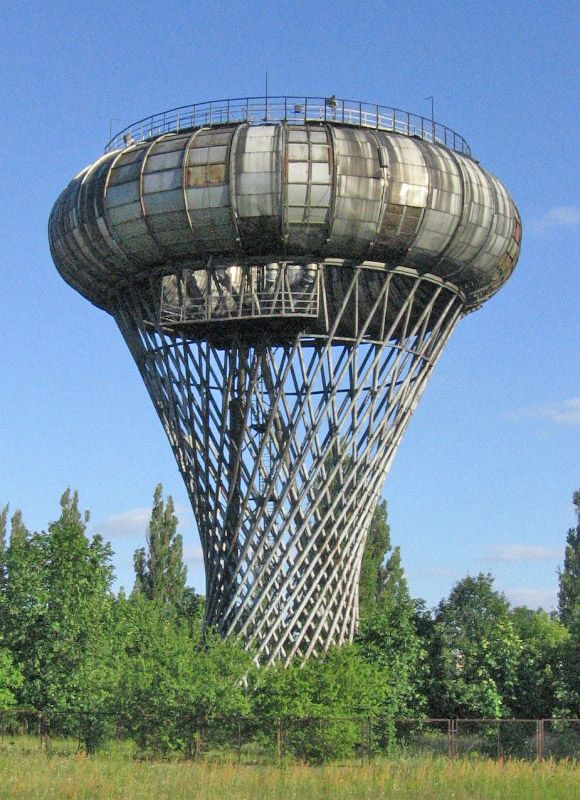}
\caption{Scrolls ``in the wild'' cf. \cite[Section 2]{eisenbudH:on-varieties-of-minimal-degree}.
From left to right: A hyperboloid model, a hyperboloid at Kobe Port Tower, Kobe, Japan, and cooling hyperbolic towers at Didcot Power Station, UK.
These are all examples of quadric surface scrolls with a double ruling.}
\end{figure}

\begin{definition}
  \label{definition:scroll}
  Suppose we have projective scheme $X \subset \bp^n$ which is abstractly
  isomorphic to a projective bundle $X \cong \bp \sce \xrightarrow \pi \bp^1$,
  where $\sce \cong \sco_{\bp^1}(a_1) \oplus \cdots \oplus \sco_{\bp^1}(a_k)$
  with $a_1 \geq \cdots \geq a_k \geq 1$.
  Then, $X$ is a {\bf scroll of type $a_1, \ldots, a_k$} 
  if it is embedded into $\bp^n$ by the complete linear series
  of the ``relative $\sco_{\bp \sce}(1)$ for $\pi$.''
  Here, the ``relative $\sco_{\bp \sce}(1)$ for $\pi$'' denotes the invertible
  sheaf as defined in \cite[Exercise 17.2.D]{vakil:foundations-of-algebraic-geometry}.
  We notate a scroll of type $a_1, \ldots, a_k$ as $\scroll {a_1, \ldots, a_k}$.
  A {\bf scroll} is any projective scheme for which there exists
  some sequence $a_1, \ldots, a_k$ so that $X$ is a scroll of
  type $a_1, \ldots, a_k$.
  We call a scroll {\bf balanced} if $a_1 - a_k \leq 1$.
  \end{definition}

\section[Construction of the degree 3 surface scroll]{An extended example: the construction of the degree 3 surface scroll}
\label{subsection:segre-cubic-surface-construction}

It is a fairly well known statement that every smooth surface scroll 
in $\bp^4$ is ``swept out''
by ``the lines connecting'' two curves $L, M,$
both isomorphic to $\bp^1$.
Here, $L$ is embedded as a line 
(by $\sco_{\bp^1}(1)$)
and $M$ is embedded as a conic
(by $\sco_{\bp^1}(2)$).
Let $H$ denote the two plane spanned
by the conic $M$.

We will see later in \autoref{proposition:equivalence-vector-bundle-swept-planes}
that this description of a scroll in terms of a variety swept out by linear spaces
is equivalent to the definition given in \autoref{definition:scroll}.

\begin{figure}
	\centering
	\includegraphics[scale=.25]{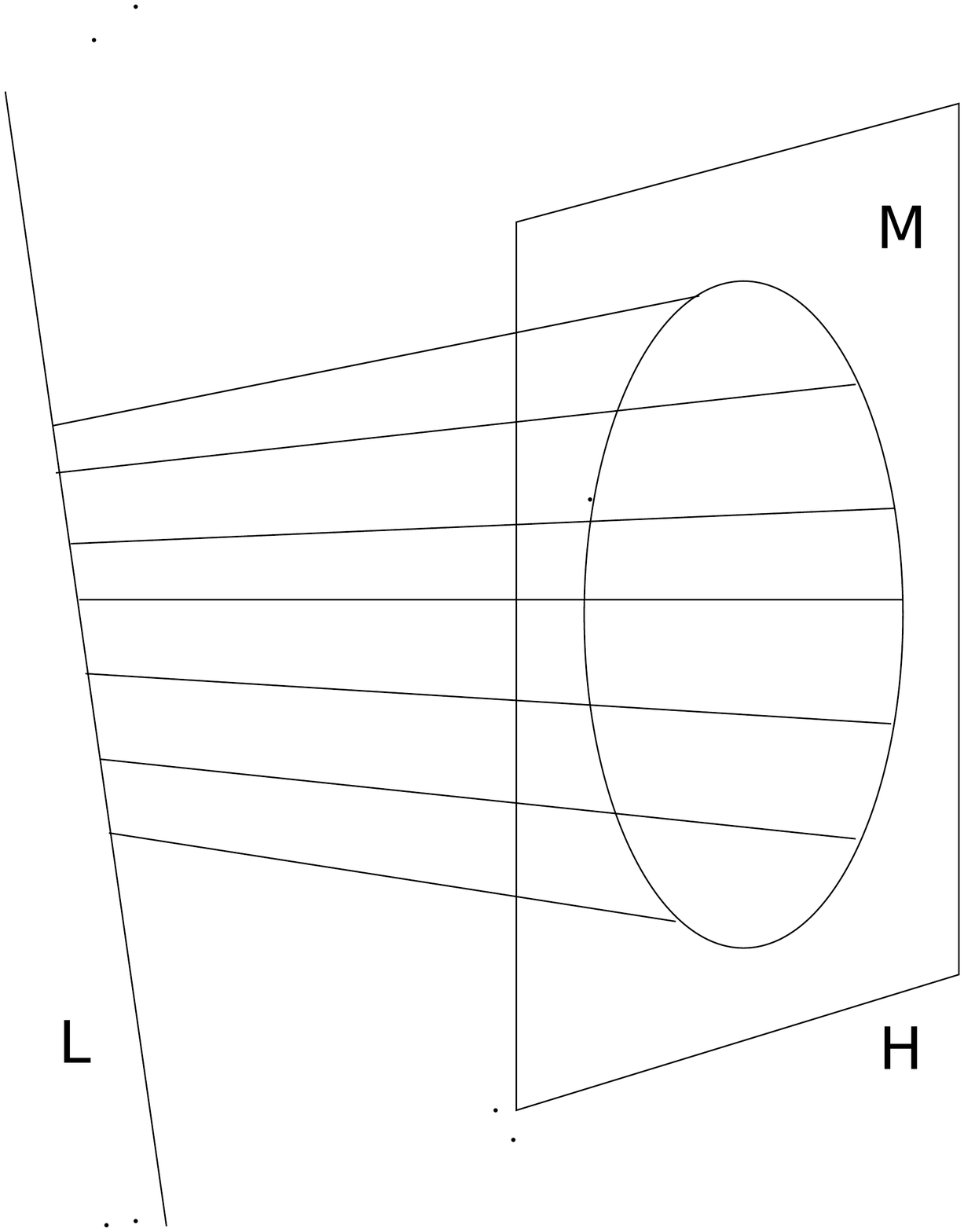}
	\caption{A degree three surface scroll $\scroll {2,1}$ made by
lines joining a line $L$ and a plane conic $M$ which spans a $2$-plane $H$.}
\end{figure}

The purpose of this section is not to prove that every scroll
appears as such, or even to define what a scroll is. Instead,
the purpose is to make sense of the statement
``the lines connecting two copies of $\bp^1$.''

This is a surprisingly tricky, but standard argument, which
is often glossed over. The key inputs are Grauert's theorem
and the universal property of the Grassmannian.

The idea will be to first construct a map $\bp^1 \rightarrow G(2,5)$,
and then use this map to describe a closed subscheme
$X \subset \bp^4$ whose closed points are in bijection with closed points
on the lines connecting points
on $M$ and $L$.
Finally, we show this scheme
in $\bp^4$ is smooth of dimension 2 and degree 3.

\ssec{Constructing a map to the Grassmannian}

In order to construct a map to the Grassmannian, $G(2,5)$ we will have to construct
a surjective map from the trivial sheaf of rank $5$ on $\bp^1$ to a sheaf
of rank $2$ on $\bp^1$.

To start, 
let $L$ be a line and let $H$ be a two plane spanning $\bp^4$. By spanning
$\bp^4$, we mean that there are no hyperplanes containing both $L$
and $H$.
Hence, $L \cap H = \emptyset$. Then, let $M$ be a smooth conic in $H$.
By assumption, we have maps $f_1:L \cong \bp^1$ and $f_2:M \cong \bp^1$, so these isomorphisms,
together with the embeddings of $L$ and $M$ into $\bp^4$ induce embeddings
of $L$ and $M$ into $\bp^1 \times \bp^4$. Let $L \cup M$ be the scheme theoretic disjoint union of $L$, $M$,
embedded in $\bp^1 \times \bp^4$. 
We have an exact sequence
\[\begin{tikzcd}
		0 \ar {r} & \sci_{L \cup M} \ar {r} & \sco_{\bp^4 \times \bp^1} \ar {r} & \sco_{L \cup M} \ar {r} & 0.
 \end{tikzcd}\]

Now, the next step will be to construct a rank 3 locally free sheaf whose fiber at a point will
be the hyperplanes containing the line joining $f_1(p)$ and $f_2(p)$,
and the dual locally free sheaf will will have fiber
consisting of the set of points in that line.

Let $\pi_1: \bp^1 \times \bp^4 \rightarrow \bp^1, \pi_2 : \bp^1 \times \bp^4 \rightarrow \bp^4$ be the natural projections.
Let us now twist the above exact sequence by $(\pi_2)^* \sco_{\bp^4}(1)$.
The pullback of an locally free sheaf is locally free, and
so we obtain an exact sequence
\begin{equation}
	\label{equation:product-sequence}	
\begin{tikzcd}
	0 \ar {r} & \sci_{L \cup M} \otimes  (\pi_2)^* \sco_{\bp^4}(1)\ar {r}
	& \sco_{\bp^4 \times \bp^1} \otimes  (\pi_2)^* \sco_{\bp^4}(1) \\
		\ar {r}{\eta} 
		& \sco_{L \cup M} \otimes  (\pi_2)^* \sco_{\bp^4}(1)\ar {r} & 0 
\end{tikzcd}
\end{equation}

Now, in order to obtain a map from $\bp^1$ 
to the Grassmannian, we will have to push this
forward to $\bp^1$.
Applying the pushforward functor, we obtain
an exact sequence

\begin{equation}
	\label{equation:pushforward-sequence}	
	\begin{tikzcd}
		0 \ar {r} & (\pi_1)_*\left(  \sci_{L \cup M}\otimes  (\pi_2)^* \sco_{\bp^4}(1)\right)\ar {r}
		& (\pi_1)_*\left( \sco_{\bp^4 \times \bp^1} \otimes  (\pi_2)^* \sco_{\bp^4}(1) \right) \\
		\ar {r} & (\pi_1)_*\left( \sco_{L \cup M} \otimes  (\pi_2)^* \sco_{\bp^4}(1)  \right) \ar {r} & R^1(\pi_1)_*\left(\sci_{L \cup M}\otimes  (\pi_2)^* \sco_{\bp^4}(1)  \right)
\end{tikzcd}
\end{equation}

Our map to the Grassmannian will eventually be determined by the dual
of the first three nonzero terms of the above exact sequence \eqref{equation:pushforward-sequence}, but 
there are a few things we have to check first.

First, we will show these three sheaves in ~\eqref{equation:product-sequence} are locally free by verifying
the hypotheses of Grauert's theorem (see \cite[Theorem 28.1.5]{vakil:foundations-of-algebraic-geometry}).
Observe that the projection
$\pi_1$ is proper.
There are many ways to see this, but one is by the cancellation theorem for proper maps 
\cite[Theorem 10.1.19(ii)]{vakil:foundations-of-algebraic-geometry}, because
both $\bp^1 \times \bp^4, \bp^1$ are proper over $\spec \bk$.
Next, note that $\bp^1$ is reduced and locally Noetherian. Next, we will check that
each of the three sheaves in \eqref{equation:product-sequence} is flat over $\bp^1$ and has constant
Hilbert polynomials on the fibers.

First, to check
$\left( (\pi_2)^* \sco_{\bp^4}(1) \right)$
is flat, note that it is a invertible sheaf on $\bp^4 \times \bp^1$. So,
by \cite[Exercise 24.2.D]{vakil:foundations-of-algebraic-geometry}, since $\bp^4 \times \bp^1$ is flat over $\bp^1$,
we obtain this sheaf is flat over $\bp^1$.
Furthermore, since the Hilbert polynomial of $\left( (\pi_2)^* \sco_{\bp^4}(1) \right)$
on each fiber over $\bp^1$ is that of a hyperplane in $\bp^4$.
Hence, the pushforward of this sheaf to $\bp^1$ is flat, by 
Grauert's theorem in the form of \cite[Theorem III.9.9]{Hartshorne:AG}.

Second, to check
$\sco_{L \cup M} \otimes  (\pi_2)^* \sco_{\bp^4}(1)$
is flat over $\bp^1$, note that for $t \in \bp^1$
the restriction of this sheaf to $\pi_1^{-1}(t)$ is
$\sco_{f_1^{-1}(t) \cup f_2^{-1}(t)} \otimes \sco_{\bp^4}(1)$.
This sheaf is supported at the two distinct points
$f_1^{-1}(t)$ and $f_2^{-1}(t)$,
that is, it is a union of two skyscraper sheaves. This evidently
has Hilbert polynomial 2. Then, since we are working over the reduced
scheme $\bp^1$, and all fibers have Hilbert polynomial $2$, the sheaf
is flat, again by \cite[Theorem III.9.9]{Hartshorne:AG}.

Hence, the second and third nonzero sheaves from 
~\eqref{equation:product-sequence} are locally free and satisfy
\begin{align}
	\nonumber
	H^0(\bp^1,(\pi_1)_* \scf) \otimes \kappa(q) \cong H^0(q, \scf|_{\pi_1^{-1}(q)}),
\end{align}
where $\scf$ ranges over the second and third nonzero sheaves from ~\eqref{equation:product-sequence}.

So, taking the fiber of $\eta$
over a closed point of $\bp^1$ we obtain a map
\begin{align}
	\nonumber
	H^0(\bp^4, \pi_2^* \sco_{\bp^4}(1)) \rightarrow H^0(\bp^4, \sco_{f_1^{-1}(t) \cup f_2^{-1}(t)} \otimes \sco_{\bp^4}(1))).
\end{align}
This map is surjective because there is a hyperplane in $\bp^4$ whose pullback vanishes on
$f^{-1}_1(t)$ but not on $f^{-1}_2(t)$, and visa versa. Hence, we
obtain that on the level of fibers of closed points, the sequence

\begin{equation}
	\label{equation:pushforward-left-sequence}	
	\begin{tikzcd}
		0 \ar {r} & (\pi_1)_*\left(  \sci_{L \cup M}\otimes  (\pi_2)^* \sco_{\bp^4}(1)\right)\ar {r} 
		& (\pi_1)_*\left( \sco_{\bp^4 \times \bp^1} \otimes  (\pi_2)^* \sco_{\bp^4}(1) \right) \\
		\ar {r} & (\pi_1)_*\left( \sco_{L \cup M} \otimes  (\pi_2)^* \sco_{\bp^4}(1)  \right) \ar {r} & 0
\end{tikzcd}
\end{equation}
is right exact, since the penultimate map is surjective. Since it holds
on the level of fibers of closed points, by Nakayama's lemma, it holds
on the level of stalks of closed points. 
Now, since the cokernel is supported on a closed set and not
supported on any closed points, it has empty support, so the
cokernel is $0$.

Therefore, 
$(\pi_1)_*\left(  \sci_{L \cup M}\otimes  (\pi_2)^* \sco_{\bp^4}(1)\right)$
is the kernel of a surjective map of locally free sheaves,
hence locally free.

So, we have finally checked that we have an exact sequence of locally free
sheaves. Furthermore, the fiber of the first sheaf consists of hyperplanes
through $f_1^{-1}(t), f_2^{-1}(t)$. Therefore, dualizing the sequence,
and working over dual projective space, we obtain that the fibers
of the corresponding dual sheaf are the points contained in all
such hyperplanes, or, in other words, the line through
$f_1^{-1}(t), f_2^{-1}(t)$.

Now, to show that this induced a map to the Grassmannian, it suffices
to show that, over $\bp^1$, the sheaf
\begin{align}
	\nonumber
	(\pi_1)_*\left( \sco_{\bp^4 \times \bp^1} \otimes  (\pi_2)^* \sco_{\bp^4}(1) \right) 
	\cong (\pi_1)_* (\pi_2)^* \sco_{\bp^4}(1)
\end{align}
is the trivial sheaf on $\bp^1$ of rank $5$. To do this, we produce
an explicit isomorphism. Let $x_0,\ldots, x_4$ be the coordinate functions
on $\bp^4$ and let $y_i := \pi_2^* x_i$ be their pullbacks. Then,
I claim that $y_i$ determine an isomorphism
$(\pi_1)_* (\pi_2)^* \sco_{\bp^4}(1) \cong \sco_{\bp^1}^{\oplus 5}$.
By the universal property of maps to the structure sheaf, a map
to $\sco_{\bp^1}^{\oplus 5}$ is determined by 5 global sections of the
sheaf. Furthermore, we see that these sections define a surjective
map on every fiber. Independence of the $y_i$ follows from the fact that
the $x_i$ are independent, implying
each fiber is isomorphic to $\bp^4$.
Now, since the map is surjective at closed points, it is surjective.

Hence, we have a surjective map of locally free sheaves of rank $5$. The
kernel is a locally free sheaf of rank 0, hence $0$, implying that
this map is an isomorphism, as desired.

\ssec{Construction of the variety}

So, we have constructed a map to the Grassmannian $\iota: \bp^1 \rightarrow G(2,5)$.
It satisfies the property that a point $p \in \bp^1$ is sent to
the line joining a point on $M$ and a point on $L$. It only remains
to realize the union of these lines as a subscheme of $\bp^4$.

For this, we invoke a technique ubiquitous in algebraic geometry.
We consider an incidence correspondence

\begin{align}
	\nonumber
	\Phi = \left\{ (p,\ell)\in \bp^4 \times G(2,5): p \in \ell, \ell \in \iota(\bp^1) \right\}.
\end{align}
\noindent
If $\pi: \Phi \rightarrow \bp^4$ is the natural projection, then the desired
rational normal scroll is simply $\pi(\Phi)$. Indeed, the closed points
of $\pi(\Phi)$ are those $p$ so that $p \in \ell$ with $\ell \in
\iota(\bp^1)$, as desired. There is one issue, which
is to show that this is a smooth subscheme. This issue will be
taken up in the next section.

However, before verifying smoothness, we 
stoop down to a level of detail usually not seen in an algebraic geometry text.
	We explain, in detail what the incidence correspondence $\Phi$ really means,
and why it is closed.

\begin{remark}
	\label{remark:}
I include this description of $\Phi$ because for a long time, I was confused about
the details of why incidence correspondences were closed, and what they mean.
\end{remark}

We now explain what $\Phi$ means.
First, consider $G(2,5) \rightarrow \bp^9$ via the Pl\"ucker embedding,
with coordinates $x_{ij}$ for $0 \leq i < j \leq 4$, and $\bp^4$
with coordinates $x_0, \ldots, x_4$. First, we know that the Pl\"ucker
embedding is a closed embedding, cut out by the five $4 \times 4$ Pfaffians
of a $5 \times 5$ skew symmetric matrix of linear forms, whose upper triangular part
consists of the 10 coordinates functions $x_{ij}$. So we have described
$\bp^4 \times G(2,5)$ as a closed subscheme of $\bp^4 \times \bp^9$.
Next, from the previous section, we have constructed a map
$\bp^1 \rightarrow G(2,5)$.

\begin{lemma}
	\label{lemma:closed-embedding}
	The map $\iota:\bp^1 \rightarrow G(2,5)$ constructed
	above is a closed embedding.
\end{lemma}
\begin{proof}
	Because $\bp^1, G(2,5)$ are both projective, they are, in particular,
	proper over $\spec \bk$. Thus, by the cancellation theorem for
	proper maps the map $\iota$ is also proper, and so $\iota(\bp^1)$
	is closed. To show it is a closed embedding,
	we only need show the map has degree 1, by Riemann Hurwitz.
	 This is true because
	in the construction, we saw that two different points of $\bp^1$ were
	mapped to two different lines.
\end{proof}

So, we now obtain that $\iota(\bp^1)$ is a closed subscheme of $G(2,5)$,
in particular, we can write down local equations defining it.
We have so far described the locus of points

\begin{align*}
	\left\{ (p,\ell) : p \in \bp^4, \ell \in \iota(\bp^1) \right\}
\end{align*}

To complete the construction of $\Phi$, it suffices to realize the condition
$p \in \ell$ as a closed condition via equations. 
Indeed, 
suppose we have coordinates $x_{ij}$ for the
Grassmannian and $x_i$ for projective space.
We can associate to a line in $\bp^4$ the hyperplane $A = \sum_i a_{ij} x_{ij}$ in
$\bp^9$ corresponding to a hyperplane in the Pl\"ucker
embedding of the Grassmannian $G(2,5)$.
We can also associate the 
to a point in $\bp^4$ the hyperplane $B = \sum_i a_i x_i$ 
The condition that the point is contained in the line
is the same as $A \wedge B = 0$, when thinking of the
$a_{ij}$ as $a_i \wedge a_j$. Then, grouping these
into four equations depending on the four basis
vectors gives four equations in the $x_k, x_{ij}$,
for each of the $a_i \wedge a_j \wedge a_k$ terms.

\ssec{The image of $\pi: \Phi\rightarrow \bp^4$ is a smooth, closed embedding of degree 3}

In the remainder of this section, our goal is to show $\pi$ is a smooth closed embedding of degree $3$.
This will complete the construction of degree $3$ surface scrolls.

For the remainder, let us denote $S := \pi(\Phi)$. We want to show that 
$S$ is a smooth surfaces of degree $3$. The first step will be to show it is
of degree $3$. Note that by \cite[Exercise 8.3.A]{vakil:foundations-of-algebraic-geometry},
to show $S$ is reduced, it suffices
to show $\Phi$ is reduced.

\begin{lemma}
	\label{lemma:phi-reduced}
	The incidence correspondence $\Phi$ is reduced. In particular, 
	$S$ is reduced.
\end{lemma}
\begin{proof}
	If we knew $\Phi$ were reduced, we would immediately obtain $S$ is reduced
	from \cite[Exercise 8.3.A]{vakil:foundations-of-algebraic-geometry},
	because $S$ is the image of a reduced scheme.
	First, note that $\iota(\bp^1)$ is reduced, by \cite[Exercise 8.3.A]{vakil:foundations-of-algebraic-geometry},
	because
	$\bp^1$ is reduced. Now, we have a natural projection $\Phi \ra \iota(\bp^1)$,
	so that all the fibers are reduced lines. Let us now examine this question
	locally on the level of rings. The map $\Phi \ra \iota(\bp^1)$ corresponds
	locally to a map of rings $g:B \ra A$ where $B$ is reduced, and the fibers over
	closed points of $B$ are reduced. We want to show $A$ is reduced. Suppose not. So, there is some $a \in A$ with $a^m = 0$. Then, choose a maximal ideal $\fm \in \spec A$.  We have $a \in \fm$. Let $\fp = g^{-1}(\fm)$. We obtain $A/g(\fp)$ is reduced.  So, the image of $a$ in $A/\phi(\fp)$ must be $0$ implying $a \in g(\fp)$.  To conclude, it suffices to show the map $g$ is injective, as then we would
	obtain $g^{-1}(a) = 0$, so $a = 0$.
	This follows from the following sublemma, \autoref{lemma:surjective-on-schemes-implies-injective-on-rings}.
\end{proof}
\begin{lemma}
	\label{lemma:surjective-on-schemes-implies-injective-on-rings} Suppose $g:B \ra A$ is a map of rings with $B$ reduced, so
	that the associated map of
	schemes $g^* \spec A \ra \spec B$ is surjective.
	Then, $g$ is injective.
\end{lemma}
\begin{remark}
	\label{remark:}
	The reducedness hypothesis on $B$ is necessary. For example, we can have
	the map $A \otimes \bk[t]/t^2 \ra A$ which is surjective on schemes but
	not injective on rings.
	
	The converse will hold if the extension $g$ is integral, which is precisely
	the going up theorem.
\end{remark}

\begin{proof}[Proof of \autoref{lemma:surjective-on-schemes-implies-injective-on-rings}]
	Suppose $g^*$ is surjective. Let $b \in B$ with $g(b) = 0$. Then, $g(b)$
	in $\fp$ for all $\fp \in \spec A$, and hence in $\phi^{-1}(\fp)$ for all
	$\fp \in \spec A$. In particular, since $g^*$ is surjective, 
	$a \in \fq$ for all $\fq \in \spec B$. But, the intersection of all primes
	is the nilradical, which is $0$ in a reduced ring. Hence, $b = 0$.
\end{proof}

\begin{lemma}
	\label{lemma:degree-3}
	The degree of $S$ is $3$. Further, let $r$ be a point on $L$. Then, the intersection of $S$ with the hyperplane
	spanned by $H$ and $r$, 
	$T = \overline{H, r}\subset \bp^4,$ is the union of $M$ and a line,
	$L_r$, with reduced scheme structure. Here, $L_r$ is a line
	joining $r$ and a point on $M$.
\end{lemma}
\begin{proof}
	From the construction of the surface, we know that $L, M \subset S$,
	and recall $M \subset H \cong \bp^2 \subset \bp^4$ is a conic.
	Consider the hyperplane $T$ spanned by
	$H$ and a point $p \in L$. We know there is some line $L_p$
	joining $L$ and a point in the conic, and so this line is necessarily contained
	in $T$. However, if there were any point $q \in S$ with $q \in T$, then there is some point
	$r \in L$ and a line $L_r$ joining $r$ and a point $s$ on the conic, which contains $q$.
	If $q \in T$, since $s \in T$, it follows that all of $L_r \in T$. In particular, $r \in T$.
	This implies that all of $L \subset T$, by Bezout's theorem, because two points of
	$L$ are in $T$, and both are degree $1$. This is, of course, a contradiction to the assumption
	that $L, H$ span $T$. Therefore, we obtain that $T \cap S = L_r \cup M$, at least
	set theoretically.
	To calculate the degree, we only need verify this intersection is generically reduced.
	First, at a point $x$ of $M$ other than $L_r \cap M$,
	we may note that the tangent space to $S$ at $x$ is not contained in
	$T$, since the line joining $x$ and the corresponding point on $L$ is not
	contained in $T$. Therefore, the intersection is transverse at $x$,
	and so the scheme structure on $M$ is generically reduced.
	To conclude, we only need know that the scheme structure on $L_r$ is generically
	reduced. This holds from Bertini's theorem, since the
	we can consider the pencil of hyperplanes parameterized by points $r$ on $L$,
	which contain both $r$ and $M$. A general member of this linear system
	will intersect $S$ smoothly away from its base locus by Bertini's theorem.
	But then, it follows that if $r$ was chosen generally, $L_r$ will be smooth away
	from $L_r \cap M$, and hence the intersection $T \cap S$ will be generically reduced.
	Therefore, since $M$ has degree $2$ and
	$L_r$ has degree $1$, $S \cap T$ has degree $3$. Ergo, $S$ has degree $3$.
\end{proof}

\begin{lemma}
	\label{lemma:flat-map-to-p1}
	Let $L_r$ be as in the statement of \autoref{lemma:degree-3}. The sheaf
	$\scl := \sco_S(L_r)$ defines a flat map $f_\scl: S \ra \bp^1$ so that each fiber over a point in
	$\bp^1$ is isomorphic to $\bp^1$. \end{lemma}
\begin{proof}
	First, consider the invertible sheaf $\scl := \sco_S(L_r)$.
	I claim we know at least two elements of $H^0(S,\scl)$. First, since $S \cap T$ is an effective
	divisor, there is the element $1$. Second, let
	$r, q \in L$ be a point, and let $T_q, T_r$ be the hyperplane spanned by $q, H$, with
	corresponding linear forms $F_q, F_r$.
	Then, $F_q/F_r$ defines a rational section of $\sco_{\bp^4}(1)$, whose restriction to
	$S$, call it $s := F_q/F_r|_S$, satisfies $s \in H^0(S, \scl)$, with divisor $\di s = L_q - L_r$.
	Now, we claim the rational functions $1,s$ span a subspace of $H^0(S, \scl)$ which
	defines a map $S \ra \bp^1$, whose fibers
	are the lines joining a point in $L$ to a point in $M$.
	First, we claim the linear system spanned by $1, s$ is basepoint free.
	To see this, note that $1$ only vanishes along $L_r$, while $s$ only vanishes along
	$L_q$. Therefore, the system is basepoint free, and we obtain a map
	$S \ra \bp^1$.
	Call this map $f_\scl$.
	Indeed, the fiber over a point $[\alpha, \beta] \in \bp^1$ is the vanishing locus
	of the linear form $\alpha \cdot 1 + \beta \cdot s = \frac{\alpha \cdot F_r + \beta \cdot F_q}{F_r}$.
	Letting $x_0, x_1$ be the linear coordinates on $L$ dual to the points $r,q$, if we let
	$t := \alpha x_0 + \beta x_1$, we have $L_t = f_\scl^{-1}([\alpha, \beta])$.
	Then, since $\bp^1$ is reduced, and all the fibers of $f_\scl$ are reduced lines,
	we obtain $f_\scl$ is a flat map because maps to a reduced scheme so that
	the fiber has constant Hilbert polynomial are flat
	by \cite[Theorem III.9.9]{Hartshorne:AG}.
\end{proof}
\begin{remark}
	\label{remark:}
	In fact, the fiber of $f_\scl$ over a point $p$ is a line in $\bp^3$ corresponding
	to the point of $G(2,5)$,
	determined by the fiber over $p$ in the exact sequence ~\eqref{equation:pushforward-left-sequence}.
\end{remark}

\begin{corollary}
	\label{corollary:}
	The scheme $S$ is integral and two dimensional.
\end{corollary}
\begin{proof}
	We have exhibited a flat map $f_\scl: S \ra \bp^1$ whose fibers are 1 dimensional lines.
	Then, by \cite[Exercise 11.4.C]{vakil:foundations-of-algebraic-geometry}, $S$ is integral
	and $2$ dimensional.
\end{proof}

Now, we can use our map $S \ra \bp^1$ to show $S$ is smooth.

\begin{theorem}
	\label{theorem:}
	$S$ is a smooth surface of degree $3$.
\end{theorem}
\begin{proof}
	We have already shown $S$ is an integral surface of degree $3$. We want to show $S$
	is smooth. Note that $\pi: S \ra \spec \bk$ factors through $f_\scl$. That is,
	\begin{equation}
		\nonumber
		\begin{tikzcd}
			S \ar {rr}{f_\scl} \ar {rd}{\pi} && \bp^1 \ar {ld} \\
			 & \spec \bk & 
		 \end{tikzcd}\end{equation}
	commutes. 
	Now, it is clear $\bp^1 \ra \spec \bk$ is smooth, so to complete the proof,
	we only need show $f_\scl$ is smooth, as then the composition will be smooth.
	For this, we use \cite[Theorem 25.2.2]{vakil:foundations-of-algebraic-geometry}, which tells us that to show $f_\scl$
	is smooth, it suffices to show it is locally finitely presented, flat,
	and the fibers are smooth varieties of dimension $n$.
	We have already seen \autoref{lemma:flat-map-to-p1} that
	$f_\scl$ is flat. Further, it is locally finitely presented
	by the cancellation property for locally finitely presented maps,
	since we know $S$ is a variety, hence locally finitely presented over $\spec \bk$.
	To complete the proof, we only need show the fibers are smooth of dimension 1.
	However, the fibers are reduced copies of $\bp^1$, as shown in 
	\autoref{lemma:flat-map-to-p1}, so we are done.
\end{proof}

\section[Constructing a scroll]{Constructing a scroll by joining rational normal curves}

In this section, we will generalize the construction given in \autoref{subsection:segre-cubic-surface-construction},
albeit in less detail.

\begin{proposition}
	\label{proposition:construction-of-scrolls-by-joining-curves}
	Fix $k \in \bz$ and a non-increasing sequence $a_1, \ldots, a_k$ with $a_k \geq 1$. Let $n := k-1 + \sum_{i=1}^k a_i$
	Let $H_1, \ldots, H_k$ be a collection of linear subspaces in $\bp^{n}$ 
	with $\dim H_i = a_i$ so that $H_1, \ldots, H_k$ together span $\bp^{n}$.
	Let $C_1, \ldots, C_k$ be rational normal curves with embeddings $\iota_i: C_i \ra \bp^n$ so that $C_i$ spans $H_i \subset \bp^n$
	and choose isomorphisms $f_i: C_i \ra \bp^1$.
	Now, let $D_i$ be the image of $C_i$ under the map
	\begin{align*}
		\iota_i \times f_i : C_i \ra \bp^{1} \times \bp^n.
	\end{align*}
	and define the projections
	\begin{equation}
		\nonumber
		\begin{tikzcd}
			\qquad & \bp^{n} \times \bp^1 \ar {ld}{\pi_1} \ar {rd}{\pi_2} & \\
			\bp^{1} && \bp^n 
		\end{tikzcd}\end{equation}
	Then, we have an exact sequence of locally free sheaves on $\bp^1$
	\begin{equation}
		\label{equation:scroll-as-map-from-line-to-Grassmannian}
		\begin{tikzcd}
			0 \ar {r} & (\pi_1)_* \left(\sci_{D_1 \cup \cdots \cup D_k} \otimes (\pi_2)^*\sco_{\bp^n}(1)\right) \\ 
			\ar {r} & (\pi_1)_*\left( \sco_{\bp^n \times \bp^1} \otimes (\pi_2)^* \sco_{\bp^n}(1)\right)  \\
			\ar {r} & (\pi_1)_* \left(\sco_{D_1 \cup \cdots \cup D_k} \otimes (\pi_2)^*\sco_{\bp^n}(1)\right) \ar {r} & 0.
		\end{tikzcd}\end{equation}
	where the first sheaf is of rank $k$ and the second is of rank $n+1$. The dual of ~\eqref{equation:scroll-as-map-from-line-to-Grassmannian} determines a map
	\begin{align*}
		\iota: \bp^1 \ra G(k,n+1)
	\end{align*}
	Then, define the incidence correspondence
	\begin{align*}
		\Phi := \left\{ (p, H) \in \bp^n \times G(k, n+1) : p \in H, H \in \iota(\bp^1) \right\} 
	\end{align*}
	with projections	
	\begin{equation}
		\nonumber
		\begin{tikzcd}
			\qquad & \Phi \ar {ld}{\pi} \ar {rd}{\eta} & \\
			\bp^n && G(k,n+1).
		\end{tikzcd}\end{equation}
	We have that $\pi(\Phi) \subset \bp^n$ is a smooth scroll of type $a_1, \ldots, a_k$.
\end{proposition}
\sssec*{Idea of Proof}
	The proof follows precisely the same argument as given throughout \autoref{subsection:segre-cubic-surface-construction}, but
	with the constants slightly modified.
	The idea is to use the $C_i$ to construct a map from $\bp^1$ to the Grassmannian,
	where a point on $\bp^1$ is sent to the plane spanned by the corresponding points on each $C_i$.
	Then, we can realize this curve in the Grassmannian as a subscheme of $\bp^n$ using an incidence
	correspondence. Finally, we can check the resulting scheme is smooth by observing that it maps
	to our original $\bp^1$ with smooth linear fibers.
\begin{proof}
First, in order to apply the universal property
of the Grassmannian to produce the map $\iota: \bp^1 \rightarrow G(k, n+1)$, we verify that the sequence of sheaves on $\bp^1$
given in 
~\eqref{equation:scroll-as-map-from-line-to-Grassmannian}
is a sequence of locally free sheaves, with the middle one
being trivial of rank $n+1$ and the first trivial of rank $k$.

We show the sheaves in 
~\eqref{equation:scroll-as-map-from-line-to-Grassmannian}
are locally free.
By Grauert's theorem, since $\bp^1$ is reduced, and the sheaves
before being pushed forward to $\bp^1$ are flat,
it suffices
to show that the three sheaves have Hilbert polynomials
which are the same on all fibers over closed points of $\bp^1$.
Further, because the kernel of a map of locally free sheaves is
locally free, we will only check the latter two have locally
constant Hilbert polynomials on fibers.
Then, the fiber over a closed point of
\begin{align*}
	(\pi_1)_*\left( \sco_{\bp^n \times \bp^1} \otimes (\pi_2)^* \sco_{\bp^n}(1)\right) 
\end{align*}
is $\sco_{\bp^n}(1)$, and is clearly independent of the point.
Also, the fiber over a closed point $t \rightarrow \bp^1$
of
\begin{align*}
(\pi_1)_* \left(\sci_{D_1 \cup \cdots \cup D_k} \otimes (\pi_2)^*\sco_{\bp^n}(1)\right)
\end{align*}
is a union of skyscraper sheaves over $f_1^{-1}(t), \ldots, f_k^{-1}(t)$,
and is therefore of rank $k$ everywhere, as desired.
This implies all sheaves in the sequence are locally free.
Note that the map between the last two nonzero sheaves in
~\eqref{equation:scroll-as-map-from-line-to-Grassmannian}
is surjective because it is surjective on fibers, since we can
find hyperplanes on $\bp^n$ containing $f_i^{-1}(t)$ but
not $f_j^{-1}(t)$ for $i \neq j$.

To obtain a map to the Grassmannian, we still need to show
the middle term of 
~\eqref{equation:scroll-as-map-from-line-to-Grassmannian}
is trivial of rank $n+1$. Note that the third sheaf has
rank $k$, as we saw above. Therefore, the dual of 
~\eqref{equation:scroll-as-map-from-line-to-Grassmannian}
will determine a map to $G(k,n+1)$ once we show the middle
term is trivial of rank $n+1$.
We know it has rank $n+1$, because that is the rank of its
fibers, so it suffices to check the middle term is trivial.
However, we get a set of trivializing sections given by
pulling back the coordinate functions of $\bp^n$. These are
independent because they are independent at every fiber.

We have now constructed the map to $G(k,n+1)$.
To complete the proof, it suffices to show
that $\pi(\Phi)$ is a smooth closed embedding of degree
$d := \sum_{i=1}^k a_i$.

First, we show that $\pi(\Phi)$ has degree $n$.
To see this, take a codimension $k-1$ plane $P$ spanned
by $H_1$ and $a_i$ general points on $D_i$ for $2 \leq i \leq k$.
The resulting intersection $\pi(\Phi) \cap P$ is the union
of the rational normal curve $D_1$, together with $a_i$
joining $D_1$ to $D_i$ for $2 \leq i \leq k$. So, if we showed
this intersection is reduced, we would know the degree of $\pi(\Phi)$
is $a_1 + (a_2 + \cdots + a_k)$, as claimed.

Suppose we knew $\pi(\Phi)$ is reduced. We will now show this implies
$\pi(\Phi) \cap P$ is generically reduced.
Then, if we consider the linear system of hyperplanes containing $H_1$,
Bertini's theorem implies that a general member of this linear system will intersect
$\pi(\Phi)$ smoothly, away from $H_1$, the base locus of the linear system.
Further, the intersection with $D_1 = H_1 \cap \pi(\Phi)$ is also generically reduced,
since the intersection is transverse
at any point $x$ on $D_1$ so that $P \cap \pi(\Phi)$ does not contain a line in $\pi(\Phi),$ (other than possibly
$D_1,)$ meeting $x$.

So, to show this intersection $\pi(\Phi) \cap P$ is reduced, it suffices to show
that $\pi(\Phi)$ is reduced. In turn, by \cite[Exercise 8.2.A]{vakil:foundations-of-algebraic-geometry},
it suffices to show $\Phi$ is reduced. By the same reasoning,
we know $\iota(\bp^1) \subset G(k,n+1)$ is reduced, and the fibers
of the map $\Phi \rightarrow \iota(\bp^1)$ are all reduced
$(k-1)$-planes. It follows that $\Phi$ is reduced, essentially
using \autoref{lemma:surjective-on-schemes-implies-injective-on-rings}.

To conclude, we only need show that $\pi(\Phi)$ is smooth.
For this, 
note that the map $\eta$ determines a map
$\Phi \rightarrow \iota(\bp^1) \cong \bp^1$ in which
all fibers are $(k-1)$-planes.
Let $Q = \pi^{-1}(t)$ for $t$ a closed point of $\bp^1$.
Then, the invertible sheaf $\sco_{\pi(\Phi)}(\pi(Q))$
determines a map to $\bp^1$ because it has a two dimensional
space of global sections, corresponding to the distinct fibers
of the map $\eta$. This determines a map $\overline \eta: \pi(\Phi) \rightarrow \bp^1$,
where each fiber is the fiber is a plane.

Since all fibers are planes, they have the same Hilbert
polynomial, and so the map is flat.
So, to show $\pi(\Phi)$ is smooth, we only need to show
that the map $\overline \eta$ is smooth, because then
the composition of $\overline \eta$ with $\bp^1 \rightarrow \spec \bk$
will also be smooth.
Now, by \cite[Theorem 25.2.2]{vakil:foundations-of-algebraic-geometry},
it suffices to show $\pi$ is locally finitely presented
and has smooth fibers of dimension $k-1$.
Note that $\pi$ is locally finitely presented by the cancellation
property for locally finitely presented maps, and we already
know the fibers are smooth $(k-1)$-planes.
\end{proof}

\section[Equivalence of descriptions of scrolls]{Equivalence of various descriptions of scrolls}

In this section, we will discuss various equivalent definitions of scrolls, with a particular focus on scrolls of minimal degree.
In particular, we will describe implications between the following descriptions. 
\begin{enumerate}
	\item[\customlabel{custom:swept-planes}{scrolls-1}]  the scheme swept out by joining rational normal curves with planes, via compatible isomorphisms, as constructed in \autoref{proposition:construction-of-scrolls-by-joining-curves},
	\item[\customlabel{custom:vector-bundle}{scrolls-2}] a projective bundle over $\bp^1$, together with a particular embedding in $\bp^n$, as defined
		in \autoref{definition:scroll},
	\item[\customlabel{custom:minors}{scrolls-3}] a variety cut out the the two by two minors of a two row matrix,
	\item[\customlabel{custom:Grassmannian}{scrolls-4}] a rational curve in a Grassmannian.

\end{enumerate}

The precise statement of these implications are given in \autoref{proposition:equivalence-vector-bundle-swept-planes},
 \autoref{proposition:equivalence-minors-swept-planes}, and \autoref{proposition:equivalence-Grassmannian-swept-planes}.

Note that some precision is required to
verify that the above are actually varieties of minimal degree.
Although the construction of \autoref{proposition:construction-of-scrolls-by-joining-curves}
was fairly arduous, it will prove easiest to verify that all other descriptions of scrolls are equivalent
to this one.
We start with a description of the equivalence of \ref{custom:swept-planes} and \ref{custom:vector-bundle}.

\begin{proposition}
	\label{proposition:equivalence-vector-bundle-swept-planes}
	A scroll $X \subset \bp^{k-1 + \sum_{i=1}^k a_i}$ of type $a_1, \ldots, a_k$ as defined in \autoref{definition:scroll}
	is isomorphic 
	to the scroll of type $a_1, \ldots, a_k$
	as defined by joining rational normal curves via $k$-planes
	in \autoref{proposition:construction-of-scrolls-by-joining-curves}.
	Further, every scroll as constructed in 
	\autoref{proposition:construction-of-scrolls-by-joining-curves},
	can be realized as such an embedding of a projective bundle on $\bp^1$.
\end{proposition}
\sssec*{Idea of Proof}
Starting with a scroll, which is abstractly $\bp \sce$ as in \autoref{definition:scroll},
we first construct the rational normal curves $D_i$ as corresponding to
the quotient map from $\sce$ to one of the direct sum components of $\sce$.
Then, we verify that the planes spanned by corresponding points on the $D_i$
are indeed the images of the fibers of $\pi:\bp\sce \ra \bp^1$.

\begin{proof}[Proof expanding that given in \protect{\cite{eisenbudH:on-varieties-of-minimal-degree}}]

Both in the description as a projective bundle over $\bp^1$ and as in the construction from	
\autoref{proposition:construction-of-scrolls-by-joining-curves},
the resulting scheme will be reduced. Therefore, to show these constructions
are equivalent, it suffices to show that the set of closed points
contained in the image agree. That is, it suffices to show
the projective bundle can be described
as $(k-1)$-planes joining rational normal curves, and conversely,
that given any set of rational normal curves with isomorphisms
to $\bp^1$, the planes joining the fibers determine the image of such a projective bundle.

We will start by showing the image of a projective bundle can be described
set theoretically as follows:
We will construct $k$ rational normal curves $D_1, \ldots, D_k$ with maps
$f_i :D_i \ra \bp^1$, so that the union of the $(k-1)$-planes joining
$f_1^{-1}(p), \ldots, f_k^{-1}(p)$ form the image of the projective bundle.
Observe that for each $i, 1 \leq i \leq k$, we have a surjection
\begin{align*}
	\pi_i : \sce \cong \oplus_{j=1}^k \sco_{\bp^1}(a_j) \rightarrow \sco_{\bp^1}(a_i).
\end{align*}
This induces an inclusion in the other direction of schemes
\begin{align*}
	\phi_i: \bp \sco_{\bp^1} (a_i) \ra \bp \sce.
\end{align*}
Observe 
\begin{align*}
	\phi_i^* \sco_{\bp \sce}(1) = \sco_{\bp \sco_{\bp^1}(a_i)}(1) \cong \sco_{\bp^1}(a_i).
\end{align*}
Define $D_i$ to be the image of $\bp^1$ embedded into 
$\bp^{k-1 + \sum_{i=1}^k a_i}$ via the linear system $\phi_i^* \sco_{\bp \sce}(1)$.
These copies of $\bp^1$ span all of $\bp^{k-1 + \sum_{i=1}^k a_i}$, because
the coordinates of this projective space correspond to the
$\sum_{i=1}^k (a_i+1)$ sections of $\sce$.
In particular, the $i$th direct summand of $\sce$ has $a_i+1$ sections generating
$\sco_{\bp \sco(a_i)}(1)$, corresponding
to the $a_i+1$ sections spanned by the image $D_i$.
This shows that $D_1, \ldots, D_k$ span $\bp^{k-1 + \sum_{i=1}^k a_i}$.
Further, we have natural maps $f_i: D_i \ra \bp^1$
given by the composition
$D_i \ra \bp \sce \ra \bp^1$.

To complete this direction of the proof, we only need show that the linear space
spanned by the preimages $f_i^{-1}(p)$ is precisely the image of a fiber of the
projective bundle.
However, we may note that taking $p \ra \bp^1$ an inclusion, we have a fiber product
\begin{equation}
	\nonumber
	\begin{tikzcd} 
		\bp^{k-1} \ar {r} \ar {d} & \bp \sce \ar {d} \\
		p \ar {r} & \bp^1.
	\end{tikzcd}\end{equation}
The top map corresponds to the restriction map of locally free sheaves
\begin{align*}
	\oplus_{i=1}^k \sco_{\bp^1}(a_i) \ra \oplus_{i=1}^k \sco_p (a_i)
\end{align*}
where the latter map is evaluation at the point $p$. 
This map sends $p$ to $f_i^{-1}(p)$, because
the restriction
\begin{align*}
	H^0(\bp^1, \sco_{\bp^1}(a_i)) \ra H^0(\bp^1, \sco_p(a_i))
\end{align*}
has kernel given by the $a_i$ dimensional space of global sections vanishing at $p$.
The intersection of the dual hyperplanes in $\sco_{\bp^1}(a_i)$ is
precisely $f_i^{-1}(p)$.
Therefore, we see that the fiber of $\bp \sce$ over $p \in \bp^1$
is a copy of $\bp^{k-1}$. We already see it contains
$f_1^{-1}(p), \ldots, f_k^{-1}(p)$, so if we show it is linearly embedded,
it will necessarily be the plane spanned by these $k$ points.
To see this, we have a restriction of global sections
\begin{align*}
	H^0(\bp \sce, \sco_{\bp \sce}(1)) \ra H^0(\bp \sce, \sco_{\bp \oplus_{i=1}^k \sco_p (a_i)}(1)).
\end{align*}
Note that the latter map has image which is $k$ dimensional, implying that the image of the fiber
is contained in a $(k-1)$-plane. Further, the global sections of 
$\sco_{\bp \sce}(1)$ vanishing on these $k$ points are precisely those containing the span
of $f_1^{-1}(p), \ldots, f_k^{-1}(p)$, and so their
intersection is the plane spanned by
$f_1^{-1}(p), \ldots, f_k^{-1}(p)$, as claimed.

This completes the description of the set theoretic locus of the image of $\bp \sce$
and shows it is a variety swept out by $k-1$ planes joining $k$ rational normal curves
which span the ambient projective space. 

Conversely, suppose we have a set of $k$ rational
normal curves $D_1, \ldots, D_k$, where $\deg D_i = a_i$,
inside $\bp V$, where $\dim V = k + \sum_{i=1}^k a_i$. Suppose further
we have maps from each of these curves to a fixed copy of $\bp^1$
and that these curves span $\bp V$.
Then, we may view the span of $D_i$ as
$\bp H^0(\sco_{\bp^1}(a_i))$.
As above, the image of
$\sce := \oplus_{i=1}^k \sco_{\bp^1}(a_i)$ under $\sco_{\bp \sce}(1)$
contains the rational curves $D_i$ as hyperplane sections inside the planes
corresponding to the inclusions
$H^0(\sco_{\bp^1}(a_i)) \ra H^0(\oplus_{i=1}^k \sco_{\bp^1}(a_i))$.
As shown above, this is indeed the image of a locally free sheaf
swept out by linear spaces joining the points in a given fiber of the map $f_i: D_i \ra \bp^1$.
\end{proof}

Next, we explain the sense in which \ref{custom:swept-planes} and \ref{custom:minors} are equivalent.
This will take some work, but we now give the proof assuming
\autoref{lemma:reduced-minors}, which says that the scheme cut out by
the two by two minors of a matrix of the coordinate linear forms is reduced.

\begin{proposition}
	\label{proposition:equivalence-minors-swept-planes}
	Let $k \in \bz, a_1 \geq a_2 \geq \cdots \geq a_k$ be integers with $a_k \geq 1$.
	Let 
	\begin{align*}
		\scroll {a_1, \ldots, a_k} \subset \bp^{\sum_{i=1}^k a_i + k - 1} \cong \proj k[x_{0,0}, \ldots, x_{0,a_1}, x_{1,0}, \ldots, x_{1,a_2}, \ldots, x_{k,0}, \ldots, x_{k,a_k}]
	\end{align*}
	be a rational normal scroll of type $a_1, \ldots, a_k$ as constructed in 
	\autoref{proposition:construction-of-scrolls-by-joining-curves}.
	Then, we can write $X$ as the scheme defined by the vanishing of the $2 \times 2$ minors of the matrix
	\begin{align*} M:=
	\begin{pmatrix}
		x_{1,0} & x_{1,1} & \cdots & x_{1,a_1 - 1} & x_{2,0} & \cdots & x_{2, a_2-1} & \cdots & x_{k, 0} & \cdots & x_{k, a_k-1}  \\
		x_{1,1} & x_{1,2} & \cdots & x_{1, a_1} & x_{2, 1} & \cdots & x_{2, a_2} & \cdots & x_{k,1} & \cdots & x_{k,a_k}
	\end{pmatrix}.
	\end{align*}
	Conversely, any scheme cut out by the two by two minors of the above matrix is isomorphic to $\scroll {a_1, \ldots, a_k}$
	as constructed in 
	\autoref{proposition:construction-of-scrolls-by-joining-curves}.
\end{proposition}
\begin{proof}[Proof assuming \autoref{lemma:reduced-minors}]
	To prove this proposition, we will verify two things: 
	\begin{enumerate}
\item First, we show that the underlying set of closed points of the variety defined by the minors of $M$
	are those swept out by the $(k-1)$-planes joining $k$ distinct rational normal curves of degree $a_1, \ldots, a_k$.
	Conversely, if we are given $k$ such rational normal curves of degree $a_1, \ldots, a_k$ together spanning
	all of $\bp^{k-1 + \sum_{i=1}^k a_i}$, then the set of points swept out by $(k-1)$-planes joining them
	form the support of closed points which are the vanishing of the minors of $M$.

\item Second, we show that the scheme cut out by the minors of $M$ is reduced.
			\end{enumerate}
	Since a projective reduced finite type scheme over $\spec \bk$ is determined by its closed points, verifying the above
	two assertions suffices.

	Note that the second item above is precisely the content of the following \autoref{lemma:reduced-minors}.
	So, to complete the proof, we only need	
	verify the first assertion.

	Let $V$ be the scheme cut out by the $2 \times 2$ minors of $M$. Note that we have a natural map
	$V \ra \bp^1$ given by taking the ratio of the first row to the second row. More precisely, on the patch
	of $\bp^n$ where one of $x_{i,j}, x_{i,j+1} \neq 0$, we define the map $\phi: V \cap D(x_{i,j}x_{i,j+1}) \ra \bp^1$ via the pair of divisors
	$[x_{i,j}:x_{i,j+1}]$. These maps agree on the intersection of their domain of definition,
	and hence patch to a map on $V$.
	Further, define the $a_k-1$ planes $H_i = \spn(x_{i,0}, \ldots, x_{i,a_i})$ for $1 \leq i \leq k$.
	Next, define $C_i := V \cap H_i$. Observe that $C_i$ is a rational normal curve of degree $a_i$. At least set theoretically,
	the map $\phi|_{C_i}:C_i \ra \bp^1$ is a bijection on closed points, as the closed points of this curve
	are precisely those of the form
	$x_{i,j} = a^{a_i-j} b^j$, which are mapped to $[a,b]$ under $\phi$.
	That $C_i$ is reduced follows from 
	\autoref{lemma:reduced-minors}.
	
	Now, consider the closed points in $\phi^{-1}([a,b])$. This is all points of the form
	\begin{align*}
		\left[ a^{a_1}p_1, a^{a_1-1}b p_1, \ldots, b^{a_1}p_1, a^{a_2} p_2, \ldots, b^{a_2} p_2, \ldots, a^{a_k}p_k, \ldots, b^{a_k} p_k \right]
	\end{align*}
	which precisely forms the plane spanned by $p_i = \phi|_{C_i}^{-1}([a,b])$. In other words, set theoretically,
	such a scroll is swept out by the $k-1$ planes joining the preimage in $C_i$ under $\phi|_{C_i}$ of $[a,b] \in \bp^1$.
	This shows that the set theoretic vanishing locus of the minors is precisely that constructed as in
	\autoref{proposition:construction-of-scrolls-by-joining-curves}.
	
	Conversely, if we start with rational normal curves $C_1, \ldots, C_k$ so that $C_i$ spans a hyperplane $H_i$
	of dimension $a_i - 1$, we can realize $C_i$ as the vanishing of the minors of the matrix
	\begin{align*}M_i :=
		\begin{pmatrix}
			x_{i,0} & x_{i,1} & \cdots & x_{i,a_i-1} \\
			x_{i,1} & x_{i,2} & \cdots & x_{i,a_i}
		\end{pmatrix}.
	\end{align*}
	This is easily seen because every rational normal curve can be written in the above form since
	all rational normal curves are related by an automorphism of the ambient projective space, since there is a unique
	invertible sheaf of degree $a_i$ on $\bp^1$.
	Again, the curve defined by the vanishing of the minors of $M_i$ is isomorphic to $\bp^1$. The isomorphism is given by taking the ratio of the two rows of the matrix.
	This is a bijection on closed points, and since the domain is reduced, the map is an isomorphism.
	Then, as above, the minors of the resulting matrix
	\begin{align*} M:=
	\begin{pmatrix}
		x_{1,0} & x_{1,1} & \cdots & x_{1,a_1 - 1} & x_{2,0} & \cdots & x_{2, a_2-1} & \cdots & x_{k, 0} & \cdots & x_{k, a_k-1}  \\
		x_{1,1} & x_{1,2} & \cdots & x_{1, a_1} & x_{2, 1} & \cdots & x_{2, a_2} & \cdots & x_{k,1} & \cdots & x_{k,a_k}
	\end{pmatrix}.
	\end{align*}
	precisely span the $(k-1)$-planes obtained by joining the corresponding points $\phi|_{C_i}^{-1}([a,b])$.
	To be precise, the map $\phi: V \ra \bp^1$ is 
	given by taking the ratio of the first and second rows of $M$.

\end{proof}

To complete the proof of \autoref{proposition:equivalence-minors-swept-planes}, we only
need prove \autoref{lemma:reduced-minors}.
We do this now, assuming
\autoref{lemma:determinantal-hilbert-function} and 
\autoref{lemma:determinantal-hilbert-function}.
	\begin{lemma}
		\label{lemma:reduced-minors}
		Let $Y$ denote the scheme cut out by the minors of $M$.
		Then $Y$ is reduced.
		Further,
		if we let $d = \sum_{i=1}^k a_i$ and has Hilbert function equal to
		\begin{align*}
			p_Y(m) = d \binom{m + k - 1}{k} + \binom{m+k-1}{k-1}.
		\end{align*}
	\end{lemma}
	\begin{remark}
		\label{remark:}
		One proof of this \autoref{lemma:reduced-minors} is given in
		~\cite[Lemma 2.1]{eisenbudH:on-varieties-of-minimal-degree}. However, this uses
		the somewhat involved tools of perfect and unmixed rings.
		Another proof is given in ~\cite[p.\ 71, Section II.3, Proposition]{ACGH:I}.
		We now give an alternate proof by comparing Hilbert functions.
	\end{remark}
	\sssec*{Idea of Proof of \autoref{lemma:reduced-minors}}
	The idea is to show that the Hilbert polynomial of a reduced scroll is the same as the
	Hilbert polynomial of one cut out by the minors of a two by two matrix of linear forms.
	In \autoref{lemma:determinantal-hilbert-function}, we explicitly compute the Hilbert polynomial
	of the ideal generated by the two by two minors of a matrix, using
	a somewhat straightforward combinatorial bijection.
	On the other hand, in \autoref{lemma:lower-bound-on-hilbert-function-of-scrolls},
	we relate the Hilbert polynomial of a scroll of dimension $k$ to a scroll
	of dimension $k-1$.
	We do this using the fact that the hyperplane section of a scroll is a scroll,
	as shown in \autoref{lemma:hyperplane-section-of-scroll}.
	\begin{proof}[Proof assuming \autoref{lemma:determinantal-hilbert-function} and \autoref{lemma:lower-bound-on-hilbert-function-of-scrolls}]
		Let $Y$ denote the scheme cut out by the minors of $M$ and let $X := Y_{red}$
		be the scheme $Y$ with the reduced scheme structure. We have a natural inclusion $X \ra Y$
		and to show $Y$ is reduced, it suffices to show this inclusion is an isomorphism of schemes.
		First, by ~\cite[Exercise 18.6.F(b)]{vakil:foundations-of-algebraic-geometry}, in order to show $X \ra Y$ is an isomorphism, it suffices
		to show the have the same Hilbert function.
	
		For convenience, define
		\begin{align*}
			f_{d,k}(m) := d \binom{m + k - 1}{k} + \binom{m+k-1}{k-1}.
		\end{align*}
		To complete the proof, it suffices to show
		\begin{align*}
			h_Y(m) = h_X(m) = f_{d,k}(m)
		\end{align*}
		for $m > 0$ 
		where $h_Y, h_X$ are the Hilbert functions of $X$ and $Y$ respectively.
		Further, because $X \ra Y$ is a closed embedding, we know $h_Y(m) \geq h_X(m)$ for $m > 0$ by
		~\cite[Exercise 18.6.F(a)]{vakil:foundations-of-algebraic-geometry}.
		So, it suffices to show $h_X(m) = f_{d,k}(m)$ for $m > 0$ and $h_Y(m) = f_{d,k}(m)$, because we will
		then obtain $f(m)= h_X(m) = h_Y(m) = f_{d,k}(m)$ for $m \gg 0$, and so all three functions
		must be equal.

		We will show this by induction on the dimension.
		Now, we can complete the proof of the computation of the Hilbert function of a scroll, while simultaneously
		verifying that $M$ is reduced. We will induct on the dimension of the scroll.
		The base case is when $k = 1$. In this case, we know the Hilbert polynomial is $h_X(m) = md + 1 = f_{d,1}(m)$,
		as this is the Hilbert polynomial of a rational normal curve in $\bp^{d-1}$.
		By \autoref{lemma:determinantal-hilbert-function}, this is also the Hilbert function of the minors of the matrix
		$M$, in the case $k = 1$.
		Now, inductively, assume that this lemma holds for scrolls of dimension $k - 1$.
		By \autoref{lemma:lower-bound-on-hilbert-function-of-scrolls}, we know
		that $f_{d,k}(m) = h_X(m)$. However, we also know from \autoref{lemma:determinantal-hilbert-function}
		that $h_Y(m) = f_{d,k}(m)$.
		So, $h_Y(m) = f_{d,k}(m) = h_X(m)$
		and so we obtain that $X \ra Y$ is an isomorphism and the Hilbert polynomial
		of a scroll of degree $d$ and dimension $k$ is $f_{d,k}(m)$.
	\end{proof}

	In order to prove \autoref{proposition:equivalence-minors-swept-planes}, we still have to prove
	both \autoref{lemma:determinantal-hilbert-function} and \autoref{lemma:lower-bound-on-hilbert-function-of-scrolls}.
	We start by showing \autoref{lemma:determinantal-hilbert-function}.
		\begin{lemma}
			\label{lemma:determinantal-hilbert-function}
			Letting $Y$ denote the scheme cut out by the minors of $M$,
			we have $h_Y(m) = f_{d,k}(m)$.
		\end{lemma}
		\begin{remark}
			\label{remark:}
				This proof relies on techniques from the theory of Gr\"{o}bner bases.
			For further reference on Gr\"{o}bner bases, see
			Eisenbud \cite[Chapter 15]{Eisenbud:commutativeAlgebra}
			and Cox--Little--O'Shea \cite[Chapter 2]{coxLS:ideals-varieties-and-algorithms}.
		\end{remark}
		
		\begin{proof}
			Let $I \subset S := \bk[x_{1,0}, \ldots, x_{k,a_k}]$ denote the
			ideal generated by the $2 \times 2$ minors of $M$. We are then looking
			for the Hilbert function of $I \subset k[x_{1,0}, \ldots, x_{k,a_k}]$.
			Choose the graded reverse lexicographic ordering on the monomials
			of $S$ with $x_{1,0} < \cdots < x_{k,a_k}$.
			With respect to this ordering. We may note that the generators of $I$ form a Gr\"{o}bner basis for $I$,
			as can be seen from the S-pair criterion for Gr\"{o}bner bases.
			This is a direct algebraic computation, as is carried out in
			\cite[Lemma 2.2]{andrewPU:divisors-on-rational-normal-scrolls}.

			Since the second graded piece of $I$ actually forms a Grobner basis
			for $I$, we obtain that the initial ideal of $I$ is generated by terms of the form
			$x_{i,j}x_{k,l}$ for all $x_{i,j},x_{k,l}$ so that either $i \neq k$ or
			($i = k$ and $|j-l| > 1$).
			Since the Hilbert function of the initial ideal of $I$ is equal to the Hilbert
			function of $I$, we can compute the Hilbert function of the initial ideal of $I$.
			The generators of the $m$th graded piece of
			$S/I$ are the images of the monomials $m$ of the form
			$m = x_{i_1,j_1}\cdots x_{i_m, j_m}$
			subject to the conditions that
			\begin{enumerate}
				\item $i_1 \leq i_2 \leq \cdots \leq i_m$ and $j_t < j_{t+1}$ if $i_t = i_{t+1}$
				\item there is a unique $1 \leq i \leq k$ so that 
					\begin{enumerate}
						\item whenever $i_t > i$ then $j_t = a_{i_t}$
						\item whenever $i_t < i$ then $j_t = 0$
						\item if $i = i_r = i_{r+1} = \cdots = i_s$
							then $j_s - j_r \leq 1$.
					\end{enumerate}
						\end{enumerate}
			These form a basis for the quotient $S/I$ because no such term has as a factor
			an element of the initial ideal of $I$, and every other degree $m$ monomial does.
			Finally, we can see there are $f_{d,k}(m)$ counting the number
			of monomials with the conditions prescribed above. 
			First, we will show that those monomials $m$ so that all variables
			in the expansion of $m$ are of the form $x_{i_t,0}$ (that is, $j_t = 0$ for all $t$)
			account for precisely $\binom{m+k-1}{k-1}$ such terms.
			Indeed, if all $j_t$ are $0$, then such monomials $m$ are in bijection
			with monomials of the form $x_{1,0}^{c_1} \cdots x_{k,0}^{c_k}$
			where the sum $\sum_{i=1}^k c_i = m$. There are $\binom{m+k-1}{k-1}$ such
			monomials.
			(One slick way of seeing this is to note that these polynomials form a basis
			for the degree $m$ polynomials on $\bp^{k-1} \cong \proj \bk[x_{1,0}, \ldots, x_{k,0}]$.)

			It remains to show that there are $d \cdot \binom{m+k-1}{k}$ remaining such
			monomials with some $j_t \neq 0$. To do this, we will find a bijection between
			such monomials and
			a choice $(S,\alpha)$ where $S$ is a size $k$ subset of $\{1, \ldots, m+k-1\}$ and $\alpha$
			is an element of the set of tuples 
			\begin{align*}
			\left\{(1,1), \ldots, (1,a_1), (2,1), \ldots, (2,a_2), \ldots, (k, 1), \ldots, (k,a_k)\right\}.
			\end{align*}
			Note that this set has size equal to $d$.
			Given a subset and 
			\begin{align*}
				S = \{n_1, \ldots, n_k\} \subset \{ 1, \ldots, m+k-1\}
			\end{align*}
			and a tuple $\alpha$,
			we will construct a monomial $m$ which is a product of variables which has variables of the form
			$x_{i,j}, x_{i,j+1}$ with $(i,j+1) = \alpha$.
			That is, if $\alpha = (i,j)$ then
			associate to $(S, \alpha)$ the monomial
			\begin{align*}
			x_{1,0}^{n_1-1} x_{2,0}^{n_2-n_1 - 1} \cdots x_{i-1, 0}^{n_{i-1}-n_{i-2}-1}, x_{i, j-1}^{n_i-n_{i-1}-1} x_{i,j}^{n_{i+1} -n_i} x_{i+1, a_{i+1}}^{n_{i+2} - n_{i+1}} \cdots x_{k, a_k}^{m+k-1 - n_k}.
			\end{align*}
			\begin{exercise}
				\label{exercise:}
				Complete the above monomial count by showing this construction given
				above is indeed a bijection.
			\end{exercise}
		\end{proof}
		We next prove the second lemma assumed in the proof of \autoref{proposition:equivalence-minors-swept-planes},
		although to prove this, we will
		assume \autoref{lemma:hyperplane-section-of-scroll}.

		\begin{lemma}
			\label{lemma:lower-bound-on-hilbert-function-of-scrolls}
			Assuming that $h_Z(m) = f_{d,k-1}(m)$ for all scrolls $Z$ of dimension $k-1$
			and degree $d$, We have $h_X(m) = f(m)$ for $m > 0$, where $X \in \minhilb d k$.
		\end{lemma}
		\begin{proof}[Proof assuming \autoref{lemma:hyperplane-section-of-scroll}]
			To prove this, we will induct on the dimension of the scroll.
			Using ~\autoref{lemma:hyperplane-section-of-scroll},
			we have that there exists a hyperplane $H$ with $H \cap X$ a scroll of degree $d$ and dimension $k-1$.
			By our inductive hypothesis, we know the Hilbert function of $H \cap X$.

			Now, we have an exact sequence
			\begin{equation}
				\nonumber
				\begin{tikzcd}
					0 \ar {r} &  \sco_X(m-1) \ar {r} & \sco_X(m) \ar {r} & \sco_{X \cap H}(m) \ar {r} & 0.
				\end{tikzcd}\end{equation}
			Note that this is also exact on global sections because 
			$H^1(X, \sco_X) = 0$ holds. 
			One way to see this is that
			$X$ is birational to $\bp^{\dim X}$ and cohomology of the structure
			sheaf is a birational invariant by
			\cite[Theorem 1]{chatzistamatiouR:higher-direct-images-of-the-structure-sheaf-in-positive-characteristic}.
			One can also show $H^1(X, \sco_X) = 0$
			via more elementary means (essentially by
			the Leray spectral sequence)
			using
			\autoref{lemma:cohomology-of-O(1)-as-dual-bundle},
			which we could prove now, but we choose
			to defer until later.

			So, we obtain
			\begin{align*}
				h^0(X, \sco_{X \cap H}(m)) = h^0(X, \sco_X(m)) - h^0(\sco_X(m-1)),
			\end{align*}
			or, in other words, 
			\begin{align*}
				f_{d,k-1}(m) = h_{X \cap H}(m) = h_X(m) - h_X (m-1).
			\end{align*}
			Now, observe that
			\begin{align*}
				f_{d,k-1}(m) = f_{d,k}(m) - f_{d,k}(m-1)
			\end{align*}
			as follows from elementary algebra, which boils down to
			Pascal's identity for binomial coefficients.
			Hence, we obtain
			\begin{align*}
				h_X(m) - h_X(m-1)  = f_{d,k}(m) - f_{d,k}(m-1)
			\end{align*}
			and $h_X(m) = f_{d,k}(m)$ for $m \leq 1$. So, we have that
			$h_X(m) = f_{d,k}(m)$ for all $m \in \bz$, as claimed.
		\end{proof}

		We finally complete our proof of \autoref{lemma:lower-bound-on-hilbert-function-of-scrolls},
		by showing \autoref{lemma:hyperplane-section-of-scroll}.

				\begin{lemma}
				\label{lemma:hyperplane-section-of-scroll}
				A general hyperplane section of a smooth scroll $X \in \minhilb d k$
				is a smooth scroll $Y \in \minhilb d {k-1}$.
			\end{lemma}
			\begin{proof}
				Say $X \cong \bp \sce$.
				Choose a hyperplane section corresponding to a quotient sheaf $\sce \ra \sce'$ of the same
				degree given by
				\begin{equation}
					\nonumber
					\begin{tikzcd}
						0 \ar {r} &  \sco_{\bp^1} \ar {r} & \sce \ar {r} & \sce' \ar {r} & 0.
					\end{tikzcd}\end{equation}
				Once one constructs a surjection $\sce \ra \sce'$ of invertible sheaves on $\bp^1$
				of the same degree, the kernel will necessarily be $\sco_{\bp^1}$.
				Then, $\sco_X(1)$ will restrict to $\bp \sce'$ and embed it as
				a scroll of dimension $k-1$ and degree $d$.
				
				To complete the proof, it suffices to show we can construct such a map $\sce \ra \sce'$.
				But, this is just a matter of writing down some simple polynomials. As just one example,
				we can start by just take the map
				$\sce \cong \oplus_{i=1}^k \sco_{\bp^1}(a_i) \ra \sco_{\bp^1}(a_1 + a_2) \bigoplus \left(\oplus_{i=3}^k \sco_{\bp^1}(a_i) \right)$
				which is the direct sum of the multiplication map 
				$\sco_{\bp^1}(a_1) \oplus \sco_{\bp^1}(a_2) \ra \sco_{\bp^1}(a_1 + a_2)$
				and the identity map on the later factors.
			\end{proof}
			This concludes the proof of \autoref{proposition:equivalence-minors-swept-planes}.
			We next state and prove an implication between the two descriptions of scrolls \ref{custom:swept-planes}
			and \ref{custom:Grassmannian}
	
\begin{proposition}
	\label{proposition:equivalence-Grassmannian-swept-planes}
	Any smooth scroll as in \autoref{proposition:construction-of-scrolls-by-joining-curves}
	is realized as the image $\pi(\Phi)$ where $\Phi$ is the incidence correspondence defined by some embedding
	$\iota: \bp^1 \ra G(k,d + k)$.
\end{proposition}
\begin{proof}
	The construction given in 
	\autoref{proposition:construction-of-scrolls-by-joining-curves}
	actually defines the scroll via the map $\bp^1 \ra G(k, d+k)$. 
	This completes the proof.
\end{proof}
\begin{remark}
	\label{remark:}
	Unlike \autoref{proposition:equivalence-vector-bundle-swept-planes} and \autoref{proposition:equivalence-minors-swept-planes},
	there is not an immediate converse to \autoref{proposition:equivalence-Grassmannian-swept-planes}.
	
	However, it is worth noting that whenever we have a nondegenerate embedding $\bp^1 \ra G(k,d+k)$,
	we can construct a variety in $\bp^{d+k-1}$ as follows:
	For any nondegenerate embedding $\iota: \bp^1 \ra G(k,d + k)$
	(meaning that the image does not lie in some $G(k, d+k-1) \subset G(k,d)$)
	we define the incidence correspondence
	\begin{align*}
		\Phi := \left\{ (p, H) \in \bp^n \times G(k, n+1) : p \in H, h \in \iota(\bp^1) \right\}
	\end{align*}
with projections	
\begin{equation}
	\nonumber
		\begin{tikzcd}
			\qquad & \Phi \ar {ld}{\pi} \ar {rd} & \\
			\bp^n && G(k,n+1).
		\end{tikzcd}\end{equation}
	Then, consider the image $\pi(\Phi) \subset \bp^{d+k-1}$.
	
	However, even if we impose that the degree of $\pi(\Phi)$ is $d$,
	not all such images will be smooth scrolls. For example, it is possible that the image can be a cone
	over a scroll of one lower dimension.
	Therefore, in order to obtain a converse to \autoref{proposition:equivalence-Grassmannian-swept-planes},
	one will essentially have to impose that the map $\bp^1 \ra G(k,d+k)$ be constructed as in \autoref{proposition:construction-of-scrolls-by-joining-curves}.
\end{remark}

\begin{remark}
	\label{remark:}
	It is interesting to note that in \autoref{proposition:equivalence-Grassmannian-swept-planes},
	we do not distinguish what type of scroll it is. That is, we do not specify the sequence
	$a_1, \ldots, a_k$ corresponding to the smooth scroll.
	This stands in stark contrast to the descriptions given in
	\autoref{proposition:equivalence-vector-bundle-swept-planes} and \autoref{proposition:equivalence-minors-swept-planes}.
	
	This leads to the following vague question:

	\begin{question}
		\label{question:}
		Can one distinguish the isomorphism class of a scroll in terms of the geometric properties of the
		embedding $\bp^1 \ra G(k, d +k)$ corresponding to that scroll from 
		\autoref{proposition:equivalence-Grassmannian-swept-planes}?
	\end{question}
	
\end{remark}

\chapter[Minimal degree preliminaries]{Preliminaries on varieties of minimal degree}
\label{section:preliminaries-on-varieties-of-minimal-degree}

\section{The definition of varieties of minimal degree}
In this chapter, we start by recalling the definition of a variety of minimal degree.
Following this, in \autoref{subsection:fano-schemes-of-scrolls}
we describe the Fano scheme of $k$-planes in a fixed scroll.
Finally, in \autoref{subsection:smoothness-of-the-hilbert-scheme-of-varieties-of-minimal-degree},
we show that all smooth varieties of minimal degree correspond to smooth points of the
Hilbert scheme.

\begin{definition}
	\label{definition:variety-of-minimal-degree}
	A variety $X \subset \bp^n$ of dimension $k$ and degree $d$ is of {\bf minimal degree} if
	$d + k = n + 1$ and $X$ is nondegenerate (i.e., not contained in a hyperplane).
\end{definition}
\begin{remark}[Reason for the name ``minimal degree'']
	\label{remark:}
	As the name suggests, there are no nondegenerate varieties with degree
	less than a variety of minimal degree.
	Indeed, one can prove this by induction on the dimension of such a variety.

	\begin{exercise}
		\label{exercise:reason-for-minimal-degree}
		Show that any nondegenerate variety $X \subset \bp^n$ of dimension $k$ and
		degree $d$ satisfies $d \geq n - 1 - k$.
		{\it Hint:} Use induction, on $k$. For the inductive step, intersect
		$X$ with a hyperplane not containing any of the generic points of $X$.
	\end{exercise}
\end{remark}

The main structure theorem for irreducible varieties of minimal degree is the following.

\begin{theorem}[~\cite{eisenbudH:on-varieties-of-minimal-degree}, Theorem 1]
	\label{theorem:classification-of-varieties-of-minimal-degree}	
	If $X$ is an irreducible nondegenerate variety over an algebraically closed field of minimal degree in $\bp^n$
	then $X$ is either smooth or a cone over a smooth irreducible nondegenerate variety of minimal degree in $\bp^n$.
	If $X$ is a smooth irreducible nondegenerate variety of minimal degree in $\bp^n$,
	then $X$ is either either a quadric hypersurface,
	a rational normal scroll, the image of
	$\bp^2$ under the 2-Veronese embedding
	$\nu_2(\bp^2) \rightarrow \bp^5$,
	or $\bp^n$ itself.
\end{theorem}

\section{Fano schemes of scrolls}
\label{subsection:fano-schemes-of-scrolls}

Before coming to the main result of the section, \autoref{lemma:linear-spaces-in-varieties-of-minimal-degree},
which describes the Fano scheme of a scroll,
we prove a quick but useful lemma about varieties of minimal degree, whose degree equals their dimension.

\begin{lemma}
	\label{lemma:segre-isomorphic-to-scroll}
	The scroll $\scroll {1^k}$ of dimension $k$ in $\bp^{2k-1}$ is, up to automorphism of $\bp^{2k-1}$,
	isomorphic to the Segre embedding $\bp^1 \times \bp^{k-1} \ra \bp^{2k-1}$.
\end{lemma}
\begin{proof}
	By the classification of minimal \autoref{theorem:classification-of-varieties-of-minimal-degree}, up to automorphism, there is a
	unique smooth variety of degree $k$ and dimension $k$ in $\bp^{2k-1}$. Since the Segre map is an embedding, the image is smooth.
	So, to complete the proof, we only need know that the Segre embedding of $\bp^1 \times \bp^{k-1} \ra \bp^{2k-1}$ has degree $k$,
	which follows from \cite[Proposition 2.11]{Eisenbud:3264-&-all-that}.
\end{proof}

We now come to the description of the Fano scheme of a scroll.
We devote the remainder of this section to proving this characterization

\begin{proposition}
	\label{lemma:linear-spaces-in-varieties-of-minimal-degree}
	Suppose $X \subset \bp^n$ is a smooth scroll of minimal degree $d$ and dimension $k$, where $\pi: X \cong \bp\sce \ra \bp^1$ is the projection, for $\sce$ a
	locally free sheaf
	on $\bp^1$. Further, let $X \cong \scroll {a_1, \ldots, a_s, 1^j}$ with $a_1 \geq \cdots \geq a_s > 1$.
	Suppose that $X \subset \bp^n$ is embedded so that it is the span of the planes joining rational normal curves
	$C_1, \ldots, C_s, L_1, \ldots, L_j$ where $L_1, \ldots, L_j$ are lines and $C_1, \ldots, C_s$ are rational normal curves
	of degree at least $2$. Let $t$ be an integer with $1 \leq t \leq k - 1$.
Then, the $t$-planes contained in $X$ are of one of the following two forms:
	\begin{enumerate}
		\item[\customlabel{custom:line-small}{1}] If $j > 1$, let $P \subset \bp V$ be $P = \bp W$ for a two dimensional
			subspace $W \subset V$, and let $L_1, \ldots, L_j$ be the projectivizations of planes $P_1, \ldots, P_j \subset V$.
			Then, $W \in \spn(P_1, \ldots, P_j)$.
			
		\item[\customlabel{custom:line-large}{2}] $P$ is contained in some $(k-1)$-plane which is the fiber of the projection map $\pi:X \ra \bp^1$.
	\end{enumerate}
	The Fano scheme of $t$-planes in $X$ is smooth. If $j \geq 1$ and $t= 1$, it has two irreducible components, corresponding to planes of type
	\autoref{custom:line-small} and \autoref{custom:line-large}. It has one irreducible component 
	corresponding to planes of type \autoref{custom:line-large} otherwise.
	The component of planes of type \autoref{custom:line-large} is isomorphic to
the Grassmannian bundle $\gbundle {t+1} {\sce}$ over $\bp^1$. If it exists, meaning $t = 1$ and $j \geq 1$,
the component of planes of type \autoref{custom:line-small} is isomorphic to $\bp^{j-1}$.
\end{proposition}
\begin{proof}
We prove this by first giving a set theoretic description of the points of the
Fano scheme of planes in $X$, then constructing it as a scheme, and finally showing that the scheme
structure is reduced.

\ssec{Set theoretic description}
\label{sssec:set-theoretic-description-of-lines}
We will start by set theoretically describing the Fano scheme of $t$-planes in $X$.

As in \autoref{proposition:equivalence-minors-swept-planes}, we may write any given scroll
$X$ as the vanishing locus of the $2 \times 2$ minors of
	\begin{align*} M:=
	\begin{pmatrix}
		x_{1,0} & x_{1,1} & \cdots & x_{1,a_1 - 1} & x_{2,0} & \cdots & x_{2, a_2-1} & \cdots & x_{k, 0} & \cdots & x_{k, a_k-1}  \\
		x_{1,1} & x_{1,2} & \cdots & x_{1, a_1} & x_{2, 1} & \cdots & x_{2, a_2} & \cdots & x_{k,1} & \cdots & x_{k,a_k}
	\end{pmatrix}.
	\end{align*}
	Let $D_i$ the the rational normal curve defined by the vanishing of the $2 \times 2$ minors of
	\begin{align*} M_i :=
		\begin{pmatrix}
			x_{i,0} & x_{i, 1} & \cdots & x_{i, a_i-1} \\
			x_{i,1} & x_{i,2} & \cdots & x_{i,a_i}
		\end{pmatrix}
	\end{align*}
	considered as a curve $D_i \subset \bp^{a_i}$.	
	Next, define $W_i \subset \bp^{d+k-1}$ to be the linear span of $C_i$.	
	Define the linear subspace of $\bp^{d+k-1}$
	\begin{align*}
		V_i :=V(x_{1,0}, \ldots, x_{i-1, a_{i-1}}, x_{i+1,0}, \ldots, x_{k, a_k})
	\end{align*}
	for $1 \leq i \leq k$.
	Consider the $k+1$ projections defined for $1 \leq i \leq k$ by
\begin{align*}
	\pi_i : \bp^{d+k-1} \setminus V_i  & \rightarrow \bp^{a_i}\\
	\left( x_{1,0}, \ldots, x_{k,a_k} \right) & \mapsto \left( x_{i, 0}, \ldots, x_{i, a_i} \right).
\end{align*}
Observe that the image of $X \setminus V_i$ is precisely the rational normal curve $D_i$. Now suppose we had a linear space
$P \subset X$. The image $\pi_i(P \setminus V_i) \subset D_i$ is necessarily a linear subspace of $\bp^{a_i}$ if it is nonempty.
Further, because the intersection of all $V_i$ is empty, for each closed point $p \in P$, there must be some projection map
$\pi_i$ defined at $p$. So, up to reordering, we may assume that $\pi_1$ is defined at some point $p \in P$.
If $\pi_1(P\setminus V_1)$ contains two points, it necessarily contains the line joining them, which implies $D_1$ contains
the line joining them. But, if $D_1$ is a rational normal curve containing a line, it must be a line.

Hence, there are two options. Either $\pi_1(P \setminus V_1)$ is a point contained in $C_1$ or else it is a line and $C_1$ is a 1-dimensional
linear subspace of $\bp^{d+k-1}$. In particular, if $C_i$ has degree at least $2$, the image must then be a point.

There are now two cases to analyze:
\begin{enumerate}
	\item There exists some $i$ so that $\pi_i(P \setminus V_i)$ is a line.
	\item For all $i$, we have $\pi_i(P \setminus V_i)$ is either empty or a point.
\end{enumerate}

By \autoref{lemma:set-theoretic-line-small}, in the first case $P$ is a line and is of type \ref{custom:line-small}.
By \autoref{lemma:set-theoretic-line-large}, in the second case $P$ is of type \ref{custom:line-large}.

\begin{lemma}
	\label{lemma:set-theoretic-line-small}
	If there is some $i$ so that $\pi_i(P \setminus V_i)$ is a line,
	then $P$ is a line. Further, 
	$P$ is a line of type ~\ref{custom:line-small}.
\end{lemma}
\begin{proof}
	 First, after reordering, assume that $\pi_i(P \setminus V_i)$
are lines for $1 \leq i \leq l$, but not for $i > l$. In particular, $L_1, \ldots, L_l$ are a subset of the lines $L_1, \ldots, L_j$ defining the scroll.
We first claim that $P \subset V_i$ for $i > l$. Suppose this were not the case. Then, $\pi_i(P \setminus V_i) \neq \emptyset$.
In particular, $\pi_i$ is defined on some open subset of the plane $P$. Choose two points $p_1, p_2 \in P \setminus V_1 \cup V_i$ mapping to distinct points
under $\pi_1$. Then, by the condition that the minors of $M$ all vanish, we see that $p_1, p_2$ must also map to distinct points under $\pi_i$,
contradicting the assumption that the image $\pi_i(P \setminus V_i)$ is not a line for $i > l$.
So, we conclude that the image of $\pi_i(P \setminus V_i) = \emptyset$ implying $P \subset V_i$ for $i > l$.

Next, the condition that all minors of $M$ vanish imply that if $p \notin V_i$,
for $1 \leq i \leq l$,
then $p \notin V_w$ for all $1 \leq w \leq l$ and the the value of $\pi_w(p)$ is defined
and $\pi|_{L_1} \pi_1(p) = \cdots = \pi|_{L_l} \pi_l (p)$.
Here, we are abusing notation by viewing $\pi|_{L_i}$ as the
uniquely determined isomorphism
from $\pi_i(L_i) \rightarrow \bp^1$ compatible with $\pi$.
Because of this, for all $p \in P$, we have that $p \notin V_i$ for $1 \leq i \leq l$, as the projection is defined at $p$.
Additionally, since the projection $\pi_1(p)$ uniquely determines the projection
$\pi_i(p)$ for $1 \leq i \leq l$, we obtain that $\pi_1$ is an isomorphism on $P$.
In particular, $P$ is a line.

To complete the proof, we check that $P$ is a linear combination of the lines $L_1, \ldots, L_l$.
To see this, suppose $p \in P$ has $\pi_1(p) = [1,0]$. Then, observe we have
$\pi_i(p) = [r_i, 0]$ for $1 \leq i \leq l$.
It follows that $P$ is the projectivization of 
a linear combination of the $2$-planes 
whose projectivizations are $L_1, \ldots, L_l$ in the ratio
$[1, r_2, r_3, \ldots, r_l]$ because all $2 \times 2$ minors of $M$ vanish at all
points $p \in P$. In this case, $P$ is of
type ~\ref{custom:line-small}.
\end{proof}

\begin{lemma}
	\label{lemma:set-theoretic-line-large}
	If for all $i$, we have that $\pi_i(P \setminus V_i)$ is either empty or a point,
	then $P$ is of type \ref{custom:line-large}.
\end{lemma}
\begin{proof}
	Suppose the nonempty projections, after possibly reordering the $C_i$, are the points $p_0, \ldots, p_l$ where $t \leq l \leq k$.
We wish to show that in this case, $P$ is contained in a fiber of the projection map.
Choose a point $q$ so that $q \notin V_1 \cup \cdots \cup V_l$. This holds true for a general $q \in P$.
We will show that all points of $P$ lie in the same fiber of the map to $\bp^1$ that
$q$ lies in. Because $M_1, \ldots, M_l$ are all nonzero at $q$, the condition that the minors of
$M$ all vanish at $q$ imply that the ratio of the first row to second row of $M_i$ at
$q$ are all the same.
Then, because the image $\pi_i(P \setminus V_i)$ consists of a single point, it must be
that all of $P$ is contained in the same fiber, because the open subset $P \setminus \cup_i V_i$ is contained in that fiber.
Therefore, $P$ is contained in the fiber of $\pi$, so it is of type ~\ref{custom:line-large}.
\end{proof}

\ssec{Scheme theoretic description}
\label{sssec:scheme-theoretic-description-of-lines}

To complete the proof, it suffices to give the scheme theoretic description.
We start by describing the reduced scheme structure of the Hilbert scheme, using the
set theoretic description above.
The reduced scheme structure for lines of type ~\ref{custom:line-small}
is $\bp^{j-1}$ by \autoref{lemma:reduced-scheme-small-lines}
while the reduced scheme structure for planes of
type ~\ref{custom:line-large}
is a Grassmannian bundle $\gbundle {t+1}{\sce}$ from
\autoref{lemma:reduced-scheme-large-lines}.

\begin{lemma}
	\label{lemma:reduced-scheme-small-lines}
	When $j \geq 1$ and $t = 1$, the reduced scheme structure
	of the component of the Fano scheme whose closed points correspond to lines
	of type ~\ref{custom:line-small} is $\bp^{j-1}$.
\end{lemma}
\begin{proof}
	Then, note that we
have a $(2j-1)$-plane $\Lambda$ spanned by the lines $L_1, \ldots, L_j$.
Then, we have a scroll $\scroll {1^j} \cong X \cap \Lambda$.
Further, by \autoref{lemma:segre-isomorphic-to-scroll},
we have $\scroll {1^j} \cong \bp^1 \times \bp^{j-1}$.
We now construct a closed subscheme of $\bp^{j-1} \times \Lambda$
flat over $\bp^{j-1}$, which will correspond to a component of the
Hilbert scheme, isomorphic to $\bp^{j-1}$.
Choose a basis of $\bp^{j-1}$ given by $x_1, \ldots, x_j$, and then
define the map
\begin{align*}
	X \cong \bp^{j-1} \times \bp^1 &\ra \bp^{j-1} \times \bp^{d+k-1} \\
	\left( [x_1, \ldots, x_j ], [a,b] \right) &\mapsto \left( \left[ x_1, \ldots, x_j \right], \left[ x_1 a, x_1b, x_2a, x_2b, \ldots, x_ja, x_j b, 0, \ldots, 0 \right] \right)
 \end{align*}
where we have reordered coordinates so that $L_1, \ldots, L_j$ appear as the first $j$ rational curves in the matrix representation
of $X$.
In particular, the image of this
map is a subscheme flat over the projection $X \cong \bp^1 \times \bp^{j-1} \ra \bp^{j-1}$.
It is flat by ~\cite[Theorem III.9.9]{Hartshorne:AG} because all fibers have the
same Hilbert polynomial and the base is reduced. This exhibits a subscheme of the Hilbert
scheme isomorphic to $\bp^{j-1}$ with closed points corresponding to lines of type ~\ref{custom:line-small}.
\end{proof}

\begin{lemma}
	\label{lemma:reduced-scheme-large-lines}
	The reduced scheme structure of the component of the Fano scheme
	whose closed points correspond to planes of type ~\ref{custom:line-large}
	is a Grassmannian bundle $\gbundle {t+1}{\sce}$.
\end{lemma}
\begin{proof}
	Let $\scs$ be the projectivization of the universal subbundle
of $\gbundle {t+1}{\sce}$, whose fiber over a point is the corresponding line. See ~\cite[Subsection 3.2.3]{Eisenbud:3264-&-all-that}
for a more detailed description of this bundle over the Grassmannian.
To do this, we construct a subscheme
\begin{align*}
	\phi: \scs \ra \gbundle{t+1}{\sce}\times \bp^{d+k-1}
\end{align*}
sending a line $L$ which lies in a fiber of the map $\pi$
to the pair of that line and the the line inside $\bp^{d+k-1}$.

This map can be constructed analogously to the map $\bp^{j-1} \times \bp^1 \ra \bp^{j-1} \times \bp^{d+k-1}$
above.

\begin{exercise}
	\label{exercise:}
	Write down the map $\phi$ explicitly.
	{\it Possible approach:} One method is to describe this map locally over an trivialization of the
	Grassmannian bundle, and check the transition functions. 
\end{exercise}

The resulting map $\scs \ra \gbundle{t+1}{\sce}$ gotten by composing
$\phi$ with the projection $\gbundle{t+1}{\sce} \times \bp^{d+k-1} \ra \bp^{d+k-1}$
is flat by ~\cite[Theorem III.9.9]{Hartshorne:AG}
because the fibers all have the same Hilbert polynomial and the base $\gbundle{t+1}{\sce}$
is reduced.
\end{proof}

\ssec{Reduced scheme structure}
\label{sssec:reduced-scheme-structure-of-lines}

In ~\autoref{sssec:scheme-theoretic-description-of-lines},
we constructed two subschemes of the Fano
scheme of planes in $X$. The closed points of the first correspond bijectively
to those points of the type ~\ref{custom:line-small}.
The closed points of the second component correspond bijectively to those of the type ~\ref{custom:line-large}.

Then, by ~\autoref{sssec:set-theoretic-description-of-lines},
these planes correspond to all planes in the Fano
scheme of planes in $X$. To complete the argument, it suffices to show that the Fano scheme
of planes is smooth, and hence automatically reduced. To do this, we will show
\begin{align*}
	h^0(P, N_{P/X}) \leq
	\begin{cases}
		j-1 = \dim \bp^{j-1}  & \text{ if } P \text{ is of type ~\ref{custom:line-small}} \\
		1 + (t+1)(k-t-1) = \dim \gbundle{t+1}{\sce} & \text{ if } P \text{ is of type ~\ref{custom:line-large}} \\
	\end{cases}
\end{align*}
It then follows that the Fano scheme is smooth because
$H^0(P, N_{P/X})$ is canonically identified with the tangent space to the
Fano scheme (i.e., the Hilbert scheme) using \cite[Theorem 1.1(b)]{Hartshorne:deformation}.
Note that the reverse inequality holds automatically
because we have shown that the reduced components of the Fano scheme have
the above dimensions.
Therefore, this
computation will show equality actually holds, and so
this Fano scheme is smooth.
In particular, this also implies if $t = 1$ and $j\neq 0$,
the two irreducible components
are connected components, because a smooth point lies on a unique connected
component.

Note also that because the two sets of planes from ~\ref{custom:line-small} and ~\ref{custom:line-large}
are always disjoint, this will imply that the Fano scheme of planes
in $X$ has two irreducible components when $j \geq 1$ and $t = 1$, and one
irreducible component otherwise.

So, to complete the proof of \autoref{lemma:linear-spaces-in-varieties-of-minimal-degree}, we now have two cases, depending on whether the plane is of type
~\ref{custom:line-small} or ~\ref{custom:line-large}.
These are covered in \autoref{lemma:smooth-line-small-fano-scheme}
and \autoref{lemma:smooth-line-large-fano-scheme} respectively.
\end{proof}
\begin{lemma}
	\label{lemma:smooth-line-small-fano-scheme}
	Let $X$ be a scroll of dimension $k$,
	with $X \cong \scroll {a_1, \ldots, a_s, 1^j}$ as in \autoref{lemma:linear-spaces-in-varieties-of-minimal-degree}.
	Let $P \subset X$ be a $t$-plane contained in $X$,
	with inclusion $\eta: P \rightarrow X$
Suppose that $P$ is of type ~\ref{custom:line-small}.
	Then, 
	\begin{align*}
		h^0(P, N_{P/X}) \leq j-1.
	\end{align*}
\end{lemma}
\sssec{Idea of proof of \autoref{lemma:smooth-line-small-fano-scheme}}
This proof is simply a matter of chasing exact sequences.
We first use the normal exact sequence to relate $N_{P/X}$
to $T_X|_P$. Then, we use the tangent exact sequence to relate
$T_X|_P$ to $T_\pi|_P$. Then, we use the Euler exact sequence
to relate $T_\pi|_P$ to $\pi^* \sce (1)|_P$, which we can finally
compute.
\begin{proof}
	Note that in this case, $P$ is necessarily a line
	by \autoref{lemma:set-theoretic-line-small}.
	Since $P$ and $X$ are both smooth, the sequence
	\begin{equation}
		\nonumber
		\begin{tikzcd}
			0 \ar {r} &  T_P \ar {r} & T_X|_P \ar {r} & N_{P/X} \ar {r} & 0 
		\end{tikzcd}\end{equation}
	is exact.
	Note that $P \cong \bp^1$ and so $H^1(P, T_P) \cong H^1(\bp^1, \sco_{\bp^1}(2)) = 0$. Therefore, we have an exact sequence on global
	sections
	\begin{equation}
		\nonumber
		\begin{tikzcd}
			0 \ar {r} &  H^0(P, T_P) \ar {r} & H^0(P, T_X|_P) \ar {r} & H^0(P, N_{P/X}) \ar {r} & 0.
		\end{tikzcd}\end{equation}
	The first term is three dimensional. Therefore, to conclude,
	we only have to show
	\begin{align*}
		h^0(P, T_X|_P) \leq (j - 1) + 3
	\end{align*}

Next, since $\pi$ is a smooth map, the relative tangent sequence
is exact, and, restricting this to $P$, we have the exact sequence
\begin{equation}
	\nonumber
	\begin{tikzcd}
		0 \ar {r} &  T_\pi|_P \ar {r} & T_X|_P \ar {r} & \pi^* T_{\bp^1}|_P \ar {r} & 0.
	\end{tikzcd}\end{equation}
Observe that $\pi^* T_{\bp^1}|_P = \eta^* \pi^* T_{\bp^1}$. However, by construction of the map $\pi$
using the description of a scroll as the variety swept out by planes joining a collection of lines
given in \autoref{proposition:equivalence-vector-bundle-swept-planes}, the composition
$\pi \circ \eta$ is an isomorphism. Therefore,
$h^0(\bp^1, \pi^* T_{\bp^1}|_P) = h^0(\bp^1, T_{\bp^1}) = 3$.
Therefore, we get a left exact sequence on cohomology implying
\begin{align*}
	h^0(P, T_X|_P) &\leq h^0(P, T_\pi|_P) + h^0(P, \pi^* T_{\bp^1}|_P) \\
	&= h^0(P, T_\pi|_P) + 3
\end{align*}
So, to complete the proof, it suffices to show
\begin{align*}
	h^0(P, T_\pi|_P) = j-1.
\end{align*}

For this, we will use the Euler exact sequence \cite[Theorem 4.5.13]{brandenburg:tensor-categorical-foundations-of-algebraic-geometry} (see also
\cite[Subsection 21.4.9]{vakil:foundations-of-algebraic-geometry})
\begin{equation}
	\nonumber
	\begin{tikzcd}
		0 \ar {r} &  \sco_{\bp \sce} \ar {r} & \pi^*(\sce^\vee)(1)  \ar {r} & T_\pi \ar {r} & 0.
	\end{tikzcd}\end{equation}
Restricting this to $P$, the resulting sequence will be exact
on global sections because $H^1(P, \sco_{\bp \sce}) = H^1(P, \sco_P) = 0$, and taking global sections, we have
\begin{align*}
	h^0(P, T_\pi|_P) &= h^0(P, \pi^*(\sce^\vee)(1)) - h^0(P, \sco_P) \\
	&= h^0(P, \pi^*(\sce^\vee)(1)|_P) - 1.
\end{align*}

Therefore, to complete the proof, we only need show
\begin{align*}
h^0(P, \pi^*(\sce^\vee)(1)|_P) = j.
\end{align*}
But now, observe that
\begin{align*}
	\pi^*(\sce^\vee)(1)|_P &\cong \eta^* (\pi^*(\sce^\vee)(1))\\
	&\cong \eta^* \pi^*(\sce^\vee) \otimes \eta^* \sco_{\bp \sce}(1) \\
	&\cong \sce^\vee \otimes \sco_P(1),
\end{align*}
where here we are crucially using that $\pi \circ \eta: P \rightarrow \bp^1$ is an isomorphism.
Now, writing
\begin{align*}
	\sce \cong \sco_P(a_1) \oplus \cdots \oplus \sco_P(a_s) \oplus \bigoplus_{i = 1}^j \sco_P(1),
\end{align*}
with $a_1 \geq \cdots \geq a_s > 1$, we obtain that
\begin{align*}
	\sce^\vee \otimes \sco_P(1) & \cong
\sco_P(1-a_1) \oplus \cdots \oplus \sco_P(1-a_s) \oplus \bigoplus_{i = 1}^j \sco_P.
\end{align*}
Because $a_1 \geq \cdots \geq a_s > 1$,
we obtain that
\begin{align*}
	h^0(P, \sce^\vee \otimes \sco_P(1)) = j,
\end{align*}
as desired, completing the proof.
\end{proof}

\begin{lemma}
	\label{lemma:smooth-line-large-fano-scheme}
	Let $P \subset X$ be a $t$-plane contained in $X$, a scroll
	of dimension $k$,
	with $X \cong \scroll {a_1, \ldots, a_s, 1^j}$ as in \autoref{lemma:linear-spaces-in-varieties-of-minimal-degree}.
	Suppose that $P$ is of type ~\ref{custom:line-small}.
	Then, 
	\begin{align*}
		h^0(P, N_{P/X}) \leq 1+(t+1)(k-t-1).
	\end{align*}
\end{lemma}
\sssec{Idea of proof of \autoref{lemma:smooth-line-large-fano-scheme}}
We prove this via a method analogous to \autoref{lemma:smooth-line-small-fano-scheme}. That is
we first use the normal exact sequence to relate $N_{P/X}$
to $T_X|_P$. Then, we use the tangent exact sequence to relate
$T_X|_P$ to $T_\pi|_P$. Then, we use the Euler exact sequence
to relate $T_\pi|_P$ to $\pi^* \sce (1)|_P$, which we can finally
compute.
\begin{proof}
	Since $P$ and $X$ are both smooth, the sequence
	\begin{equation}
		\nonumber
		\begin{tikzcd}
			0 \ar {r} &  T_P \ar {r} & T_X|_P \ar {r} & N_{P/X} \ar {r} & 0 
		\end{tikzcd}\end{equation}
	is exact.
	Note that by the Euler exact sequence for
	$P \cong \bp^t$,
	we have	that
	\begin{align*}
		h^0(P, T_P) &= (t+1)^2 - 1 \\
		h^1(P, T_P) &= 0.
	\end{align*}
	Therefore, we have an exact sequence on global
	sections
	\begin{equation}
		\nonumber
		\begin{tikzcd}
			0 \ar {r} &  H^0(P, T_P) \ar {r} & H^0(P, T_X|_P) \ar {r} & H^0(P, N_{P/X}) \ar {r} & 0.
		\end{tikzcd}\end{equation}
	To conclude our proof, we only need show
	\begin{align*}
		h^0(P, T_X|_P) &= h^0(P, N_{P/X}) + h^0(P, T_P) \\
		&= h^0(P, N_{P/X}) + \left( (t+1)^2 - 1 \right) \\
		&\leq (1 + (t+1)(k-t-1)) + \left( (t+1)^2 - 1 \right) \\
		&= tk + k.
	\end{align*}
Next, since $\pi$ is a smooth map, the relative tangent sequence
is exact, and, restricting this to $P$, we have the exact sequence
\begin{equation}
	\nonumber
	\begin{tikzcd}
		0 \ar {r} &  T_\pi|_P \ar {r} & T_X|_P \ar {r} & \pi^* T_{\bp^1}|_P \ar {r} & 0.
	\end{tikzcd}\end{equation}
Observe that $\pi^* T_{\bp^1}|_P \cong \sco_P$ because
$P$ lies in a fiber of $\pi$.
Therefore, we get a left exact sequence on cohomology implying
\begin{align*}
	h^0(P, T_X|_P) &\leq h^0(P, T_\pi|_P) + h^0(P, \pi^* T_{\bp^1}|_P) \\
	&= h^0(P, T_\pi|_P) + 1
\end{align*}
So, to complete the proof, it suffices to show
\begin{align*}
	h^0(P, T_\pi|_P) = tk + k - 1.
\end{align*}

For this, we will use the Euler exact sequence, \cite[Theorem 4.5.13]{brandenburg:tensor-categorical-foundations-of-algebraic-geometry} (see also
\cite[Subsection 21.4.9]{vakil:foundations-of-algebraic-geometry})
\begin{equation}
	\nonumber
	\begin{tikzcd}
		0 \ar {r} &  \sco_{\bp \sce} \ar {r} & \pi^*(\sce^\vee)(1)  \ar {r} & T_\pi \ar {r} & 0.
	\end{tikzcd}\end{equation}
Restricting this to $P$, the resulting sequence will be exact
on global sections because $H^1(P, \sco_{\bp \sce}) = H^1(P, \sco_P) = 0$. Taking global sections, we have
\begin{align*}
	h^0(P, T_\pi|_P) &= h^0(P, \pi^*(\sce^\vee)(1)) - h^0(P, \sco_P) \\
	&= h^0(P, \pi^*(\sce^\vee)(1)|_P) - 1.
\end{align*}

Therefore, to complete the proof, we only need show
\begin{align*}
h^0(P, \pi^*(\sce^\vee)(1)|_P) = tk + k.
\end{align*}
But now, observe that
\begin{align*}
	\pi^*(\sce^\vee)(1)|_P &\cong \eta^* (\pi^*(\sce^\vee)(1))\\
	&\cong \eta^* \pi^*(\sce^\vee) \otimes \eta^* \sco_{\bp \sce}(1) \\
	&\cong \sco_P^k \otimes \sco_P(1) \\
	&\cong \sco_P(1)^k
\end{align*}
where here we are crucially using that $\eta^* \pi^* \scf \cong \sco_P^{\rk \scf}$ because $P$ is contained in a fiber of $\pi$.
But then, we have
\begin{align*}
	h^0(P, \pi^*(\sce^\vee)(1)|_P) &= h^0(P, \sco_P(1)^k) \\
	&= k(t+1)
\end{align*}
as desired, completing the proof.
\end{proof}

\section[Smoothness of minimal degree Hilbert schemes]{Smoothness of the Hilbert scheme of varieties of minimal degree}
\label{subsection:smoothness-of-the-hilbert-scheme-of-varieties-of-minimal-degree}

In this section, we will check that for $X$ a smooth variety of minimal degree, $X$ is a smooth
point of the Hilbert scheme. The primary tactic for showing this is to show that $H^1(X, N_X) = 0$.
Hence, by \cite[Corollary 6.3]{Hartshorne:deformation}, $X$ is a smooth point of the Hilbert scheme.
Note that in the case $X$ is a local complete intersection (e.g., that $X$ is smooth) it has no local
obstructions by \cite[Corollary 9.3]{Hartshorne:deformation}, and so we may apply \cite[Corollary 6.3]{Hartshorne:deformation}
in such a case.

By \autoref{theorem:classification-of-varieties-of-minimal-degree}, we only have to deal with quadric surfaces,
the Veronese surface, and scrolls. We start with the easy case of quadric surfaces, which follows from the
more general \autoref{lemma:hypersurface-smooth-hilbert-scheme}. Then, we move on to showing smoothness
of the $d$th Veronese embedding $\nu_d:\bp^n \ra \bp^{\binom{n+d}{d}-1}$, when viewed as a point of the
Hilbert scheme in \autoref{proposition:veronese-smooth-hilbert-scheme}. In particular, this implies that any Veronese surface in $\bp^5$ under the second
Veronese embedding is a smooth point of the Hilbert scheme.
Finally, in \autoref{proposition:regular-minimal-degree-hilbert-schemes}, we show
all smooth scrolls correspond to smooth points of the Hilbert scheme.

We begin with showing hypersurfaces are smooth points of the Hilbert scheme.
\begin{lemma}
	\label{lemma:hypersurface-smooth-hilbert-scheme}
	If $X \subset \bp^n$ is any hypersurface, then it is a smooth point of the Hilbert scheme
	and satisfies $H^1(X, N_{X/\bp^n}) = 0$.
\end{lemma}
We give two proofs.
\begin{proof}[Proof 1]
	If $X$ is a hypersurface, it is a divisor, and so the normal bundle is $\sco_X(X)$. Observe that
	$H^1(X, \sco_X(X)) = 0$ as follows from the long exact sequence on cohomology associated to
	the short exact sequence of sheaves
	\begin{equation}
		\nonumber
		\begin{tikzcd}
			0 \ar {r} &  \sco_{\bp^n} \ar {r} & \sco_{\bp^n}(X) \ar {r} & \sco_X(X) \ar {r} & 0 
		\end{tikzcd}\end{equation}
	as $H^2(\bp^n, \sco_{\bp^n}) = H^1(\bp^n, \sco_{\bp^n}(\deg X)) = 0$. 
	Therefore, $H^1(X, \sco_X(X)) = 0$,
and so \cite[Corollary 6.3]{Hartshorne:deformation} and \cite[Corollary 9.3]{Hartshorne:deformation},
$X$ is a smooth point of the Hilbert scheme.
\end{proof}
\begin{proof}[Proof 2]
	By \cite[Proposition 28.3.6]{vakil:foundations-of-algebraic-geometry}, the Hilbert scheme
	of degree $d$ in $\bp^n$ is $\bp^{\binom{n+d}{d}-1}$. In particular,
	all points of this Hilbert scheme are smooth points,
	as $\bp^r$ is smooth for any $r$.
\end{proof}

Next, we show Veronese varieties are smooth points of the Hilbert scheme.

\begin{proposition}
	\label{proposition:veronese-smooth-hilbert-scheme}
	Let $N := \binom{n+d}{d}-1$. Let $X := \nu_d(\bp^n) \subset \bp^N$ be
	the image of $\bp^n$ under the $d$th Veronese embedding.
	Then, $X$ is a smooth point of the Hilbert scheme and has $H^1(X, N_{X/\bp^N}) = 0$.
	Further, $H^0(X, N_{X/\bp^N}) = (N+1)^2 - (d+1)^2$, and so $\dim \hilb X = (N+1)^2 - (n+1)^2$.
\end{proposition}
\begin{proof}
	As usual, by \cite[Corollary 6.3]{Hartshorne:deformation} and \cite[Corollary 9.3]{Hartshorne:deformation},
	it suffices to show $H^1(X, N_{X/\bp^N}) = 0$. For simplicity of notation, define $N := \binom{n+d}{d}-1$.

For this, by the normal exact sequence,
\begin{equation}
	\nonumber
	\begin{tikzcd}
		0 \ar {r} & T_X  \ar {r} & T_{\bp^{N}}|_X \ar {r} & N_{X/\bp^N} \ar {r} & 0 
	\end{tikzcd}\end{equation}
in order to show $H^1(X, N_{X/\bp^N}) = 0$, it suffices to show
\begin{align*}
	H^2(T_X) &= 0 \\
	H^1(T_{\bp^N}|_X) &= 0
\end{align*}
These both follow from the Euler exact sequence.
Since $X$ is abstractly isomorphic to $\bp^n$, we have the Euler exact sequence
\begin{equation}
	\nonumber
	\begin{tikzcd}
		0 \ar {r} &  \sco_{\bp^n}\ar {r} & \sco_{\bp^n}(1)^{n+1} \ar {r} & T_X \ar {r} & 0.
	\end{tikzcd}\end{equation}
In particular, $H^2(X, T_X) = 0$ because $H^1(\bp^n, \sco_{\bp^n}(1)) = 0$ and $H^2(\bp^n, \sco_{\bp^n}) = 0$.

To complete the proof, we only need show
$H^1(T_{\bp^N}|_X) = 0$. This too follows from the Euler exact sequence, this time applied to $\bp^N$.
We have
\begin{equation}
	\nonumber
	\begin{tikzcd}
		0 \ar {r} & \sco_{\bp^N}|_X \ar {r} & \sco_{\bp^N}(1)^{N+1}|_X \ar {r} & T_{\bp^N}|_X \ar {r} & 0.
	\end{tikzcd}\end{equation}
So, to show $H^1(T_{\bp^N}|_X) = 0$, it suffices to show $H^2(X, \sco_{\bp^N}|_X) = 0$ and
$H^1(X, \sco_{\bp^N}(1)^{N+1}|_X) = 0$.
First, since $X \cong \bp^n$, we have
\begin{align*}
H^2(X, \sco_{\bp^N}|_X) = H^2(\bp^n, \sco_{\bp^n}) = 0.
\end{align*}
Second, we see
\begin{align*}
H^1(X, \sco_{\bp^N}(1)^{N+1}|_X) = H^1(X, \sco_{\bp^n}(d)^{N+1}) = 0.
\end{align*}
The statement that $H^0(X, N_{X/\bp^N}) = (N+1)^2 - (n+1)^2$
follows by chasing these same sequences. Next, $X$ is a smooth point of the Hilbert scheme,
and $H^0(X, N_{X/\bp^N})$ is the tangent space to the Hilbert scheme at $X$, by
\cite[Theorem 1.1(b)]{Hartshorne:deformation}.
Thus, it follows that $\dim \hilb X = h^0(X, N_{X/\bp^N})$, completing the proof.
\end{proof}

\begin{corollary}
	\label{corollary:veronese-interpolation-numerics}
	If $X := \nu_d(\bp^n) \subset \bp^{\binom{n+d}{d}-1}$ is the image of $\bp^n$ under some $d$th Veronese embedding,
	then $\hilb X$ satisfies interpolation 
	if there exists some smooth Veronese variety through $\binom{n+d}{d} + n + 1$ general points.
\end{corollary}
\begin{proof}
	This follows from \autoref{proposition:veronese-smooth-hilbert-scheme}
	and the definition of interpolation because
	\begin{align*}
		\frac{\binom{n+d}{d}^2 - (n+1)^2}{\binom{n+d}{d} - (n-1)} = \binom{n+d}{d} +n +1.
	\end{align*}
\end{proof}

To show that all varieties of minimal degree are smooth points of the Hilbert scheme, it only remains to show
that smooth scrolls are smooth points of the Hilbert scheme.
This is the content of \autoref{proposition:regular-minimal-degree-hilbert-schemes},
whose proof will occupy much of the remainder of this section.
\begin{remark}
	\label{remark:}
	The proof of \autoref{proposition:regular-minimal-degree-hilbert-schemes},
	is structured in reverse order from what one usually finds in a math book.
	Usually, one might first prove several lemmas, and then combine them to prove
	our main result \autoref{proposition:regular-minimal-degree-hilbert-schemes}.
	However, here, we instead reduce the proof of \autoref{proposition:regular-minimal-degree-hilbert-schemes}
	to showing several lemmas, which we in turn reduce to showing other lemmas, and so on.

	The reason for this structure is that the proof will be fairly involved.
	By first reducing the proof to smaller chunks, we will motivate the proofs
	of the smaller chunks. We present the proof in this reverse order,
	as this is the logical order in which one would would 
	try to solve the problem.
\end{remark}

\begin{proposition}
	\label{proposition:regular-minimal-degree-hilbert-schemes}
	Let $X \subset \bp^n$ be a smooth variety of minimal degree $d$ and dimension $k$. 
	Then, $H^1(X, N_{X/\bp^n}) = 0$. Thus, the corresponding point $[X] \in \minhilb d k$ is smooth.	
\end{proposition}
\begin{proof}[Proof assuming \autoref{lemma:cohomology-restricted-tangent} and \autoref{proposition:cohomology-as-vector-bundle}]
	First, suppose we knew $H^1(X, N_{X/\bp^n}) = 0$. By \cite[Corollary 6.3]{Hartshorne:deformation}, which is applicable because $X$ is smooth, using
	\cite[Corollary 9.3]{Hartshorne:deformation}, we obtain that
	$\hilb X$ is smooth at $[X]$. Therefore, it suffices to show $H^1(X, N_{X/\bp^n}) = 0$.
	
	Next, we have the normal exact sequence
	\begin{equation}
		\nonumber
		\begin{tikzcd}
			0 \ar {r} & T_X \ar {r} & T_{\bp^n}|_X \ar {r} & N_{X/\bp^n} \ar {r} & 0.
		\end{tikzcd}\end{equation}
	From the associated long exact sequence on cohomology, in order to show
	$H^1(X, N_{X/\bp^n}) = 0$, it 
	suffices to show $H^1(T_{\bp^n}|_X) = 0$ and $H^2(T_X) = 0$. 
	First, by \autoref{proposition:cohomology-as-vector-bundle} we have $H^2(T_X) = 0$. 
	To complete the problem, it suffices to show
	$H^1(T_{\bp^n}|_X) = 0$.
	But, this is the content of \autoref{lemma:cohomology-restricted-tangent}.
\end{proof}

So, our next goals are to prove \autoref{lemma:cohomology-restricted-tangent} and \autoref{proposition:cohomology-as-vector-bundle}.
We start with a proof of \autoref{lemma:cohomology-restricted-tangent}, which in turn assumes
\autoref{lemma:cohomology-of-O(1)-as-dual-bundle}.

\begin{lemma}
	\label{lemma:cohomology-restricted-tangent}
	For $X$ a variety of minimal degree in $\bp^n$, we have $H^1(T_{\bp^n}|_X) = 0$ and $H^0(T_{\bp^n}|_X) = (n+1)^2-1$.
\end{lemma}
\begin{proof}[Proof assuming \autoref{lemma:cohomology-of-O(1)-as-dual-bundle}]
	We have an exact sequence on $\bp^n$ given by
	\begin{equation}
		\nonumber
		\begin{tikzcd}
			0 \ar {r} &  \sco_{\bp^n} \ar {r} & \sco_{\bp^n}(1)^{n+1} \ar {r} & T_{\bp^n} \ar {r} & 0.
		\end{tikzcd}\end{equation}
	Restricting this to $X$, we obtain
	\begin{equation}
		\nonumber
		\begin{tikzcd}
			0 \ar {r} &  \sco_X \ar {r} & \sco_X(1)^{n+1} \ar {r} & T_{\bp^n}|_X \ar {r} & 0.
		\end{tikzcd}\end{equation}
	So, to show $H^1(X, T_{\bp^n}|_X) = 0$, it suffices to show $H^1(X, \sco_X(1)) = 0$ and
	$H^2(X, \sco_X) = 0$.
	These both follow from \autoref{lemma:cohomology-of-O(1)-as-dual-bundle}.
	
	Finally, we compute $H^0(T_{\bp^n}|_X)$. Since $X$ is rational, we again have
	$H^1(X, \sco_X) = 0$. Additionally, since $X$ is connected, we know $H^0(X, \sco_X) = 1$. Therefore, to conclude the proof,
	it suffices to show $H^0(\sco_X(1)) = n+1$.
	Finally, this follows from \autoref{lemma:cohomology-of-O(1)-as-dual-bundle}.
\end{proof}
\begin{remark}
		\label{remark:}
			Note one can also deduce $H^2(X, \sco_X) = 0$
			by using the fact that the 	
			cohomology of the structure sheaf is a birational
			invariant.
			This is much easier to show in characteristic $0$ than in characteristic $p > 0$,
			but follows in positive characteristic from
			\cite[Theorem 1]{chatzistamatiouR:higher-direct-images-of-the-structure-sheaf-in-positive-characteristic}.
	Here is a proof in the characteristic $0$ case: Define $h^{i,j}(X) := h^j(X, \Omega^i_{X/\bk})$,
	then from Hodge theory (see the aside following (21.5.11.1) in \cite{vakil:foundations-of-algebraic-geometry})
	we have $h^{0,i}(X) = h^{i,0}(X) = h^0(X, \Omega^i_{X/\bk})$ and the latter is a birational invariant
	by ~\cite[Exercise 21.5.A]{vakil:foundations-of-algebraic-geometry}.
	\end{remark}

In order to prove \autoref{lemma:cohomology-restricted-tangent} we prove
\autoref{lemma:cohomology-of-O(1)-as-dual-bundle}, assuming
\autoref{lemma:pushforward-of-O-1-on-projective-bundle}.

\begin{lemma}
	\label{lemma:cohomology-of-O(1)-as-dual-bundle}
	Let $\sce \cong \oplus_i \sco_{\bp^1}(a_i)$ be a locally free sheaf on $\bp^1$ with $a_i > 0$ and
	projectivization $\pi: \bp \sce \ra \bp^1$.
	Then, $H^i(\bp \sce, \sco_{\bp \sce}(m)) \cong H^i(\bp^1, \sym^m \sce)$.
	In particular, $H^i(\bp \sce, \sco_{\bp \sce}(m)) = 0$ for $i \geq 2$.
\end{lemma}
\begin{proof}[Proof of \autoref{lemma:cohomology-of-O(1)-as-dual-bundle},
	assuming \autoref{lemma:pushforward-of-O-1-on-projective-bundle}]
Observe that the last statement $H^i(\bp \sce, \sco_{\bp \sce}(m)) = 0$ for $i \geq 2$ follows
immediately from the isomorphism $H^i(\bp \sce, \sco_{\bp \sce}(m)) \cong H^i(\bp^1, \sym^m \sce)$
because $\bp^1$ is only $1$-dimensional.

	By, \autoref{lemma:pushforward-of-O-1-on-projective-bundle},
we have $R^i \pi_* \sco_{\bp \sce}(m) = 0$ for $i > 0$ while $\pi_* \sco_{\bp \sce}(m) \cong \sym^m \sce$.
So, the Leray spectral sequence, ~\cite[Theorem 23.4.5]{vakil:foundations-of-algebraic-geometry}
tells us there is a spectral sequence with $E_2$ term given by $H^q(\bp^1, R^p \pi_* \sco_{\bp \sce}(m))$
abutting to $H^{p+q}(\bp \sce, \sco_{\bp \sce}(m))$.
However, we know $H^q(\bp^1, R^p \pi_* \sco_{\bp \sce}(m))= 0$ whenever $p > 0$, by
\autoref{lemma:pushforward-of-O-1-on-projective-bundle},
and therefore, the spectral sequence has already converged at the $E_2$ page, with
\begin{align*}
	H^{p+q}(\bp \sce, \sco_{\bp \sce}(m)) & \cong H^{p+q}(\bp^1, \pi_* \sco_{\bp \sce}(m)) 
	\\
	&\cong H^{p+q}(\bp^1, \sym^m \sce)
\end{align*}
as claimed.
	\end{proof}

We conclude the proof of \autoref{lemma:cohomology-restricted-tangent} by proving \autoref{lemma:pushforward-of-O-1-on-projective-bundle}.

\begin{lemma}
	\label{lemma:pushforward-of-O-1-on-projective-bundle}
	Let $\bp \sce \subset \bp^n$ be a rational normal scroll with structure map
	$\pi: \bp \sce \ra \bp^1$. Then, the higher derived pushforward of $\sco_{\bp \sce}(m)$ for $m \geq 0$ satisfy
	\begin{align*}
		R^i \pi_* \sco_{\bp \sce}(m) \cong
		\begin{cases}
			\sym^m \sce & \text{ if } i = 0 \\
			0 & \text{ if } i > 0 \\
		\end{cases}
	\end{align*}
\end{lemma}
\begin{proof}
First, the $i = 0$ case follows from
~\cite[Proposition 9.3]{Eisenbud:3264-&-all-that}.
Note that we are using the convention that $\bp \sce := \proj \sym^\bullet \sce$
instead of the convention $\bp \sce := \proj \sym^\bullet \sce^\vee$, adopted in
~\cite{Eisenbud:3264-&-all-that}, which explains the lack of a dual.

To complete the proof, we only need show that $R^i \pi_* \sco_{\bp \sce}(1) = 0$ for $i > 0$. 
Technically, this is also stated in 
~\cite[Proposition 9.3]{Eisenbud:3264-&-all-that}, but the proof is somewhat lacking in details,
so we provide them here.

To see this,
note that $\sco_{\bp \sce}(m)$ is flat over $\bp^1$ because the fibers are all
have the same Hilbert polynomial and $\bp^1$ is reduced, by \cite[Theorem III.9.9]{Hartshorne:AG}.
Hence, we have $h^i(\bp \sce_q, \sco_{\bp \sce}(1)|_{X_q}) \cong h^i(\bp^{k-1}, \sco_{\bp^{k-1}}(1))= 0$ when $i > 0$,
because all fibers of $\pi$ are isomorphic to a projective space $\bp^{k-1}$.
Hence, by Grauert's theorem, ~\cite[Grauert's Theorem 28.1.5]{vakil:foundations-of-algebraic-geometry},
we have that $R^i \pi_* \sco_{\bp \sce}$ is locally free. It is then $0$ because it has rank $0$ at all
closed points.
\end{proof}

To conclude the proof of \autoref{proposition:regular-minimal-degree-hilbert-schemes}, we only need
prove
\autoref{proposition:cohomology-as-vector-bundle}.
We do this assuming
\autoref{lemma:endomorphisms-and-pulled-back-endomorphisms},
\autoref{lemma:relative-tangent-and-pulled-back-endomorphisms},
and \autoref{lemma:tangent-and-relative-tangent}.

\begin{proposition}
	\label{proposition:cohomology-as-vector-bundle}
	Let $\sce \cong \oplus_i \sco_{\bp^1}(a_i)$ be a locally free sheaf on $\bp^1$ with all $a_i > 0$
	and projectivization $\pi: \bp \sce \ra \bp^1$.
	Then, 
	\begin{align*}
		h^i(\bp \sce, T_{\bp \sce}) 
		=
		\begin{cases}
			h^i(\bp^1, \End \sce) & \text{ if } i > 0 \\
			h^i(\bp^1, \End \sce) + 2 & \text{ if } i = 0.
		\end{cases}
	\end{align*}
	In particular $H^i(\bp \sce, T_{\bp \sce}) = 0$ for $i \geq 2$.
\end{proposition}
\begin{proof}[Proof assuming \autoref{lemma:endomorphisms-and-pulled-back-endomorphisms},
\autoref{lemma:relative-tangent-and-pulled-back-endomorphisms},
and \autoref{lemma:tangent-and-relative-tangent}]

The main content of this proof is the use of the Leray spectral sequence which we use to relate $\End \sce$ to $\pi^* \sce^\vee(1)$,
the relative Euler exact sequence which we use to relate $\pi^*\sce(1)$ to $T_\pi$,
and the relative tangent sequence which we use to relate $T_\pi$ to $T_{\bp \sce}$.
Combining these, we have,
\begin{align*}
	h^0(\bp^1, \End \sce) &= h^0(\bp \sce, \pi^* \sce^\vee(1)) & \text{ by } \autoref{lemma:endomorphisms-and-pulled-back-endomorphisms}\\
	&= h^0(\bp \sce, T_\pi) + \delta_{i0} & \text{ by } \autoref{lemma:relative-tangent-and-pulled-back-endomorphisms}\\
	&= h^0(\bp \sce, T_{\bp \sce}) - 2 \delta_{i0} & \text{ by } \autoref{lemma:tangent-and-relative-tangent}
\end{align*}
where $\delta_{ij}$ is the Kronecker $\delta$ function.
We have that $H^i(\bp \sce, T_{\bp \sce}) = 0$ for $i \geq 2$ because it is isomorphic to
$H^i(\bp^1, \End \sce)$ which is $0$ when $i \geq 2$ since $\bp^1$ is only $1$-dimensional.
\end{proof}

It now only remains to prove \autoref{lemma:endomorphisms-and-pulled-back-endomorphisms},
\autoref{lemma:relative-tangent-and-pulled-back-endomorphisms},
and \autoref{lemma:tangent-and-relative-tangent}.
We start with \autoref{lemma:endomorphisms-and-pulled-back-endomorphisms}.

\begin{lemma}
	\label{lemma:endomorphisms-and-pulled-back-endomorphisms}
	Let $\sce \cong \oplus_i \sco_{\bp^1}(a_i)$ be a locally free sheaf on $\bp^1$ with all $a_i > 0$
	and projectivization $\pi: \bp \sce \ra \bp^1$.
	Then,
	\begin{align*}
	H^i\left(\bp \sce, \pi^* (\sce^\vee)(1)\right) \cong H^i(\bp^1, \End \sce)
	\end{align*}
\end{lemma}
\begin{proof}
	First, because $R^0 \pi_*$ is the same functor as $\pi_*$, by
\autoref{lemma:pushforward-of-O-1-on-projective-bundle}, with $m = 1$, we have $\pi_* \sco_{\bp \sce}(1) \cong \sce$.
Hence, applying the projection formula ~\cite[Exercise 18.8.E]{vakil:foundations-of-algebraic-geometry},
we obtain an isomorphism
\begin{align*}
	\left( R^i \pi_* \sco_{\bp \sce}(1) \right) \otimes \sce^\vee \cong R^i \pi_* \left( \sco_{\bp \sce}(1) \otimes \pi^* \sce^\vee \right).
\end{align*}
In particular, the left hand side is $0$ for $i > 0$ by \autoref{lemma:pushforward-of-O-1-on-projective-bundle}
with $m = 1$,
and is $\sce \otimes \sce^\vee \cong \End \sce$ when $i = 0$. Therefore, we have
\begin{align*}
R^i \pi_* \left( \sco_{\bp \sce}(1) \otimes \pi^* \sce^\vee \right) \cong
	\begin{cases}
		\End \sce & \text{ if } i = 0 \\
		0 & \text{ if } i > 0.
	\end{cases}
\end{align*}

Now, applying the Leray spectral sequence ~\cite[Theorem 23.4.5]{vakil:foundations-of-algebraic-geometry},
we see that there is a spectral sequence with $E_2$ term given by
\begin{align*}
H^q \left(\bp^1, R^p \pi_*\left( \sco_{\bp \sce}(1) \otimes \pi^* \sce^\vee \right) \right)
\end{align*}
abutting to
\begin{align*}
	H^{p+q}\left(\bp \sce, \sco_{\bp \sce}(1) \otimes \pi^* \sce^\vee \right).
\end{align*}
However, because 
$H^q \left(\bp^1, R^p \pi_*\left( \sco_{\bp \sce}(1) \otimes \pi^* \sce^\vee \right) \right) = 0$ for
$p > 0$, we obtain that the spectral sequence has already converged at the $E_2$ page with
\begin{align*}
	H^{p+q}\left(\bp \sce, \pi^* \sce^\vee(1) \right) &\cong H^{p+q}\left(\bp^1, R_0 \pi_* \left(\pi^* \sce^\vee (1)\right) \right) \\
	& \cong H^{p+q}(\bp^1, \End \sce).
\end{align*}
\end{proof}

We have now reduced our task to proving \autoref{lemma:relative-tangent-and-pulled-back-endomorphisms}
and \autoref{lemma:tangent-and-relative-tangent}.
We next tackle \autoref{lemma:relative-tangent-and-pulled-back-endomorphisms}.

	\begin{lemma}
		\label{lemma:relative-tangent-and-pulled-back-endomorphisms}
			Let $\sce \cong \oplus_i \sco_{\bp^1}(a_i)$ be a locally free sheaf on $\bp^1$ with all $a_i > 0$
	and projectivization $\pi: \bp \sce \ra \bp^1$.
	Then, 
	\begin{align*}
		h^i(\bp \sce, T_{\pi}) 
		=
		\begin{cases}
			h^i(\bp \sce, \pi^* (\sce^\vee)(1)) & \text{ if } i > 0 \\
			h^i(\bp^1, \pi^* (\sce^\vee)(1)) -1 & \text{ if } i = 0. \\
		\end{cases}
	\end{align*}
	\end{lemma}
	\begin{proof}
		First, by \cite[Theorem 4.5.13]{brandenburg:tensor-categorical-foundations-of-algebraic-geometry} (see also
\cite[Subsection 21.4.9]{vakil:foundations-of-algebraic-geometry})
we have the Euler exact sequence
\begin{equation}
	\nonumber
	\begin{tikzcd}
		0 \ar {r} &  \sco_{\bp \sce} \ar {r} & \pi^*(\sce^\vee)(1)  \ar {r} & T_\pi \ar {r} & 0 
	\end{tikzcd}\end{equation}
where $T_\pi$ is the relative tangent sheaf of $\pi: \bp \sce \ra \bp^1$.

We now use these sequences to determine the cohomology of $T_{\bp \sce}$.
First, since $\bp \sce$ is rational, we have $H^i(\bp \sce, \sco_{\bp \sce}) = \delta_{i0}$.
Here, we are using \autoref{lemma:cohomology-of-O(1)-as-dual-bundle}.
This implies that for $i > 0$, we have an isomorphism 
\begin{align*}
	H^i(\bp \sce, \pi^* (\sce^\vee)(1)) \cong H^i(\bp \sce, T_\pi).
\end{align*}
Further, when $i = 0$,
\begin{align*}
	h^0(\bp \sce, \pi^* (\sce^\vee)(1)) -1 = h^0(\bp \sce, T_\pi).
\end{align*}
\end{proof}
	
Finally, to conclude our proof of
\autoref{proposition:regular-minimal-degree-hilbert-schemes}, it suffices to prove
\autoref{lemma:tangent-and-relative-tangent}.
Following our style so far, we prove this
assuming two more lemmas,
\autoref{lemma:cohomology-of-pullback-of-tangent-sheaf} and
\autoref{lemma:relative-tangent-exact-on-global-sections}.

\begin{lemma}
	\label{lemma:tangent-and-relative-tangent}
	Let $\sce \cong \oplus_i \sco_{\bp^1}(a_i)$ be a locally free sheaf on $\bp^1$ with all $a_i > 0$
	and projectivization $\pi: \bp \sce \ra \bp^1$.
	Then, 
	\begin{align*}
		h^i(\bp \sce, T_{\bp \sce}) 
		=
		\begin{cases}
			h^i(\bp \sce, T_{\pi}) & \text{ if } i > 0 \\
			h^i(\bp \sce, T_{\pi}) +3 & \text{ if } i = 0 \\
		\end{cases}
	\end{align*}
	\end{lemma}
\begin{proof}[Proof assuming \autoref{lemma:cohomology-of-pullback-of-tangent-sheaf} and
\autoref{lemma:relative-tangent-exact-on-global-sections}]

Because $\pi$, is smooth and $\bp^1$ is smooth over $\bk$,
we have that the relative tangent sequence
\begin{equation}
	\nonumber
	\begin{tikzcd}
		0 \ar {r} &  T_\pi \ar {r} & T_{\bp \sce} \ar {r} & \pi^* T_{\bp^1} \ar {r} & 0.
	\end{tikzcd}\end{equation}
is exact.

	By \autoref{lemma:cohomology-of-pullback-of-tangent-sheaf},
	\begin{align*}
	H^i(\bp \sce, \pi^* T_{\bp^1}) =0
\end{align*}
for all $i > 0$. Since we also know the relative tangent sequence 
is exact on global sections,
by \autoref{lemma:relative-tangent-exact-on-global-sections},
we have an isomorphism
\begin{align*}
	H^i(\bp \sce, T_\pi) \cong H^i(\bp \sce, T_{\bp \sce})
\end{align*}
for all $i > 0$. So, for $i > 0$, we have the claimed isomorphisms
\begin{align*}
	H^i(\bp \sce, T_{\pi}) \cong H^i(\bp \sce, T_{\bp \sce})
\end{align*}
coming from the relative tangent sequence.

Finally, in the case $i = 0$, we obtain that
\begin{align*}
	h^i(\bp \sce, T_\pi) = h^i(\bp \sce, T_{\bp \sce}) + 3	
\end{align*}
because the relative tangent sequence 
is exact on global sections, by \autoref{lemma:relative-tangent-exact-on-global-sections},
and because $h^1(\bp \sce, \pi^* T_{\bp^1}) = 0$ by
\autoref{lemma:cohomology-of-pullback-of-tangent-sheaf}.
\end{proof}

To conclude our proof, we only need show
\autoref{lemma:cohomology-of-pullback-of-tangent-sheaf} and
\autoref{lemma:relative-tangent-exact-on-global-sections}.
First, we dispense with 
\autoref{lemma:cohomology-of-pullback-of-tangent-sheaf}.

\begin{lemma}
	\label{lemma:cohomology-of-pullback-of-tangent-sheaf}
	Let $\pi: \bp \sce \ra \bp^1$ be a projective bundle. Then, we have
	\begin{align*}
		H^i(\bp \sce, \pi^* T_{\bp^1}) = H^i(\bp^1, T_{\bp^1}) =
		\begin{cases}
			3 & \text{ if } i = 0 \\
			0 & \text{ if } i \neq 0.
		\end{cases}
	\end{align*}
\end{lemma}
\begin{proof}
	Of course, since $T_{\bp^1} \cong \sco_{\bp^1}(2)$, in order to complete the proof,
	we only need show
	\begin{align*}
		H^i(\bp \sce, \pi^* T_{\bp^1}) = H^i(\bp^1, T_{\bp^1}).
	\end{align*}
By \autoref{lemma:pushforward-of-O-1-on-projective-bundle}, applied with $m = 2$, 
we have that 
\begin{align*}
	R^i \pi_* \pi^* T_{\bp^1} \cong R^i \pi_* \pi^* \sco_{\bp^1}(2) = 0
\end{align*}
when $i > 0$ and
\begin{align*}
	\pi_* \pi^* T_{\bp^1} \cong \sco_{\bp^1}(2) \cong T_{\bp^1}  
\end{align*}
when $i = 0$. 
Finally, by applying
the Leray spectral sequence ~\cite[Theorem 23.4.5]{vakil:foundations-of-algebraic-geometry},
we obtain that there is a spectral sequence with $E_2$ term given by $H^q(\bp^1, R^p \pi_* \pi^* T_{\bp^1})$
abutting to $H^{p+q}(\bp \sce, \pi^* T_{\bp^1})$. Because $H^q(\bp^1, R^p \pi_* \pi^* T_{\bp^1}) = 0$ for $p > 0$,
the spectral sequence converges at the $E_2$ page and we have
\begin{align*}
	H^{p+q}(\bp \sce, \pi^* T_{\bp^1}) \cong H^{p+q}(\bp^1, T_{\bp^1}).
\end{align*}
\end{proof}

We have finally reduced proving \autoref{proposition:regular-minimal-degree-hilbert-schemes} to showing
\autoref{lemma:relative-tangent-exact-on-global-sections}.
We do this now.

\begin{lemma}
	\label{lemma:relative-tangent-exact-on-global-sections}
	Let $\pi: \bp \sce \ra \bp^1$ be a projective bundle where $\sce \cong \oplus_{i=1}^k \sco_{\bp^1}(a_i)$ with $a_i > 0$ for all $i$. The sequence
	\begin{equation}
	\nonumber
	\begin{tikzcd}
		0 \ar {r} &  H^0(\bp \sce, T_\pi) \ar {r} & H^0(\bp \sce, T_{\bp \sce}) \ar {r} & H^0(\bp \sce, \pi^* T_{\bp^1}) \ar {r} & 0 
	\end{tikzcd}\end{equation}
gotten by taking global sections of the relative tangent sequence
is exact.
\end{lemma}
\begin{proof}
	Because $\pi$ is smooth and $\bp^1$ is smooth over $\bk$, the relative tangent sequence is exact,
	and so the above sequence on cohomology is automatically left exact.
	Because $H^0(\bp \sce, \pi^* T_{\bp^1}) = H^0(\bp^1, T_{\bp^1}) = 3$,
	in order to show this is exact, it suffices to show
	\begin{align*}
		h^0(\bp \sce, T_\pi) + 3 \leq h^0(\bp \sce, T_{\bp \sce}).
	\end{align*}
	Further, recall that $H^0(\bp \sce, T_{\bp \sce})$ parameterizes first order automorphisms 
	while $H^0(\bp \sce, T_\pi)$ parameterizes first order automorphisms fixing each fiber of $\pi: \bp \sce \ra \bp^1$.
	The inclusion $H^0(\bp \sce, T_\pi) \ra H^0(\bp \sce, T_{\bp\sce})$ the expresses the natural inclusion
	of these two parameterizations.
	Hence, it suffices to exhibit a three dimensional family of automorphisms not fixing $\bp \sce$ fiber by fiber.
	We will now do this.
	
	If $X \cong \bp \sce \in \scroll {a_1, \ldots, a_k}$, and $X$ is embedded into $\bp^{d+k-1}$ so that it has degree $d$ and dimension $k$,
	we have an explicit description of $X$ as the vanishing
	of the $2 \times 2$ minors of
	\begin{align*} M:=
	\begin{pmatrix}
		x_{1,0} & x_{1,1} & \cdots & x_{1,a_1 - 1} & x_{2,0} & \cdots & x_{2, a_2-1} & \cdots & x_{k, 0} & \cdots & x_{k, a_k-1}  \\
		x_{1,1} & x_{1,2} & \cdots & x_{1, a_1} & x_{2, 1} & \cdots & x_{2, a_2} & \cdots & x_{k,1} & \cdots & x_{k,a_k}
	\end{pmatrix}.
	\end{align*}
	Now, acting on $M$ on the left by an element of $GL_2(\bk)$
	will produce a new matrix whose rank one locus agrees with that of $M$
	and this automorphism sends fibers of the structure map $\pi :\bp \sce \ra \bp^1$
	to fibers of $\pi$. 
	The elements of $GL_2(\bk)$ act on $\bp^1 = \im \pi$ by linear fractional transformations.
	In particular, they act triply transitively and hence define a three dimensional
	family of automorphisms preserving the fibers of $\pi$.

\end{proof}

To conclude this section,
we calculate the dimension of the Hilbert scheme by finding
the dimension of $H^0(X, N_{X/\bp^n})$ for $X$ a scroll.

\begin{proposition}
	\label{proposition:scroll-hilbert-scheme-dimension}
	If $X$ is a smooth scroll of minimal degree $d$ and dimension $k$ in $\bp^n$ then $h^0(X, N_{X/\bp^n}) = (d+k)^2 - k^2 - 3$. In particular, 
	$\dim \minhilb d k =  (d+k)^2 - k^2 - 3$.
\end{proposition}
\begin{proof}
	First, by \autoref{proposition:regular-minimal-degree-hilbert-schemes}, we know $\hilb X$ is smooth at $[X]$.
	So we know $\dim \hilb X = \dim T_{[X]} \hilb X = h^0(X, N_{X/\bp^n})$, where the latter
	equality holds by \cite[Theorem 1.1(b)]{Hartshorne:deformation}.
	So, to conclude the proof, we only need show $h^0(X, N_{X/\bp^n}) = (d+k)^2 - k^2 - 3$.
	We will compute $H^0(X, N_{X/\bp^n})$ by the exact sequence
	\begin{equation}
	\nonumber
	\begin{tikzcd}
		0 \ar {r} &  T_X \ar {r} & T_{\bp^n}|_X \ar {r} & N_{X/\bp^n} \ar {r} & 0.
	\end{tikzcd}\end{equation}
As we saw in \autoref{lemma:cohomology-restricted-tangent}, we have $H^1(X, T_{\bp^n}|_X) = 0$, so we can consider the
exact sequence on cohomology
\begin{equation}
	\nonumber
	\begin{tikzcd}
		0 \ar {r} &  H^0(X, T_X) \ar {r} & H^0(X, T_{\bp^n}|_X)  \ar {r} & H^0(N_{X/\bp^n}) \\
		\ar {r} & H^1(X, T_X) \ar {r} & 0.
	\end{tikzcd}\end{equation}
We know from \autoref{lemma:cohomology-restricted-tangent}, that $H^0(X, T_{\bp^n}|_X) = (n+1)^2 - 1$.
We want to show $h^0(N_{X, \bp^n}) = (n+1)^2 - k^2 - 2$. So, to complete the proof, it suffices to show that
\begin{align*}
\chi(X, T_X) = h^0(X, T_X) - h^1(X, T_X) = k^2 +2. 
\end{align*}
This suffices because dimension is additive in exact sequences, so we will obtain
\begin{align*}
h^0(X, N_{X/\bp^n}) = (n+1)^2 - 1 - k^2 -2 = (n+1)^2 -k^2 - 3.
\end{align*}

In order to show $\chi(X, T_X) = k^2 + 2$, by \autoref{proposition:cohomology-as-vector-bundle}, it suffices to show that
\begin{align*}
\chi(\bp^1, \End \sce) = h^0(\bp^1, \End \sce) - h^1(\bp^1, \End \sce) = k^2.
\end{align*}
Now, writing $\sce = \oplus_{j=1}^k \sco_{\bp^1}(a_j)$,
we have
\begin{align*}
	\End \sce \cong \sce \otimes \sce^\vee \cong \oplus_{i,j} \sco_{\bp^1}(a_i - a_j).
\end{align*}
Therefore,
\begin{align*}
	\chi(\bp^1, \End \sce) &\cong \chi(\bp^1, \oplus_{i,j} \sco_{\bp^1}(a_i - a_j)) \\
	&= \sum_{i,j} \chi(\sco_{\bp^1}(a_i - a_j)) \\
	&= \sum_{i,j} (a_i - a_j + 1) \\
	&= \sum_{i,j} 1 \\
	&= k^2
\end{align*}
as desired.
\end{proof}

\begin{remark}
	\label{remark:}
	One can also compute $\dim \minhilb d k$ in a different way. 
	We now provide a sketch of this alternate approach, omitting
	many of the details:

	Using the fact that the general member of $\minhilb d k$ is a balanced
	scroll (which holds because one can degenerate a balanced scroll to any unbalanced scroll, via a corresponding exact sequence of locally free sheaves),
	the dimension of the Hilbert scheme will be equal to the dimension of the space of balanced scrolls.
	One can then check that this dimension is the difference of the dimensions of the automorphisms of projective
	space and the automorphisms of a balanced scroll. Since the automorphisms of projective space are $(d+k)^2 - 1$ dimensional, it suffices
	to show that automorphisms of a balanced scroll are $k^2 + 2$ dimensional. 

	Let $X = \bp \sce$ for $\sce$ an invertible sheaf on $\bp^1$.
	Note that $\dim Aut(X) = \dim H^0(\bp^1, \End \sce) + 2 = \dim \bp H^0(\bp^1, \End \sce) +3 $ by \autoref{proposition:cohomology-as-vector-bundle}.
	But, since $X$ is balanced,
	we can write $\sce \cong \sco_{\bp^1}(a_1) \oplus \cdots \oplus \sco_{\bp^1}(a_k)$ where $a_k - a_1 \leq 1$. Suppose $a_1 = a_2 = \cdots a_t$ and either $t = k$ or $a_{t+1} = a_t + 1$.
	Then,
	\begin{align*}
		\End \sce \cong \sce \otimes \sce^\vee \cong \sco_{\bp^1}^{\oplus (t^2 + (k-t)^2)} \oplus \sco_{\bp^1}(1)^{\oplus t(k-t)} \oplus \sco_{\bp^1}(-1)^{\oplus t(k-t)}.
	\end{align*}
Therefore,
\begin{align*}
	h^0(\bp^1, \End \sce) &\cong h^0 ( \bp^1, \sco_{\bp^1}^{\oplus (t^2 + (k-t)^2)} \oplus \sco_{\bp^1}(1)^{\oplus t(k-t)} \oplus \sco_{\bp^1}(-1)^{\oplus t(k-t)}) \\
		&= t^2 + (k-t)^2 + 2t(k-t) + 0 \\
		&= (t + (k-t))^2 \\
		&= k^2
\end{align*}
So, by \autoref{proposition:cohomology-as-vector-bundle}, the
vector space of
first order automorphisms of such a balanced scroll is $k^2 + 2$ dimensional.
Hence,
\begin{align*}
\dim \minhilb d k = (n+1)^2 -1 - k^2 - 2 = (n+1)^2 - k^2 -3,
\end{align*}
as claimed.
\end{remark}

\chapter[Minimal degree degenerations]{Degenerations of varieties of minimal degree}
\label{section:degenerations-of-varieties-of-minimal-degree}

We have now arrived at what is arguably the most technical chapter of this thesis.
In this chapter we examine the degeneration of a scroll in $\minhilb d k$ to the union
of a $k$-plane and a scroll in $\minhilb {d-1} k$, meeting along a $(k-1)$-plane of the ruling
of the degree $d-1$ scroll.
In \autoref{subsection:constructing-a-degeneration-using-equations},
we construct such a degeneration (see \autoref{proposition:hilbert-scroll-degeneration}), showing that these degenerate scrolls lie
in the same component of the Hilbert scheme as a smooth scroll.
In \autoref{subsection:irreducibility-in-the-locus-of-singular-scrolls},
we show (see \autoref{proposition:minutely-broken-closure})
that if a family of scrolls has special fiber isomorphic to the degenerate scroll 
described above, then a general fiber will either be isomorphic to the degenerate scroll
or it will be a smooth scroll.
We will use this to prove interpolation for scrolls by using
it to show that certain such degenerate scrolls are isolated points in their fiber
of the map $\pi_2$, from \autoref{definition:interpolation}.
Along the way, we will prove 
\autoref{proposition:distinct-hilbert-polynomials-implies-reducible-source},
a fact which is interesting in its own right
about the source of a map to a curve, where all fibers have two irreducible components.
Although very believable, we could not find this technical fact in the literature.

\section{Constructing a degeneration using equations}
\label{subsection:constructing-a-degeneration-using-equations}

In this section, our main aim is to prove \autoref{proposition:hilbert-scroll-degeneration},
describing a degeneration of a smooth scroll, which will be crucial for proving that scrolls
satisfy interpolation. We start by defining this degeneration.

\begin{definition}
	\label{definition:least-broken}
	Fix a projective space $\bp^n$ and let $H \subset \bp^n$ be a hyperplane.
	Let $X \subset H$ be a smooth variety of minimal degree with $[X] \in \minhilb {d-1} k$, and let $Y \subset \bp^n$ be a $k$-plane
	so that $X \cap Y \cong \bp^{k-1} \subset \bp^n$ is of type \ref{custom:line-large}
	(that is, $X \cap Y$ is a fiber of the projection $X \ra \bp^1$)
	and $W$ is nondegenerate.
	Let $W: = X \cap Y$.
	Then, we define $\brokengeneral d k$ to be the locally closed
	locus in the Hilbert scheme
	of all such $W$ and define $\broken d k$ to be the closure in the Hilbert
	scheme of $\brokengeneral d k$.
\end{definition}
\begin{remark}
	\label{remark:}
	Note that in the notation $\broken d k$, we use the superscript ``broken'' to
	refer to varieties which are the union of a plane and a scroll of one lower degree.
	In the literature, broken scrolls are sometimes used to mean
	any union of two scrolls meeting along a plane of the ruling,
	where one is not required to be a plane.
	We will not consider these more general types of broken
	scrolls in this thesis.
\end{remark}

\begin{proposition}
	\label{proposition:hilbert-scroll-degeneration}
	Fix a projective space $\bp^n$ and let $W \subset \bp^n$ satisfy $[W] \in \broken d k$. Then, $[W] \in \minhilb d k$.
\end{proposition}
\begin{proof}[Proof assuming \autoref{lemma:degeneration-of-scrolls}]
By definition of the functor of points, which is represented by the Hilbert scheme of minimal rational normal scrolls $X$ of
degree $d$ and dimension $k$ in $\bp^{d+k-1}$, maps
\begin{align*}
	\ba^1 \ra \hilb X
\end{align*}
are in natural bijection with maps closed subschemes $Y \subset \ba^1 \times \bp^n$,
flat over $\ba^1$.
Such a closed subscheme is constructed in \autoref{lemma:degeneration-of-scrolls}.
In particular, because $\ba^1$ is irreducible, and the image of an irreducible scheme
is irreducible, all fibers of the map $\pi: \scx \ra \ba^1$, as defined in \autoref{lemma:degeneration-of-scrolls}
lie in a single component of the Hilbert scheme.
Since the Hilbert scheme of rational normal scrolls is smooth at a point corresponding to
a smooth rational normal scroll by \autoref{proposition:regular-minimal-degree-hilbert-schemes}, it is, in particular, integral at such a point. Therefore,
the fiber of $\scx$ over the origin must lie in the same component that a smooth rational normal scroll lies in.
That is, $X \cup Y$ corresponds to a point in $\minhilb d k$.
\end{proof}

To complete the proof of \autoref{proposition:hilbert-scroll-degeneration},
we only need prove \autoref{lemma:degeneration-of-scrolls}, which we now do.

\begin{lemma}
	\label{lemma:degeneration-of-scrolls}
	Let $\scroll {a_1, \ldots, a_k} \subset \bp^n$ be a rational normal scroll with $a_1 \geq a_2 \geq \cdots \geq a_k$ and $a_i > a_{i+1}$
	(where we assume this condition holds vacuously for $i = k$).
	Let $X \cong  \scroll {a_1, \ldots, a_i-1, a_{i+1},  a_k}$ be a variety spanning a hyperplane $H \subset \bp^n$
	and let $Y \cong \bp^{k-1} \subset \bp^n$.
	Then,
consider the projective scheme $\scx$, with map $\pi: \scx \ra \ba^1 := \spec \bk[t],$ so that 
\begin{align*}
	\scx &\subset \bp^n \times_{\spec \bk} \ba^1 \\
	\bp^n &:= \proj \bk[x_{1,0}, \ldots, x_{1,a_1}, x_{2,0}, \ldots, x_{2,a_2}, \ldots, x_{k,0}, \ldots, x_{k,a_k}] 
\end{align*}
where $\scx$ has ideal sheaf generated by the minors of
\begin{align*}
	\begin{pmatrix}
		x_{1,0} & x_{1,1} & \cdots & x_{1,a_1 - 1} & \cdots & x_{i,0} & \cdots & tx_{i, a_i-1} & \cdots & x_{k, 0} & \cdots & x_{k, a_k-1}  \\
		x_{1,1} & x_{1,2} & \cdots & x_{1, a_1} & \cdots & x_{i, 1} & \cdots & x_{i, a_i} & \cdots & x_{k,1} & \cdots & x_{k,a_k}
	\end{pmatrix}.
	\end{align*}
Then, $\pi: \scx \ra \ba^1$
is a flat family whose fiber over the origin is isomorphic to the $W := X \cup Y$ with $[W] \in \broken d k$ and whose fiber over any other closed point of $\ba^1$ is isomorphic to
$\scroll {a_1, \ldots, a_k}$.
\end{lemma}
\begin{proof}
	To prove this, we will produce an explicit degeneration. Without loss of generality, we may assume $i = 1$. If not, one may
	simply reorder the variables in the following proof.
So, we are assuming that $\scx$ is
defined by the vanishing of the minors of
\begin{align*}
	\begin{pmatrix}
		x_{1,0} & x_{1,1} & \cdots & tx_{1,a_1 - 1} & x_{2,0} & \cdots & x_{2, a_2-1} & \cdots & x_{k, 0} & \cdots & x_{k, a_k-1}  \\
		x_{1,1} & x_{1,2} & \cdots & x_{1, a_1} & x_{2, 1} & \cdots & x_{2, a_2} & \cdots & x_{k,1} & \cdots & x_{k,a_k}
	\end{pmatrix}.
	\end{align*}
We claim that this produces the desired flat family, with map $\pi: \scx \ra \ba^1$ given by the $t$ coordinate.
Since the base is reduced, by ~\cite[Theorem III.9.9]{Hartshorne:AG},
it suffices to show all fibers have the same Hilbert polynomial.

When $t \neq 0$, we have that $\pi^{-1}(p) \cong \scroll {a_1, \ldots, a_k}$ when $p$ a closed point other than the origin.
This follows
for general $t \neq 0$ by the alternate proof of ~\cite[Proposition 9.12]{harris:a-first-course}
which holds in all dimensions, as given in ~\cite[p.\ 108-109]{harris:a-first-course}.
To prove the family is flat, by \cite[Theorem III.9.9]{Hartshorne:AG},
it suffices to show that the fiber over $t = 0$ had Hilbert polynomial equal to that of a smooth scroll.

So, to complete the proof, it only remains to check that the preimage of the origin is precisely the degeneration into a degree $d-1$ rational normal
scroll and a linear subspace, and has the same Hilbert polynomial as a smooth scroll.
Note that the degeneration is the scheme cut out by the minors of $M$, with $M$ defined by
\begin{align*} M:=
	\begin{pmatrix}
		x_{1,0} & x_{1,1} & \cdots & 0 & x_{2,0} & \cdots & x_{2, a_2-1} & \cdots & x_{k, 0} & \cdots & x_{k, a_k-1}  \\
		x_{1,1} & x_{1,2} & \cdots & x_{1, a_1} & x_{2, 1} & \cdots & x_{2, a_2} & \cdots & x_{k,1} & \cdots & x_{k,a_k}
	\end{pmatrix}.
	\end{align*}
	Now, we see that the minors of $M$ are precisely
	\begin{align*}
	I_1 := \left\{  x_{1, a_1} x_{i,j}: (i,j) \neq (t, a_t), 1 \leq t \leq k \right\}  
	\end{align*}
	together with the minors of the matrix
	\begin{align*}
		N :=
	\begin{pmatrix}
		x_{1,0} & x_{1,1} & \cdots & x_{1, a_1-2} & x_{2,0} & \cdots & x_{2, a_2-1} & \cdots & x_{k, 0} & \cdots & x_{k, a_k-1}  \\
		x_{1,1} & x_{1,2} & \cdots & x_{1, a_1-1} & x_{2, 1} & \cdots & x_{2, a_2} & \cdots & x_{k,1} & \cdots & x_{k,a_k}
	\end{pmatrix}.
	\end{align*}
	Define $I_2$ to be the ideal generated by the minors of $N$.
	Now, we claim $I_1 + I_2 = I_{X \cup Y}$, where $X \cup Y \in \broken d k$ with $X$ a scroll of degree $d-1$ and $Y \cong \bp^{k}$ a plane.

	Explicitly, we will take $X$ defined by $I_X = I_1 + (x_{1, a_1})$, and $Y$ defined by $I_Y = \left\{ x_{i,j}: (i,j) \neq (t,a_t), 1 \leq t \leq k \right\}$.
	Now, since $I_{X \cup Y} = I_X \cap I_Y$, it suffices to check $I_1 + I_2 = I_X \cap I_Y$.

	We first check $I_1 + I_2 \subset I_X \cap I_Y$. Since $I_1 \subset  I_X$ and $I_1 \subset I_Y$ we have $I_1 \subset I_X \cap I_Y$. Similarly, since $x_{1, a_1} \in I_X$, we have
	$I_2 \subset I_X$. Analogously, we also have $I_2 \subset I_Y$. Therefore, $I_2 \subset I_X \cap I_Y$.
	Hence, $I_2 + I_1 \subset I_X \cap I_Y$.
	The reverse inclusion is similarly easy to check, as the generators of the intersection $I_X \cap I_Y$ are precisely the generators of $I_1$, which
	already lie in $I_X$, together with the generators of $I_2$, which are the generators of the intersection of $(x_{1,a_1}) \cap I_Y$.

	So, we conclude $I_{X \cup Y} = I_1 + I_2$. Note that $Y \cong \bp^{k}, X \in \minhilb {d-1} k$. Further, scheme theoretically,
	$X \cap Y \cong \bp^{k-1}$ because $I_{X \cap Y} = I_X + I_Y = I_Y + (x_{1, a_1})$.
	This gives the desired description of the the fiber over $t = 0$. To conclude, we only need verify that the Hilbert
	polynomial of this fiber over $t=0$ agrees with the Hilbert polynomial of the generic fiber.

	That the Hilbert polynomials agree follows from \autoref{lemma:hilbert-polynomial-intersection}, 
	In more detail, we have
\begin{align*}
	p_{X \cup Y}(m) &= p_X(m) + p_Y(m) - p_{X \cap Y}(m) \\
	&= \left( (d-1) \binom{m + k - 1}{k} + \binom{m+k-1}{k-1} \right) + \binom{m+k}{k} - \binom{m+k - 1}{k-1} \\
	&= \left( (d-1) \binom{m + k - 1}{k} + \binom{m+k-1}{k-1} \right) + \binom{m+k-1}{k} \\
	&= d \binom{m + k - 1}{k} + \binom{m+k-1}{k-1},
\end{align*}
which is indeed the Hilbert function of a smooth scroll of degree $d$ and dimension $k$ by \autoref{lemma:reduced-minors}.
\end{proof}

\section[$\broken d k$ is a component of $\minhilbsing d k$]{Showing $\broken d k$ is an irreducible component of $\minhilbsing d k$}
\label{subsection:irreducibility-in-the-locus-of-singular-scrolls}

In this section our main aim is to prove 
\autoref{proposition:minutely-broken-closure},
	which tells us that $\broken d k$ sits inside
$\minhilbsing d k$ as an irreducible component.

\begin{remark}
	\label{remark:coskun-degenerations-of-scrolls}
	We present a proof of \autoref{corollary:isolated-points-in-minutely-broken-scrolls}
using \autoref{proposition:distinct-hilbert-polynomials-implies-reducible-source},
a general fact about reducibility of a family of varieties with a reducible special fiber.
However, one can also prove it using a suitable generalization of
\cite[Proposition 4.1]{coskun:degenerations-of-surface-scrolls}.
Explicitly, here is the generalization:
Suppose $\scx \ra B$ is a flat family with $B$ a curve
so that the general fiber is a smooth scroll, and that the special fiber is nondegenerate.
Then, the special fiber is a connected variety whose irreducible components are
scrolls. 
In particular, since the number of components in a family is upper semicontinuous in reduced families,
and nondegeneracy is an open condition,
\autoref{corollary:isolated-points-in-minutely-broken-scrolls}
follows without much work.
\end{remark}

From \autoref{corollary:isolated-points-in-minutely-broken-scrolls},
this gives a way of checking whether we have produced
an isolated point of a fiber of $\pi_2$ from the definition of interpolation, \autoref{definition:interpolation},
and will enable us to prove scrolls satisfy interpolation.

Our first goal is to prove \autoref{proposition:distinct-hilbert-polynomials-implies-reducible-source}.
This will be crucial in showing that if we have a degeneration of some family whose general member
is reducible to a general element of $\broken d k$, then that family actually has general member in $\broken d k$.

\begin{proposition}
	\label{proposition:distinct-hilbert-polynomials-implies-reducible-source}
	Suppose $f:X \ra Y$ is a flat map of projective varieties over an algebraically
	closed field $\bk$ where $Y$ is a smooth connected projective
	curve and the geometric fibers over the closed points of $Y$ all have two components
	and are reduced.
	Further, suppose that there is one closed geometric point $p \in Y$ 
	so that the fiber over $p$ has two
	irreducible components $Z_1$ and $Z_2$ with distinct Hilbert polynomials.
	Then $X$ is reducible with two components $X = X_1 \cup X_2$.
	Further, up to permutation
	of these components,
	we have $X_i|_p = Z_i$ for $i = 1, 2$.
\end{proposition}
\sssec*{Idea of Proof}
To show $X$ has two irreducible components, we show the fiber of $f$ over the generic point has two components.
To show this, we first show the generic fiber of $f$ has two irreducible components with different 
Hilbert polynomials after a finite base change is made.
Then, we note that the Galois group of this finite base change cannot
permute the two components because they have different Hilbert polynomials, and
so the generic fiber must have had two components before the finite base change.

\begin{remark}
	\label{remark:}
	Before proceeding with the proof of \autoref{proposition:distinct-hilbert-polynomials-implies-reducible-source}, 
	we briefly outline an alternate approach,
	suggested by Anand Patel,
	with the caveat that this approach is destined to failure in characteristic $p > 0$.

	The idea for this alternate approach is to split apart the components
	of each fiber of $f$ by taking the normalization of $X$, and then argue they must
	have the same Hilbert polynomial by using Stein factorization.
	More specifically, let $\tilde X \ra X$ be the normalization of $X$.
	Stein factorization of $\tilde X \ra Y$ gives a factorization
	$\tilde X \xrightarrow h Y' \xrightarrow g Y$ where $h$ has connected fibers and
	$g$ is finite. If we knew that the fibers of $\tilde X \ra Y$ were the normalizations
	of the corresponding fibers of $X \ra Y$, the fibers of $\tilde X \ra Y$
	would consist of two components, one for each irreducible component of
	the fibers of $X \ra Y$. If $X$ were irreducible, $\tilde X$ would also be irreducible,
	and so $Y'$ would also be irreducible.
	Then, since $f$ is flat, every fiber of $\tilde X \ra Y'$ would have the same
	Hilbert polynomial, contradicting the assumption that the two components
	over the special fiber have distinct Hilbert polynomials.

	It is worth noting that the proof sketch above uses
	characteristic $0$ in an essential way for the statement that the normalization of the fiber
	is the fiber of the normalization. More specifically, this is the following lemma.
	\begin{lemma}
		\label{lemma:normalization-of-fiber}
Suppose $X \ra Y$ is a map of reduced connected projective schemes of finite type over an algebraically closed field of characteristic 0, where $Y$ is a smooth connected curve. Let $\tilde X \ra X$ be the normalization of X. Then, for all but finitely many closed points $y \in Y$
we have that $\tilde X_y$ is the normalization of $X_y$.
	\end{lemma}
	\begin{proof}
		A sketch is given in the comments of 
		\cite{MO:why-is-the-normalization-of-a-general-fiber-the-general-fiber-of-the-normalization}
		by nfdc23.
	\end{proof}
	
	Note that \autoref{lemma:normalization-of-fiber} is emphatically false in characteristic $p$:
	quasi-elliptic fibrations provide a counterexample in characteristics $2$ and $3$, as pointed out by
	Jason Starr in the comments to 
	\cite{MO:why-is-the-normalization-of-a-general-fiber-the-general-fiber-of-the-normalization}.
	
	As an additional fun observation, we can give \autoref{lemma:normalization-of-fiber}
	the following funny slogan:
\begin{displayquote}
	Normal families have normal members, provided the families do not have a positive characteristic.
\end{displayquote}
\end{remark}

We now return to the proof of \autoref{proposition:distinct-hilbert-polynomials-implies-reducible-source}.

\begin{proof}[Proof of \autoref{proposition:distinct-hilbert-polynomials-implies-reducible-source} assuming \autoref{lemma:two-components-over-separable-closure}, \autoref{lemma:component-by-component-specialization},
	and \autoref{lemma:galois-preserves-hilbert-polynomial}]

Let $\eta$ denote the generic point of $Y$,
		let $X|_\eta$ denote the fiber of $f$ over $\eta$,
		and let $K(Y)$ denote the fraction field of $Y$.
		
		By flatness of $f$, in order to show
		$X$ has two irreducible components, it suffices to show
		$X|_\eta$ has two irreducible components.

Let $L$ be the finite Galois field extension of $K(Y)$ constructed in
\autoref{lemma:two-components-over-separable-closure},
so that $(X|_\eta)_L$ has two irreducible components.
Define $\tilde Y$ to be the normalization of $Y$ in $L$,
as defined in \cite[Exercise 9.7.I]{vakil:foundations-of-algebraic-geometry}.
Let $\tilde \eta$ be the generic point of $\tilde Y$.
Define $\tilde X$ as the fiber product
\begin{equation}
	\nonumber
	\begin{tikzcd} 
		\tilde X \ar {r} \ar {d}{\tilde f} & X \ar {d}{f} \\
		\tilde Y \ar {r} & Y.
	\end{tikzcd}\end{equation}
Now, directly from the definitions, we have
\begin{align*}
	\tilde X |_{\tilde \eta} = \left( X|_\eta \right)_L.
\end{align*}
Therefore, $\left( X|_\eta \right)_L$ is the generic fiber
of $\tilde f$.
Further, because $f$ is flat, so is $\tilde f$.
Since $\tilde X|_{\tilde \eta}$ 
has two components, $\tilde X$ must also have two components,
by flatness of $\tilde f$.
Call these two components $\tilde X_1$ and $\tilde X_2$.
Let $\tilde p \in \tilde Y$ be a closed point
with $f(\tilde p) = p$.
Note that $(\tilde X|_{\tilde p}) \cong X|_p$
and so $(\tilde X|_{\tilde p})$ has two components which we can identify with
$Z_1$ and $Z_2$.
Next, $Z_i = \tilde X_i|_{\tilde p}$ by \autoref{lemma:component-by-component-specialization}.
So, the Hilbert polynomial of $\tilde X_i$ agrees with that
of $Z_i$. In particular, the Hilbert polynomials of
$\tilde X_1|_{\tilde \eta}$ and $\tilde X_2|_{\tilde \eta}$ are distinct.

Hence, by \cite[Lemma 11.2.15]{vakil:foundations-of-algebraic-geometry},
applied to the field extension $L/K(Y)$,
the Galois group $Gal(L/K(Y))$ acts transitively on the
irreducible components of $\tilde X|_{\tilde \eta}$ a given component of $X|_\eta$.
So, by \autoref{lemma:galois-preserves-hilbert-polynomial},
since the two components $\tilde X_i|_{\tilde \eta}$ 
have different Hilbert polynomials,
they must lie over distinct points of $X|_\eta$.
In other words, $X|_\eta$ must have to components,
call them $X_1$ and $X_2$.
Finally, by another application of \autoref{lemma:component-by-component-specialization},
we obtain that $X_i|_p = Z_i$.
\end{proof}

We now proof \autoref{lemma:two-components-over-separable-closure},
\autoref{lemma:component-by-component-specialization},
and \autoref{lemma:galois-preserves-hilbert-polynomial}
in succession, which will complete the proof of
\autoref{proposition:distinct-hilbert-polynomials-implies-reducible-source}.
	\begin{lemma}
		\label{lemma:two-components-over-separable-closure}
		With the same notation as \autoref{proposition:distinct-hilbert-polynomials-implies-reducible-source},
		Let $\eta$ denote the generic point of $Y$,
		let $X|_\eta$ denote the fiber of $f$ over $\eta$,
		and let $K(Y)$ denote the fraction field of $Y$.
		Then, there is some finite Galois extension $L$ of $K(Y)$
		so that the base change of the generic fiber $(X|_\eta)_L$
	has two irreducible components.	
	\end{lemma}
	\sssec*{Idea of Proof}
	We do this in three steps, first by showing that the base change of the
	generic fiber to the algebraic closure has two components,
	then by showing the base change to the separable closure has two components,
	and finally by finding 
	the desired finite extension. The step regarding the separable closure
	is necessary to ensure separability of $L/K(Y)$.
	\begin{remark}
		\label{remark:finite-base-change-subscheme}
	Here is a summary of why,
	if we start with a scheme $X$ over a field $K$, and consider
	an algebraic base change by $\spec L \ra \spec K$ to $X_L$,
	for any closed subscheme $Z \subset X_L$, we can in fact find a finite extension
	$\spec M \ra \spec K$ and a closed subscheme $W \subset X_M$ so that $Z = W_L$.

	If the extension $L/K$ has two components, one of the components will
	have ideal defined by a finite number of algebraic elements. These elements each generate
	a finite extension, and therefore the extension generated by all of them will be finite.
	After just passing to this finite extension, we will already see this component.

	The above is written out in more detail in \cite[Lemma 3.2.6]{liu:algebraic-geometry-and-arithmetic-curves}.
	We include this argument here as it is a very standard
	technique, ubiquitous in Galois theory, which allows one to reduce from the case
	of algebraic extensions to finite extensions.
\end{remark}

	\begin{proof}
First, we show $(X|_\eta)_{\overline {K(Y)}}$ has two irreducible components.
Under the assumption that the fibers of $f$ are reduced, we obtain
that the number of geometrically irreducible components is upper semicontinuous
by \autoref{proposition:upper-semicontinuous-number-of-components}.
Since all geometric fibers over closed points of $Y$ have two irreducible components,
the generic geometric fiber must also have two irreducible components.
In other words, $X_{\overline {K(Y)}}$ has two irreducible components.

Next, we show that $(X|_\eta)_{K(Y)^s}$ has two irreducible components, where
$K(Y)^s$ denotes the separable closure of $K(Y)$.
By \cite[Proposition 2.7]{liu:algebraic-geometry-and-arithmetic-curves},
any base extension by a purely inseparable field extension is a homeomorphism.
Therefore, since $K(Y)^s \rightarrow \overline {K(Y)}$
is purely inseparable, we obtain that $(X|_\eta)_{\overline {K(Y)}} \cong (X|_\eta)_{K(Y)^s}$ are
homeomorphic. In particular, $(X|_\eta)_{K(Y)^s}$ has two irreducible components if $(X|_\eta)_{K(Y)^s}$ does.

Finally, by \cite[Lemma 3.2.6]{liu:algebraic-geometry-and-arithmetic-curves}
	(see \autoref{remark:finite-base-change-subscheme} for a discussion of how to prove this),
we have a finite separable extension $L'/K(Y)$ so that the base change $(X|_\eta)_{L'}$
has two components. Taking the normalization of $L'$ inside $K(Y)^s$
produces the desired Galois extension $L/K$.
	\end{proof}

\begin{lemma}
	\label{lemma:component-by-component-specialization}
	Suppose we have a flat map of projective schemes
	$g: X \ra Y$ where $Y$ is a smooth integral curve.
	Further, assume that all the fibers of $g$ are reduced
	and composed of
	precisely two irreducible components.
	Let $\eta$ denote the generic point of $Y$.
	Then for any closed point $p \in Y$,
	if we write $X|_p = Z_1 \cup Z_2$ and $X = X_1 \cup X_2$,
	we have that $Z_i = X_i|_p$.
\end{lemma}
\begin{proof}
First, let $g_i$ be the restriction of $g$ to $X_i$.
Then, since $g$ is flat, it is dominant when restricted to
any component. Then, the restrictions $g_i$
are flat, 
because they are dominant maps to a curve. 

Therefore, the Hilbert polynomial of the fibers of $g_i$
must be locally constant over $Y$.
If $X_i|_p$ were irreducible, 
then the intersection of $X_i$ with $X|_p$ must be either $Z_i$
or $Z_{1-i}$, and because the two components have different Hilbert polynomials,
it must be $Z_i$.

So, we only need rule out the possibility that the restriction of $X_i$ to $Z_i$
is reducible. This is where we use the assumption that the fibers of $g$ are reduced.
If the restriction of $X_i$ to $Z_i$ is reducible,
then we may assume that one of the components of $X_i|_p$, call it $W$
is strictly contained in $Z_{1-i}$.
But then, $X|_p$ will have the same reduction
as $(X_i|_p \setminus W) \cup X_{1-i}|_p$. In particular,
$X|_p$ will be nonreduced, since the Hilbert polynomial
of $(X_i|_p \setminus W) \cup X_{1-i}|_p$
is less than that of $X|_p$, but a reduced scheme
has Hilbert polynomial which is minimal among all
schemes with the same reduction.
Therefore, we have a contradiction, and the $Z_i$ were in fact irreducible.
\end{proof}

\begin{lemma}
	\label{lemma:galois-preserves-hilbert-polynomial}
	Suppose $L/K$ is a finite Galois extension
	and $Gal(L/K)$ acts on a variety $X$.
	Then, for every irreducible component $X_i \subset X$
	and $\sigma \in Gal(L/K)$, we have that $\sigma(X_i)$
	is an irreducible component of $X$ with the same Hilbert polynomial
	as $X_i$.
\end{lemma}
\begin{proof}
	First, because $\sigma \in Gal(L/K)$ is invertible, it must send irreducible
	components to irreducible components. To complete the proof, it suffices
	to show $\sigma$ preserves the Hilbert polynomial. Showing this is just a matter
	of unwinding the definitions.

	Let $\scl$ be the
invertible sheaf giving the embedding $X \ra \bp^n$ so that,
as a function of $m \in \bz$, 
the Hilbert series is $h_X(m) = H^0(X, \scl^{\otimes m})$.
Then, we obtain that the map $\sigma: X_i \ra \sigma(X_i)$ determines an isomorphism
on cohomology
\begin{align*}
	H^0(X_i, \scl^{\otimes m}|_{X_i}) \ra H^0(X_i, \scl^{\otimes m}|_{\sigma(X_i)}).
\end{align*}
So, these two vector spaces have the same dimension over $K$.
But then, since $L$ is finite dimensional over $K$, these vector spaces have
the same dimension over $L$.
\end{proof}

The above \autoref{proposition:distinct-hilbert-polynomials-implies-reducible-source}
will help us show that reducible varieties cannot degenerate to a general element
of $\broken d k$. However, in order to show that only smooth scrolls or other scrolls
in $\broken d k$ cannot degenerate to scrolls in $\broken d k$, we will
need to rule out the possibility that cones can degenerate to scrolls in $\broken d k$.
This is done in \autoref{lemma:cone-degeneration}.
But, to prove this, we need a quick lemma showing that any two smooth curves are isomorphic \'etale locally.

\begin{lemma}
    \label{lemma:curve-etale-isomorphism}
    If $C$ and $D$ are two projective
    curves, then any two smooth points $p \in C, q \in D$
    have some \'etale neighborhood on which they are isomorphic.
    In other words, there exists Zariski open sets $U \subset C, V \subset D$
    and a curve $W$ so with \'etale maps $W \rightarrow U$ and $W \rightarrow V$.
  \end{lemma}
  \begin{proof}
    Since the notion of being \'etale locally isomorphic is transitive,
    it suffices to show that any curve $C$ is \'etale locally isomorphic to
    $\bp^1$, at a smooth point. For this, we only need produce
    an map $C \ra \bp^1$ which is \'etale at a given smooth point $p$ of $C$.
    However, this can be accomplished by taking an invertible sheaf
    $\scl$ on $C$ of degree at least $2g+1$, which will determine
    a basepoint free map $C \ra \bp^{g+1}$.
    Then, a general two dimensional subspace of $H^0(C, \scl)$ will determine
    a map $C \ra \bp^1$ which is unramified at $p$.
    Removing the branch points and singular locus of $\pi$
    yields the desired \'etale map.
  \end{proof}

\begin{lemma}
  \label{lemma:cone-degeneration}
The locus $\broken d k$ cannot be in the closure of the locus of cones.	  
\end{lemma}
\begin{remark}
  \label{remark:}
  In fact, the same exact proof shows that every generically
  smooth degeneration of a cone
  is a cone.
\end{remark}
\sssec*{Idea of the proof of \autoref{lemma:cone-degeneration}}
If some cones degenerated to a generically smooth variety which
is not a cone, we can find a point $p_0$ in the special fiber
which is the limit of the cone points.
Taking a point $p$ in the special fiber so that the line
$\overline {p_0, p}$ is not contained in the special fiber,
we can slightly perturb $p_0$ and $p$ so that the lines through
these perturbations are contained in the nearby fibers.
But, this is a contradiction because the limit of lines contained
in the nearby fibers will be a line contained in the special fiber.

\begin{proof}
  Suppose $\broken d k$ did lie in the closure of the locus of cones.
  Then, a general $[W] \in \broken d k$ could be constructed as
  a fiber of a flat map $\pi:\scx \rightarrow Y$ where $Y$ is a smooth curve and a general
  fiber of the map is a cone.
  Say $X = \pi^{-1}(y_0)$.
  Now, by \cite[Theorem 25.2.2]{vakil:foundations-of-algebraic-geometry},
   if $A \rightarrow B$ is a map of schemes, $x \in f^{-1}(b)$,
   and $B$ is smooth, then $x$ is a smooth point of $f^{-1}(b)$
   if and only if $x$ is a smooth point of $A$.
  So, in our case at hand, the singular locus of $\pi: \scx \rightarrow Y$
  contains all points corresponding to the cone points of fibers of $\pi$.
  Now, let $C$ be the 
  be the closure of the locus of cone points,
  More formally, we can take $C_1$ to be the singular locus of $\pi$
  over the open set points in why whose fibers are cones,
  and then take $C$ to be the closure of $C_1$ in $\scx$.

  We obtain that $C \subset \scx$ is a curve mapping dominantly, hence flatly to $Y$.
  If necessary, replace $C$ by its reduction so that the fiber over a general point
  will be reduced.
  Then, by flatness, the fiber over every point will be of degree 1. 
  Then, $Y \cap C$ is some reduced point $p_0 \in Y$.
  Choose a smooth point $p \in X$ so that there is no line in $X$ joining $p_0$ and $p$.
  This is possible by \autoref{lemma:linear-spaces-in-varieties-of-minimal-degree},
  as the lines through any point $p_0 \in Y$ do not span all of $Y$.
  
  There is then an irreducible curve $D \subset \scx$ whose fiber over $y_0$ contains $p$,
  but so that $D$ is not contained in $X$ because $p$ is a smooth point of $X$.
  After taking an \'etale base change of $Y$ if necessary, we may assume
  that $D \rightarrow Y$ is an isomorphism in a Zariski neighborhood of
  $p$. In other words, we can find a curve $D$ which is isomorphic to $Y$
  in an \'etale local neighborhood, as proven in \autoref{lemma:curve-etale-isomorphism}.

  We have already constructed a curve
  $C \subset \scx$ whose fiber over $y_0$ is $p_0$.
  Then, for a general $y \in Y$, we know $\pi^{-1}(y)$ is 
  a cone whose cone point is $C \cap \pi^{-1}(y)$.
  Therefore, the line 
  \begin{align*}
  L_y := \overline{C \cap \pi^{-1}(y), D \cap \pi^{-1}(y)}
  \end{align*}
  is contained $\pi^{-1}(y)$.
  Next, construct the family of lines $Z \subset \scx$ whose fiber
  over a point $y \in Y$ is the line $L_y$, assuming $\pi^{-1}(y)$ is a cone.

  \begin{exercise}
    \label{exercise:}
    Verify the family $Z \subset \bp_Y^{d+k-1}$ described above can be constructed scheme
    theoretically.
    {\it Hint:  Scheme theoretically, we can construct this family of lines as in the construction
	    of varieties of minimal degree \autoref{proposition:construction-of-scrolls-by-joining-curves} 
	    by taking those hyperplanes in the ambient projective
    space containing both $C$ and $D$, and then constructing a map to the Grassmannian
    corresponding to an appropriate exact sequence of locally free sheaves.}
  \end{exercise}
 
  The family $Z$ is in fact contained in $\scx$ because it is reduced and
  it is set theoretically contained in $\scx$, as the general fiber
  of $Z$ over $Y$, $L_y$,
  is contained in $\scx$.
  It follows that the fiber of $Z$ over $y_0$ is contained in $\scx$,
  as $\scx$ is projective. It is also a line because the family $Z$
  is flat over $Y$, and the general fiber is a line.
  Hence, $\pi^{-1}(y_0)$ is a line in $Y$ containing both $p_0$
  and $p$, contradicting the assumption that there was no line
  in $Y$ containing both points.
    \end{proof}

    In further preparation for our proof of \autoref{proposition:distinct-hilbert-polynomials-implies-reducible-source},
    we will need to know that a degenerate variety cannot degenerate to a nondegenerate variety.
 
\begin{lemma}
	\label{lemma:general-fiber-nondegenerate}
	Suppose $X \ra Y$ is a flat proper map of $\bk$-varieties, together with an embedding $X \ra \bp^n_Y$ compatible with the projection to $Y$.
	Then, if one fiber over a closed point of $Y$ is nondegenerate under this embedding,
	the general fiber is nondegenerate.
\end{lemma}
\begin{proof}
	Consider the exact sequence
	\begin{equation}
		\nonumber
		\begin{tikzcd}
			0 \ar {r} & \sci_{X/\bp^n_Y} \ar {r} & \sco_{\bp^n_Y}(1) \ar {r} & \sco_X(1) \ar {r} & 0 
		\end{tikzcd}\end{equation}
	This induces a left exact sequence on cohomology
	\begin{equation}
		\nonumber
		\begin{tikzcd}
			0 \ar {r} & H^0(\bp^n_Y, \sci_{X/\bp^n_Y}) \ar {r} & H^0(\bp^n_Y, \sco_{\bp^n_Y}(1)) \ar {r} & H^0(\bp^n_Y, \sco_X(1))
		\end{tikzcd}\end{equation}
	Now, suppose that one fiber over a closed point $y \in Y$ is nondegenerate. Equivalently, we have that for that fiber
	the restriction, $H^0(\bp^n_{\kappa(y)}, \sci_{X/\bp^n_Y}|_y) = 0$. Therefore, by upper semicontinuity of
	cohomology in flat proper families, we obtain that this cohomology group vanishes for a general closed fiber of $Y$.
	Therefore, the general fiber of $Y$ is nondegenerate.
\end{proof}

With all the above preparation, we are finally able to prove
\autoref{proposition:distinct-hilbert-polynomials-implies-reducible-source}.

\begin{proposition}
	\label{proposition:minutely-broken-closure}
	Suppose $(k,d) \neq (2,4)$.
	Let $\minhilbsing d k \subset \minhilb d k$ denote the closed locus of singular scrolls
	in $\minhilb d k$.
	Then, $\broken d k$ is an irreducible component of
	$\minhilbsing d k$.	
\end{proposition}
\begin{remark}
	\label{remark:}
	\autoref{proposition:minutely-broken-closure} emphatically does not hold in the case $k = 2, d = 4$, as the Veronese surface in $\bp^5$ will degenerate to the union of
	a plane and a line.
\end{remark}
\sssec*{Idea of proof of \autoref{proposition:minutely-broken-closure}}
The statement is asking to show that we cannot
find a family of singular scrolls not in $\broken d k$
degenerating to
an element of $\brokengeneral d k$.
If the general member were irreducible,
it would have to be a cone. Then,
the special fiber
would also be a cone by \autoref{lemma:cone-degeneration}.
On the other hand, if the general member were reducible,
then the total space of the family is reducible,
and a general fiber will be the union of a plane
and a scroll in $\minhilb {d-1} k$
using \autoref{proposition:distinct-hilbert-polynomials-implies-reducible-source}.
Then, we see that in a general fiber, the scroll must meet the $k$-plane
in a $(k-1)$-plane of the ruling, by \autoref{lemma:intersection-of-components-of-fibers}.

\begin{proof}
First, $\broken d k$ is irreducible. 
\begin{exercise}
	\label{exercise:}
	Show that $\broken d k$ is indeed irreducible.
	{\it Hint:}
	Let $\Lambda$ be a $k$-plane.
	Show that there is a surjective map from an irreducible subscheme of the product
	$\broken {d-1} k \times \hilb \Lambda$, corresponding
	to pairs of scrolls in $[X] \in \minhilb {d-1} k$ and $[Y] \in \hilb \Lambda$ 
	with $[X \cup Y] \in \broken d k$.
	Show that the resulting subscheme of $\broken {d-1} k \times \hilb \Lambda$ is irreducible
	by showing it has a map
	to the Hilbert scheme of $(k-1)$-planes in $\bp^n$, with irreducible fibers of the same dimension.
\end{exercise}

To complete the proof, we need to show $\broken d k$ is a full component of $\minhilbsing d k$.

For the sake of contradiction, suppose there were some irreducible component of $Z \subset \minhilbsing d k$
strictly containing some irreducible component of $W \subset \minhilbsing d k$.
Since the Hilbert scheme is projective, so are $Z$ and $W$. Now,
choose a point $w \in \brokengeneral d k$ (recall this
means it is the union of a smooth scroll and a plane meeting along a $(k-1)$-plane of the ruling
of the smooth scroll).
Let the corresponding variety be $\overline w$ so that $[\overline w] = w \in \brokengeneral d k$.
Then, take a subscheme $T \subset Z$ containing $w$ so that $T \cap W$ consists of finitely many points, which
exists as $Z$ is projective. Further, take $T$ to be 1 dimensional. This may be accomplished by
intersecting $T$ with hyperplanes containing $w$.
Further, we may assume $T$ is irreducible, by taking an irreducible component of $T$ containing
$w$. Further, we may assume $T$ is reduced by replacing $T$ with its reduced subscheme structure
if necessary.

Now, $T$ is an integral projective curve. Let $Y \xrightarrow g T$ be the normalization
and let $y \in Y$ be a point mapping to $w$ under the composition  $Y \xrightarrow g T \xrightarrow \iota \minhilb d k$.
Observe also that under $g \circ \iota$, only finitely
many points of $Y$ map to points in $W$. To complete the proof, it suffices to show
show that a general point of $Y$ maps to a point in $W$, as this will be a contradiction.

Let $X$ be the pullback of the universal family over the Hilbert scheme
along $g \circ \iota$ so that we have a fiber product
\begin{equation}
	\nonumber
	\begin{tikzcd} 
		X \ar {r}{g \circ \iota} \ar {d}{f} & \uhilb {\overline w}|_W \ar {d} \\
		Y \ar {r} & W
	\end{tikzcd}\end{equation}
Now, by construction, the fiber of $f$ over $y$ lies in
$\brokengeneral d k$.

Note that this fiber corresponds to a reduced scheme. Therefore, 
since the locus of fibers which are reduced is open on the target,
by \cite[Th\'eor\`eme 12.2.4(v)]{EGAIV.3}.
we may assume that all fibers of $f$ are reduced, after possibly replacing $Y$
with an open subscheme.
Now, by \autoref{proposition:upper-semicontinuous-number-of-components},
if the general closed fiber of $Y$ is not irreducible,
the general closed fiber of $Y$ has at least two components.

We now have two cases, depending on whether the general closed fiber
of $Y$ is reducible or irreducible.

\vskip.1in
\noindent
{\sf Case 1: The general closed fiber of $Y$ is irreducible}

First, let us show that if the general closed fiber of $Y$ is irreducible, it is smooth.
By \cite[Theorem 1]{eisenbudH:on-varieties-of-minimal-degree},
we know that the general closed fiber is either smooth or a cone over a smooth variety,
as we are assuming $(k,d) \neq (2,4)$. 
So, by we only need rule out the possibility
that the general closed fiber is a cone over a variety of minimal degree of lower dimension.
This follows from \autoref{lemma:cone-degeneration}.

\vskip.1in
\noindent
{\sf Case 2: The general closed fiber of $Y$ is reducible}

So, it only remains to deal with the case that every fiber of $f$ is reducible.
After replacing $Y$ by an open subset of $Y$ if necessary, we may assume that
every fiber of $f$ has two components.

Now, by \autoref{proposition:distinct-hilbert-polynomials-implies-reducible-source},
$X$ has two irreducible components, call them $X_1, X_2$.
Further, because $Y$ is a smooth curve, the map $X \ra Y$ being flat implies
that every generic point of $X$ maps to the generic point of $Y$, by
~\cite[Exercise 24.4.K]{vakil:foundations-of-algebraic-geometry}.
So, each restriction map $X_i \ra Y$ is dominant, hence flat.

Further, by \autoref{proposition:distinct-hilbert-polynomials-implies-reducible-source}, up to interchanging $1$, and $2$, we have
that $f_1^{-1}(\overline y)$ is a $k$-plane and $f_2^{-1}(\overline y)$
is a (reduced) scroll of degree $d -1$ and dimension $k$.

Now, because the maps $f_1, f_2$ are flat, every fiber of $f_1$ must be
a $k$-plane, and every fiber of $f_2$ must be a degree $d-1$, dimension $k$ scroll.

To complete the proof, it suffices to show that in every fiber, these two
schemes meet along a $(k-1)$-plane which is a plane of the ruling of the scroll.
This suffices because we would then obtain that the general member of $Y$ maps
to a point in $W$, which would be a contraction. However,
this fact is precisely the following \autoref{lemma:intersection-of-components-of-fibers}.
\end{proof}
\begin{lemma}
	\label{lemma:intersection-of-components-of-fibers}
	Retaining the notations and assumptions from the proof and statement of
	\autoref{proposition:minutely-broken-closure},
	the general closed fiber of $X \rightarrow Y$ consists of an element in
	$\brokengeneral d k$.
\end{lemma}
\begin{proof}
	Now, we will show that in every fiber of $f$, the fiber of $f_1$ and $f_2$
meet along a $(k-1)$-plane of the ruling. To show this, we will show that in any other case, the fiber of
$f$ would have Hilbert polynomial different from that of $f^{-1}(\overline y)$.

To start, we know the general fiber is nondegenerate, by \autoref{lemma:general-fiber-nondegenerate}.
So, for any point $p \in Y$, we must have that
$f_1^{-1}(p)$ and $f_2^{-1}(p)$ together span $\bp^n$. Next, because the degree
of $f_1^{-1}(p)$ has degree $d - 1$, its span is a projective space of at most of dimension $n-1$, and
because $f_2^{-1}(p)$ is a plane, its span necessarily has dimension $k$. Therefore, the span
of $f_1^{-1}(p)$ intersects the span of $f_2^{-1}(p)$ in a $(k-1)$-dimensional linear subspace.
Now, because we know $f, f_1, f_2$ are all flat, we have that
\begin{align*}
	p_{f^{-1}(p)} &= p_{f^{-1}(\overline y)} \\
	p_{f_1^{-1}(p)} &= p_{f_1^{-1}(\overline y)} \\
	p_{f_2^{-1}(p)} &= p_{f_2^{-1}(\overline y)}.
\end{align*}
Therefore, using this and \autoref{lemma:hilbert-polynomial-intersection},
we obtain
\begin{align*}
	p_{f_1^{-1}(p) \cap f_2^{-1}(p)} &= p_{f_1^{-1}(p)} + p_{f_2^{-1}(p)} - p_{f^{-1}(p)} \\
	&= p_{f_1^{-1}(\overline y)} + p_{f_2^{-1}(\overline y)} - p_{f^{-1}(\overline y)} \\
	&= p_{f_1^{-1}(\overline y) \cap f_2^{-1}(\overline y)}.
\end{align*}
But, since $f_1^{-1}(\overline y) \cap f_2^{-1}(\overline y)$ is a $(k-1)$-plane,
and the only scheme with Hilbert polynomial equal to that of a $(k-1)$-plane is a $(k-1)$-plane,
we obtain that $f_1^{-1}(p) \cap f_2^{-1}(p)$ is a $(k-1)$-plane.
It only remains to show that this $(k-1)$-plane is one of the ruling planes of the scroll.
However, this holds by \autoref{lemma:linear-spaces-in-varieties-of-minimal-degree}
because the only $(k-1)$-planes contained in a scroll are the ruling planes,
so long as $k > 2$ or both $k = 2$ and $d > 3$.
\end{proof}

We finally arrive at the main point of this section, which is how we can
use \autoref{proposition:minutely-broken-closure}
to verify interpolation.
Essentially, \autoref{corollary:isolated-points-in-minutely-broken-scrolls}
gives us a criterion for checking whether $[W] \in \broken d k$
is an isolated point in its fiber, which enables us to verify
interpolation.
Note that an alternate proof of \autoref{corollary:isolated-points-in-minutely-broken-scrolls}
follows from \autoref{remark:coskun-degenerations-of-scrolls}.

\begin{corollary}
	\label{corollary:isolated-points-in-minutely-broken-scrolls}
	Suppose $Z \subset \minhilb d k$ so that $p \in Z \cap \broken d k$ is an isolated point
	of both $Z \cap \broken d k$ and does not lie in the closure of $Z \cap \minhilbsmooth d k$.
	Further assume that $p \in \broken d k$ is general.
	Then, $p$ is an isolated point of $Z$.
\end{corollary}
\begin{proof}
	Let $\scq$ be the union of the irreducible components of $\minhilbsing d k \setminus \broken d k$.
	Taking $p$ to lie outside of the closed subset $\scq \subset \minhilb d k$
	we obtain that $p \in Z \setminus \scq$.
	Note that $Z \setminus \scq$ is nonempty precisely by \autoref{proposition:minutely-broken-closure}.
	Since $\minhilbsing d k$ is closed, $p$ will be isolated in $Z$ if and only if it is isolated
	in $Z \setminus \scq$.
	Then, by construction $\minhilb d k \setminus \scq$ only consists of
	points corresponding to smooth irreducible scrolls and scrolls in $\broken d k$.
	So, by our hypotheses, $p$ is an isolated point of $Z \setminus \scq$.
	Ergo, $p$ is an isolated point of $Z$.
\end{proof}

\chapter[Interpolation of scrolls]{Interpolation of varieties of minimal degree}
\label{section:interpolation-of-varieties-of-minimal-degree}

In this chapter, we prove that any smooth variety of minimal degree satisfies interpolation.
By \autoref{theorem:classification-of-varieties-of-minimal-degree}, we can
show this separately in the cases of a Hilbert scheme whose general member is a quadric hypersurface,
the $2$-Veronese surface in $\bp^5$ and a smooth rational normal scroll.
We know from \autoref{lemma:balanced-complete-intersection} that quadrics satisfy interpolation.
Apart from the special case of the $2$-Veronese surface, which we now dispense with, we shall concentrate
on showing that rational normal scrolls satisfy interpolation.
The general idea of the proof is to fix the dimension of scrolls, and induct on
the degree.
The elements of this inductive argument
is summarized verbally in \autoref{theorem:scrolls-interpolation} and
pictorially in \autoref{figure:induction-schematic}.

\section{Interpolation of the $2$-Veronese}
\label{ssec:2-veronese-interpolation}
In this section, we begin by summarizing a result going back to Coble \cite[Theorem 19]{coble:associated-sets-of-points}, which follows
more rigorously from work in Dolgachev \cite[Theorems 5.2 and 5.6]{dolgachev:on-certain-families-of-elliptic-curves-in-projective-space}, showing there are four $2$-Veronese surfaces in $\bp^{5}$ containing $9$ general points over an algebraically closed field of characteristic $0$.
We are further able to show in \autoref{theorem:counting-veronese-interpolation} that there are four such surfaces in all characteristics other than $2$, while there are only
$2$ such surfaces in characteristic 2.
This suggests that these two such surfaces may ``count twice.'' In fact, they do, in the sense that
the map to $(\bp^5)^9$ is not separable in characteristic $2$. This yields an example of how
interpolation is not equivalent to interpolation of the normal bundle over a field of characteristic $2$, as given in
\autoref{corollary:failure-of-2-veronese-vector-bundle-interpolation}.

The key to finding $2$-Veronese surfaces through $9$ points is to find a genus $1$ curve through the $9$ points,
and then find a $2$-Veronese surface containing that curve. We start off by understanding $2$-Veronese
surfaces containing a genus $1$ curve.

\begin{proposition}
  \label{proposition:veronese-surfaces-through-genus-1-curves}
  Let $\bk$ be an algebraically closed field and let $E \subset \bp^5_\bk$
  be a genus $1$ curve, embedded by a complete linear series of degree
  $6$.
  If $\ch k \neq 2$, there are precisely four $2$-Veronese surfaces
  containing $E$ and if $\ch \bk = 2$, there are precisely two $2$-Veronese
  surfaces containing $E$.
\end{proposition}
\begin{remark}
  \label{remark:}
  It is shown in
     \cite{dolgachev:on-certain-families-of-elliptic-curves-in-projective-space}[Theorem 5.6] that there are four exactly four $2$-Veronese surfaces containing
  a given genus 1 curve of degree $6$ in $\bp^5$
over a field of characteristic $0$.
  However, the proof given there does not make it completely clear
  why there is a unique $2$-Veronese surface through $E$
  corresponding to each chosen square root of the line bundle embedding
  $E$. Therefore, we now repeat the proof in more detail,
  and generalize it to all characteristics.
\end{remark}

\begin{proof}
  Say $E \ra \bp^5$ is given by the invertible sheaf $\scl$. For any degree three invertible sheaf
$\scm$ with $\scm^{\otimes 2} \cong \scl$, we can map $E \ra \bp^2$ using $\scm$. Then, the composition of
$E \ra \bp^2$ with the $2$-Veronese map $\bp^2 \ra \bp^5$ will send $E$ to $\bp^5$ by $\scl$ and so we have
constructed a $2$-Veronese surface containing $E$. Since there are two such sheaves $\scm$
in characteristic $2$ and four in all other characteristics (since a general genus $1$ curve has
two $2$ torsion points in characteristic $2$ and four such points in all other characteristics), it suffices to show these are the only 
$2$-Veronese surfaces containing $E$.
That is, we only need show that for each square root $\scm$ of $\scl$, there is a unique $2$-Veronese surface
$X \cong \bp^2$ containing $E$ so that the map $E \ra \bp^2$ is given by a basis for the global sections of $\scm$.
 
  First, note that if an automorphism fixes $E$ pointwise then it fixes all of $\bp^5$.
  This holds because $E$ spans $\bp^5$, and so a linear automorphism fixing $E$ pointwise would
also fix a basis for the vector space $H^0(\sco_E(1))$ which satisfies $\bp H^0(\sco_E(1)) \cong \bp^5$. Hence, such an automorphism would 
fix all of $\bp^5$.

 Suppose we have two $2$-Veronese surfaces $X$ and $X'$ containing
  $E$ so that $E$ 
  we have a map $\phi_1: E \ra X$ and $\phi_2 : E \ra X'$ so that
  both maps $\phi_1$ and $\phi_2$ are given by the same degree $3$ invertible
sheaf $\scm$, together with a choice of basis for $H^0(E, \scm)$.
  We will show that there exists an automorphism $\phi:\bp^5 \ra \bp^5$
  fixing $E$ pointwise and sending $X$ to $X'$. Since any automorphism of $\bp^5$
  fixing $E$ pointwise is the identity, this would imply $X = X'$, and would complete the proof.
 
  First, we show there is an automorphism $\phi: \bp^5 \ra \bp^5$
  fixing $E$ as a set and taking
  $X$ to $X'$.
  We know there is an automorphism $\psi:\bp^5 \ra \bp^5$ with $\psi(X) = X'$.
  Say $\psi$ sends the curve $E \subset X$
  to some curve $E' := \psi(E) \subset X'$.
  Next, by our assumption that $E$ and $E'$ are two curves on $X'$ both given by
  global sections associated to the same invertible sheaf $\scm$,
  there is some automorphism of $\psi': X' \ra X'$ with $\psi'(E') = E$.
  Thus, taking $\phi := \psi' \circ \psi$, we see
  $\phi(X) = X'$ and $\phi(E) = E'$ as sets.
  If we could arrange for $\phi|_E = \id$, we would be done, as then $\phi = \id$.

  Hence, it suffices to show that $\phi|_E$ is an automorphism of $E$ fixing both 
  $X$ and $X'$.    

  Let $A(E, \scm)$ denote the automorphisms $\pi:E \ra E$ with $\pi^* \scm \cong \scm$.
  Note that we have an exact sequence
\begin{equation}
  \nonumber
  \label{equation:}
  \begin{tikzcd} 
    0 \ar{r} & E[3] \ar {r} & A(E, \scm) \ar{r} & \bz/2 \ar{r} & 0 
  \end{tikzcd}\end{equation}
  where the generator of the quotient $\bz/2$ is the hyperelliptic involution and the
  subset $E[3]$ is a torsor over the $6$ torsion of $E$ with any given
  choice of origin.
  In particular, if we choose a point $p$ so that $\scm \cong \sco_E(3p)$,
  we have that $E[3]$ is precisely translation by $6$-torsion.

  It suffices to show that any element of $A(E, \scm)$
  fixes the $2$-Veronese surface we constructed above corresponding
  to $\scm$.

  But, if we view $E \ra \bp^2$ by a completely linear system corresponding
  to $\scm$, the automorphisms
  $A(E, \scm)$ are precisely the automorphisms of $\bp^2$
  fixing $E \subset \bp^2$ as a set.
  These automorphisms of $\bp^2$ extend to automorphisms on $\bp^5$
  with $\bp^2 \ra \bp^5$ embedded via the $2$-Veronese map.
  Therefore, they also fix the $2$-Veronese surface, as desired.
\end{proof}

\begin{theorem}
  \label{theorem:counting-veronese-interpolation}
  Through 9 general points in $\bp^5_\bk$ there exist precisely four
  $2$-Veronese surfaces $\bp^2 \ra \bp^5$ if $\bk$
  is an algebraically closed field with $\ch \bk \neq 2$
  and precisely two $2$-Veronese surfaces $\bp^2 \ra \bp^5$
  if $\bk$ is an algebraically closed field with $\ch \bk = 2$.
  In particular, the $2$-Veronese surface satisfies interpolation.
\end{theorem}
\begin{proof}
  Fix $9$ general points $p_1, \ldots, p_9 \in \bp^5$.
  First, by \cite{dolgachev:on-certain-families-of-elliptic-curves-in-projective-space}[Theorem 5.2], there is a unique genus 1 curve embedded by
  a complete linear series through 9 general points in $\bp^5$.
  Call this curve $E$.
  Next, by \autoref{proposition:veronese-surfaces-through-genus-1-curves},
  there are four $2$-Veronese surfaces containing $E$
  if $\ch \bk \neq 2$ and two $2$-Veronese surfaces containing $E$
  if $\ch \bk = 2$.
  To complete the proof, 
  it suffices to show that every $2$-Veronese surface containing
  $p_1, \ldots, p_9$
  also contains $E$.
  Consider such a $2$-Veronese surface $X \subset \bp^5_\bk$ containing
  $p_1, \ldots, p_9$.
  Choosing an isomorphism $\phi: \bp^2_\bk \cong X$,
  we have nine points $q_1, \ldots, q_9$ on $\bp^2$
  so that $\phi(q_i) = p_i$.
  Then, since $p_1, \ldots, p_9$ were general on $\bp^5_\bk$,
  we have that $q_1, \ldots, q_9$ are general on $\bp^2_\bk$,
  and so there is a degree $3$ genus 1 curve $C$ passing through
  $q_1, \ldots, q_9$ on $\bp^2_\bk$, by
  ~\autoref{lemma:balanced-complete-intersection}.
  The image of $\phi(C)\subset X$ is a degree $6$ genus $1$ curve
  containing $p_1, \ldots, p_9$.
  Since $E$ is the unique genus $1$ degree $6$ curve $E$ containing
  $p_1, \ldots, p_9$, we must have $\phi(C) \cong E$,
  and therefore $E \subset X$.
\end{proof}

\section[Normal bundle interpolation of the $2$-Veronese]{The normal bundle of the $2$-Veronese does not satisfy interpolation in characteristic 2}
The main goal of this section, is to show that,
while the $2$-Veronese surface satisfies interpolation in all characteristics,
its normal bundle fails to satisfy interpolation in characteristic $2$.
The key idea is that in characteristic $2$, there are only
two $2$-Veronese surfaces through $9$ points while in characteristic $0$
there are $4$. We can realize this number as the degree of a map
over $\spec \bz$, which will show it is not separable in
in characteristic $2$. We prove failure of normal bundle
interpolation of the $2$-Veronese in characteristic $2$ in
\autoref{corollary:failure-of-2-veronese-vector-bundle-interpolation}.

We start by setting up our situation, over a general base ring.
Let $X$ be a smooth $2$-Veronese surface and let $R$ be a ring.
Let $\ringhilb X R$ denote
the irreducible component of the Hilbert scheme over $\spec R$
whose general member is a $2$-Veronese surface and let
$\ringuhilb XR$ denote the universal family over $\ringhilb XR$. Define
\begin{align*}
  \Phi_R := \ringuhilb XR \times_{\ringhilb XR} \cdots \times_{\ringhilb XR} \ringuhilb XR,
\end{align*}
where there are $9$ copies of $\ringuhilb XR$.

For $R$ an arbitrary ring, define the maps $f_R, g_R, \pi_R$ as the natural projections
\begin{equation}
\nonumber
  \label{equation:}
  \begin{tikzcd} 
    \ringuhilb XR \times_{\ringhilb XR} \Phi_R \ar{d}{f_R} \\
    \Phi_R \ar {d}{g_R} \\
    \left( \bp^5_R \right)^9 \ar {d} \\
    \spec R.
  \end{tikzcd}\end{equation}
Let $\pi_R := g_R \circ f_R$.
Let $\pi_R^{\text{sing}}$ denote the singular locus of $\pi_R$.
In other words, it is the locus in 
$\ringuhilb XR \times_{\ringhilb XR} \Phi_R$ over which the relative
sheaf of differentials
$\Omega_{\pi_R}$ is not a locally free sheaf of rank equal to
$\dim  \ringuhilb XR \times_{ \ringhilb XR} \Phi_R  - \dim\left( \bp^5_R \right)^9 = 2$.
We have $\pi_R(\pi_R^{\text{sing}}) \subset \left( \bp^5_R \right)^9$.
Define the open subscheme 
\begin{align*}
V_R := \left( \bp^5_R \right)^9 \setminus \pi_R\left( \pi_R^{\text{sing}} \right).
\end{align*}
So, we can enlarge our diagram to
\begin{equation}
  \label{equation:w-veronese-diagram}
  \begin{tikzcd} 
    \pi_R^{-1}(V_R) \ar{r}\ar{dd} & \ringuhilb XR \times_{ \ringhilb XR} \Phi_R \ar{d}{f_R} \\
    \qquad & \Phi_R \ar {d}{g_R} \\
    V_R \ar{r} & \left( \bp^5_R \right)^9 \ar {d} \\
    \qquad & \spec R.
  \end{tikzcd}\end{equation}
Next, let $W_R \subset V_R$ be the open subset on which the fibers
of $g_R$ have dimension $0$. We further enlarge our diagram to
\begin{equation}
  \label{equation:v-veronese-diagram}
  \begin{tikzcd} 
    \pi_R^{-1}(W_R)\ar{r}\ar{dd} &\pi_R^{-1}(V_R) \ar{r}\ar{dd} & \ringuhilb XR \times_{ \ringhilb XR} \Phi_R \ar{d}{f_R} \\
    \qquad & & \Phi_R \ar {d}{g_R} \\
    W_R \ar{r} & V_R \ar{r} & \left( \bp^5_R \right)^9 \ar {d} \\
    \qquad & \qquad & \spec R.
  \end{tikzcd}\end{equation}

We have now completed the setup.
In what follows, we will show that the normal bundle of the $2$-Veronese
fails to satisfy interpolation using the idea that
$\pi_\bz$ is a generically finite map. Its degree is $4$
because every point in the target has four preimages over the geometric
generic point of $\spec \bz$.
It follows that the fiber of $\pi_{\overline{\mathbb F_2}}$ 
also has degree $4$, but since a general point has only $2$ preimages,
the map must be inseparable, implying that the normal bundle fails
to satisfy interpolation.
This proof is essentially carried out in
\autoref{subsection:finishing-failure-proof}.
In order to make sense of the degree of $\pi_\bz$ and
$\pi_{\overline{\mathbb F_2}}$, we will need to know that it is flat
over $W_\bz$ and that $W_\bz$ has nonempty fiber over $\spec \mathbb F_2$.
Flatness is established in \autoref{subsection:flatness}
while the nonemptiness of the fiber is established in
\autoref{subsection:surjectivity}.

\subsection{Flatness of $\pi_R$}
\label{subsection:flatness}

We prove that $\pi_R|_{g^{-1}_R(V_R)}$ is flat in
\autoref{lemma:flatness-of-g}.
For this, we need a simple scheme theoretic result, which we first prove.

\begin{lemma}
  \label{lemma:regular-source-of-smooth-map}
  Let $f:X \ra Y$ be a smooth map of schemes and suppose $Y$ is regular.
  Then $X$ is regular.
\end{lemma}
\begin{proof}
Regularity is a local notion, so we will choose $\eta$ a (not necessarily
closed) point of $X$.
It suffices to show $X$ is regular at $\eta$.
Then, by assumption, $f(\eta)$ is a regular point of $Y$,
so we can choose a regular sequence
$y_1, \ldots, y_t$ in $\sco_{Y, f(\eta)}$
cutting out the maximal ideal.
Then, recall each $y_i$ corresponds to a Cartier divisor 
$D_i$ in a neighborhood of $f(\eta)$.
Let $f^*(D_i)$ be the corresponding Cartier divisors
in a neighborhood of $\eta$.
Repeatedly using
\cite[Exercise 12.2.C]{vakil:foundations-of-algebraic-geometry},
for the smooth map $f:X \ra Y$, in order to show
$\eta$ is regular, it suffices to show
$f^{-1}(f(\eta))$ is regular.
But, we know $f^{-1}(f(\eta))$ is smooth over $K(f(\eta))$
by
\cite[Theorem 25.2.2(iii)]{vakil:foundations-of-algebraic-geometry},
and hence regular by
\cite[Theorem 12.2.10(b)]{vakil:foundations-of-algebraic-geometry}.
\end{proof}

\begin{lemma}
  \label{lemma:flatness-of-g}
  Suppose $R$ is a regular ring.
  Then, the map $g_R|_{g_R^{-1}(W_R)}: g_R^{-1}(W_R) \ra W_R$ is flat.
\end{lemma}
\begin{proof}
  We aim to apply miracle flatness, see 
  \cite[Theorem 26.2.11]{vakil:foundations-of-algebraic-geometry}
 to the map 
  $g_R|_{g_R^{-1}(W_R)}$.
  If $W_R = \emptyset$, the statement is vacuous,
  so we may assume $W_R \neq \emptyset$.
  A standard dimension count reveals that
  $\dim \Phi_R = \dim \left( \bp^5_R \right)^9 = \dim R + 45$.
  Since $W_R$ is by definition the locus of $g_R$
  whose fibers are $0$ dimensional,
  we obtain that each fiber of $g_R|_{g_R^{-1}(W_R)}$
  has dimension $\dim g_R^{-1}(W_R) - \dim W_R = 0$.
  Hence, to apply miracle flatness,
  we only need verify that $W_R$ is regular and
  $g_R^{-1}(W_R)$ is Cohen-Macaulay.
  Of course, $W_R$ is regular because it is an open subset of
  $\left( \bp^5_R \right)^9$, which is regular
  by \autoref{lemma:regular-source-of-smooth-map}.
  So, it suffices to show $g_R^{-1}(W_R)$ is Cohen-Macaulay.
  To do this, it suffices to show $g_R^{-1}(W_R)$ is regular,
  since regular implies Cohen-Macaulay, by 
  \cite[Exercise 26.1.F]{vakil:foundations-of-algebraic-geometry}.
  
  So, to complete the proof, we will show $g_R^{-1}(W_R)$ is regular.
  Let $T_R \subset \ringhilb X R$, be the complement of the image of the singular locus of the map
  $\tau:\ringuhilb X R \ra \ringhilb X R$.
  Observe that the structure map $T_R \ra \spec R$
  it is flat of finite type and all fibers are smooth, using
  \autoref{proposition:veronese-smooth-hilbert-scheme}.
  Since $\spec R$ is regular, we also have
  $T_R$ is regular from \autoref{lemma:regular-source-of-smooth-map}.
  It follows that the $9$-fold fiber product
  $\Psi := \tau^{-1}(T_R) \times_{T_R} \cdots \times_{T_R} \tau^{-1}(T_R)$
  is regular, as it has a smooth map to $T_R$.
  Since $g_R^{-1}(V_R)$ is an open subscheme of $\Psi$,
  we have that $g_R^{-1}(V_R)$ is also regular.
\end{proof}

\subsection{Surjectivity of $W_\bz \ra \spec \bz$}
\label{subsection:surjectivity}

The main goal of this subsection is to prove that
all fibers of $W_\bz \ra \spec \bz$ are nonempty, as is done in
\autoref{lemma:w-surjective}.
Of course, we will ultimately care about the fiber over $\spec \mathbb F_2$.
The idea of this proof is to first show that $V_\bz \ra \spec \bz$ is surjective,
as is done in \autoref{lemma:v-surjective},
and then show that $W_\bz \ra \spec \bz$ is also surjective,
since the locus of $V_\bz \setminus W_\bz$ will be codimension $2$.
To establish this locus is codimension $2$, we need a general
result on codimension being the difference of dimensions, proven
in \autoref{lemma:codimension-is-difference-of-dimensions}.

\begin{lemma}
  \label{lemma:v-surjective}
  The structure map $V_\bz \ra \spec \bz$ is surjective.
\end{lemma}
\begin{proof}
  Let $\eta$ be a point of $\spec \bz$ so that we have an inclusion
  $\spec K(\eta) \ra \spec \bz$.
  Then, by 
  \cite[Exercise 28.3.1]{vakil:foundations-of-algebraic-geometry},
  we know $\ringhilb X K(\eta) = \ringhilb X \bz \times_{\spec \bz} \spec K(\eta)$.
  Therefore, \eqref{equation:w-veronese-diagram} with
  $R = K(\eta)$ is the base change of
  \eqref{equation:w-veronese-diagram} with $R = \bz$
  along the map $\spec K(\eta) \ra \spec \bz$.

  Further, \eqref{equation:w-veronese-diagram} with
  $R = \overline{K(\eta)},$ the algebraic closure of $K(\eta)$,
  is the base change of
  \eqref{equation:w-veronese-diagram} with $R = K(\eta)$
  along the map $\spec \overline{K(\eta)} \ra \spec K(\eta)$.

  Note that for a map $\iota:\spec S \ra \spec R$, 
  the singular locus of $\pi_S$ is the preimage of the singular locus
  of $\pi_R$ 
  because the singular locus of $\pi_S$ is characterized as the locus where $\Omega_{\pi_S}$
  is not locally free, and $\Omega_{\pi_S} = \iota^* \Omega_{\pi_R}$
  by
  \cite[Theorem 21.2.27]{vakil:foundations-of-algebraic-geometry}.

  So, in order to show the fiber over $\eta$ is nonempty, it suffices
  to show that the fiber over $\overline K(\eta)$ is nonempty.
  However, in this case, we obtain that
  the fiber over $\overline K(\eta)$
  is precisely the singular locus of the map $\pi_{K(\eta)}$.
  Since a general member of $\ringhilb X K(\eta)$ is a smooth $2$-Veronese
  surface, we know that 
  $f_{\overline{K(\eta)}}(\pi_{\overline{K(\eta)}}^{\text{sing}}) \subset \Phi_{\overline{K(\eta)}}$
  is a strict inclusion.
  Next,
  since the $2$-Veronese surface satisfies interpolation
  (though we will see its normal bundle does not satisfy interpolation
  in characteristic $2$ in \autoref{corollary:failure-of-2-veronese-vector-bundle-interpolation})
  in all characteristics by \autoref{theorem:counting-veronese-interpolation},
  the map $g_{\overline {K(\eta)}}$ is generically finite. 
  Therefore, $g_{\overline{K(\eta)}}(f_{\overline{K(\eta)}}(\pi_{\overline{K(\eta)}}^{\text{sing}})) = \pi_{\overline{K(\eta)}}(\pi_{\overline{K(\eta)}}^{\text{sing}}) \subset \Phi_{\overline{K(\eta)}}$
  is also as strict inclusion.
  In particular, $V_{\overline {K(\eta)}} \subset \left( \bp^5_{\overline {K(\eta)}} \right)^9$ is a dense open subset.
  Since $V_{\overline{K(\eta)}}$ is nonempty and is the base change
  of $V_{K(\eta)}$, we have that $V_{K(\eta)}$ is also nonempty.
 \end{proof}

\begin{lemma}
  \label{lemma:finite-residue-fields-over-z}
  Let $X$ be a scheme of finite type over $\spec \bz$.
  Then, if $p$ is a closed point, the fraction field $\kappa(p)$
  is a finite field.
\end{lemma}
\begin{proof}
  First, we claim that $\kappa(p)$ cannot have characteristic $0$.
  If it did, it would necessarily map to the generic point of $\spec \bz$.
  But, the generic point of $\spec \bz$ is not constructible, so
  this would contradict Chevalley's theorem, 
  \cite[Theorem 7.4.2]{vakil:foundations-of-algebraic-geometry}.
  So, we know $\kappa(p)$ has prime characteristic.
  Thus, we must have that $p$ maps to $\spec \mathbb F_q$, for some
  prime $q$. But then, we obtain that $\kappa(p)$ is a finitely generated
  scheme over $\mathbb F_q$, and so it is actually a finite extension of
  $\mathbb F_q$, by the Nullstellensatz.
  Hence, $\kappa(p)$ is a finite field.
\end{proof}

\begin{lemma}
  \label{lemma:codimension-is-difference-of-dimensions}
  Suppose $X \ra \spec \bz$ is an irreducible
  scheme of finite type over $\spec \bz$.
  Let $\eta$ be a point of $X$.
  Then $\dim \overline \eta + \dim \sco_{X, \eta} = \dim X$.
\end{lemma}
\begin{proof}
  This result with $\spec \bz$ replaced by
  a field is precisely
  \cite[Theorem 11.2.9]{vakil:foundations-of-algebraic-geometry}.
  We may assume $X \ra \spec \bz$ is dominant as otherwise
  we have that $X$ is of finite type over a finite field, in which case
  the result follows from 
  \cite[Theorem 11.2.9]{vakil:foundations-of-algebraic-geometry}.
First, note by
  \cite[Remark 11.2.10]{vakil:foundations-of-algebraic-geometry}
  that finitely generated $\bz$ algebras are catenary.
So, to complete the proof, we only need show that every closed point
$\eta$ of $X$ has the same height, which is equal to $\dim X$.
Note that every closed point of $X$ maps to a closed point of $\spec \bz$
by \autoref{lemma:finite-residue-fields-over-z}.
Say $\dim X = n + 1$.
We know the fiber $X_{\mathbb F_p}$ of $X$ over $\mathbb F_p$
must have dimension at most $n$ because every generic point of
$X_{\mathbb F_p}$ lies in the closure of the generic point of $X$, which lies
over the generic point of $\bz$.
If we knew that every component of $X_{\mathbb F_p}$ has dimension $n$,
we would be done, because then every closed point
would have the same height, equal to $n + 1$.
However, since $\bz$ is a regular curve, and the map $X \ra \spec \bz$
is dominant, it is flat by 
\cite[Exercise 24.4.K]{vakil:foundations-of-algebraic-geometry}.
Therefore, by 
\cite[Corollary 8.2.8]{liu:algebraic-geometry-and-arithmetic-curves},
all fibers of $X \ra \spec \bz$ have the same pure dimension.
Note that we can apply this as $\spec \bz$ is universally catenary
by 
\cite[Corollary 8.2.16]{liu:algebraic-geometry-and-arithmetic-curves}.
Hence, they must all have pure dimension $n$, as desired.
\end{proof}

\begin{lemma}
  \label{lemma:w-surjective}
  The map $W_\bz \ra \spec \bz$ is surjective.
\end{lemma}
\begin{proof}
  Define $X_R$ to be the locus where $g_R$ 
  from \eqref{equation:v-veronese-diagram} 
  has fiber dimension
  more than $0$
  and define $Y_R$ to be the complement of $X_R$ in $(\bp_R^5)^9$.
  Then, observe that $W_R = Y_R \cap V_R$.
  From \autoref{lemma:v-surjective}, we know $V_\bz$
  maps surjectively onto $\spec \bz$.
  So, to show that $W_\bz$ maps surjectively onto $\spec \bz$,
  it suffices to show that $Y_\bz$ maps surjectively onto
  $\spec \bz$.

  Let $\eta$ be a point in $\spec \bz$.
  Then, we want to show that $X_\bz$ does not contain
  all of $(\bp^5_{K(\eta)})^9$.
  To show this, it suffices to show that $X_\bz \subset (\bp^5_{K(\eta)})^9$
  is a subset of codimension at least $2$.
  We know that $g_\bz^{-1}(X_\bz) \subset \Phi_\bz$ is a strict closed
  subset. Therefore, 
  it has codimension at least $1$.
  Since the fibers are $1$ dimensional,
  $X_\bz = g_\bz(g_\bz^{-1}(X_\bz))$ has codimension at least $2$,
  where here we are using \autoref{lemma:codimension-is-difference-of-dimensions}.

  To spell this out in more detail, if $X_\bz$ had codimension $1$,
  it would have dimension $\dim \left( \bp^5_\bz \right)^9 - 1 = \dim \Phi_\bz - 1$.
  Then, since $g_\bz$ has one dimensional fibers over $X_\bz$, 
  $g_\bz^{-1}(X_\bz)$ would have dimension at least 
  $\dim \Phi_\bz$, meaning that it is a full irreducible
  component of $\Phi_\bz$, hence all of $\Phi_\bz$ by irreducibility
  of $\Phi_\bz$.
  \end{proof}

\subsection{Proving failure of interpolation of the normal bundle}
\label{subsection:finishing-failure-proof}

In this section, we conclude the proof that the normal
bundle to the $2$-Veronese surface fails interpolation in characteristic
$2$.
We show this by first showing the map $\pi_\bz$ has degree $4$
in \autoref{lemma:degree-4-map},
and using this to deduce that $\pi_{\overline {\mathbb F_2}}$
also has degree $4$ and is inseparable.
This implies that the normal bundle to the $2$-Veronese does not satisfy
interpolation, as shown in
\autoref{corollary:failure-of-2-veronese-vector-bundle-interpolation}.

We start by proving a general scheme theoretic result
relating the number of preimages of a map to its degree,
which will be crucially used in proving that $\pi_{\bz}$ has degree $4$.

\begin{lemma}
  \label{lemma:number-of-preimages}
  Suppose $f: X \ra Y$ is a finite dominant map of integral varieties
  inducing a map $K(Y) \rightarrow K(X)$ of fraction fields.
  Then, there is a dense open subset of $Y$ over which every
  geometric fiber has $d$ preimages, where $d$ is the separable degree
  of $K(X)/K(Y)$ (by which we mean the degree of the maximal separable
  subextension of $K(X)/K(Y)$).
  In particular, the degree $f$ is the product of the inseparable
  degree of $f$ and the number of preimage of a general point
  of $Y$.
\end{lemma}
\begin{proof}
  Since any field extension can be factored as a separable extension
  followed by a purely inseparable extension, the degree of a field extension is the product of its separable degree and its inseparable degree.
  So, the second statement follows from the first, if we know the separable
  degree is equal to the number of preimages.

  To show the separable degree is equal to the number of preimages over
  a general point, we can consider separately the cases that the map
  $f$ is separable and purely inseparable, after possibly
  replacing $Y$ by a smaller open set.
  Furthermore, we may assume that $K(Y) \rightarrow K(X)$
  is unigenerated as in general we can factor this inclusion as
  a composition of unigenerated extensions.
  Further, since the statement is only claimed over an open subset
  of $Y$, we may assume $Y = \spec B$ is affine, in which case $X = f^{-1}(\spec B) = \spec A$
  is also affine, as the map is finite.
  Then, since the corresponding map on fraction fields is unigenerated,
  we may assume, after possibly further localizing, that $B = A[t]/g(t)$.

  In the case that the extension is separable, the roots of $g(t)$
  are all distinct. This means that each fiber over a point in $A$ has $\deg g(t) = \deg f$
  preimages over the algebraic closure.
  In the purely inseparable case,
  $g(t)$ is a polynomial which has a single root over the algebraic
  closure, and so each point has a single preimage.
  \end{proof}

\begin{lemma}
  \label{lemma:degree-4-map}
  The map $g_R: \Phi_R \ra \left( \bp^5_R \right)^9$ given in
  \eqref{equation:v-veronese-diagram} is a map which is
  generically of degree $4$
  for $R = \overline \bz, \overline{\mathbb F_p}, \overline \bq$, as $p$ varies
  over all prime numbers.
\end{lemma}
\begin{proof}
  To prove this statement about generic degree, it suffices to show the map
  $g_R^{-1}(W_R) \ra W_R$ has degree $4$, as $R$ varies over
  all geometric points of $\bz$,
  since $W_R$ is nonempty by
  \autoref{lemma:w-surjective}.

  By \autoref{lemma:flatness-of-g},
  we know $g_\bz|_{g_\bz^{-1}(W_\bz)}$ is flat.
  Therefore,
  $g_R|_{g_\bz^{-1}(V_\bz)}$ is also flat
  as flatness is preserved by base change.
  Further, it is quasifinite by construction.
  It is proper because properness is preserved by base change,
  and $g_R$ is proper.
  Hence, by
  \cite[Theorem 29.6.2]{vakil:foundations-of-algebraic-geometry},
  $g_\bz|_{g_\bz^{-1}(V_\bz)}$ is finite.
  Because $g_\bz|_{g_\bz^{-1}(V_\bz)}$ is finite and flat, we have
  $(g_\bz)_* \sco_{g_\bz^{-1}(V_\bz)}$ is a locally free sheaf
  on $V_\bz$, by
  \cite[24.4.H]{vakil:foundations-of-algebraic-geometry}.

  By definition, the degree is the rank of this locally free sheaf.
  Since the rank will be preserved by base change, to calculate the rank,
  we only need calculate the rank over the geometric generic fiber.
  By \autoref{theorem:counting-veronese-interpolation},
  over the geometric generic fiber over $\spec \bz$, each closed point will have four preimages.
  Therefore, since $\overline \bq$ has characteristic $0$,
  all field extensions are separable. In particular, by
  \autoref{lemma:number-of-preimages},
  the number of preimages of a general point is equal to the degree.
  Therefore, the degree of $g_{\overline \bq}$ is $4$,
  and hence the degree of $g_\bz$ is also $4$.

    Since $(g_\bz)_* \sco_{g_\bz^{-1}(V_\bz)}$ is a locally free sheaf of rank $4$,
  if we base change the map $g_\bz$ to $g_{\overline {K(\eta)}}$ for
  $\eta$ a point of $\spec \bz$,
  we obtain that $(g_{\overline {K(\eta)}})_* \sco_{g_{\overline {K(\eta)}}^{-1}(V_{\overline {K(\eta)}})}$
  is also locally free of rank $4$, 
  as desired.
  Of course, we are implicitly using
  \autoref{lemma:w-surjective}, as we need to know
  that $W_{\overline {K(\eta)}}$ is nonempty.
\end{proof}

\begin{corollary}
  \label{corollary:failure-of-2-veronese-vector-bundle-interpolation}
  The map $g_{\overline {\mathbb F_2}}$ from
  \eqref{equation:v-veronese-diagram}
  is inseparable.
  In particular, the $2$-Veronese surface
  is an example of a variety which satisfies interpolation
  but its normal bundle does not satisfy vector bundle interpolation.
  That is, 
  \ref{interpolation-pointed}
  of \autoref{theorem:equivalent-conditions-of-interpolation}
  does not imply
  \ref{cohomological-definition}
  of
  \autoref{theorem:equivalent-conditions-of-interpolation}
  in characteristic $2$.
\end{corollary}
\begin{proof}
  First, as we showed in \autoref{lemma:degree-4-map},
  the degree of $g_\bz$ is $4$.
  Therefore, since $g_{\overline {\mathbb F_2}}$, is a base change
  of $g_\bz$, it also has degree $4$.

  From 
  \autoref{theorem:counting-veronese-interpolation},
  we know 
  each closed point of $W_{\overline {\mathbb F_2}}$
  will have two preimages under $g_{\overline{\mathbb F_2}}$.
  So, by  
  \autoref{lemma:number-of-preimages},
  the number of preimages of a general point is equal to the separable
  degree,
  which implies the separable degree of
  $g_{\overline{\mathbb F_2}}$ must be $2$,
  and hence the inseparable degree must be $4/2 = 2$.

  Note that $g_{\overline{\mathbb F_2}}$ is the map
  $\eta_2$ from \autoref{proposition:irreducible-incidence}.
  If it is not separable, then $d \eta_2$ is not surjective.
  So by \autoref{proposition:irreducible-incidence},
  the corresponding map $\tau$ from \autoref{proposition:irreducible-incidence}
is not surjective,
  which precisely means that
  the $2$-Veronese surface fails to satisfy vector bundle
  interpolation in characteristic $2$, by ~\eqref{equation:sequence-cohomological-interpolation-definition}
  in the definition of interpolation for vector bundles.
\end{proof}

\section{A teaser: rational normal curves}

Before moving on to the general setting, to get a feel for how the argument will proceed,
let's carry it out the special case that the scrolls are one dimensional.
That is, let's examine the case when the scrolls are rational normal curves.
It is a classical result that a unique rational normal curve passes
through a general set of $n+3$ points in $\bp^n$. In fact, you may
have even seen this result in your first course on algebraic geometry 
\cite[Theorem 1.18]{Harris1998}.

\begin{remark}
	\label{remark:}
	In addition to knowing that rational normal curves satisfying interpolation,
	it is also known that all rational normal surface scrolls satisfy interpolation, as follows from
	Coskun's thesis \cite{coskun:degenerations-of-surface-scrolls}.
	So, the main object of this section is to prove interpolation for scrolls of dimension at least 3.
	Nevertheless, the proof still applies to dimensions $1$ and $2$.
\end{remark}

Let's now trace through the example of rational normal curves.

\begin{figure}
	\centering
	\includegraphics[scale=.2]{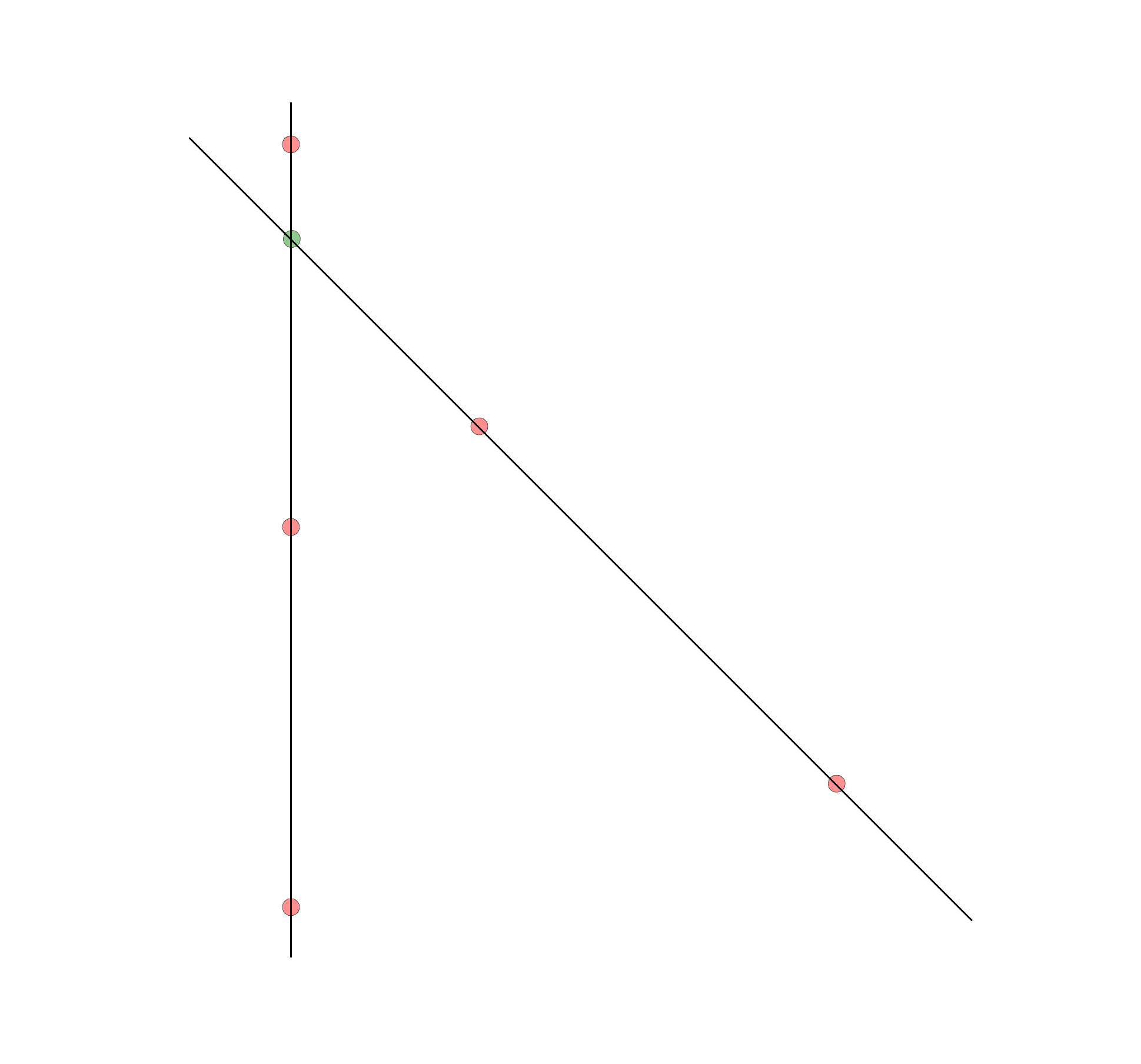}
	\includegraphics[scale=.2]{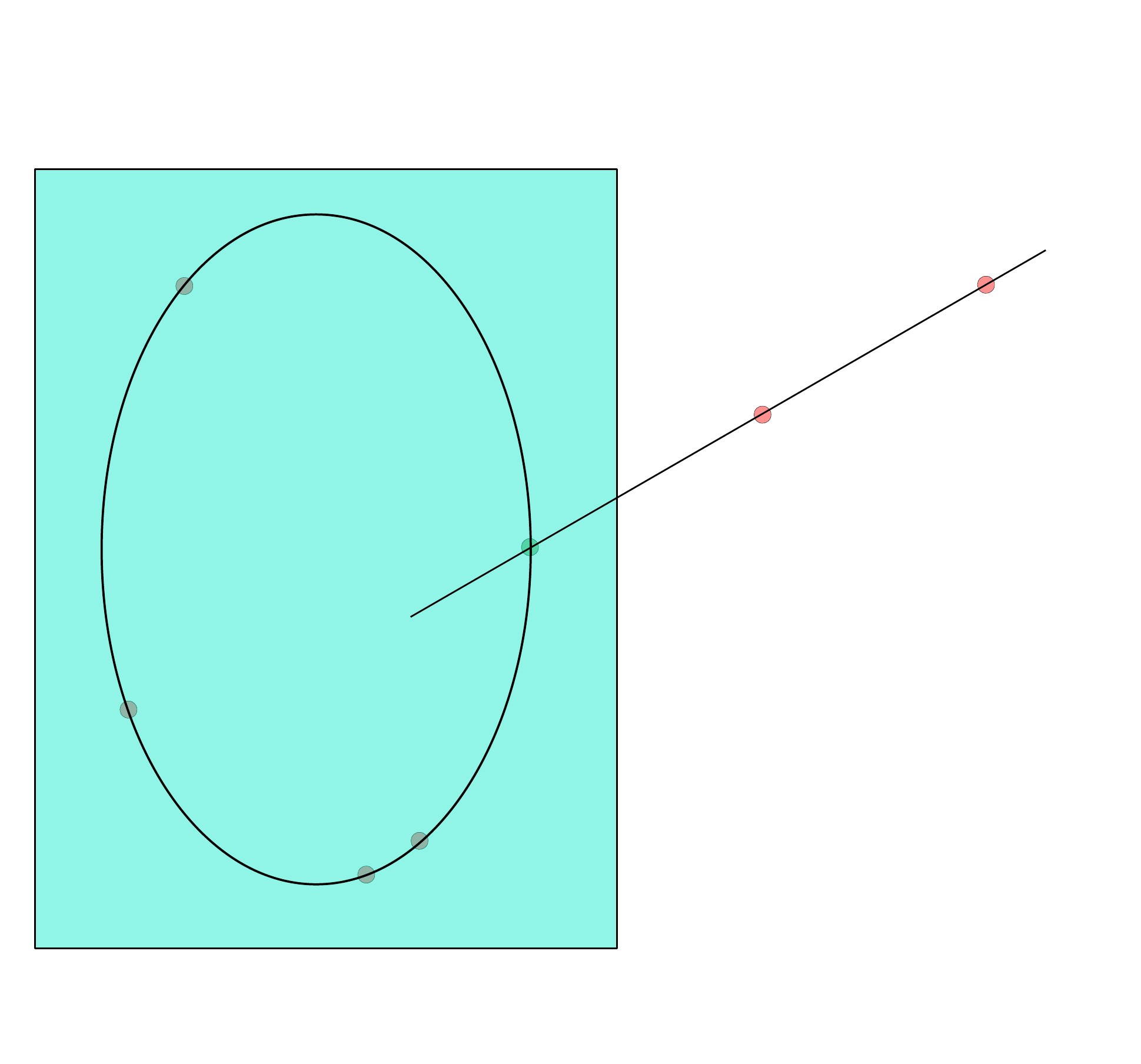}
	\includegraphics[scale=.2]{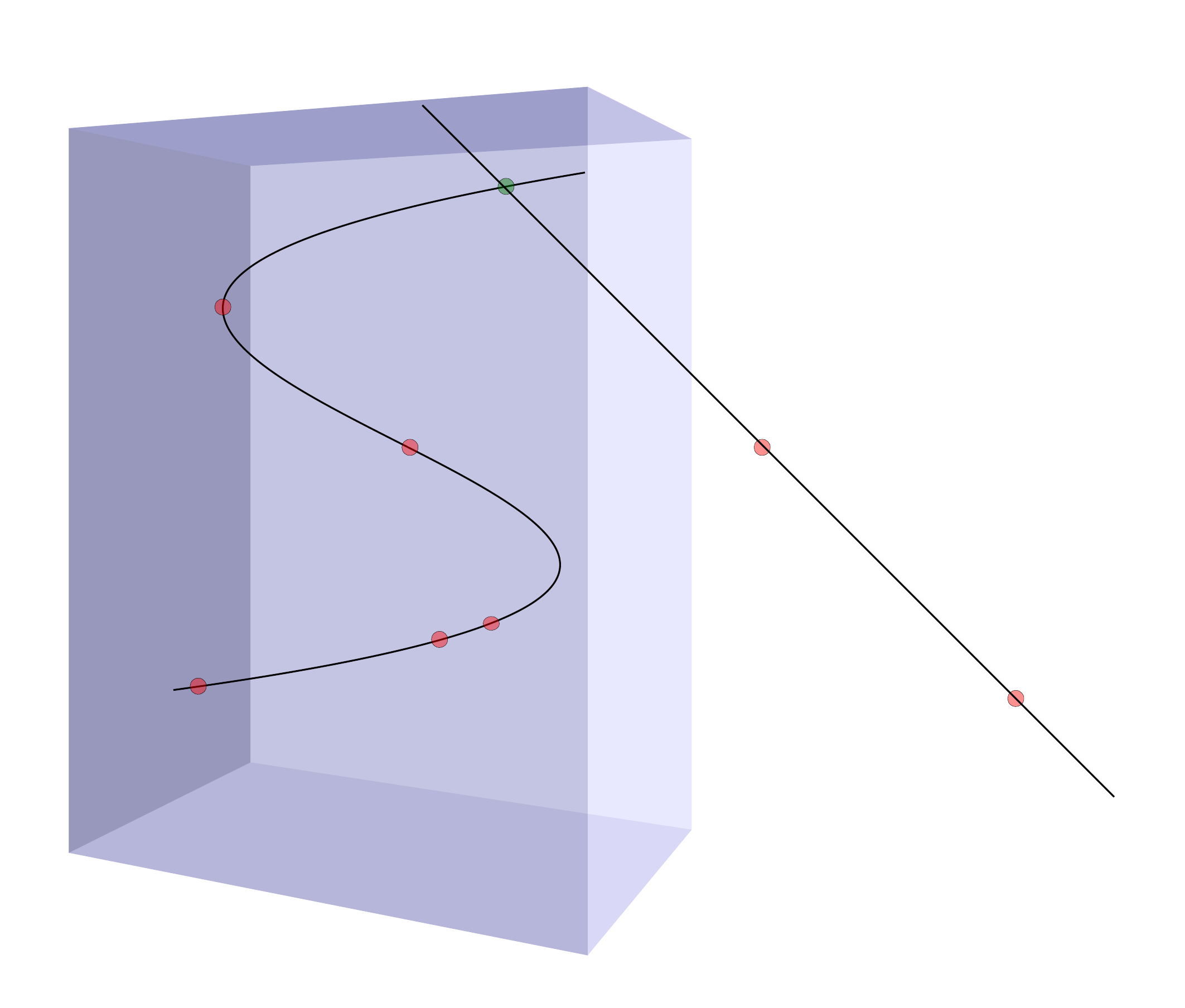}
	\caption{A pictorial description of the inductive degeneration for showing rational normal curves
	satisfy interpolation. Degenerations are drawn for rational normal curves of degrees 2, 3, and 4.}
\end{figure}

\begin{example}
	\label{example:rational-normal-curve-interpolation}
	We'd like to show that there is a rational normal curve $C$ passing through $n+3$ points
	in $\bp^n$. The result is clear when $n = 2$ by \autoref{lemma:balanced-complete-intersection}.
	This is just saying that through any $5$ points in $\bp^2$ we can find a conic curve.

	Now, inductively assume there exists a rational normal curve passing through $n+2$ points
	in $\bp^{n-1}$. We will show there is a rational normal curve in $\bp^n$ passing through
	$n+3$ points, $p_1, \ldots, p_{n+3}$. 
	Start by specializing $p_3, \ldots, p_{n+3}$ to general points in a hyperplane $H \cong \bp^{n-1} \subset \bp^n$.
	Note that because a smooth rational normal curve spans $\bp^n$, it cannot be contained in any hypersurface,
	and so, by Bezout's theorem, its intersection with any hypersurface must be a scheme of degree $n$.
	Therefore, there are no smooth rational normal curves passing through such a configuration of points.

	However, there {\emph is} a degenerate rational normal curve passing through this collection of points.
	Namely, let $\ell$ be the line joining $p_1$ and $p_2$ and let $q := \ell \cap H$.
	Then, by the inductive hypothesis, there is a unique rational normal curve $D \subset H$ of degree $n+2$ containing
	$q, p_3, \ldots, p_{n+3}$. Let $C := \ell \cup D$.
	This defines a unique ``degenerate'' rational normal curve passing through $n+2$ general points
	in $H$ and two additional general points in $\bp^n$. Note that this curve indeed lies in the irreducible
	component of the Hilbert scheme whose general member is a smooth rational normal curve, as follows
	from \autoref{proposition:hilbert-scroll-degeneration}.

	Then, by \autoref{corollary:isolated-points-in-minutely-broken-scrolls}, we obtain
	that this degenerate rational normal
	curve, together with the $n+3$ points, corresponds to an isolated point in the fiber of the map
	\begin{equation}
		\nonumber
		\begin{tikzcd} 
			\left\{ \left( p_1, \ldots, p_{n+3}, C  \right) \subset (\bp^n)^{n+3} \times \minhilb n 1: p_1 \in C, \ldots, p_{n+3} \in C \right\}  \ar {d}\\
			(\bp^n)^{n+3}
		\end{tikzcd}\end{equation}
	Hence, by \autoref{theorem:equivalent-conditions-of-interpolation}, 
	rational normal curves satisfy interpolation.
\end{example}
\begin{remark}
	\label{remark:}
	In \autoref{example:rational-normal-curve-interpolation}, we were able to show rational normal curves satisfy
	interpolation. One may still want to know how many rational normal curves there are through a general set
	of $n+3$ points. Using the above method, one can see there is a unique one,
	as one can specialize points into a hyperplane 1 by 1, and note that there will only be a degenerate curve
	once we specialize the $n+1$st point to the hyperplane. Then, one can inductively show there is precisely 1 curve,
	by showing (with some work) that there is precisely 1 degenerate curve.

	\begin{exercise}
		\label{exercise:}
		Make the above discussion precise, showing that there is exactly one rational normal curve
		through $n+3$ points.
	\end{exercise}
	Similarly, in his thesis, Coskun was able to count the number of rational normal scrolls meeting a given
	set of linear spaces. However, finding the number in higher dimensions seems to be a more difficult task,
	as the degenerations get much thornier in dimension 3 and above. In the remainder of this thesis,
	we shall primarily be concerned just with finding whether there exists a scroll passing through
	a set of points, and not the question of how many there are.

	With that said, we do answer some enumerative questions for scrolls of dimension
	$k$ and degree $k$ using \autoref{subsection:a-combinatorial-interlude}.
\end{remark}

\section{The set-up for scrolls}
In the remainder of this chapter, we prove that scrolls satisfy interpolation.
We'll start by stating the precise number of points and linear spaces a scroll must
pass through to satisfy interpolation.

\begin{lemma}
	\label{lemma:interpolation-numerics}
	The condition of interpolation means that we can find a degree $d$ dimension $k$ scroll 
	\begin{itemize}
		\item containing $d + 2k + 1$ general points and meeting a general $d-2k + 1$-plane if $d \geq 2k - 1$,
		\item containing $d + 2k + 2$ general points and meeting a general $2(d-k)$-plane if $k \leq d \leq 2k - 2$.
	\end{itemize}
\end{lemma}
\begin{proof}
	By \autoref{proposition:scroll-hilbert-scheme-dimension}, we know
	\begin{align*}
		\dim \minhilb d k = (d + k)^2 - k^2 - 3.
	\end{align*}
	Additionally, for any $[Y] \in \minhilb d k$, since $n  = d + k - 1$, we have $\codim Y = n - k = d - 1$.
	Now, we may write
	\begin{align*}
		\dim \minhilb d k
		&= (d+k)^2 - k^2- 3 \\
		&= (d-1) \left( d+2k + 1 \right) + 2k - 2.
	\end{align*}
	Note that we always have $d \geq k$ and $2k - 2 \leq d-1$ when $d \geq 2k - 1$. Therefore, in the case $d > 2k - 1$.
	interpolation means the scheme passes through $d + 2k + 1$ points and meet a $(d - 2k + 1)$-plane.
	In the case $d = 2k - 1$, the scheme must pass through $d + 2k + 2$ points, which is the same as passing through
	$d + 2k +1$ points and a general $d - 2k +1$ plane, since a $d - 2k + 1$ plane is a point in this case.
	Next, when $k < d \leq 2k - 2$, we have $d-1 <2k - 2 < 2(d-1)$, and so for interpolation to hold, such
	schemes must pass through $d + 2k + 2$ points and meet a $2(d-k)$-plane.
	In the special case that $k = d$, the scheme must pass through $d + 2k + 3$ points, which is the same
	as passing through $d + 2k + 2$ points and meeting a $2(d-k)$-plane, as a $2(d-k)$-plane is a point.
	Therefore, by \autoref{definition:interpolation}, $\minhilb d k$ satisfies interpolation if
	it satisfies the conditions of this lemma.
\end{proof}

We prove that scrolls satisfy interpolation by induction.
We fix a dimension $k$ and induct on the degree of $k$-dimensional varieties.
In order to prove the theorem for a variety of dimension $k$ and degree $d$, we will
make the following inductive hypotheses.
\begin{enumerate}
	\item[\customlabel{custom:ind-high}{ind-1}] When $d > k$, we assume $\minhilb {d-1} k$ satisfies interpolation.
	\item[\customlabel{custom:ind-mid}{ind-2}] When $k + 1 \leq d \leq 2k - 1$, we assume there is a variety of degree $d - 1$ containing a general $(2k-d-1)$-plane and containing $2d$ points.
\end{enumerate}

\section[Degree at least $2k-1$]{Inductive degeneration for degree at least $2k-1$}

In this section we show that degree $d$ scrolls satisfy interpolation for $d \geq 2k-1$,
assuming our inductive hypothesis \autoref{custom:ind-high}.
The main result of this section is \autoref{lemma:high-induction}.

\begin{proposition}
	\label{lemma:high-induction}
	Assuming induction hypothesis ~\ref{custom:ind-high}, if $d > 2k-1$ then $\minhilb d k$ satisfies interpolation.
\end{proposition}
The idea of this proof is to specialize all but two of the points to a hyperplane, and then find a reducible scroll of degree
$d$ which is the union of a scroll of degree $d-1$ in a hyperplane and $\bp^k$, meeting along $\bp^{k-1}$,
as pictured in \autoref{figure:high-induction}.

\begin{figure}
	\centering
	\includegraphics[scale=.35]{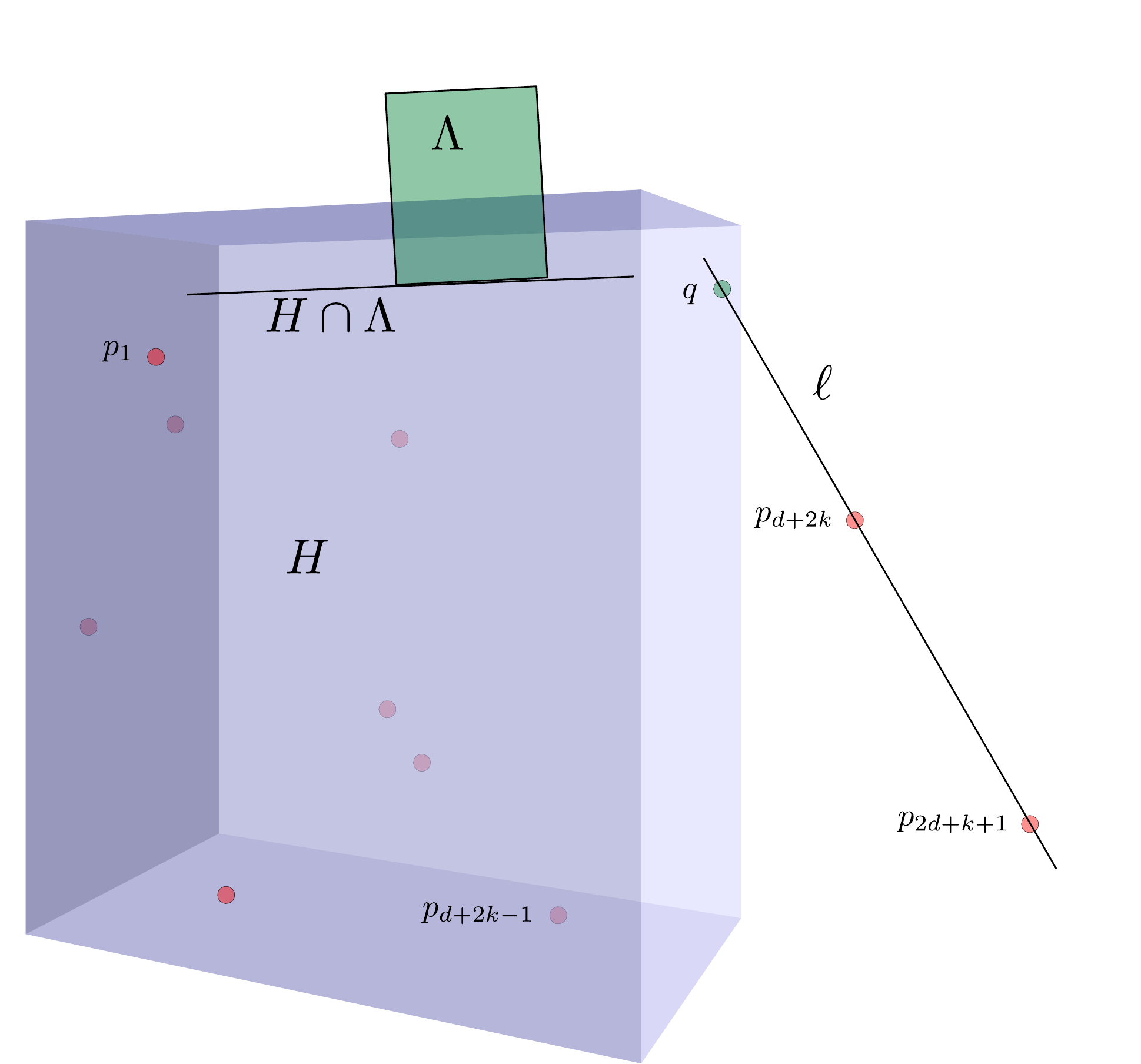}
	\caption{A visualization of the idea of the proof of \autoref{lemma:high-induction},
	where one inductively specializes all but two points to lie in a hyperplane.
}
	\label{figure:high-induction}
\end{figure}

\begin{proof}
	By \autoref{lemma:interpolation-numerics}, $\minhilb d k$ satisfying interpolation means that we can find 
	some variety corresponding to a point in this Hilbert scheme, passing through
	a general collection of $d + 2k +1$ points and a general 
	$(d -2k + 1)$-plane, $\Lambda$. 	Now specialize $d - 2k - 1$ of the points $p_1, \ldots, p_{d+2k-1}$ to a general hyperplane $H \subset \bp^{d+k-1}$. Note further that
	since $H$ was chosen generally, we have $H \cap \Lambda$ is a general $((d-1)-2k+1)$-plane. 
	Further, let $\ell$ be the line through $p_{d+2k}, p_{d+2k+1}$ and let $q := \ell \cap H$.
	By induction hypothesis ~\ref{custom:ind-high}, there is
	a degree $d-1$ dimension $k$ variety of minimal degree which contains the $(d-1) + 2k + 1$ points $p_1, \ldots, p_{d+2k-1}, q$ and meets
	the $((d-1)-2k+1)$-plane $H \cap \Lambda$. Call that variety $X$. Now, since $q \in X$, there is a unique $k-1$-plane
	contained in $X$ and containing $q$, as follows from \autoref{lemma:linear-spaces-in-varieties-of-minimal-degree}.
	Let that $k-1$-plane be $P$. Then, let $Y$ be the $k$ plane spanned by $\ell$ and $P$.
	Then, the variety $X \cup Y$ with reduced scheme structure lies in $\minhilb d k$ by \autoref{proposition:hilbert-scroll-degeneration}.
	
	We claim further that $X \cup Y$ is an isolated point in the set of all elements of $\minhilb d k$ containing $p_2, \ldots, p_{d + 2k - 1}$ and $\Lambda$.
	This would then show that $\minhilb d k$ satisfies equivalent criterion ~\ref{interpolation-isolated} of interpolation.
	Since the points and plane were chosen generally subject to the requirement that $k+d$ of the points were contained
	in a hyperplane, by \autoref{corollary:isolated-points-in-minutely-broken-scrolls}, it suffices to show there are only finitely many
	scrolls in $\broken d k \cup \minhilbsmooth d k$ containing $p_1, \ldots, p_{d - 2k+1},$ and meeting $\Lambda$.

	This now follows from a dimension count. First, we will show there are only finitely many
	scrolls in $\broken d k$ containing $p_1, \ldots, p_{d - 2k+1}$ and meeting $\Lambda$. 
	Because of the specialization of the points, any $[X \cup Y] \in \broken d k$ with $Y \cong \bp^{k-1}$ 
	containing this set of points must satisfy
	$X \subset H$.
	Therefore, if the points and $\Lambda$ were chosen generally, any such scroll must pass through
	$p_1, \ldots, p_{d-2k-1}, q$, and meet $\Lambda \cap H$.
	Now, there are only finitely many scrolls $X \subset H$ containing $p_1, \ldots, p_{k+d}, q$ and meeting
	$\Lambda \cap H$ by
	\autoref{lemma:interpolation-numerics}.
	Let $P$ be the unique $(k-1)$-plane contained in $X$ and containing $q$. This exists and is unique by
	\autoref{lemma:linear-spaces-in-varieties-of-minimal-degree}.
	Then, $Y$ is uniquely determined to be the plane containing $p_{k+d}, p_{k+d+1}$ and $P$.
	Hence, there are only finitely many scrolls in $\broken d k$.

	By \autoref{corollary:isolated-points-in-minutely-broken-scrolls},
	to complete the proof, it suffices to show there are only finitely many smooth scrolls containing
	$p_1, \ldots, p_{k+d+1}$ and meeting $\Lambda$.
	This follows from \autoref{lemma:ind-high-smooth-finite}, which we prove next.
\end{proof}

\begin{lemma}
	\label{lemma:ind-high-smooth-finite}
	Let $p_1, \ldots, p_{d+2k+1}$ be points in $\bp^n$ so that $p_1, \ldots, p_{d+2k-1}$ are contained in a hyperplane $H \subset \bp^n$,
	but the points are otherwise general, and let $\Lambda$ be a general $(d-2k+1)$ plane.
	Then, there are only finitely many smooth scrolls containing $p_1, \ldots, p_{d+2k+1}$ and meeting $\Lambda$.
\end{lemma}
We give two proofs, using the two techniques described in \autoref{remark:interpolation-strategy}.

\begin{proof}[Proof using \ref{technique-conditions}]
The condition that $X$ pass through a point imposes $n-k = d-1$ conditions. 
When we specialize a point to a hyperplane, the hyperplane will intersect $X$ in a scroll of dimension $k-1$
in $\bp^{n-1}$. Hence, a point will still impose $n-1 - (k-1) = n - k$ conditions. 
Finally, the condition to meet a $(d - 2k + 1)$-plane imposes $n - k - (d-2k+1)$ conditions.
Adding these up, we obtain a total of
\begin{align*}
	(d+2k+1)(n-k) + n-k-(d-2k+1) &= (k+d)^2 - k^2 - 3 \\
	&= \dim \minhilb d k
\end{align*}
conditions.
\end{proof}

For the second proof of \autoref{lemma:ind-high-smooth-finite} using \ref{technique-dimension},
we will first need a lemma, showing that a hyperplane section of a scroll is a scroll.

	\begin{lemma}
		\label{lemma:smooth-scroll-hyperplane-section}
		Let $[X] \in \minhilbsmooth d k$ be a smooth scroll. Then, for any hyperplane $H \in (\bp^n)^\vee$, we have
		$X \cap H \in \minhilb d {k-1} $.
	\end{lemma}
	\begin{remark}
		\label{remark:}
		Note that the statement of \autoref{lemma:smooth-scroll-hyperplane-section} asserts not only
		that $X \cap H$ has dimension $k-1$, but also
		that
		$X \cap H$ lies in the closure of the locus of smooth scrolls.
		This additional fact about the irreducible component in which a hyperplane section
		of a smooth scroll lies, 
		which takes up the bulk of the proof of \autoref{lemma:smooth-scroll-hyperplane-section}, is 
		not needed in the proof of \autoref{lemma:ind-high-smooth-finite}.
		Nevertheless, we prove
		it as a bonus.
	\end{remark}
	
	\begin{proof}
		First, note there
		are no hyperplanes $[H] \in (\bp^n)^\vee$ that contain an
		associated point of $X$, since $X$ has no degree $1$ components. 
		Therefore, by Bezout's theorem, for any hyperplane $[H] \in (\bp^n)^\vee$, we have $H \cap X$
		lies in the connected component of the Hilbert scheme of varieties of minimal degree $d$ and dimension $k-1$.
	
		We need to show that $H \cap X$ further lies in
		$\minhilb {d} {k-1}$. That is, we claim it lies in the irreducible component
		whose general member is a smooth scroll of degree $d$ and dimension $k-1$.
		To see this, observe that by Bertini's theorem, there is some hyperplane $[J] \in (\bp^n)^\vee$ so that $X \cap J$ is smooth.

		Now, consider the family
		\begin{align*}
			\left\{ X \cap H : [H] \in (\bp^n)^\vee \right\} \subset \bp^n \times (\bp^n)^\vee \ra (\bp^n)^\vee
		\end{align*}
		By \cite[Theorem III.9.9]{Hartshorne:AG}, the projection map to $(\bp^n)^\vee$ is flat, because $(\bp^n)^\vee$ is reduced, and all fibers have the same
		Hilbert polynomial, equal to that of a member of $\minhilb d {k-1}$, as shown above.

		Therefore, by the functorial definition of the Hilbert scheme $\sch$, we obtain a map
		$(\bp^n)^\vee \ra \sch,$ which sends the point $[J] \in (\bp^n)^\vee$ to a smooth scroll.
		By \autoref{proposition:regular-minimal-degree-hilbert-schemes}, we know that a smooth scroll is a smooth point of the Hilbert
		scheme, and is therefore contained in the irreducible component $\minhilb d {k-1}$, and not contained in any other irreducible components.
		Hence, since $(\bp^n)^\vee$ is irreducible, it must map completely to $\minhilb d {k-1}$, as claimed.
	\end{proof}

	\begin{proof}[Proof of \autoref{lemma:ind-high-smooth-finite} using \ref{technique-dimension}]

	For the remainder of the proof, fix a hyperplane $H \subset \bp^n$. Define the incidence correspondence
\begin{align*}
	\Phi := 
	\left( \uhilb X \times_{\hilb {p_1}} H \right) &\times_{\hilb X} \cdots \times_{\hilb X} \left( \uhilb X \times_{\hilb {p_{d+2k-1}}} H \right)
	\\
	&\times_{\hilb X} \uhilb X \times_{\hilb X} \uhilb X \times_{\hilb X} \left( \uhilb X \times_{\bp^n} \hilb \Lambda \right)
\end{align*}
To complete the proof, it suffices to show that 
\begin{align*}
	\dim \Phi &= \dim H^{d+2k-1} \times (\bp^n)^2 \times G(d -2k + 2, n+1) \\
	&= (n-1)(d+2k-1)+ 2n +(d-2k+2)(n+2k-d-1)
\end{align*}
because then the projection map
\begin{align*}
	\Phi \ra (\bp^n)^{d+2k-1} \times (\bp^n)^2 \times G(d -2k + 2, n+1)
\end{align*}
will either be generically finite or not dominant, and in either case, there
will only be finitely many smooth varieties through a general set of such
points, meeting a general plane $\Lambda$.
Now, consider the projection
\begin{align*}
	\Phi \ra \minhilbsmooth d k
\end{align*}
It suffices to show that the dimension of $\minhilb d k$ plus the dimension of every fiber of this map is
\begin{align*}
(n-1)(d+2k-1)+ 2n +(d-2k+2)(n+2k-d-1).
\end{align*}
Noting that, by definition, $n = d+k-1$, we can rewrite the above dimension as
as
\begin{align*}
(d+k-2)(d+2k-1)+ 2(d+k-1) +(d-2k+2)(n+2k-d-1).
\end{align*}

We know that $\dim \minhilbsmooth d k = (d+k)^2 - k^2 - 3$ because smooth scrolls are smooth points of $\minhilb d k$
by \autoref{proposition:regular-minimal-degree-hilbert-schemes} and because this is equal to the dimension
of $\minhilb d k$ by \autoref{proposition:scroll-hilbert-scheme-dimension}.
We claim
\begin{align*}
	\dim \uhilb X \times_{\bp^n} H - \dim \minhilbsmooth d k &= k-1 \\
	\dim \uhilb X - \dim \minhilbsmooth d k &= k \\
	\dim \uhilb X \times_{\bp^n} \hilb \Lambda - \dim \minhilbsmooth d k &= k+ (d-2k+1)(3k-2)
\end{align*}
Indeed, the last two statements are immediate consequences of \autoref{lemma:interpolation-dimension}.
Finally, the first statement follows because the intersection $X \cap H$ is always
$k-1$ dimensional, by \autoref{lemma:smooth-scroll-hyperplane-section}, and so the fibers of the map
\begin{align*}
	\uhilb X \times_{\bp^n} H \ra \minhilbsmooth d k
\end{align*}
are all $k-1$ dimensional, implying the difference of the dimensions is $k-1$.

We have now completed the proof because
\begin{align*}
	\dim \Phi &= ((d+k)^2 -k^2 - 3) \\
	 & \qquad + \left(  (k-1)(d+2k-1) + (2k) + (k +(d-2k+1)(3k-2))
\right) \\
&= (d+k-2)(d+2k-1)+ 2(d+k-1) \\
& \qquad +(d-2k+2)(n+2k-d-1) \\
&= \dim (\bp^n)^{d+2k-1} \times (\bp^n)^2 \times G(d -2k + 2, n+1).
\end{align*}
\end{proof}

\section[Degree between $k+1$ and $2k-1$]{Inductive degeneration for degree between $k+1$ and $2k-1$}

The main result of this section is \autoref{corollary:mid-induction}
which lets us verify scrolls of degree between $k+1$ and $2k-1$
satisfy interpolation, assuming \autoref{custom:ind-mid} inductively.
We prove this assuming two lemmas, 
\autoref{lemma:finite-containing-plane-and-points} and
\autoref{lemma:finite-smooth-scrolls-mid-to-ladder}.
These two lemmas are essentially just dimension counts.

\begin{figure}
	\centering
	\includegraphics[scale=.3]{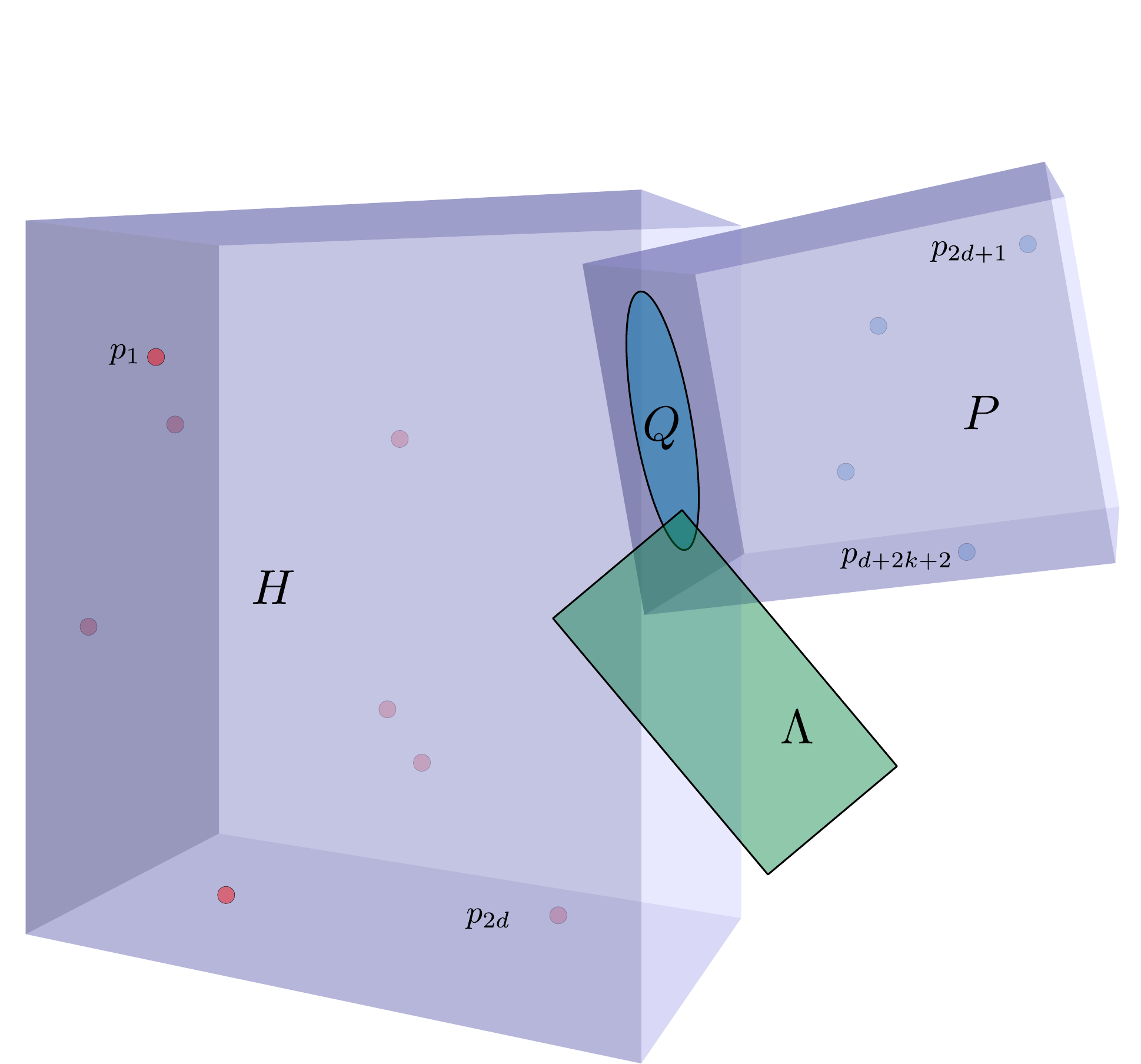}
	\caption{A visualization of the idea of the proof of \autoref{corollary:mid-induction},
	where one inductively specializes $2d$ points to lie in a hyperplane $H$, and the additional plane $\Lambda$
	to meet the intersection of $H$ and the span of the remaining $2k-d+2$ points.
}
	\label{figure:}
\end{figure}

\begin{proposition}
	\label{corollary:mid-induction}
	Let $k+1 \leq d \leq 2k-1$. Assuming induction hypothesis ~\ref{custom:ind-mid} holds for varieties of degree $d-1$ then $\minhilb {d} k$ satisfies interpolation.
\end{proposition}
\begin{proof}[Proof assuming \autoref{lemma:finite-containing-plane-and-points} and \autoref{lemma:finite-smooth-scrolls-mid-to-ladder}]
	By \autoref{lemma:interpolation-numerics}, we would like to show there is a degree $d-1$, dimension $k$ variety passing through
	$d + 2k + 2$ points, $p_1, \ldots, p_{d+2k+2}$, and meeting a general $2(d-k)$ plane $\Lambda$. 
	Choose a general hyperplane $H$ and specialize
	$2d$ of these points, $p_{1}, \ldots, p_{2d}$, to $H$.
	Now, the first $2k-d+2$ points remain general, and span some $2k-d+1$ plane $P$.
	Specialize $\Lambda$ so that $\Lambda \cap H \cap P \neq \emptyset$, but so that $\Lambda$ is otherwise general.

	Then, by induction hypothesis
	~\ref{custom:ind-mid} for varieties of degree $d$, there is a variety of degree $d-1$ containing the $2d$ points
	$p_1, \ldots, p_{2d}$ and containing $Q$. Call this variety $X$. If the points are chosen generally, we can 
	assume this variety is smooth.
	
	Further, if the points are chosen generally, we will have some $(k-1)$ plane $R$ with $Q \subset R \subset X$,
	as we now explain.
	Note, by \autoref{lemma:linear-spaces-in-varieties-of-minimal-degree}, if $\dim Q \neq 1$, then we will
	have that $Q$ is contained in some $(k-1)$-plane. If, instead $\dim Q = 1$, then, 
	at least when $k > 2$, we know there are at most two components of the Fano scheme of lines
	contained in $X$, and one of these has dimension $k$, which is strictly larger than 
	the other component (if it exists), from \autoref{lemma:linear-spaces-in-varieties-of-minimal-degree}.
	Therefore, since	
	$p_1, \ldots, p_{d+2k+2}$ and $\Lambda$ are chosen generally,
	we may assume that this line $Q$ lies in the component of the Fano scheme of larger dimension. In this case,
	the line $Q$ lies in some $(k-1)$-plane, again by \autoref{lemma:linear-spaces-in-varieties-of-minimal-degree}.

	Define $Y$ to be the plane spanned by $R, p_{2d+1}, \ldots, p_{d+2k+2}$.
	Then, we obtain then obtain that $X \cup Y$ corresponds 
	to a point in $\minhilb d k$ by \autoref{proposition:hilbert-scroll-degeneration}.
	Note that because we specialized $\Lambda$ to meet $Q$, we automatically obtain that $\Lambda$ meets $X \cup Y$.

	It remains only to show that $X \cup Y$ is an isolated point.
	First, by \autoref{lemma:finite-containing-plane-and-points}, there are only finitely many
	schemes of degree $d-1$ containing $p_1, \ldots, p_{2d}$ and $Q$.
	For each such degree $d-1$ scroll, $X$, there will be at most one plane $Y$ containing $p_{2d+1}, \ldots, p_{d+2k+2}$ so that
	$Y \cap X \cong \bp^{k-1}$, as follows from \autoref{lemma:linear-spaces-in-varieties-of-minimal-degree}.
	By \autoref{lemma:finite-smooth-scrolls-mid-to-ladder},
	there are also only finitely many smooth scrolls containing $p_1, \ldots, p_{d+2k+2}$ and meeting $\Lambda$.
	Since there are only finitely many scrolls in $\broken d k$ and $\minhilbsmooth d k$ containing the points and meeting
	the plane, it follows from \autoref{corollary:isolated-points-in-minutely-broken-scrolls}
	that $X \cup Y$ is an isolated point of the set of schemes containing the points and meeting $\Lambda$,
	and so it satisfies interpolation by \autoref{theorem:equivalent-conditions-of-interpolation}.
\end{proof}

To complete the proof of \autoref{corollary:mid-induction}, we only need show
 \autoref{lemma:finite-containing-plane-and-points} and \autoref{lemma:finite-smooth-scrolls-mid-to-ladder}.
We first show \autoref{lemma:finite-containing-plane-and-points}.

\begin{lemma}
	\label{lemma:finite-containing-plane-and-points}
	Fix $k$ and suppose $k + 1 \leq d \leq 2k - 1$.
	Let $p_1, \ldots, p_{2d}$ be a general set of $2d$ points in $\bp^{n-1}$ and let $\Lambda$ be a general $2k-d$ plane in $\bp^{n-1}$.
	Then, there are only finitely many schemes in $\minhilb {d-1} k $
	containing $p_1, \ldots, p_{2d}, \Lambda$.
\end{lemma}
We give two proofs using the two techniques from \autoref{remark:interpolation-strategy}.

\begin{proof}[Proof using \ref{technique-conditions}]
	The condition that a variety of dimension $k$ meet a point in $\bp^n$
	is $n-k$ conditions. Next, containing a $(2k-d)$ plane imposes
	$\dim G(2k-d+1, n) - \dim G(2k-d+1, k) -1$ conditions because there
	is a $\dim G(2k-d+1, n)$ dimensional space of $2k-d$ planes, and a
	$\dim G(2k-d+1, k) -1$ dimensional space of planes in any given $k$-dimensional
	scroll.
	Adding these up, the total number of conditions is
	\begin{align*}
		&(2d)(n-k) + \dim G(2k-d+1, n) - \dim G(2k-d+1, k) -1\\
		&= (k + (d-1))^2 - k^2 - 3 \\
		&= \dim \minhilb {d-1} k
	\end{align*}
	as desired.
\end{proof}
\begin{proof}[Proof using \ref{technique-dimension}]
	Let $\flag \Lambda X$ denote the flag scheme of pairs of $(2k-d)$-planes $\Lambda$ and elements of $\minhilb d k$.
	That is, it is the pullback of the usual flag Hilbert scheme along the map from
	the irreducible component of the Hilbert scheme
	$\minhilb d k$ into the Hilbert scheme.
	Define
	\begin{align*}
		\Phi := \uhilb X \times_{\hilb X} \cdots \times_{\hilb X} \uhilb X \times_{\hilb X} \flag \Lambda X
	\end{align*}

	To complete the lemma it suffices to show that	
	\begin{align*}
		\dim \Phi &= \dim (\bp^{n-1})^{2d} \times G(2k-d+1, n) 
	\end{align*}
	
	To calculate the dimension of $\Phi$, we claim
	\begin{align*}
		\dim \uhilb X - \dim \minhilb {d-1} k &= k \\
		\dim \flag \Lambda X - \dim \minhilb {d-1} k  &= 1 + (2k-d+1)(-k+d-1).
	\end{align*}
	The first equality holds by \autoref{lemma:interpolation-dimension} while the second holds from
	\autoref{lemma:linear-spaces-in-varieties-of-minimal-degree}.
	
	Finally, we see
	\begin{align*}
		\dim \Phi &= \dim \minhilb {d-1} k + 2d(\dim \Phi_{(n-1)} - \dim \minhilb {d-1} k ) \\
		& \qquad + 
		(\dim \Phi_{(2d-2k)} - \dim \minhilb {d-1} k ) \\
		&= ((k+d-1)^2 -k^2 -3) + 2d(k) \\
		& \qquad + (1 + (2k-d+1)(-k+d-1)) \\
		&= (k+d-2)(2d) + (2k-d+1)(2d-k-2) \\
		&= \dim (\bp^{n-1})^{2d} \times G(2k-d+1, n) 
	\end{align*}
	completing the proof.
\end{proof}

We now complete the proof of \autoref{corollary:mid-induction} by proving
\autoref{lemma:finite-smooth-scrolls-mid-to-ladder}.

\begin{lemma}
		\label{lemma:finite-smooth-scrolls-mid-to-ladder}
		Let $k+1 \leq d \leq 2k$ and $n = d + k - 1$.
		Let $p_1, \ldots, p_{2d}$ be a general set of points in a hyperplane $H \subset \bp^n$ and let
		$q_1, \ldots, q_{2k-d+2}$ be general points in $\bp^n$. Further, let
		$P$ be the plane spanned by $q_1, \ldots, q_{2k-d+2}$ and let $Q := P \cap H$.
		Specialize $\Lambda \subset \bp^n$ to be a $2(d-k)$-plane 
		meeting $Q$. Then, there are only finitely many elements in $\minhilbsmooth d k$ meeting
		$\Lambda$ and containing $p_1, \ldots, p_{2d}, q_1, \ldots, q_{2k-d+2}$.
	\end{lemma}
	We give two proofs, one using \ref{technique-conditions}, and one using \ref{technique-dimension}.
	However, since the proof using \ref{technique-conditions} involves a different setup than those of previous
	proofs, we gloss over some of the details of this alternate setup, as they are more than sufficiently addressed in the second
	proof by \ref{technique-dimension}.
	\begin{proof}[Proof 1, by \ref{technique-conditions}]
	We prove this using \ref{technique-conditions}. The proof can also be done using \ref{technique-dimension},
		but is much trickier to set up.
		Following the explanation given in \autoref{remark:interpolation-strategy} and 
		\autoref{example:twisted-cubic-interpolation-dimension-count},
		we count the number of conditions imposed by the various objects.
		Recall from \autoref{proposition:scroll-hilbert-scheme-dimension} that
		$\dim \minhilb d k = (d+k)^2- k^3- 3$.
		First, each of the points $q_1, \ldots, q_{2d+2}$ impose $n - k$ conditions.
		That is, the subscheme of the Hilbert scheme $\minhilb d k$ passing through $q_i$ has codimension
		$n - k$. Therefore, in total, these impose $(2k-d+2)(n - k)$ conditions.
		Next, because each of the points $p_1, \ldots, p_{2d+2}$ lie in $H \cong \bp^{n-1}$,
		and any hyperplane section of a scroll is $k-1$ dimensional,
		(or, more precisely, one can consider the universal family restricted to $H$, which is
		the fiber product $\uhilb X \times_{\bp^n_{\hilb X}} (H \times_{\bp^n} \hilb X)$,)
		we obtain that passing through a point $p_i$ also imposes $(n-1) - (k-1) = n-k$ conditions.
		
		Finally, since $\Lambda$ is a $(2d - 2k)$ plane, it is
		\begin{align*}
			((d+k-1) - k) - (2d-2k) = 2k - d - 1
		\end{align*}
		conditions
		for a $k$ dimensional variety to meet it.
		Adding up all these conditions yields \begin{align*}
			(d-1)(2d) + (d-1)(2k-d+1) + (2k-d-1) = (d+k)^2 - k^2 - 3,
		\end{align*}
		as desired.
	\end{proof}
	\begin{proof}[Proof 2, by \ref{technique-dimension}]
	Construct the incidence correspondences 
	\begin{align*}
			\Psi &:= \left\{ \left(x, q, X, p_1, \ldots, p_{2d}, q_1, \ldots, q_{2k-d+2}, P, Q, \Lambda \right) \right. \\
			& \qquad \subset 
			\bp^n \times \bp^n \times
			\minhilbsmooth d k \times
			H^{2d} \times
			(\bp^n)^{2k-d+2} \\
			& \qquad \times G(2k-d + 2, d+k) \times
			G(2k-d+1, d+k-1) \\
			& \qquad \times 
			G(2d-2k+1, d+k) \\
			& \qquad : x \in X \cap \Lambda, q \in P \cap H \cap \Lambda, p_i \in X, q_i \in X, \\
			& \qquad \left. P = \overline{q_1, \ldots, q_{2k-d+2}}, Q = P \cap H \right\}, \\
		B &:= \left\{ \left (p_1, \ldots, p_{2d}, q_1, \ldots, q_{2k-d+2},t, \Lambda \right ) \right. \\
			& \qquad \subset (\bp^n)^{2k-d+2}  \times H \times G(2d-2k+1, d+k) \\
			& \qquad \left. : t \in \overline{q_1, \ldots, q_{2k-d+2}}, t \in \Lambda \right\}.
		\end{align*}
		The definition of $\Psi$ above is only intended as a set theoretic description, so as to introduce notation
		for the components of $\Psi$.

		We now give a scheme theoretic construction of $\Psi$.
		First, take
		\begin{align*}
			\Psi_1 &:=  \left( \left(\uhilb X \times_{\bp^n_{\hilb X}} H_{\hilb X} \right) \times_{\hilb X} \cdots \times_{\hilb X} \left(\uhilb X \times_{\bp^n_{\hilb X}} H_{\hilb X} \right)\right) \\
			& \qquad \times_{\hilb X} \left(\uhilb X \times_{\hilb X} \cdots \times_{\hilb X} \uhilb X\right)  \\
			\Psi_2 &:= \Psi_1 \times_{\hilb X} \left( \uhilb \Lambda \times_{\bp^n} \uhilb X \right)
		\end{align*}
		where there are $2d$ copies of $\uhilb X \times_{\bp^n_{\hilb X}} H_{\hilb X}$, corresponding to the $p_i$,
		and $2k-d+2$ copies of $\uhilb X$, corresponding to the $q_i$.
		Then, define $B'$ to be the irreducible component of the reduced scheme of triples $(P,Q,b)$ with $b = (p_1, \ldots, p_{2d}, q_1, \ldots, q_{2k-d+2},t,\Lambda) \in B, P \in G(2d-2k+2, d+k), Q \in G(2k-d+1,d+k-1)$
		whose general member consists of independent points $q_1, \ldots, q_{2k-d+2}$ so that
		$P=\overline{q_1, \ldots, q_{2k-d+2}}$ and so that $P$ intersects $H$ transversely with $Q = P \cap H$.
		Finally, define $\Psi := \Psi_2 \cap B'$ to be the scheme theoretic intersection.
		Now, consider the projection maps
		\begin{equation}
			\nonumber
			\begin{tikzcd}
				\qquad & \Psi \ar{ld}{\pi_1} \ar{rd}{\pi_2} & \\
				\minhilbsmooth d k && B. 
			\end{tikzcd}\end{equation}
		Following \ref{technique-dimension}, it suffices to show that $\dim \Phi = \dim B$.

		So, we will now show $\dim \Phi = \dim B$. First, we compute $\dim B$.
		Observe that we have a projection
		\begin{align*}
			B \xrightarrow \eta H^{2d} \times (\bp^n)^{2k - d + 2}.
		\end{align*}
		Note further that the fibers of this map $\eta$ are all isomorphic to a Schubert cell
		inside $G(2d-2k + 1, d +k)$ of $(2d-2k)$-planes meeting a given $2k-d$ plane.
		This Schubert cell has codimension $k-1$ in $G(2d-2k+1, d+k)$ by
		~\cite[Theorem 4.1]{Eisenbud:3264-&-all-that}. Therefore, the fibers of $\eta$
		are
		\begin{align*}
			(2d - 2k + 1)(2k-d-1) - (k-1).
		\end{align*}
		Now, because the target of $\eta$ is $2d(d+k-2) + (2k-d+2)(d+k-1)$-dimensional,
		and the fibers are irreducible of the same dimension,
		\autoref{lemma:irreducible-base-and-fibers}
		implies that $B$ is irreducible of dimension
		\begin{align*}
			&(2d - 2k + 1)(2k-d-1) - (k-1) +2d(d+k-2) \\
			& \qquad + (2k-d+2)(d+k-1).
		\end{align*}

		To complete the proof, it suffices to show that
		\begin{align*}
			\dim \Phi &= (2d - 2k + 1)(2k-d-1) - (k-1) +2d(d+k-2) \\
			& \qquad + (2k-d+2)(d+k-1),
		\end{align*}
		as the latter is equal to $\dim B$.	
		From the definition of $\Psi$,
		we have a natural projection map	
		\begin{align*}
			\Psi \xrightarrow {\pi_3} \Psi_1
		\end{align*}
		
		We claim it suffices to show that the fibers of $\pi_3$ are $k + (2k-d) +(2d-2k-1)(3k-d-1)$ dimensional. 
		To see this, we will show that the sum of dimension of a general fiber of $\Psi$, plus the dimension of $\Psi$ of $\pi_3$ is equal to the dimension of $B$.
		Now, we can compute the dimension of $\Psi_1$ using the projection $\Psi_1 \ra \hilb X$. Observe
		\begin{align*}
		\dim \uhilb X \times_{\bp^n_{\hilb X}} H_{\hilb X}- \dim \minhilbsmooth d k &= k-1 \\
		\dim \Phi_{q_i} - \dim \minhilbsmooth d k &= k.
	\end{align*}
	The first equality holds by \autoref{lemma:smooth-scroll-hyperplane-section}, as all fibers of the projection map
	\begin{align*}
		\uhilb X \times_{\bp^n_{\hilb X}} H_{\hilb X} \ra \minhilbsmooth d k
	\end{align*}
	are $k-1$ dimensional.
	Therefore, we obtain that 
		\begin{align*}
			\dim \Psi_1 &= \dim \minhilbsmooth d k + 2d (\dim \Phi_{p_i} - \dim \minhilbsmooth d k) \\
			& \qquad + (2k-d+2)(\dim \Phi_{q_i} - \dim \minhilbsmooth d k) \\
			&= (d+k)^2 - k^2 -3 + 2d + (2k-d+2)k
		\end{align*}
		Using this, if we knew the fibers of $\pi_3$ are 
		\begin{align*}
			k + (2k-d) +(2d-2k-1)(3k-d-1)
		\end{align*}
		 dimensional,
		we would obtain that
		\begin{align*}
			\dim \Psi &= \dim \Psi_1+ k + (2k-d) +(2d-2k-1)(3k-d-1) \\
			&=  (d+k)^2 - k^2 -3 + 2d + (2k-d+2)k + k + (2k-d) \\
			& \qquad +(2d-2k-1)(3k-d-1) \\
			&= (2d - 2k + 1)(2k-d-1) - (k-1) +2d(d+k-2) \\
			& \qquad  + (2k-d+2)(d+k-1) \\
				&= \dim B.
		\end{align*}

	So, to complete the proof, we will now show the fibers of $\pi_3$ are $k + (2k-d) +(2d-2k-1)(3k-d-1)$ dimensional.
	For this, observe that we can describe the fiber of $\pi_3$ over a particular general point 
	\begin{align*}
	(X, p_1, \ldots, p_{2d}, q_1, \ldots, q_{2k-d+2})
	\end{align*}
	as quintets
	\begin{align*}
		F :=& \left\{ \left( \Lambda, q, x, P, Q \right): \right. \\
			& \qquad \left. P = \overline {q_1, \ldots, q_{2k-d+2}}, Q = P \cap H, q \in Q, x \in X, x \in \Lambda, q \in \Lambda \right\}
	\end{align*}
	Further, define
	\begin{align*}
		\Psi_3 &= \left\{ \left( q, x, p_1, \ldots, p_{2d}, q_1, \ldots, q_{2k-d+2}, P, Q \right)\right. \\
		& \qquad \left. : p_1, \ldots, p_{2d} \in H, P = \overline {q_1, \ldots, q_{2k-d+2}}, Q = P \cap H, q \in Q, x \in X \right\}.
	\end{align*}
	Observe that we have commuting triangle
	\begin{equation}
		\nonumber
		\begin{tikzcd}
			\Psi \ar {rr}{\pi_4} \ar {rd}{\pi_3} && \Psi_3 \ar {ld}{\pi_5} \\
			& \Psi_1 & 
		 \end{tikzcd}\end{equation}
	
	Note that the fiber of $\pi_5$ over a point of $\Psi_1$ is
	\begin{align*}
		G &:= \left\{ \left( q, x, P, Q \right): p_1, \ldots, p_{2d} \in H, \right. \\
		& \qquad \left. P = \overline {q_1, \ldots, q_{2k-d+2}}, Q = P \cap H, q \in Q, x \in X \right\}.
	\end{align*}
Therefore, the dimension of the fibers of $\pi_3$ is the sum of the dimensions of the fibers of $\pi_4$ and $\pi_5$.
	So, to complete the proof, it suffices to show that the fibers of $\pi_5$ are $k+(2k-d)$ dimensional and the fibers of $\pi_5$
	are $(2d-2k-1)(3k-d-1)$ dimensional over a general point in the base $\Psi_3$. 
	
	We can first see that the fibers of $\pi_5$ are simply isomorphic to the product of $X$ and the $2k-d$ plane $Q$, and therefore $k+2k-d$ dimensional.
	Next, we see that the fibers of $\pi_4$ consist of $2d-2k$ planes $\Lambda$ containing $x,q$. In the case that
	$x,q$ are distinct points,
	the fiber will simply be $2k-d$ planes containing two points, which is isomorphic to the Grassmannian $G(2d-2k -1, d + k - 2)$.
	Finally,
	\begin{align*}
		\dim G(2d-2k-d, d + k -2) = (2d - 2k-1)(3k-d-1)
	\end{align*}
	as desired. This completes the proof.
	\end{proof}
	
\section{Inductively verifying \getrefnumber{custom:ind-mid}}

In this section, we show \autoref{custom:ind-mid} holds for a given degree
$d'$, assuming it holds for degree $d'-1$.
This is accomplished in \autoref{lemma:middle-induction}.

The proof of the following \autoref{lemma:middle-induction} is quite analogous to \autoref{lemma:high-induction}, in that we will specialize all but two of the points to
a hyperplane, and then find a scroll of degree d which is a union of a scroll of degree $d-1$ inside a hyperplane and a $\bp^k$, meeting along a $\bp^{k-1}$. We now prove \autoref{lemma:middle-induction} assuming \autoref{proposition:at-most-two-components} and \autoref{lemma:ind-mid-smooth-finite}.

\begin{figure}
	\centering
	\includegraphics[scale=.3]{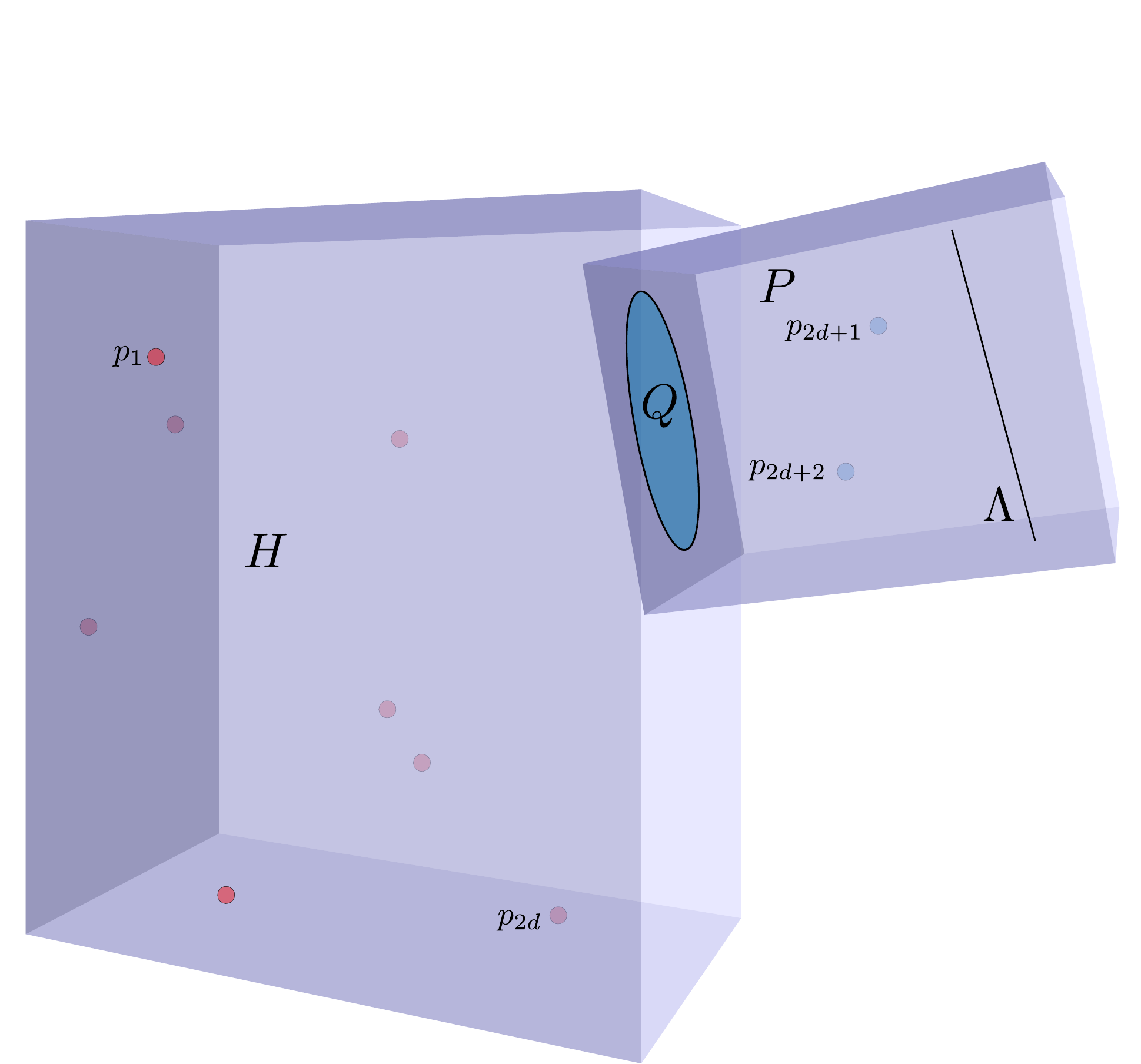}
	\caption{A visualization of the idea of the proof of \autoref{lemma:middle-induction},
	where one inductively specializes $2d$ points to lie in a hyperplane.
}
	\label{figure:}
\end{figure}

\begin{proposition}
	\label{lemma:middle-induction}
	Let $k \geq 3$. Assuming induction hypothesis \autoref{custom:ind-mid}
	holds for varieties of degree $d-1$ with $k+1 \leq d \leq 2k-2$, 
	it holds for varieties of degree $d$.
\end{proposition}
\begin{proof}[Proof assuming \autoref{proposition:at-most-two-components} and \autoref{lemma:ind-mid-smooth-finite}]
	Start with $2d+2$ general points $p_1, \ldots, p_{2d+2}$ and a general $(2k-d-1)$-plane $\Lambda$. 
	We would like to show there is an element of $\minhilb d k$ containing
	these points and $\Lambda$.

	Choose a general hyperplane $H \subset \bp^{d + k - 1}$ and specialize $\Lambda$ and $p_1, \ldots, p_{2d}$ to be contained in this hyperplane.
	Let $P$ be the $(2k-d+1)$-plane spanned by $\Lambda, p_{2d+1}$ and $p_{2d+2}$, and let $Q := P \cap H$. Then, by inductive hypothesis
	~\ref{custom:ind-mid}, there is a scroll of degree $d - 1$ and dimension $k$, call it $X$,
	containing the $2d$ points $p_1, \ldots, p_{2d}$ and the $2k-d$ plane $Q$.

	There is a unique $k-1$-plane contained in $X$ and containing $Q$,
	by \autoref{lemma:linear-spaces-in-varieties-of-minimal-degree}. Call this plane $P$. Note that here the special case
	~\ref{custom:line-small} of \autoref{lemma:linear-spaces-in-varieties-of-minimal-degree} only occurs when
	$X \cong \bp^1 \times \bp^1$ is a quadric surface, and so we need not consider this case, as we are assuming $k \geq 3$.
			
	Take $Y$ to be the span of $\Lambda, p_{2d+1}, p_{2d+2},$ and $P$.
	The variety $X \cup Y$ with reduced scheme structure lies in $\broken d k \subset \minhilb d k$
	and contains $p_1, \ldots, p_{2d+2}$, and $\Lambda$,
	as desired.
	Note that the domain and range of $\pi_2$ have the same dimension by \autoref{lemma:finite-containing-plane-and-points}.
	Furthermore, by \autoref{proposition:at-most-two-components},
	when $k > 2$, the incidence correspondence which is the domain of $\pi_2$ will have a unique component of maximal dimension,
	and possibly one other component of lower dimension. 
	If the source of $\pi_2$ from \autoref{proposition:at-most-two-components}
	has two components, since we are choosing $(p_1, \ldots, p_{2d+2}, \Lambda)$ generally,
	we may as well choose this tuple not to lie in the image of the lower dimensional component.

	Next, there are only finitely many smooth scrolls and finitely many scrolls in $\broken d k$
	containing $p_1, \ldots, p_{d+2}, \Lambda$, as follows from \autoref{lemma:ind-mid-smooth-finite}.
	So, applying \autoref{corollary:isolated-points-in-minutely-broken-scrolls}, to the map from
	the irreducible component of maximal dimension inside the domain of $\pi_2$ to the range of $\pi_2$,
	we see that the point
	the point $[X \cup Y]$ constructed above in the Hilbert scheme is isolated among all schemes containing
	$p_1, \ldots, p_{2d+2}, \Lambda$. 
	
	Then, by \autoref{lemma:isolated-fiber-implies-dominant},
	we obtain that $\pi_2$ is surjective, meaning that 
	\ref{custom:ind-mid} holds for scrolls of degree $d$ and dimension $k$.	
\end{proof}
\begin{remark}
	\label{remark:}
	The restriction that $k \geq 3$ made in the hypothesis of \autoref{lemma:middle-induction} was only
	introduced to make the proof slightly easier, so that we would not have to deal with the case of quadric surfaces in $\bp^2$.
	Of course, it is quite easy to deal with the dimension 2 case separately. This essentially amounts to showing there
	are only finitely many quadric surfaces containing 6 general points and a line. Nevertheless, interpolation in dimension
	$2$ also follows immediately from Coskun's thesis ~\cite[Example, p.\ 2]{coskun:degenerations-of-surface-scrolls}
\end{remark}

To complete the proof of \autoref{lemma:middle-induction}, we prove \autoref{proposition:at-most-two-components} and
\autoref{lemma:ind-mid-smooth-finite}. We start with \autoref{proposition:at-most-two-components}.

	\begin{lemma}
		\label{proposition:at-most-two-components}
		Let $X \in \minhilb d k$ and let
	$F$ denote the Fano scheme whose $\bk$ points are $2k-d-1$ planes contained in $X$.
	Define
	\begin{align*}
\Phi = \uhilb X \times_{\hilb X} \cdots \times_{\hilb X} \uhilb X \times_{\hilb X} F.
	\end{align*}
	where there are $2d + 2$ copies of $\uhilb X$.
	Define
	\begin{equation}
		\nonumber
		\begin{tikzcd}
			\qquad & \Phi \ar {ld}{\pi_1} \ar {rd}{\pi_2} & \\
			 \minhilb d k  && (\bp^n)^{2d+2} \times G(2k-d, n+1).
		 \end{tikzcd}\end{equation}

	Then, $\Phi$ is irreducible if $2k-d-1 \neq 1$. When $2k -d -1 = 1$, it has at most two components.
		If it has two components, one is of dimension
		\begin{align*}
			\dim (\bp^n)^{2d+2} \times G(2k-d, n+1) &= (k+d-1)(2d+2) + (2k-d)(2d-k)
		\end{align*}
		and the other is of dimension 
		\begin{align*}
			(k+d-1)(2d+2) + (2k-d)(2d-k) - 2(k-2).
		\end{align*}
	\end{lemma}
	\begin{proof}
		First, the statements regarding irreducibility of $\Phi$ follow immediately from the statements on irreducibility
		of the Fano scheme $F$ as given in \autoref{lemma:linear-spaces-in-varieties-of-minimal-degree},
		coupled with the assumption that $\hilb X$ and $\uhilb X$ are irreducible.

		Since \autoref{lemma:linear-spaces-in-varieties-of-minimal-degree}, also gives the dimensions
		of the one or two irreducible components of $F$, call them $F_i$, we have
		that the dimension of the $i$th irreducible component of $\Phi$ is
		$(\dim F_i - \dim \hilb X) + (2d+2)(\dim \uhilb X - \dim \hilb X) + \dim \hilb X$, as claimed.
		\end{proof}

		Finally, we prove \autoref{lemma:ind-mid-smooth-finite}, which finishes the proof of 
		\autoref{lemma:middle-induction}.

	\begin{lemma}
		\label{lemma:ind-mid-smooth-finite}
		Let $p_1 \ldots, p_{2d +2}$ be $2d+2$ points in $\bp^n$ and let $\Lambda$ be a $(2k-d-1)$-plane so that
		there is a hyperplane $H \subset \bp^n$ containing $p_1, \ldots, p_{2d}$ but the points and plane are otherwise
		general. Then, there are only finitely many scrolls in $\minhilbsmooth d k$ containing
		$p_1, \ldots, p_{2d+2}, \Lambda$ and there are only finitely many scrolls in $\broken d k$ containing $p_1, \ldots, p_{2d+2}, \Lambda$.
	\end{lemma}

	\begin{proof}
	First, let us show there are only finitely many scrolls in $\broken d k$ containing $p_1, \ldots, p_{2d+2}, \Lambda$.
	Suppose we took such a scroll $X \cup Y$ with $Y \cong \bp^k, X \in \minhilb {d-1}k$.
	We would necessarily have $X \subset H$, as if $X$ spanned any other hyperplane, there would not be a $k$-plane
	$Y$ so that $Y \cup H$ contains $p_1, \ldots, p_{2d+2}, \Lambda$, and in particular there can be no such
	scroll $X \cup Y$ containing $p_1, \ldots, p_{2d+2}, \Lambda$.
	So, we conclude that $Y$ contains $\Lambda, p_{2d+1}, p_{2d+2}$. Let $P$ be the plane spanned by
	$\Lambda, p_{2d+1}, p_{2d+2}$. Then, $Q:= P \cap H$ is a $(2k-d)$-plane. It now suffices to show there are only finitely
	many scrolls $X \in \minhilb d k$ containing $Q$ and planes $Y$ meeting $X$ at $\bp^{k-1} \subset H$ and containing
	$\Lambda, p_{2d+1}, p_{2d+2}$.
	Now, by \autoref{lemma:finite-containing-plane-and-points}, there are finitely many $X \in \minhilb {d-1} k$
	containing $Q, p_1, \ldots, p_{2d}$. Then, for any such $X$, there is at most 1 $k$-plane containing $Q$ which meets
	$X$ in a $(k-1)$-plane and contains $\Lambda, p_{2d+1}, p_{2d+2}$.

	To complete the proof, we only need show there are only finitely many scrolls in $\minhilbsmooth d k$ containing
	$p_1, \ldots, p_{2d+2}, \Lambda$.
	This proof is quite similar to that of \autoref{lemma:finite-containing-plane-and-points}.
	For this proof, we will use \ref{technique-conditions}.
	Observe that it is $n-k$ conditions to contain a point, while it is also $(n-1) - (k-1) = n - k$ conditions
	to contain a point in a fixed $n-1$ dimensional hyperplane.
	Therefore, a point in a fixed hyperplane and a general point impose the same number of conditions.
	Hence, since the configuration from \autoref{lemma:finite-containing-plane-and-points},
	prior to specializing $p_1, \ldots, p_{2d}$ to a hyperplane already imposed
	$\dim \minhilb d k$ conditions, the current configuration after specializing them to a hyperplane
	still imposes $\dim \minhilb d k$ conditions, completing the proof.

	\begin{exercise}
		\label{exercise:}
		Prove that there are only finitely many scrolls in $\minhilbsmooth d k$
		using \ref{technique-dimension}, following the proof using \ref{technique-dimension}
		for \autoref{lemma:finite-containing-plane-and-points}.

	\end{exercise}
\end{proof}

\section{A combinatorial interlude}
\label{subsection:a-combinatorial-interlude}

In this section, we count the number of planes containing a point and meeting $k$ lines
in \autoref{lemma:number-of-planes}.
We will only need that this number is nonzero to show scrolls satisfy interpolation.
It is fairly easy to see this number is nonzero from Schubert calculus.
Nevertheless, it is nice to know that the actual number of scrolls containing these linear spaces.
This may be helpful in finding the number of scrolls through the expected number of
linear spaces, as discussed in \autoref{question:number-of-scrolls}.

	\begin{definition}
		\label{definition:}
		A {\bf Young diagram} of size $n$ is a collection of $n$ $1 \times 1$ boxes which are left and bottom justified.
		A {\bf right-justified Young diagram} of size $n$ is a collection of $n$ $1 \times 1$ boxes which are right and top justified.
		A {\bf Young tableaux} (respectively {\bf right-justified Young tableaux}) of size $n$ is a Young diagram 
		(respectively right-justified Young diagram), together with an ordered labeling of the boxes from $1$ to $n$.
		A Young tableau (respectively right-justified Young tableau) is {\bf standard} if the entries in each
		row are increasing from left to right (respectively left to right) and bottom to top.
		Define $T(n) := \#\left\{Y: Y \text{ is a standard Young tableaux } \right\}$.
	\end{definition}

	\begin{lemma}
		\label{lemma:number-of-planes}
		The number of $(k-1)$-planes which are subspaces of $\bp^{2k-1}$ containing a point and meeting $k$ lines is
		$T(k-1)$. In particular, it is nonzero.
	\end{lemma}
	\begin{proof}
		Consider $G(k, 2k)$, the Grassmannian of $(k-1)$-planes	in $\bp^{2k-1}$. Let $A(G(k,2k))$ denote its chow ring. 
		Let $\sigma_{k-1} \in A(G(k, 2k))$ denote the class of
		the Schubert cell in the Grassmannian of $(k-1)$-planes meeting a line and let $\sigma_k$ denote the class
		of the Schubert cell in the Grassmannian of $(k-1)$-planes containing a point, as defined in
		~\cite{Eisenbud:3264-&-all-that}.	
		We are now asking to compute the product $\sigma_k \sigma_{k-1}^k \in A_0(G(k,2k)) \cong \bz$
		as the value of this class is the number of $(k-1)$-planes which meets $k$ lines and contains
		a point.
		Applying Pieri's formula \cite[Proposition 4.9]{Eisenbud:3264-&-all-that} interpreted in terms of Young diagrams via 
		\cite[Section 4.5, page 165]{Eisenbud:3264-&-all-that},
		the product $\sigma_k \cdot \sigma_{k-1}^k$ is precisely the number of ways to fill a $k \times k$ box
		by iteratively placing a set of $k$ boxes, followed by $k$ sets of $k-1$ boxes so that
		after adding a given set of boxes, the resulting shape is a Young diagram and we only add at most
		a single box to each column.

		Once we show this, we will finish the proof. This, however, follows from \autoref{lemma:young-bijection},
		which we now prove.
	\end{proof}	

		\begin{lemma}
			\label{lemma:young-bijection}
			Let $S(k)$ denote the number of ways to fill a $k \times k$ box
		by iteratively placing a set of $k$ boxes, followed by $k$ sets of $k-1$ boxes so that
		after adding a given set of boxes, the resulting shape is a Young diagram and we only add at most
		a single box to each column. Then, $S(k) = T(k-1)$, where $T(k-1)$ is the number of standard Young tableaux of size $k-1$.
		\end{lemma}
		\begin{proof}
			To compute $S(k)$, note that we must start by placing $k$ boxes on the bottom row.
			Therefore, $S(k)$ is equal to the number of ways to place $k$ sets of $k-1$ boxes so that 
			no two boxes from the $i$th set are in the same column and after adding
			the $i$th set, the resulting set of boxes is a Young diagram.
			We now give a bijection between the set of all such configurations and the set of standard Young tableau of size
			$k-1$, which will complete the proof.
		
			Note that the datum of adding $k-1$ boxes under the above constraints is the same as choosing a column
			that has more boxes than the column to its right.
			So, given a valid filling of the $(k-1) \times k$ box by $k$ sets of $k-1$ boxes, associate to it the
			Young diagram of size $k-1$ whose $i$th box is placed on the top of the column in which there is no box from the
			$i$th set of $k-1$ boxes. This is a right-justified young diagram because the $i$th box must be placed
			in a column so that the column to it's right has fewer boxes, meaning that before inserting
			the $i$th box in our standard Young tableaux, we have already placed
			some box in our standard Young tableau on the column to its right.

			To show this is a bijection, the inverse map can be given quite simply:
			Given a right justified standard Young tableaux, for all $i$ with $1 \leq i \leq k-1$, we place the $i$th set of $k-1$ boxes in all columns but
			the column in which the $i$th box of the standard Young tableaux lies in. 
			We then place the $k$th set of $k-1$ boxes in whatever $k-1$ boxes are remaining after placing the first $k$ sets of $k-1$ boxes.
			This then defines a valid filling of the $(k-1) \times k$ box by $k$ sets of $k-1$ boxes.

			It is immediate from the construction that these to processes are mutually inverse, and hence define a bijection.
		\end{proof}

\section{Base case: verifying \getrefnumber{custom:ind-mid} for degree $k$ varieties}
\label{subsection:base-case-plane}

We have essentially completed our inductive steps. Now comes the heart of the
argument, where we will examine interpolation of varieties of dimension $k$ and degree $k$.
Recall that degree $k$, dimension $k$ varieties are realized as
a Segre embedding $\bp^1 \times \bp^{k-1} \ra \bp^{2k-1}$, by \autoref{lemma:segre-isomorphic-to-scroll}.

We have two remaining parts of our argument for showing scrolls satisfy interpolation.
First, in this section, we complete the base case of the induction, which is carried out in
\autoref{proposition:segre-plane}.
We show that we can find a variety of degree $k$ and dimension $k$ containing
$2k+2$ general points and a general $(k-1)$-plane.

Then, in the next section \autoref{subsection:base-case-points},
we show that $\minhilb k k$ satisfies interpolation, meaning that we can find
a scroll in $\minhilb k k$ through $3k+3$ general points.

In order to prove 
\autoref{proposition:segre-plane},
we will first need a way of constructing a scroll through a point and $2k + 2$ points.
For this, we will need \autoref{lemma:unique-segre-containing-linear-spaces}, which we now prove.

\begin{lemma}
	\label{lemma:unique-segre-containing-linear-spaces}
	Fix $k$ lines $\ell_1, \ldots, \ell_k \subset \bp^{2k-1}$ and three $(k-1)$-planes $\Lambda_1, \Lambda_2,\Lambda_3 \subset \bp^{2k-1}$.
	Suppose that $\ell_i \cap \Lambda_j$ consists of a single point for all pairs $(i,j)$, and that no two of these intersection points
	are the same. Further, suppose the lines $\ell_1, \ldots, \ell_k$ span $\bp^{2k-1}$.
	Then, there is a unique scheme $X \subset \bp^{2k-1}$ so that $X \cong \bp^1 \times \bp^{k-1}$ is the Segre embedding and $X$
	contains $\ell_1, \ldots, \ell_k, \Lambda_1, \Lambda_2, \Lambda_3$.

	Further, the lines $\ell_1, \ldots, \ell_k$ appear as the images of lines of the form $\bp^1 \times \left\{ q_k \right\} \subset \bp^1 \times \bp^{k-1}$
	under the Segre embedding, with $q_i \in \bp^{k-1}$.
\end{lemma}
\begin{proof}
	First, we show some such scheme $X$ exists. Since the lines span all of $\bp^{k-1}$, we may choose coordinates so that
	\begin{align*}
	\ell_i = V(x_0, \ldots, x_{2i-3}, x_{2i}, x_{2i+1}, \ldots, x_{2k-1}).
	\end{align*}
	Denote the standard coordinate point by $e_i := [0, \ldots, 0, 1, 0, \ldots, 0]$, where there is a $1$ in the $i$th coordinate.
	Since automorphisms of $\bp^1$ are
	triply transitive, we may assume that 
	\begin{align*}
		\Lambda_1 \cap \ell_i &= e_{2i-2} \\
		\Lambda_2 \cap \ell_i &= e_{2i-1} \\
		\Lambda_3 \cap \ell_i &= e_{2i-2} + e_{2i-1}
	\end{align*}
Now, observe that the variety defined by the minors of the matrix
\begin{align*}
	\begin{pmatrix}
		x_0 & x_2 & \cdots & x_{2k-2} \\
		x_1 & x_3 & \cdots & x_{2k-1}
	\end{pmatrix}
\end{align*}
defines a variety of minimal degree $k$ in $\bp^{2k-1}$ by \autoref{proposition:equivalence-minors-swept-planes}, and therefore
a Segre variety by \autoref{lemma:segre-isomorphic-to-scroll}. This variety contains the linear
spaces $\ell_1, \ldots, \ell_k, \Lambda_1, \Lambda_2, \Lambda_3$
by construction, because the minors of the matrix vanish on all of these linear spaces.

To complete the proof, it suffices to show this variety is unique. 
However, since the lines $\ell_i$ meet the $(k-1)$ planes
$\Lambda_j$, they must be
type \ref{custom:line-small}, by \autoref{lemma:linear-spaces-in-varieties-of-minimal-degree}.
Then, if we had a scroll through such a configuration, say $\pi:X \ra \bp^1$,
it would necessarily satisfy
$\pi(\ell_1 \cap \Lambda_j) = \cdots = \pi(\ell_k \cap \Lambda_j)$.
In particular, since the isomorphisms $\pi|_{\ell_i}: \ell_i \ra \bp^1$ 
are determined for the three points
$\ell_i \cap \Lambda_j$, they are uniquely determined, as
automorphisms of $\bp^1$ are triply transitive.
Hence, we have specified isomorphisms $\ell_i \cong \bp^1$,
and there is then a unique scroll defined by these isomorphisms,
as follows from \autoref{proposition:equivalence-vector-bundle-swept-planes}.
More precisely, this proposition shows that every smooth scroll can be constructed
from only the datum of
a collection of $k$ rational normal curves with specified isomorphisms
to $\bp^1$.
\end{proof}

\begin{figure}
	\centering
\includegraphics[scale=.3]{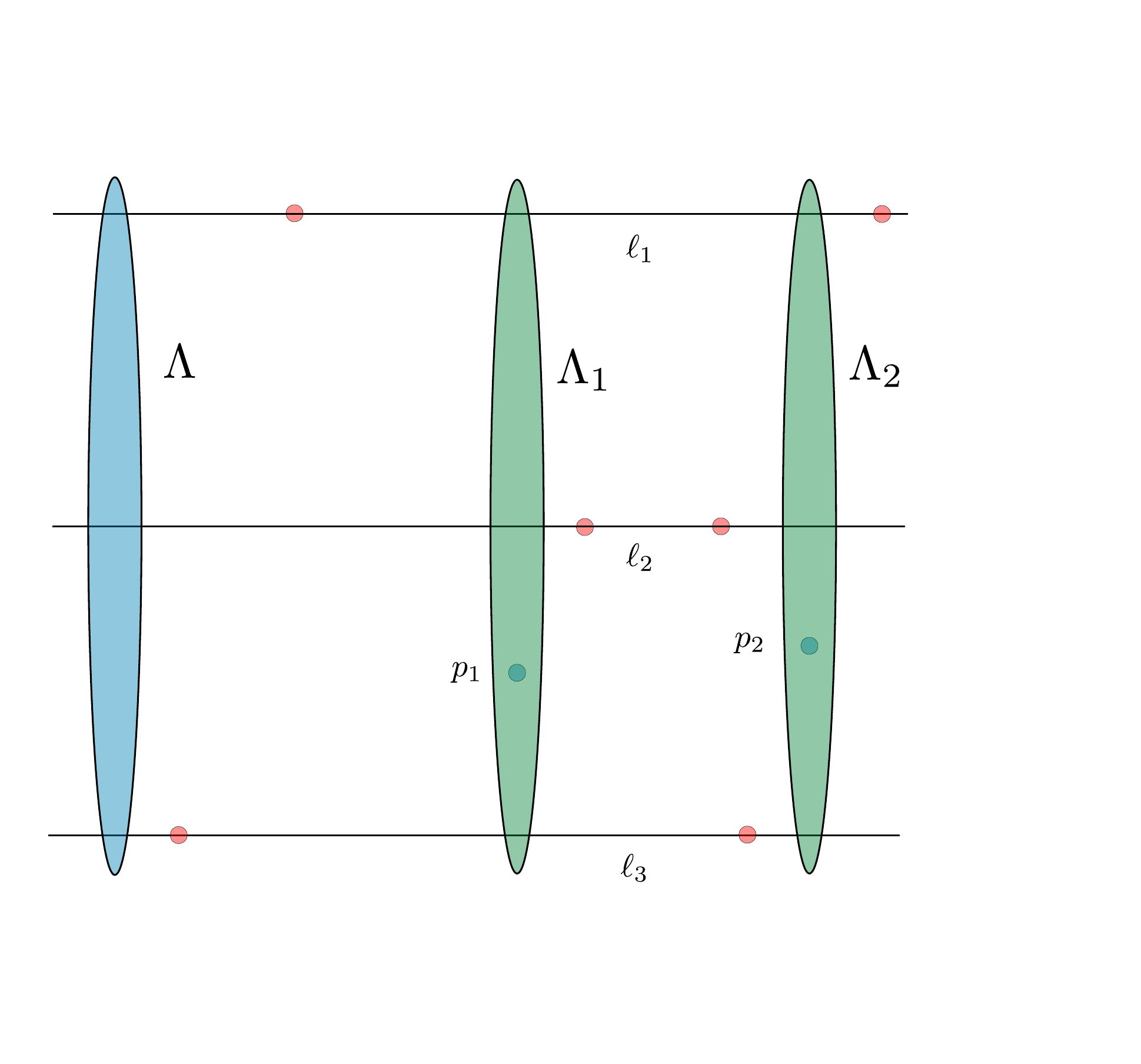}
\caption{A visualization of the idea of the proof of \autoref{proposition:segre-plane},
where one specializes pairs of points to lie on lines meeting the given $(k-1)$-plane.
}
\label{figure:}
\end{figure}

\begin{proposition}
	\label{proposition:segre-plane}
	There exist finitely many varieties of minimal degree $\scroll {1^k}$
	containing $2k + 2$ general points and a general $(k-1)$-plane.
\end{proposition}

\begin{proof}
	Label the points by $p_1, \ldots, p_{2k+2}$ and the $(k-1)$-plane by $\Lambda$.
	Define the $k$ lines
	$\ell_1, \ldots, \ell_k$, by $\ell_i := \overline {p_{2i+1}, p_{2i+2}}$.
	Start by specializing the points $p_1, \ldots, p_{2k}$ to general points
	satisfying the condition that the $k$ lines
	$\ell_1, \ldots, \ell_k$ all meet
	$\Lambda$. We claim that there are only finitely many varieties $\scroll {1^k}$ which
	contain such a configuration of points and $\Lambda$.

	To see this, note that the ideal of a scroll is generated by quadrics,
	by \autoref{proposition:equivalence-minors-swept-planes}. Note that the variety $\scroll {1^k}$ contains
	$\Lambda$ and
	three points on each line $\ell_i$, namely $p_{2i+1}, p_{2i+2}$ and $\ell_i \cap \Lambda$.
	By Bezout's theorem, each quadric containing these three points must contain the line.
	Since the variety $\scroll {1^k}$ is the intersection of the quadrics containing it, 
	and each quadric containing it contains $\ell_i$, it follows that $\ell_i \subset X$.

	Hence, it suffices to show there are a finite, nonzero number of smooth
	varieties $\scroll {1^k}$ corresponding to a point
	in $\minhilb k k$ that
	\begin{enumerate}
		\item contain a $(k-1)$-plane $\Lambda$
		\item contain $k$ lines $\ell_1, \ldots, \ell_k$, and
		\item contain $2$ points $p_1, p_2$.
	\end{enumerate}
	This suffices because the smooth locus of $\minhilb k k$ is open, and so any
	such smooth scheme would then be an isolated point inside $\minhilb k k$ containing $p_1, \ldots, p_{2k+2}, \Lambda$.
	Now, define $\Phi$ to be
	\begin{align*}
		\left\{ \left( X, p_1, \ldots, p_{2k+2}, \Lambda \right) \subset \minhilb k k \times (\bp^{2k-1})^{2k+2}\times G(k, 2k): p_i \in X, \Lambda \subset X \right\}
	\end{align*}
	and define the projection map	
	\begin{equation}
		\nonumber
		\begin{tikzcd}
			\qquad & \Phi \ar {ld} \ar {rd}{\pi_2} & \\
			\minhilb {d-1}k && (\bp^{2k-1})^{2k} \times G(k, 2k).
		\end{tikzcd}\end{equation}
	The source and target of $\pi_2$ have the same dimension by
	\autoref{lemma:finite-containing-plane-and-points}.
	So, by \autoref{lemma:isolated-fiber-implies-dominant},
	$\pi_2$ would be dominant if it has a point isolated in its fiber.

	So, to complete the proof, it suffices to show there is a finite nonzero number of smooth 
	scrolls containing $\Lambda, \ell_1, \ldots, \ell_k, p_1, p_2$.
	Next, any smooth scheme in $\minhilb k k$ is abstractly isomorphic to $\bp^1 \times \bp^{k-1}$ by
	\autoref{theorem:classification-of-varieties-of-minimal-degree} and \autoref{lemma:segre-isomorphic-to-scroll}.
	Let $X$ be some smooth scheme containing $\Lambda, \ell_1, \ldots, \ell_k, p_1, p_2$.
	By \autoref{lemma:linear-spaces-in-varieties-of-minimal-degree}, when $k > 2$, (the $k = 2$ case can easily
	be handled separately,) the only $(k-1)$-planes contained in $X$
	are those which are the image of planes of the form $\left\{ p \right\} \times \bp^{k-1}$ under the Segre embedding.
	Additionally, by \autoref{lemma:linear-spaces-in-varieties-of-minimal-degree}, any line
	contained in $X$ is either image of 
	$\left\{ p \right\}\times \ell$ or $\bp^1 \times \left\{ q \right\}$
	under the Segre embedding. Since $\ell_1, \ldots, \ell_k$ all meet $\Lambda$ at precisely 1 point, they must pull back to lines
	of the form $\bp^1 \times \left\{ q \right\}$.
	Additionally, the two remaining points $p_1$ and $p_2$ must be contained in two $(k-1)$-planes $\Lambda_1, \Lambda_2$ with
	$\Lambda_i \cap \ell_j \neq \emptyset$.
	This is because they pull back to some pair of points $q_1,q_2$, on $\bp^1 \times \bp^{k-1}$ which are contained in some $\left\{ p \right\}\times \bp^{k-1}$,
	which intersects the line $\bp^1 \times \left\{ q \right\}$ at $\left\{ p \right\}\times \left\{ q \right\}$.
	According to \autoref{lemma:number-of-planes}, there are $T(k-1)$ choices of planes containing $q_1$ which meet $\ell_1, \ldots, \ell_k$,
	where $T(m)$ is the number of standard Young diagrams of size $m$. Therefore, there are $T(k-1)^2$ choices of pairs of planes $\Lambda_1, \Lambda_2$
	containing $p_1, p_2$ and meeting lines $\ell_1, \ldots, \ell_k$.

	Now, for any such choice of planes $\Lambda_1, \Lambda_2$, there is a unique smooth scroll containing
	$\ell_1, \ldots, \ell_k, \Lambda_1, \Lambda_2, \Lambda$ by \autoref{lemma:unique-segre-containing-linear-spaces}.
	In particular, this scroll contains $\Lambda,p_1, \ldots, p_{2k+2}$, as desired. Hence, in total, there are
	$T(k-1)^2$ smooth scrolls containing such a collection of points, which is a finite, positive number, as claimed.
\end{proof}

\section{Base case: interpolation for degree $k$ varieties}
\label{subsection:base-case-points}

Before piecing our inductive argument for interpolation of scrolls together,
we have to show that varieties of minimal degree $k$ and dimension $k$
satisfy interpolation.
Perhaps surprisingly, this will turn out to be the
most subtle case of all, and will be proven in \autoref{proposition:segre-lines-interpolation}.
To simplify the clutter a bit, we'll start by outlining the argument in the case that
$k = 3$. 

\begin{example}
	\label{example:three-segre-interpolation}
	Before we go through the argument that $\minhilb k k$ satisfies interpolation 
	in detail, let's start with an example in the case $k=3$,
	illustrating the idea of this argument.
	See \autoref{figure:isolated-segre-threefold} for a pictorial description of the degeneration
	in this example.
	We will omit many of the details, as they are covered in the proof of
	\autoref{proposition:segre-lines-interpolation}.
	In this case we want to show there is a scroll $\scroll {1,1,1}$ passing through
	$12$ general points in $\bp^5$.

	Start by specializing the points so that
	\begin{enumerate}
		\item $p_1, p_2$, and $p_3$ lie on a line $\ell_1$,
		\item $p_4, p_5$, and $p_6$ lie on a line $\ell_2$, and
		\item $p_7, p_8$, and $p_9$ lie on a line $\ell_3$.
	\end{enumerate}
	By Schubert calculus, there will be two $2$-planes which contain a general point $p_{10}$ and meet all three lines $\ell_1, \ell_2$, and $\ell_3$.
	Choose one of these two choices, and call it $P_1$. Similarly, choose a two plane $P_2$ (respectively $P_3$) containing $p_{11}$ (respectively $p_{12}$) and meeting
	all three lines $\ell_1, \ell_2$, and $\ell_3$.
	Let the unique intersection point of $P_i$ with $\ell_r$ be denoted $q_{i,r} := p_i \cap \ell_r$.
	Then, recall an automorphism of $\bp^1$ is uniquely determined by where three points are sent.
	So, we obtain unique isomorphisms $\ell_i \cong \bp^1$ sending
	\begin{align*}
		q_{1, r} &\mapsto 0 \in \bp^1 \\
		q_{2, r} &\mapsto 1 \in \bp^1 \\
		q_{3, r} &\mapsto \infty \in \bp^1,
	\end{align*}
	This uniquely determines the maps $\phi_r: \ell_r \ra \bp^1$ and hence
	determines the scroll $X$.
	
	One may observe that there will be other scrolls in $\minhilb 3 3$ containing $p_1, \ldots, p_{12}$. In fact, there will be a 3 dimensional family of such scrolls.
	However, we claim that the scroll $X$ is an isolated point, meaning that it does not lie in the closure of any of the
	higher dimensional loci of scrolls containing $p_1, \ldots, p_{12}$. Recall that
	$X$ is the image of $\bp^1 \times \bp^2$ under the Segre embedding.
	Then, the only lines contained in $X$ are those which are the image of
	$\bp^1 \times \left\{ p \right\}$ for $p \in \bp^2$ (i.e., of type \ref{custom:line-small})
	and those which
	are the image of $\left\{ q \right\} \times \ell$ for 
	$q \in \bp^1$ and $\ell \subset \bp^2$ a line (i.e., of type \ref{custom:line-large}).
	Then, suppose we had a family of threefolds $\scx$.
	Inside this family, we can follow our three lines.
	Since the lines lying in the family \ref{custom:line-large} cannot degenerate
	to those lying in \ref{custom:line-small}, the only possibility is that
	the lines appeared as lines in \ref{custom:line-small} for all members of our
	family $\scx$. But, since there are only finitely many
	$X$ containing the three lines as lines in \ref{custom:line-small},
	$[X]$ must be an isolated point in its fiber.
\end{example}

Keeping outline given in the case \autoref{example:three-segre-interpolation} in mind,
we now proceed with the general proof.

\begin{proposition}
	\label{proposition:segre-lines-interpolation}
	Let $k \geq 2$ and 
	$p_1, \ldots, p_{3k+3}$ be $3k + 3$ general points
	in $\bp^{2k-1}$.
	Then, there is some variety of minimal degree $\scroll {1^k}$ containing $p_1, \ldots, p_{3k+3}$.
	That is, $\minhilb k k$ satisfies interpolation.
\end{proposition}
\begin{figure}
	\centering
	\includegraphics[scale=.35]{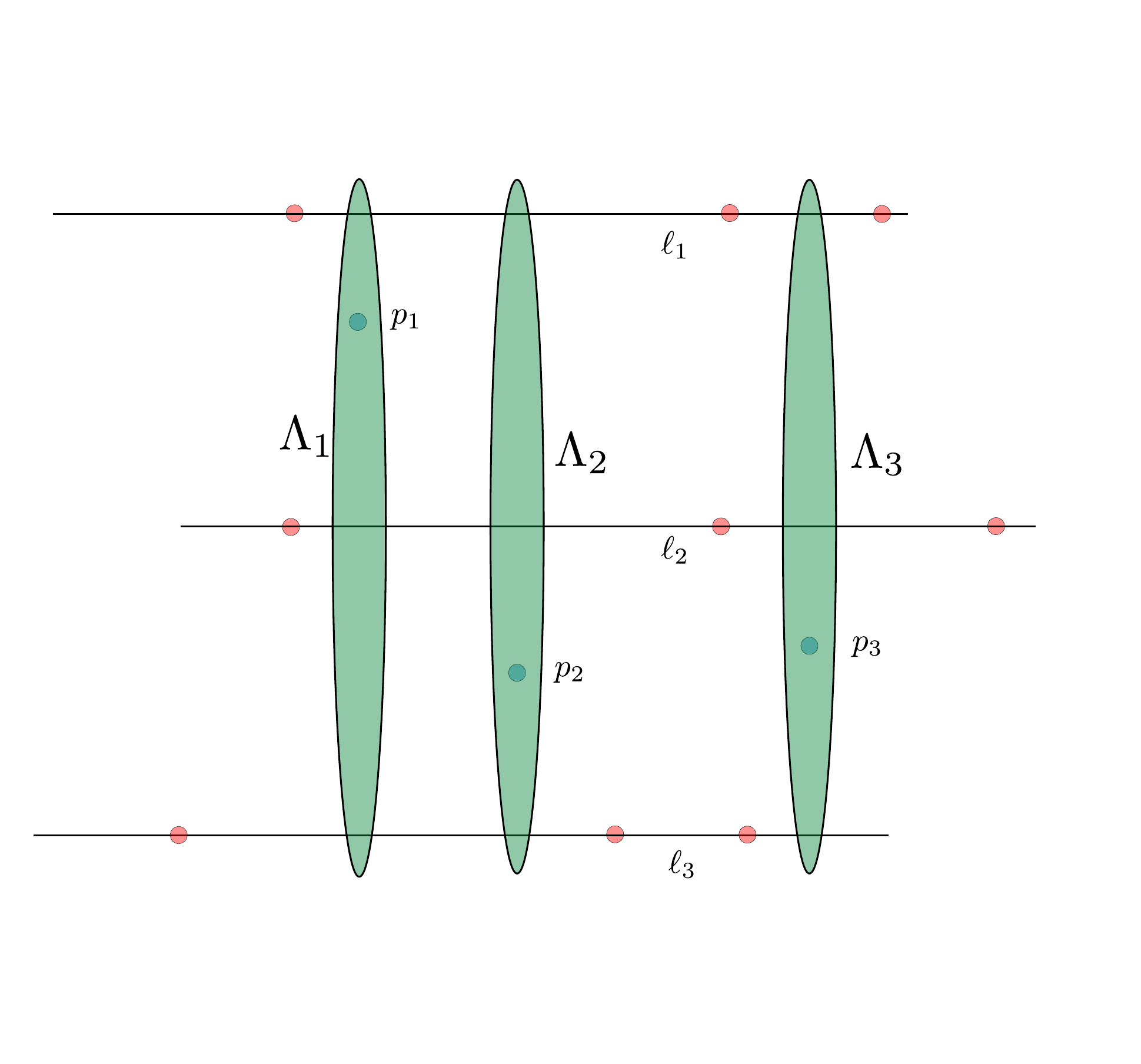}
	\caption{A visualization of the idea of the proof of \autoref{proposition:segre-lines-interpolation},
	where one specializes triples of points to lie on lines, and then finds a scroll containing those points.
}
	\label{figure:isolated-segre-threefold}
\end{figure}

\begin{proof}
	To begin with, specialize the points, as follows:
	Pick $k$ lines $\ell_{1}, \ldots, \ell_k$ so that $\ell_1, \ldots, \ell_k$
	span $\bp^{2k-1}$.
	Now, specialize the points so that
	$p_{3i+1}, p_{3i+2}, p_{3i+3}$ are distinct points on $\ell_{i}$ for $1 \leq i \leq k$.
	This leaves us with $3$ general points, which we have labeled as $p_1, p_2, p_3$.
	
	We will now examine the incidence correspondence
	\begin{align*}
		\Psi := \uhilb X \times_{\hilb X} \cdots \times_{\hilb X} \uhilb X
	\end{align*}
	where the product is taken of $3k+3$ copies of $\uhilb X$. Set theoretically (but without
	specifying a scheme structure)
	we can describe the closed points of $\Psi$ as
	\begin{align*}
	\label{equation:segre-incidence-correspondence-with-points}
		\left\{ \left(X, r_1, \ldots, r_{3k+3}\right) \subset \minhilb k k  \times (\bp^{2k-1})^{3k+3} : r_i \in X \neq \emptyset \right\}
	\end{align*}
	with projections
	\begin{equation}
		\nonumber
		\begin{tikzcd} 
		\qquad & \Psi \ar{dl}{\eta_1} \ar {dr}{\eta_2}& \\
		\minhilb k k && (\bp^{2k-1})^{3k+3}.
		\end{tikzcd}\end{equation}
	We show the fiber over the special configuration of points chosen above
	contains an isolated point. 	
	This will complete the proof because, 
	by \autoref{lemma:isolated-fiber-implies-dominant}, the map $\eta_2$ will be
	is dominant, and so this interpolation problem is satisfied.
	Here, we are using \autoref{proposition:irreducible-incidence} to say that the incidence correspondence is irreducible
	and \autoref{lemma:interpolation-dimension} to say the source and target of $\eta_2$ have the same dimension.

	Define $\scx = \eta_2^{-1} (p_1, \ldots, p_{3i+3})$.
	First, we shall exhibit a point of $\scx$, and then show it is isolated.
	By \autoref{proposition:equivalence-minors-swept-planes}, if 
	$[X] \in \minhilbsmooth k k$, the ideal of $X$ is defined by quadrics, since it is
	generated by the minors of a matrix with two rows. 
	Therefore,
	if a smooth variety of minimal degree contains three collinear points, it must contain the line through them,
	by Bezout's theorem.

	In particular, a smooth variety of minimal degree $X$ corresponds to a point in $\scx$ if and only if it contains
	the lines lines $\ell_1, \ldots, \ell_k$ and the points $p_1, p_2, p_3$.
	So, define
		\begin{align}
		\label{equation:correspondence-with-lines-and-points}
		\Phi := \left(  \uhilb {\ell_1} \times_{\bp^n} \uhilb X \right) \times_{\hilb X} \cdots \times_{\hilb X} \left(  \uhilb {\ell_k} \times_{\bp^n} \uhilb X \right)
			\times_{\hilb X} \uhilb X\times_{\hilb X} \uhilb X\times_{\hilb X} \uhilb X
	\end{align}
	with projections
	\begin{equation}
		\label{equation:projections-segre-incidence-correspondence-with-lines-and-points}
		\begin{tikzcd} 
			\qquad & \Phi \ar{dl}{\pi_1} \ar {dr}{\pi_2}& \\
			\minhilb k k && (G(2, 2k))^k \times (\bp^{2k-1})^3
		\end{tikzcd}\end{equation}

	Set theoretically, we can describe the closed points of $\Phi$ (without specifying a scheme structure)
	as
		\begin{align}
			\nonumber
		\left\{ \left( L_1, \ldots, L_k, r_1, r_2, r_3, X \right) \subset (G(2,2k))^j \times (\bp^{2k-1})^3 \times \minhilb k k : L_i \subset X, r_i \in X  \right\}.
	\end{align}
	Let $\scy := \pi_2^{-1}(\ell_1, \ldots, \ell_k, p_1, p_2, p_3)$.

	To complete the proof, it suffices to exhibit an isolated point of $\scy$, because locus of smooth schemes
	corresponding to points of $\scy$
	agrees with the locus of smooth schemes corresponding to points in $\scx$ as subschemes of $\minhilb k k$.
	\begin{remark}
		\label{remark:}
		It is worth observing that $\Phi$ is reducible, and even has multiple connected components.
		However, the set theoretic locus of smooth schemes which are fibers of $\scy$ under $\pi_1$ agrees with the
		set theoretic locus of smooth schemes which are fibers of $\scx$ under $\eta_1$.
		So, we obtain that a point of $\scx$ corresponding to a smooth scheme will be isolated in
		$\scx$ if that point is an isolated point of $\scy$.
	\end{remark}

	Now, by assumption, $\ell_1, \ldots, \ell_k, p_1, p_2, p_3$ are chosen generally.
	In particular, the lines span $\bp^{2k-1}$. Further, the points were specialized to general position so that,
	by \autoref{lemma:number-of-planes}, there are finitely many $(k-1)$-planes containing $p_i$ and meeting all lines
	$\ell_1, \ldots, \ell_k$.
	In particular, there are only finitely many three tuples of $(k-1)$-planes $(\Lambda_1, \Lambda_2, \Lambda_3)$ so that $\Lambda_i$ contains
	$p_i$ and meets all lines. Additionally, since the points were chosen generally, we may assume that the points of intersection
	$\Lambda_i \cap \ell_j$ are all pairwise distinct.
	Now, by \autoref{lemma:unique-segre-containing-linear-spaces}, there is a unique
	Segre variety $X$ containing $\Lambda_1, \Lambda_2,\Lambda_3$ and $\ell_1, \ldots, \ell_k$.

	This corresponds to a point in $\scy$, hence also one in $\scx$. 
	To complete the proof, it remains to show it corresponds to an isolated point in $\scy$.
	This follows from \autoref{lemma:isolated-segre}. 	
\end{proof}

\begin{lemma}
	\label{lemma:isolated-segre}
	Consider the incidence correspondence
	$\Phi$ defined in
	\eqref{equation:correspondence-with-lines-and-points}
	with projection maps $\pi_1$ to $\minhilb k k$ and $\pi_2$ to
	$(G(2, 2k))^k \times (\bp^{2k-1})^3$, as defined in
	\eqref{equation:projections-segre-incidence-correspondence-with-lines-and-points}.
	Let $\ell_1, \ldots, \ell_k$ be lines in $\bp^{2k-1}$ and let $p_1, p_2,p_3 \in \bp^{2k-1}$ be three points.
		Define $\scy := \pi_2^{-1}(\ell_1, \ldots, \ell_k, p_1, p_2, p_3)$.
		Write $[X] \in \scy$ as
		$X \cong \bp^1 \times \bp^{k-1}$ via
		\autoref{lemma:segre-isomorphic-to-scroll}. 
		If the lines $\ell_i$ are all 
		of type ~\ref{custom:line-small} 
		Then, $[X]$ is an isolated point of $\scy$.
		\end{lemma}
\begin{proof}
	We will assume $k \geq 3$, as the cases $k = 2$ holds since $S_{1,1}$ is a quadric hypersurface.

	First, by \autoref{lemma:number-of-planes} and \autoref{lemma:unique-segre-containing-linear-spaces},
	there are only finitely many points  $[Y] \in \scy$ so that all $\ell_i$ appear as lines of type
	~\ref{custom:line-small}.
	To conclude the proof, it suffices to show that there cannot be a family of scrolls in $\scy$
	containing $X$ so that in the general member of the family, there is some $\ell_i$ which appears as a line
	of type ~\ref{custom:line-large}.
	So, suppose for the sake of contradiction, that such a family $\scz \rightarrow C$, for $C\subset \scy$ and $[X] \in C$, does exist.

	Then, consider the relative Hilbert scheme $\sch := \sch^{t+1}(\scz/C)$ of lines inside $\scz$ over $C$.
	\begin{remark}
		\label{remark:}
		Note, throughout nearly all of the rest of this thesis, we were considering the Hilbert scheme over a ground field,
	but this is a case in which we will need to harness the construction of Hilbert schemes over arbitrary Noetherian bases.
	\end{remark}
	
	We claim that $\sch$ has two connected components one whose lines are of type
	~\ref{custom:line-small} and the other whose lines are of type ~\ref{custom:line-large}.
	In particular, this will suffice to show such a family over $C$ cannot exist,
	as inside this family, we would have a family of lines moving between the two components of $\sch$.

	To see that $\sch$ has the two connected components, we examine the map
	$\sch \ra C$. Since the fiber of $\scz$ over a closed point of $C$ is a scroll,
	we know that the fibers of $\sch$ over any closed point of $C$ has two smooth connected components
	of different dimensions, using
	\autoref{lemma:linear-spaces-in-varieties-of-minimal-degree},
	and our standing assumption that $k > 2$.
	In particular, there are two irreducible components of each of the closed
	fibers of $\sch \rightarrow Y$,
	each smooth with different Hilbert polynomials.
	By \autoref{proposition:distinct-hilbert-polynomials-implies-reducible-source},
	$\sch$ has two components, call them $\sch_1$ and $\sch_2$,
	where the restriction of $\sch_1$ to any closed fiber consists of lines of type
	~\ref{custom:line-small} and the restriction of
	$\sch_2$ to the closed fiber consists of lines of type ~\ref{custom:line-large}.
	So, since the two components of $\sch_1$ and $\sch_2$ do not intersect on any of the closed
	fibers, they are in fact two connected components, completing the proof.
	\end{proof}

\section{Spelling out the induction}
In this section, we combine the previous results from this section to prove that smooth scrolls satisfy interpolation.
We then conclude that all varieties of minimal degree satisfy interpolation, and hence also strong interpolation in characteristic
$0$.

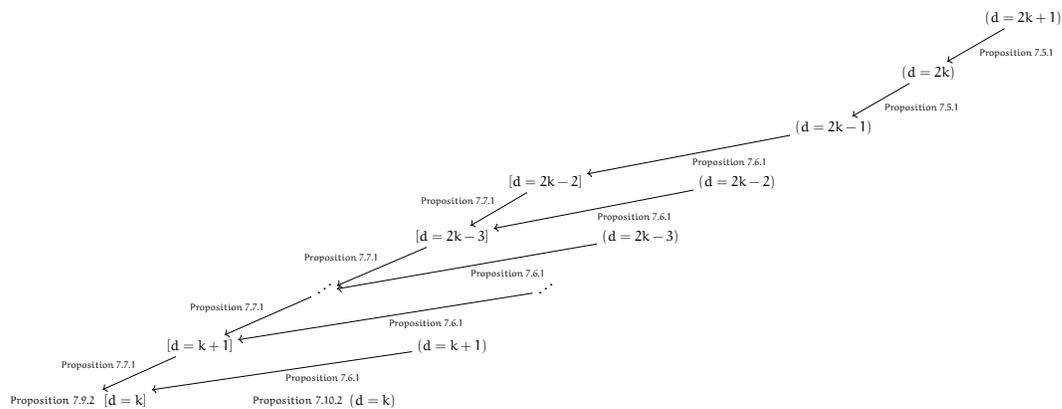
\begin{figure}
	\centering
\begin{equation}
  \nonumber
\begin{tikzpicture}[baseline= (a).base]
\node[scale=.45] (a) at (0,0){
  \begin{tikzcd}[column sep=tiny]
    \qquad && && && && && & (d = 2k+1) \ar{dll}{\begin{NoHyper}\autoref{lemma:high-induction}\end{NoHyper}} \\ 
    \qquad && && && && & (d = 2k) \ar{dll}{\begin{NoHyper}\autoref{lemma:high-induction}\end{NoHyper}} && \\
    \qquad && && && & (d = 2k-1) \ar{dlll}{\begin{NoHyper}\autoref{corollary:mid-induction}\end{NoHyper}} &&  && \\
    \qquad && && \left[d = 2k-2\right] \ar{dl}[swap]{\begin{NoHyper}\autoref{lemma:middle-induction}\end{NoHyper}} && (d=2k-2) \ar{dlll}{\begin{NoHyper}\autoref{corollary:mid-induction}\end{NoHyper}}&  &&  && \\
    \qquad && & \left[d = 2k - 3\right] \ar{dl}[swap]{\begin{NoHyper}\autoref{lemma:middle-induction}\end{NoHyper}} && (d=2k-3) \ar{dlll}{\begin{NoHyper}\autoref{corollary:mid-induction}\end{NoHyper}} &&  &&  && \\
    \qquad && \iddots \ar{dl}[swap]{\begin{NoHyper}\autoref{lemma:middle-induction}\end{NoHyper}}&& \iddots \ar{dlll}{\begin{NoHyper}\autoref{corollary:mid-induction}\end{NoHyper}} &  &&  &&  && \\
    \qquad & \left[d = k+1\right] \ar{dl}[swap]{\begin{NoHyper}\autoref{lemma:middle-induction}\end{NoHyper}} && (d = k + 1) \ar{dlll}{\begin{NoHyper}\autoref{corollary:mid-induction}\end{NoHyper}} && &&  &&  && \\
    \substack{\begin{NoHyper}\autoref{proposition:segre-plane}\end{NoHyper}}\hspace{.2cm}\left[d = k\right]  && \substack{\begin{NoHyper}\autoref{proposition:segre-lines-interpolation}\end{NoHyper}}\hspace{.2cm}(d = k)  & && &&  &&  &&
  \end{tikzcd}
};
\end{tikzpicture}
\end{equation}
\caption{
This is a schematic diagram for the proof that scrolls satisfy interpolation.
The parenthesized expressions $(d = a)$ indicate that
scrolls of degree $a$ satisfy interpolation,
while the bracketed expressions $[d = a]$ indicate that scrolls of degree $a$
satisfy the hypothesis \ref{custom:ind-mid}.
See \autoref{theorem:scrolls-interpolation} for a written proof.
The arrows point from higher degree to lower degree, and are labeled
by the proposition showing that interpolation for the lower degree
variety implies interpolation for the higher degree variety.
}
	\label{figure:induction-schematic}
\end{figure}

\begin{theorem}
	\label{theorem:scrolls-interpolation}
	If $X$ is a rational normal scroll of minimal degree
	then $\hilb X$ satisfies interpolation.
\end{theorem}
\begin{proof}
	First, if $X = \scroll {1^k}$ then this holds by
	\autoref{proposition:segre-lines-interpolation} with
	$j = 0$. 

	Second, suppose $X = \scroll {2^t, 1^{k-t}}$ with
	$1 \leq t \leq k-1$. We will show $X$ satisfies interpolation.	By \autoref{proposition:segre-plane}, inductive hypothesis ~\ref{custom:ind-mid}
	holds for $d = k+1$. By induction, assume it holds for a given degree $d-1$.
	Then, By \autoref{lemma:middle-induction}, it holds for degree $d$, provided
	$d \leq 2k-2$. Finally, we obtain that varieties of degree $t + k$
	satisfy interpolation by \autoref{corollary:mid-induction}.

	To complete the proof, it suffices to show that $X$ satisfies interpolation
	when $d > 2k - 1$. We have shown this when $d = 2k -1$. Inductively assume that
	inductive hypothesis ~\ref{custom:ind-high} holds for degree $d-1$. Then,
	by \autoref{lemma:high-induction}, it also holds for degree
	$d$. Therefore, varieties of degree $d$ satisfy interpolation.
\end{proof}

\begin{remark}
	\label{remark:}
	The question of interpolation for scrolls can be rephrased in
	an interesting alternative fashion as a question of 
	whether curves in the Grassmannian meet Schubert cycles.

	Using \autoref{proposition:equivalence-Grassmannian-swept-planes},
	every scroll in $\minhilb d k$ can be viewed as a rational curve in the Grassmannian
	$G(k, d+k).$
	The condition that the scroll meet a point is the same as the condition that
	the rational curve intersect a Schubert cell.
	Therefore, the question of interpolation can be rephrased as
	whether there is a smooth rational curve in the Grassmannian
	meeting an appropriate collection of Schubert cells.
	\autoref{theorem:scrolls-interpolation}, combined with the equivalence
	of interpolation and strong interpolation, implies that whenever
	we ``expect'' to have such a rational curve, we will indeed have one.
	That is, if the sum of the codimensions of the Schubert cells is equal
	to the dimension of the Hilbert scheme of scrolls, then we will have
	a rational curve in the Grassmannian meeting those Schubert cells.

	The equivalence between scrolls and rational curves in the Grassmannian is
	described in more detail when
	$k = 2$ in \cite[Section 3]{coskun:degenerations-of-surface-scrolls}.
	Further, in the case $k = 2$, Coskun is able to use his algorithm
	for computing the number of surface scrolls with certain
	incident conditions to calculate the Gromov-Witten invariants of
	the Grassmannian \cite[Section 9]{coskun:degenerations-of-surface-scrolls}.
	We ask whether a similar algorithm exists in \autoref{question:number-of-scrolls}.
\end{remark}

We can now prove our main theorem.

\minimalInterpolation*
\begin{proof}
	First, by \autoref{theorem:classification-of-varieties-of-minimal-degree},
	we only need show that quadric surfaces, 
	scrolls, and the $2$-Veronese embedding $\bp^2 \ra \bp^5$ satisfy
	interpolation.

	First, quadric surfaces satisfy interpolation by
	\autoref{lemma:balanced-complete-intersection}.
	Second, the $2$-Veronese embedding satisfies interpolation by
	\autoref{theorem:counting-veronese-interpolation}.
	Finally, by \autoref{theorem:scrolls-interpolation}
	if $X$ is a scroll then $\hilb X$ satisfies interpolation.
\end{proof}

\begin{corollary}
	\label{corollary:strong-interpolation-for-varieties-of-minimal-degree}
	If $\bk$ has characteristic $0$, then smooth varieties of minimal degree
	over $\bk$ satisfy strong interpolation.
\end{corollary}
\begin{proof}
	This follows from \autoref{theorem:interpolation-minimal-surfaces}
	and the equivalence of interpolation and strong interpolation
	in characteristic $0$, as proven in
	\autoref{theorem:equivalent-conditions-of-interpolation}.
\end{proof}

\section{Interpolation of $\minhilb 3 3$ through lines}

In this section, we outline, in a series of exercises,
that
there is a 1 dimensional family of degree 3,
dimension 3 scroll passing through 4 general lines and one general point.
Toward the end of the section, we also pose some related open questions.

For the remainder of this section, let $L_1, L_2, L_3, L_4$ be four lines in $\bp^5$, and let $p, q$ be two points in $\bp^5$.

\begin{exercise}
	\label{exercise:condition-counting-for-3-scrolls}
	Show that given four general lines and two general points, there are at most finitely many degree 3, dimension 3 scrolls
	in $\minhilb 3 3$ containing all four lines and both points.
	{\it Hint:} Observe that a general such scroll contains a three dimensional family of lines
	(a two dimensional family in each plane of the ruling, and a 1 dimensional family of rulings) while
	there is a $8$ dimensional family of lines in $\bp^5$. Therefore, containing a line ``imposes $8 - 3 = 5$
	conditions.`` Similarly, show that containing a point ''imposes $5 - 3 = 2$ conditions.`` Since
	there is a $\dim \minhilb 3 3= 24 = (3+3)^2 - 3^2 - 3$, and we also have $24 = 4 \cdot 5 + 2 \cdot 2$, corresponding
	to containing 4 lines and 2 points, conclude, using ~\ref{technique-conditions} that there
	are at most finitely many scrolls in $\minhilb 3 3$ containing all four lines and both points.
	In particular, there is at most a two dimensional family of scrolls containing $p, L_1, \ldots, L_4$.
\end{exercise}

\begin{exercise}
	\label{exercise:reduction-to-plane-for-3-scrolls}
	Reduce the problem of finding a one dimensional family of scrolls in $\minhilb 3 3$ through $p, L_1, \ldots, L_4$ to the problem
	of finding a 1 dimensional family of scrolls containing a two plane $P$ and three general lines,
	where the lines appear as type ~\ref{custom:line-small} on each scroll in this one dimensional family.
	{\it Hint:} Assume that $L_4$ and $p$ lie in the same ruling plane of the scroll.
\end{exercise}

\begin{exercise}
	\label{exercise:reduction-to-line-isomorphisms-for-3-scrolls}
	Let $M_1, M_2, M_3$ be three lines, each with three labeled points $p_1^i, p_2^i, p_3^i$ on $M_i$.
	Reduce the problem of finding a 1 dimensional family of scrolls in $\minhilb 3 3$ through $p, L_1, \ldots, L_4$
	to the problem of finding a one dimensional family of triples of isomorphisms $(\phi_1, \phi_2, \phi_3)$, with $\phi_i : M_i \ra \bp^1$
	so that 
	\begin{align*}
		\phi_1(p_1^1) &= \phi_2(p_1^2) = \phi_3(p_1^3),\\
		\phi_1(p_2^1) &= \phi_2(p_2^2), \\
		\phi_1(p_3^1) &= \phi_3(p_2^3), \\
		\phi_2(p_3^2) &= \phi_3(p_3^3).
	\end{align*}
	{\it Hint:} First, use \autoref{exercise:reduction-to-plane-for-3-scrolls}. 
	Show that there are finitely many lines meeting a two plane and two other general lines. Take $M_i$ to be a line
	meeting the two plane $P$, containing $p, L_4$, and meeting the two lines $L_j$ with $j \in \left\{ 1,2,3 \right\}\setminus i$.
	Show that any set of three isomorphisms $\phi_i: M_i \ra \bp^1$ sending all three points $\phi_i(M_i \cap P)$ to the same
	point in $\bp^1$ and satisfying $\phi_i(L_j \cap M_i) = \phi_i(L_k \cap M_i)$ for $\left\{ i,j,k \right\} = \left\{ 1,2,3 \right\}$
	uniquely determines a scroll in $\minhilb 3 3$ containing $P, L_1, L_2, L_3$.
\end{exercise}

\begin{exercise}
	\label{exercise:}
	Let $M_1, M_2, M_3$ be three lines.
	Show there is a one dimensional family of triples of isomorphisms $(\phi_1, \phi_2, \phi_3)$, with $\phi_i : M_i \ra \bp^1$
	satisfying the conditions given in \autoref{exercise:reduction-to-line-isomorphisms-for-3-scrolls}.
	{\it Hint:} Write the maps out as linear fractional transformations in coordinates. It may help to assume, without loss of
	generality, that $\phi^{-1}(M_i \cap P)$ is the
	point at infinity.
\end{exercise}

\begin{exercise}
	\label{exercise:}
	Piece together the above exercises to conclude that there are at least four one dimensional families of scrolls through a point and four lines.
	{\it Hint:} To get all four, take the plane $P$ spanned by $L_i$ and $p$ in \autoref{exercise:reduction-to-plane-for-3-scrolls},
	as $1 \leq i \leq 4$.
\end{exercise}

We have just shown there are at least four one dimensional families of
scrolls in $\minhilb 3 3$ through one general point and four lines.
It is natural to ask whether there will be one through four lines and two points, as is ``expected.''
This is an open question.

\begin{question}
	\label{question:}
	Is a scroll in $\minhilb 3 3$ containing four general lines and two general points?
\end{question}

For dimensional reasons, if there is a scroll through $4$ general lines and two general points,
these lines must be of type ~\ref{custom:line-large}. It is reasonable to ask whether 
we can interpolate scrolls through four the lines of type
~\ref{custom:line-small}, but, through a smaller number of points. 
While I do not know whether it is possible to have three be of type
~\ref{custom:line-large} and one of type ~\ref{custom:line-small}, or all four of type ~\ref{custom:line-small}, \autoref{exercise:scrolls-through-vertical-lines},
addresses the other two cases.
\begin{exercise}
	\label{exercise:scrolls-through-vertical-lines}
	Show that there is no scroll in $\minhilb 3 3$ that contains $L_1, \ldots, L_4$ so that either two or three
	of $L_1, \ldots, L_4$ appear as a line of type ~\ref{custom:line-small} on the scroll.
	{\it Hint:} Assume that $L_1, L_2$ are of type ~\ref{custom:line-small} while $L_3$ is of type
	~\ref{custom:line-large}. Show that $L_1, L_2, L_3$ are necessarily contained in a hyperplane in $\bp^5$,
	and so cannot be general.
\end{exercise}

But, \autoref{exercise:scrolls-through-vertical-lines} raises the following question, which is still open.

\begin{question}
	\label{question:}
	Is there a scroll in $\minhilb 3 3$ containing four general lines $L_1, \ldots, L_4$ so that at least three of them
	are of type ~\ref{custom:line-small}?
\end{question}

\chapter{Castelnuovo curves}
\label{section:castelnuovo-curves}

Our main aim in this chapter is to show
that a Castelnuovo curve 
in $\bp^r$ of degree at least $2r + 1$ does not satisfy interpolation.

Combined with previous results of interpolation for canonical curves and nonspecial curves,
this yields a nearly complete picture of interpolation for Castelnuovo curves,
as described in \autoref{theorem:castelnuovo-interpolation}.
The question of whether canonical curves satisfy interpolation
is open, although they are known to satisfy weak interpolation
in genus not equal to $4$ or $6$. See 
\cite[p.\ 108]{stevens:deformations-of-singularities}.
For more on this open question, see
\autoref{question:interpolation-for-canonical-curves}.

\section{Background on Castelnuovo curves}
We start by recalling the definition of Castelnuovo curves,
their basic properties, and the dimension of the irreducible component
of the Hilbert scheme containing a given smooth Castelnuovo curve.

\begin{definition}
	\label{definition:}
	Suppose a curve $C \subset \bp^r$ has degree $d$, which can be written in the form
	$d= m(r-1) + \varepsilon + 1$, where $0 \leq \varepsilon \leq r - 2$.
	Define
	\begin{align*}
		\pi(d, r) := \binom{m}{2}(r-1) + m \varepsilon.
	\end{align*}
	Then, $C$ is a {\bf Castelnuovo curve} if its arithmetic genus is equal to $\pi(d,r)$.
\end{definition}

\begin{remark}
	\label{remark:}
	In fact, if $C \subset \bp^r$ has degree $d$, then
	its genus is at most $\pi(d,r)$. In this way, Castelnuovo curves
	are curves of the biggest possible genus, given their degree and ambient projective space,
	as is shown in \cite[Section III.2, Castelnuovo's Bound]{ACGH:I}.

	If $d = 2r$, then $C$ is necessarily a canonical curve, while if $d < 2r$, $C$ is projectively normal
	and nonspecial.
	In the case that $d > 2r$, every Castelnuovo curve lies on some surface of minimal degree in $\bp^r$,
	by \cite[Section III.2, Theorem 2.5]{ACGH:I}.
\end{remark}

Recall that every rational normal surface scroll $X$ is a Hirzebruch surface. 
The Picard group of such a surface is generated by the class of the fiber of the projection $\pi: X \ra \bp^1$,
which we call $f$, and the class of a hyperplane section, which we call $h$.
This description of the Picard group follows from \cite[Exercise 20.2.L]{vakil:foundations-of-algebraic-geometry},
since the classes $(h, f)$ are related to the generators of the Picard group detailed in 
\cite[Exercise 20.2.L]{vakil:foundations-of-algebraic-geometry} by a change of basis.
We can now state the main theorem describing the different components of the Hilbert scheme
whose general member is a smooth Castelnuovo curve of degree at least $2r + 1$ in $\bp^r$.

\begin{theorem}[\protect{\cite[Theorem, p.\ 351, and Theorem 1.4]{ciliberto:on-the-hilbert-scheme-of-curves-of-maximal-genus-in-a-projective-space}}]
	\label{theorem:castelnuovo-hilbert-scheme-dimension}
	Let $C \subset \bp^r$ be a smooth Castelnuovo curve of degree $d$ and genus $g$, with $d \geq 2r + 1$.
	Choose $\varepsilon, m$ so that $d = m(r-1) + \varepsilon + 1$ and $g = \binom{m}{2}(r-1) + m \varepsilon$.
	Then, $C$ is a smooth point of the Hilbert scheme.
	Further, $\dim \hilb C$ is
	\begin{align*}
		\begin{cases}
		m(m+3)+10 &\text{if } r = 3,\\
		27 + \frac{a(a+3)}{2} & \text{if $C$ lies on the Veronese surface} \\
	& \text{or a degeneration of a Veronese } \\
	& \text{surface in }\bp^5 \text{ and } r = 5, \\
	& a := d/2, \varepsilon \in \left\{ 1,3 \right\},   \\
	(r-1)\left( \binom{m+1}{2}+r+2 \right)+2m & \text{if $C$ is a curve of class } \\
	& m h - (r-2-\varepsilon)f \text{ on a } \\
	& \text{rational normal surface scroll } \\
	& \text{and }\varepsilon = 0, r \geq 4, \\
	(r-1)\left( \binom{m+1}{2}+r+2 \right)+2(m-1) & \text{if $C$ is a curve of class $m h + f$} \\
	& \text{on a rational normal surface } \\
	& \text{scroll and} \varepsilon = 0, r \geq 4, \\
		(r-1)\left( \binom{m+1}{2}+r+2 \right) + (\varepsilon+2)(m+2)-4 & \text{otherwise.}
	\end{cases}
	\end{align*}
\end{theorem}

\begin{warn}
	\label{warning:}
	Even though \autoref{theorem:castelnuovo-hilbert-scheme-dimension}
	guarantees that Castelnuovo curves $C \subset \bp^r$ are smooth points of the Hilbert scheme,
	this does not mean that $H^1(C, N_{C/\bp^r}) = 0$.
	In fact, $h^1(C, N_{C/\bp^r})$ is given in 
	\cite[Proposition 2.4]{ciliberto:on-the-hilbert-scheme-of-curves-of-maximal-genus-in-a-projective-space}
	and often takes on nonzero values.
\end{warn}

\begin{remark}
	\label{remark:}
	While proving that smooth Castelnuovo curves are smooth points of the Hilbert scheme is fairly tricky,
	the idea for computing the dimensions of the components of the Hilbert scheme given in 
	\autoref{theorem:castelnuovo-hilbert-scheme-dimension} is not so bad.

	We briefly summarize the idea of how to compute the dimension of $\hilb C$,
	in the last case that $\varepsilon > 0$ and $r > 5$ in a series of exercises. The other cases can be computed similarly.

	\begin{exercise}
		\label{exercise:}
		Explain why $\dim \hilb C = h^0(C, N_{C/\bp^r})$
	{\it Hint:} You may assume that $C$ is a smooth point of the Hilbert scheme,
	by \autoref{theorem:castelnuovo-hilbert-scheme-dimension}.
	\end{exercise}

	\begin{exercise}
		\label{exercise:}
		Show that the dimension of $\hilb C$ is equal to the sum of the dimension
		of scrolls $X \subset \bp^r$ and the dimension of curves on a fixed scroll.
		{\it Hint:} Invoke the exact sequence
	\begin{equation}
		\nonumber
		\begin{tikzcd}
			0 \ar {r} &  N_{C/X} \ar {r} & N_{C/\bp^r} \ar {r} & N_{X/\bp^r}|_C \ar {r} & 0 
		\end{tikzcd}\end{equation}
	Show this is exact on global sections because $H^1(C, N_{C/X}) = 0$ (as is mentioned in
	\cite[Lemma 1.5(iii)]{ciliberto:on-the-hilbert-scheme-of-curves-of-maximal-genus-in-a-projective-space}).
	Use this to relate $h^0(C, N_{C/\bp^r})$ to $h^0(C, N_{X/\bp^r}|_C)$ and $h^0(C, N_{C/X})$.
	\end{exercise}
	
	\begin{exercise}
		\label{exercise:}
		Using that we know the dimension of scrolls in projective space, by
	\autoref{proposition:scroll-hilbert-scheme-dimension},
	reduce the problem to computing
	$h^0(C, N_{C/X})$.
	\end{exercise}

	\begin{exercise}
		\label{exercise:}
		Reduce to the case of computing $h^0(\sco_C(C))$.
		{\it Hint:} Use the exact sequence
	\begin{equation}
		\nonumber
		\begin{tikzcd}
			0 \ar {r} &  \sco_X \ar {r} & \sco_X(C) \ar {r} & \sco_C(C) \ar {r} & 0.
		\end{tikzcd}\end{equation}
	Show this sequence is exact on global sections and then use that $h^0(X, \sco_X) = 1$.
	\end{exercise}
	
	\begin{exercise}
		\label{exercise:}
		Compute $h^0(C, \sco_C(C))$ using that $h^0(\sco_C(C))$ is the self intersection of $C$ with itself
	on the surface $X$, as follows from standard intersection theory.
	{\it Hint:} To compute $h^0(C, \sco_C(C))$, use that $C$ has class
	$(m+1)h - (r-2-\varepsilon) f$. It may help to use the genus formula given in
	\cite[Subsection 2.4.1]{Eisenbud:3264-&-all-that}.
	For this approach, it may also help to use the result from 
	\cite[p.\ 3]{coskun:degenerations-of-surface-scrolls}, that
	\begin{align*}
		K_{\mathbb F_r} = -2h + (r-1)f.
	\end{align*}
	\end{exercise}

	\begin{exercise}
		\label{exercise:}
		Piece the above exercises together to compute the dimension of $\hilb C$.
	\end{exercise}
	
	\begin{exercise}
		\label{exercise:}
		Generalize the method in the above sequence of exercises to also deal with the cases when
		$\varepsilon \neq 0$ and the cases when $r \leq 5$.
	\end{exercise}
\end{remark}

\section{When Castelnuovo curves satisfy interpolation}

In this section, we show that Castelnuovo curves with degree more than $2r$ in
$\bp^r$ do not satisfy interpolation.
Combining this with previous results yields a nearly complete
description of which Castelnuovo curves satisfy interpolation,
as given in \autoref{theorem:castelnuovo-interpolation}.

\begin{proposition}
	\label{proposition:castelnuovo-curves-fail-interpolation}
	The Hilbert scheme of Castelnuovo curves $C \subset \bp^r$
	do not satisfy weak interpolation (and hence
	also do not satisfy interpolation or strong interpolation) when the degree of $C$ is at least
	$2r+1$.
\end{proposition}
\begin{proof}
	Using \autoref{theorem:castelnuovo-hilbert-scheme-dimension},
	if $C$ does not lie on a Veronese surface or degeneration thereof and
	$r > 3$, then
	the condition that $C$ satisfy interpolation
	interpolation can be expressed as the conditions that $C$
	pass through
	$\binom{m+1}{2} + r + 2$ general points, and that $C$ meet some additional general plane
	of codimension at least 3.
	So, if we can show that there are no curves $C$ passing through general $\binom{m+1}{2}+r+2$
	points and meeting a general plane of codimension 3, such curves will not satisfy interpolation.
	Since we are assuming $C$ has degree at least $2r + 1$,
	it will lie on some variety of minimal degree, and so it suffices to show there is no variety
	of minimal degree containing $\binom{m+1}{2} +r +2$ points
	and meeting a codimension $3$ plane.
	But now, since we are assuming $r > 3$,
	(and also that $m > 2$, since $d > 2r$,)
	we obtain that $\binom{m+1}{2}+r+2 \geq r + 5$.

	However, there are no surfaces of degree $r - 1$ passing through
	$r + 5$ points and meeting a codimension $3$ plane,
	as follows from \autoref{lemma:interpolation-numerics}.

	To complete the proof, it suffices to show Castelnuovo curves in $\bp^3$ or 
	lying on the Veronese surface in $\bp^5$ do not satisfy interpolation.

	First, if the curve lies on the Veronese surface or degeneration thereof, its degree must be at least $12$.
	We conclude that it must pass through at least $\frac{27 + \frac{12}{4} (\frac{12}{2}+3)}{4} = 12$
	points. But, there are no Veronese surfaces through $12$ or more general points,
	as there are only four (that is, finitely many) Veronese surfaces through $9$ points,
	by \autoref{theorem:counting-veronese-interpolation}.

	To conclude, we deal with the case of Castelnuovo curves in $\bp^3$.
	In this case, since $m \geq 2$, we obtain that the dimension of the Hilbert
	scheme is at least $20$. Therefore, if a Castelnuovo curve satisfied interpolation,
	it would need to pass through at least $10$ general points. But, it lies on a quadric surfaces,
	and there will not be any quadric surface containing $10$ general points.
\end{proof}

\castelnuovo*
\begin{proof}
	In the case $d \geq 2r+1$, this follows from \autoref{proposition:castelnuovo-curves-fail-interpolation}.
	For the case $d < 2r$, this follows from
	\cite[Theorem 1.3]{atanasovLY:interpolation-for-normal-bundles-of-general-curves}.
	as such a curve is nonspecial.
	Finally, in the case $d = 2r$, a Castelnuovo curve is a canonical curve,
	and so the statement follows from \cite[p.\ 108]{stevens:deformations-of-singularities}.
\end{proof}

\chapter{Further questions}
\label{section:further-questions}

Interpolation is a burgeoning field with oodles of open problems.
Here, we state just a few of those not already mentioned.

\section{Questions on Veronese interpolation}

We begin with some questions related to interpolation of
Veronese varieties.
We have seen the very beginnings of interpolation
for Veronese embeddings. That is, as discussed in \ref{sssec:castelnuovos-lemma},
all rational normal curves, which are the Veronese embeddings
of $\bp^1$, satisfy interpolation. In general, interpolation
of the $r$-Veronese embedding of $\bp^n$ which is the image of
$\bp^n \rightarrow \bp^{\binom{n+r}{n}-1}$ is equivalent
to the question of whether the Veronese surface
passes through $\binom{n+r}{r} + n +1$ points.
Unlike the del Pezzo surfaces and rational normal scrolls,
Veronese embeddings are a class of varieties for
which interpolation only imposes point conditions, and not
an additional linear space condition. Perhaps
this coincidence may be helpful in finding the solution to
the following question.

\begin{question}
	\label{question:veronese-interpolation}
	Does the image of the $r$-Veronese embedding $\bp^n \rightarrow \bp^{\binom{n+r}{r}-1}$ satisfy interpolation?
	That is, is there a Veronese embedding containing
	$\binom{n+r}{r} + n + 1$ general points in $\bp^{\binom{n+r}{r}-1}$?
\end{question}

If the answer to ~\autoref{question:veronese-interpolation}
is affirmative, it would be very interesting to know how many
Veronese varieties pass through the correct number of points.
From \ref{sssec:castelnuovos-lemma}
we know there is precisely one $r$-Veronese $\bp^1$ through 
$\binom{r+1}{1} + 1 +1  = r+3$ points in $\bp^r$.
Additionally,  \autoref{theorem:counting-veronese-interpolation}
tells us there are $4$ $2$-Veronese surfaces in through $9$ general
points in $\bp^5$. In the appendix, we show that there are at least $630$ $3$-Veronese surfaces through $13$ general points in $\bp^{9}$. See \autoref{remark:630}.

\begin{question}
	\label{question:veronese-enumeration}
	How many $r$-Veronese varieties of dimension
	$n$ pass through
	$\binom{n+r}{r} + n + 1$ general points in $\bp^{\binom{n+r}{r}-1}$?
\end{question}

We have also seen in Coble's work 
\cite{coble:associated-sets-of-points}
that any two $2$-Veronese
surfaces through $9$ general points in $\bp^5$
intersect along a genus 1 curve through those $9$ points.
This leads to the question:
\begin{question}
	\label{question:intersect-curve}
	Suppose there are at least two $r$-Veronese varieties
	of dimension $n$ passing through
	$\binom{n+r}{r} + n + 1$ general points in $\bp^{\binom{n+r}{r}-1}$.
	Do they have positive dimensional intersection?
\end{question}

\section{Interpolation of canonical curves}

We now engage in a brief discussion of interpolation for canonical curves.
Note that
by \autoref{theorem:castelnuovo-interpolation}
once we understand whether canonical curves satisfy interpolation,
we will have completed the description of which Castelnuovo curves satisfy interpolation.
While ~\cite{stevens:on-the-number-of-points-determining-a-canonical-curve}
showed that canonical curves of genus not equal to $4$ or $6$ satisfy weak interpolation, whether they satisfy
interpolation is still open.
This has particular relevance for the slope conjecture, as described in
\autoref{subsection:relevance-of-interpolation}.

\begin{question}
	\label{question:interpolation-for-canonical-curves}
	Do those canonical curves which satisfy weak interpolation also satisfy interpolation?
\end{question}

\section{More on interpolation of scrolls}

Next, we move onto some open questions regarding
interpolation of higher dimensional varieties.
In \autoref{theorem:interpolation-minimal-surfaces}, we proved
that smooth varieties of minimal degree satisfy interpolation.
In the case of varieties of minimal degree of dimension 1, which
are all rational normal curves, we know that
there is a unique rational normal curve through $n+3$ points in $\bp^n$.
Similarly, the number of surface scrolls through the expected number of points
and one linear subspace of the appropriate dimension
is computed to be $(n-2)(n-3)$ in
~\cite[Example, p.\ 2]{coskun:degenerations-of-surface-scrolls}.
\begin{question}
	\label{question:number-of-scrolls-satisfying-point-interpolation}
	How many scrolls in $\minhilb d k$ meet the collections of linear spaces given in \autoref{lemma:interpolation-numerics}?
\end{question}

Coskun also gives a highly efficient algorithm for computing how many surface scrolls meet a collection
of linear subspaces of the appropriate codimensions ~\cite{coskun:degenerations-of-surface-scrolls}.

\begin{question}
	\label{question:number-of-scrolls}
	Let $\Lambda_1, \ldots, \Lambda_m$ be general planes in $\bp^n$.
	Suppose $X$ is a scroll of dimension $k$ in $\bp^n$ and
	$\sum_{i=1}^m (n - k + \dim \Lambda_i) = \dim \hilb X$.
	Assuming the base field $\bk$ has characteristic $0$, 
	we find there are a nonzero finite number of scrolls in $\hilb X$
	passing through them, by
	\autoref{corollary:strong-interpolation-for-varieties-of-minimal-degree}.
	Is there an efficient algorithm, perhaps similar to that in
	~\cite{coskun:degenerations-of-surface-scrolls}
	which can be used to compute this number?
\end{question}

It seems likely that any recursive formula to
answer \autoref{question:number-of-scrolls}
will be quite nasty, involving many terms.
Further, it appears that even the answer to \autoref{question:number-of-scrolls-satisfying-point-interpolation}
is likely to be much messier than $1$ in the case of curves and $(n-2)(n-3)$ in the case of surfaces.
Nevertheless, the answer in low degree may not be so bad.
Although we have almost solely been considering smooth scrolls, scrolls with degree less than their dimension
do exist as singular varieties which are the locus where a two row matrix of linear forms has rank $1$.
It seems very likely that these scrolls satisfy interpolation, as might be proved using
degenerations similar to those in ~\cite{coskun:degenerations-of-surface-scrolls}.
That said, such a proof would likely be rather tedious and involved to write out, and so we pose it as a question.

\begin{question}
	\label{question:low-degree-scroll-interpolation}
	Do scrolls of dimension $k$ and degree $d$, with $d < k$, satisfy interpolation?
\end{question}

It is fairly immediate that when $d = 2$, such scrolls do satisfy interpolation, and we can even count the number,
which we now do in \autoref{example:number-of-singular-quadrics}.
\begin{example}
	\label{example:number-of-singular-quadrics}
	The simplest case of \autoref{question:low-degree-scroll-interpolation}
	is when the degree $d = 2$. In this case, all scrolls of dimension $k$ are simply rank $4$ quadric hypersurfaces.
	Then, we are asking for the number of such rank $4$ quadrics meeting $4k + 1$ general points.
	Inside the projective space of quadrics $\bp H^0(\bp^{k+1}, \sco_{\bp^{k+1}}(2))^\vee$,
	the number of such quadrics is the degree of the intersection of the locus of such quadrics with a
	general codimension $4k + 1$ linear subspace.
	This is simply the degree of the locus of rank $4$ quadrics in the vector space of all quadrics.
	In particular, it is positive, implying quadric scrolls satisfy interpolation.
	Using \cite[Proposition 12(b)]{harrisT:on-symmetric-and-skew-symmetric-determinantal-varieties},
	the degree of this variety is
	\begin{align*}
		\prod_{\alpha = 0}^{k-3} \frac{\binom{k+1 + \alpha}{k - 3 - \alpha}}{\binom{2\alpha + 2}{\alpha}}
	\end{align*}
	Although it takes some algebraic manipulation, when $k > 1$, the above expression simplifies to the fairly clean expression
	\begin{align*}
		\frac{(2k-2)!(2k-1)!}{(k-1)!(k-1)!k!(k+1)!}.
	\end{align*}
\end{example}

So, \autoref{example:number-of-singular-quadrics} shows degree $2$ scrolls satisfy interpolation, and even gives
a simple closed form for the number. Of course, this number can also be computed by Gromov-Witten theory,
but this is not a calculation that can be done by hand. The next case to consider is the number of degree $3$ scrolls
meeting points. Note that these scrolls will be singular if their
dimension is more than $2$.
With careful specialization arguments as in ~\cite[Section 5, Example A]{coskun:degenerations-of-surface-scrolls},
it seems that one should be able to calculate the number of degree $3$ scrolls meeting
$3k + 3$ points (and there will be finitely many such scrolls by a dimension count, as passing through $3k+3$ 
points is equivalent to interpolation).

\begin{question}
	\label{question:}
	Is there a simple form for the number of (usually singular) degree $3$ scrolls of dimension $k$ meeting $3k + 3$ general points
	in $\bp^{k+2}$?
\end{question}

\section{Questions on determinantal varieties}

We may also note that all varieties of minimal degree are determinantal.
Determinantal varieties are in many ways a fundamental object of study,
and the fact that varieties of minimal degree satisfy interpolation suggests
the possibility that determinantal varieties may as well.

\begin{question}
	\label{question:}
	Do determinantal varieties satisfy interpolation?
\end{question}

A particularly interesting special case of determinantal varieties is
that of Grassmannians.

\begin{question}
	\label{question:}
	Do Grassmannians under the Pl\"ucker embedding satisfy interpolation?
\end{question}

\section{Questions related to del Pezzo surfaces}

We conclude our survey of open questions with some questions
related to the result shown in the appendix
\autoref{theorem:main}, that del Pezzo surfaces satisfy weak interpolation.

Since weak interpolation is equivalent to interpolation
for del Pezzo surfaces of degrees $3,4,5$, and $9$,
it is immediate from \autoref{theorem:main} that 
del Pezzo surfaces of degree $3, 4, 5$ and $9$ surfaces satisfy interpolation,
while the remaining del Pezzo surfaces satisfy weak interpolation.

\begin{question}
	\label{question:}
	Do all del Pezzo surfaces satisfy strong interpolation?
	If so, how many del Pezzo surfaces meet
	a collection of points and a linear space,
	as given in \autoref{table:del-Pezzo-conditions}?
\end{question}

It was mentioned in the introduction that plane conics
constitute all anticanonically embedded Fano varieties of dimension 1 and
del Pezzo surfaces constitute those of dimension 2.
As we have seen these both satisfy weak interpolation.
Further, there is a complete classification of Fano
varieties in dimension 3 \cite{iskovskikhP:fano-varieties}. 
Unfortunately, it is immediately clear that not all
Fano varieties in dimension more than 3 satisfy interpolation.
A counter example is provided by the complete intersection of
a quadric and cubic hypersurface in $\bp^5$. 
This leads to the following question:
\begin{question}
	\label{question:}
	Which Fano threefolds, embedded by their
	anticanonical sheaf, satisfy weak interpolation?
	Which Fano threefolds, embedded by their anticanonical sheaf,
	satisfy interpolation? Which Fano varieties in dimension more than $3$ satisfy
	interpolation?
\end{question}

In another direction, we may note that surfaces of minimal degree satisfy
interpolation, that is, surfaces of degree $d-1$ in $\bp^d$ satisfy
interpolation by \autoref{theorem:interpolation-minimal-surfaces}.
In the appendix, we show that all smooth surfaces of one more than minimal
degree, which are not projections of surfaces of degree $d$ from $\bp^{d+1}$ 
as described in \cite[Theorem 2.5]{coskun:the-enumerative-geometry-of-del-Pezzo-surfaces-via-degenerations}.
satisfy interpolation. That is, surfaces of degree $d$ in $\bp^d$ satisfy interpolation. 

\begin{question}
	\label{question:}
	Do all smooth surfaces of degree $d$ in $\bp^d$ satisfy
	interpolation? Equivalently, using
	\cite[Theorem 2.5]{coskun:the-enumerative-geometry-of-del-Pezzo-surfaces-via-degenerations}
	\autoref{theorem:main},
			do projections of surfaces of minimal
	degree from a point 
	satisfy interpolation?
\end{question}

While all smooth linearly normal nondegenerate surfaces of degree $d$ in $\bp^d$
satisfy weak interpolation, note that not all surfaces of degree $d+2$
in $\bp^d$ will satisfy interpolation.
This is because the complete
intersection of a quadric and cubic hypersurface in $\bp^4$ does not
satisfy interpolation.
So, in some way, surfaces of $d + 1$ in $\bp^d$
are the turning point between surfaces satisfying interpolation
and surfaces not satisfying interpolation.
This leads naturally to the following question.

\begin{question}
	\label{question:}
	Do surfaces of degree $d + 1$ in $\bp^d$ satisfy interpolation?
\end{question}

From \autoref{theorem:interpolation-minimal-surfaces}, we know that varieties of dimension
$k$ and degree $d$ in $\bp^{d+k-1}$ satisfy interpolation.
In the appendix we see that varieties of degree $2$ and
dimension $2$ (which are nondegenerate and not projections
of varieties of minimal degree)
in
$\bp^{d + 2 - 2} = \bp^d$ satisfy interpolation.
This too  offers an immediate generalization.

\begin{question}
	\label{question:}
	Do varieties of dimension $k$ and degree $d$ in
	$\bp^{d+k-2}$ satisfy interpolation?
\end{question}

\section{General interpolation questions}

In addition to the questions of interpolation for specific varieties,
there are several general questions regarding interpolation,
which seem quite plausible, yet very little is known about them.

It is clear that if a $k$-dimensional variety lies on
a $(k+1)$-dimensional variety, and one cannot pass the
$(k+1)$-dimensional variety through the required number of linear spaces,
then the $k$-dimensional variety cannot satisfy interpolation.
Indeed, in every instance of interpolation we are aware of, 
a $k$-dimensional variety will only fail to satisfy interpolation
if it lies on such a $(k+1)$-dimensional variety.
For example, this is why certain Castelnuovo curves and non-balanced
complete intersections fail to satisfy interpolation.
To this end, we pose the following vague question.

\begin{question}
	\label{question:}
	Are there any instances of a parameter space of $k$-dimensional 
	varieties failing to satisfy
	interpolation, 
	so that there is no family of $(k+1)$-dimensional varieties
	containing all the members of this family
	of $k$-dimensional varieties which causes the
	failure of interpolation?
\end{question}

Continuing in this vein, we ask several questions about the relation
of $k$-dimensional varieties satisfying interpolation to $(k+1)$-dimensional 
varieties satisfying interpolation.
One natural construction sending $k$-dimensional varieties to $(k+1)$-dimensional 
varieties is the construction taking $X$ to the cone over $X$.
\begin{warn}
	\label{warning:}
	When we ask the question about interpolation of cones, as well as several
	later questions in this section,	
	we ask whether a certain closed subscheme of the Hilbert scheme
	satisfies interpolation, which is not necessarily an irreducible component.
	Although, strictly speaking, we have only defined interpolation
	for an irreducible component of the Hilbert scheme, the definition
	of interpolation makes sense for any closed subscheme $\sch$ of the
	Hilbert scheme.
	We say $\sch$ satisfies interpolation if
	$\sch$ parameterizes varieties of dimension
	$k$ in $\bp^n$
	with $\dim \sch := q \cdot (n - k)+ r$, for $0 \leq r < n - k$,
	such that for 
	$q$ general points and a $(n - k - r)$-dimensional plane,
	there is some $[X] \in \sch$ meeting the $q$ points and meeting the plane.
\end{warn}

\begin{example}
	\label{example:}
	If one starts with $r$ points in $\bp^{n-1}$, the cone over the points
will be a union of $r$ lines in $\bp^n$, all meeting at a common vertex.
We now show the locus of the Hilbert scheme corresponding to cones over $r$
points satisfies interpolation.

To see this, note that satisfying interpolation is equivalent to
finding a one dimensional family of cones passing through $r + 1$ points.
It is quite easy to produce such a family. Simply choose two points among
the $r+1$ points, and fit a line $L$ through them. Then, we have a one dimensional
family of cones through the $r+1$ points containing $L$,
as we vary the cone point along $L$.
Once we choose the cone point on $L$, this will determine the cone, as
the cone must consist of $L$ together with
the $r - 1$ lines joining the cone point to the remaining $r-1$ points.
\end{example}

We now ask if we can generalize the above example, which shows that
cones over $0$ dimensional varieties satisfy interpolation, to higher dimensional varieties.
\begin{question}
	\label{question:}
	If $\scg$ is a closed subscheme of the Hilbert scheme satisfying interpolation
	and $\sch$ denotes the closed subscheme of the Hilbert scheme whose members are cones
	over $\scg$, does $\sch$ satisfy interpolation?
\end{question}

A sort of inverse construction to taking a cone is taking a hyperplane section.
Analogously to the case of cones,
we pose the following question.

\begin{question}
	\label{question:}
	Let $\sch$ be a closed subscheme of the Hilbert scheme satisfying interpolation
	and let $\scg$ is the locus in the Hilbert scheme of hyperplane sections of $\sch$.
	Suppose further that all members of $\scg$ have the same Hilbert polynomial, with
	each member of $\scg$ being at least $1$ dimensional.
	If $\sch$ satisfies interpolation, does $\scg$?
	If $\scg$ satisfies interpolation, does $\sch$?
	\end{question}

Finally, at least in the case of smooth curves, we know that when a curve is embedded by
a very high degree invertible sheaf, the corresponding irreducible component of the Hilbert scheme
will satisfy interpolation. We ask whether this holds for higher dimensional curves as well.

\begin{question}
	\label{question:}
	Suppose we have an irreducible component of the Hilbert scheme $\sch$
	corresponding to varieties $[X] \in \sch$ embedded by line bundles $\scl_X$.
	Does there exist a sufficiently high $n$ so that the locus in the Hilbert
	scheme parameterizing the same abstract varieties embedded by $\scl_X^{\otimes n}$
	satisfies interpolation?
\end{question}

\chapter[Appendix]{Appendix: Interpolation of del Pezzo surfaces, with Anand Patel}
\label{section:appendix}

\section{Introduction to the appendix}
The main result of this appendix is that del Pezzo surfaces satisfy weak interpolation, over a field of characteristic $0$.
For the remainder of this chapter, we assume our base field $\bk$
is of characteristic $0$.
In many ways, del Pezzo surfaces are a natural next
class of varieties to look at.

First, as mentioned earlier, varieties of
minimal degree were shown to satisfy interpolation
in \autoref{theorem:interpolation-minimal-surfaces}. Del Pezzo surfaces
are surfaces of degree $d$ in $\bp^d$, one higher
than minimal. Further, all irreducible varieties of
degree $d$ in $\bp^d$ are either del Pezzo surfaces,
projections of surfaces of minimal degree, or cones
over elliptic curves, by \cite[Theorem 2.5]{coskun:the-enumerative-geometry-of-del-Pezzo-surfaces-via-degenerations}. So, del Pezzo constitute all 
linearly normal smooth varieties of
degree $d$ in $\bp^d$.
By analogy, all curves of degree $d-1$ in $\bp^d$
(also one more than minimal degree) have been shown
to satisfy interpolation in \cite[Theorem 1.3]{atanasovLY:interpolation-for-normal-bundles-of-general-curves}.

Second, del Pezzo surfaces constitute all Fano surfaces.
Since the only Fano curve is $\bp^1$, which
satisfies interpolation, del Pezzo surfaces
suggest themselves as a higher dimensional analog from this point of view.

Recall that a del Pezzo surface, embedded in $\bp^n$,
is a surface with ample anticanonical bundle, embedded
by the complete linear system of its
anticanonical bundle. All del Pezzo surfaces have
degree $d$ in $\bp^d$, and
all linearly normal smooth surfaces of degree $d$ in $\bp^d$ are del Pezzo
surfaces by
\cite[Theorem 2.5]{coskun:the-enumerative-geometry-of-del-Pezzo-surfaces-via-degenerations}.
We also know the dimension of the component of the Hilbert
scheme containing a del Pezzo surface from \cite[Lemma 2.3]{coskun:the-enumerative-geometry-of-del-Pezzo-surfaces-via-degenerations}, as given
in \autoref{table:del-Pezzo-conditions}, and that all del Pezzo
surfaces have $H^1(X, N_{X/\bp^n}) = 0$, by \cite[Lemma 5.7]{coskun:the-enumerative-geometry-of-del-Pezzo-surfaces-via-degenerations}.

Assuming the remainder of the appendix, we now restate and prove
the main result of the appendix.

\main*
\begin{proof}
Recall that that there is a unique component of the Hilbert
scheme of del Pezzo surfaces in
degrees $3, 4, 5, 6, 7, 9$, and there are two in degree $8$.
One component in degree $8$, which we call type $0$,
has general member abstractly isomorphic to $\mathbb F_0 \cong \bp^1 \times \bp^1$.
The other component in degree $8$, which we call type $1$,
has general member abstractly isomorphic to $\mathbb F_1$.
The cases of degree $3$ and $4$ surfaces hold by
\autoref{lemma:balanced-complete-intersection}.
The case of degree $5$ del Pezzo surfaces is
\autoref{thm:quinticDPinterpolation}.
The case of degree $6$ del Pezzo surfaces is
\autoref{thm:sexticDPweakinterpolation}.
The case of degree $8$, type $0$
surfaces is ~\autoref{thm:weakinterpolationP1xP1}.
Finally, the three remaining cases of del Pezzo surfaces in
degrees $7,8,9$ are 
\autoref{corollary:degree-8-type-1-interpolation},
\autoref{corollary:degree-7-interpolation},
and
~\autoref{thm:3veroneseexistence},
respectively.
\end{proof}

\begin{table}
\begin{tabular}
	{| c || c | c | c |}
	\hline
	Degree & Dimension & Number of Points & Additional Linear Space Dimension \\
	\hline
	\hline

	$3$ & $19$ & $19$ & None \\	\hline

	$4$ & $26$ & $13$ & None \\	\hline
	
	$5$ & $35$ & $11$ & $1$ \\	\hline
	
	$6$ & $46$ & $11$ & $2$ \\	\hline
	
	$7$ & $59$ & $11$ & $1$ \\	\hline
	
	$8$, type 0 & $74$ & $12$ & $4$ \\	\hline
	
	$8$, type 1 & $74$ & $12$ & $4$ \\	\hline
	
	$9$ & $91$ & $13$ & None \\	\hline
\end{tabular}
\vspace{.5cm}
\caption{Conditions for del Pezzo surfaces to satisfy interpolation.
Type $0$ refers to the component of the Hilbert scheme whose
general member is a degree $8$ del Pezzo surface, isomorphic to $\mathbb F_0$,
(this also includes, in its closure, del Pezzo surfaces abstractly
isomorphic to $\mathbb F_2$,)
while type $1$ refers to those isomorphic to $\mathbb F_1$.
The dimension counts are proven in \cite[Lemma 2.3]{coskun:the-enumerative-geometry-of-del-Pezzo-surfaces-via-degenerations}.
}
\label{table:del-Pezzo-conditions}
\end{table}

The remainder of this appendix is structured as follows:
In \autoref{sec:degree-5}, \autoref{sec:degree-6}, and
\autoref{sec:degree-8-type-0}, we show del Pezzo surfaces of degree $5, 6$, and
degree $8$, type $0$,
respectively,
satisfy weak interpolation. Our approach for
surfaces of degree $5,6$ and the degree $8$, type $0$ del Pezzo
surfaces is to find surfaces
through a collection of points by first finding a curve or threefold
containing the points and then a surface containing the curve or contained
in the threefold.
In \autoref{section:association}, we recall the technique
of association, in preparation for \autoref{sec:degree-7-8-9},
where we use association to prove weak interpolation of
del Pezzo surfaces of degree $7$, degree $8$, type 1, and degree $9$.  The
degree $9$ del Pezzo surface case is by far the most technically challenging
case in this appendix. 
We were led to the approach of association after reading Coble's remarkable
paper ``Associated Sets of Points" \cite{coble:associated-sets-of-points}.

\section{Degree 5 del Pezzos}

\label{sec:degree-5}
\begin{theorem}
	\label{thm:quinticDPinterpolation}
	Quintic del Pezzo surfaces satisfy
	weak interpolation.
\end{theorem}
\begin{proof}
	By \autoref{table:del-Pezzo-conditions}	it suffices to show quintic del Pezzo surfaces pass through $11$ points.
	
	Start by choosing $11$ points.
	Since degree $3$, dimension $3$	scrolls satisfy interpolation,
	by \autoref{theorem:interpolation-minimal-surfaces},
	there is such a scroll through any $12$ general points.
	Equivalently, there is a two dimensional family of scrolls
	through $11$ points, which sweeps out all of $\bp^5$.
	In any scroll in this two dimensional family, we will show
	there is a quintic del Pezzo.

	First, start with a scroll $X$ containing the $11$ points.
	Since $X$ is projectively normal and it's ideal is defined by
	$3$ quadrics, $h^0(X, \sco_X(2)) = 21 - 3 = 18.$
	Therefore, if we let $P$ be a ruling two plane of $X$,
	since $h^0(P, \sco_P(2)) = 6$, there will be an $18 - 6 = 12$
	dimensional space of quadrics on $X$ vanishing on $P$.
	Therefore, there will be a $12 - 11 = 1$ dimensional
	space of quadrics vanishing on $P$ and containing the $11$ points.
	However, the intersection of any such quadric with $X$
	is the union of $P$ and a quintic del Pezzo surface. 
	Therefore, we have produced a two dimensional family of
	quintic del Pezzo surfaces containing the $11$ points.
	\end{proof}

\begin{remark}
	\label{remark:}
	Another way to prove weak interpolation of quintic del Pezzo
	surfaces uses curves instead of three folds.
	Specifically, by \cite[Corollary 6]{stevens:on-the-number-of-points-determining-a-canonical-curve}, every genus $6$ canonical curve passes through
	$11$ general points. Then, since there is a quintic del Pezzo
	surface containing any genus 6 canonical curve,
	as proved in, among other places,
	\cite[5.8]{arbarelloH:canonical-curves-and-quadrics-of-rank-4}.
	Then, because there is a canonical curve containing these points
	and a quintic del Pezzo containing the canonical curve,
	there is a quintic del Pezzo containing these points.
\end{remark}

\section{Degree $6$ del Pezzos}
\label{sec:degree-6}
By \autoref{table:del-Pezzo-conditions}, weak interpolation for sextic del Pezzos amounts to showing that through $11$ general points $\Gamma_{11} \subset \bp^{6}$ there passes a sextic del Pezzo.

\begin{lemma}\label{lem:genus3degree9}
Through $11$ general points $\Gamma_{11} \subset \bp^{6}$ there passes a smooth degree $9$, genus $3$ curve. 
\end{lemma}

\begin{proof}
This is a special case of \autoref{thm:NaskoEricDavid}.
\end{proof}

Starting from a curve $C \subset \bp^{6}$ as in \autoref{lem:genus3degree9}, we can ``build'' a sextic del Pezzo surface containing $C$.

\begin{lemma}\label{lem:p+q+r}
Let $D$ be a general degree $9$ divisor class on a genus $3$ curve $C$. Then there exists a unique degree three effective divisor $P+Q+R$ such that $D \sim 3K_{C} - (P+Q+R)$.
\end{lemma} 

\begin{proof}
In general, if $X$ is a smooth genus $g$ curve, the natural map \[J \colon \sym^{g}C \to \pic^{g}C\] is a birational map.

In our setting, if $D$ is a general degree $9$ divisor class, $3K_C - D$ will be a general degree $3$ divisor class, and therefore can be represented by a unique degree three divisor class $P+Q+R$. Of course, by Riemann-Roch,
every degree three divisor class is effective.
\end{proof}

\begin{lemma}\label{lem:sexticcontainingC}
Let $\Gamma_{11} \subset \bp^{6}$ be general, and let $C$ be general among the degree $9$, genus $3$ curves containing $C$.  Then there is a smooth sextic del Pezzo surface containing $C$.
\end{lemma}

\begin{proof}
Embed $C \subset \bp^{2}$ via its canonical series $|K_{C}|$. The linear system $|3K_{C} - (P + Q + R)|$ on $C$ is cut out by plane cubics passing through the three points $P + Q + R$. Under the generality conditions, we can assume $P,Q,R$ are not collinear in $\bp^{2}$.  

The linear system of plane cubics through three noncollinear points maps $\bp^{2}$ birationally to a smooth sextic del Pezzo surface in $\bp^{6}$.
\end{proof}

\begin{theorem}\label{thm:sexticDPweakinterpolation}
Sextic del Pezzo surfaces satisfy weak interpolation.
\end{theorem}

\begin{proof}
To show a sextic del Pezzo satisfies interpolation,
by \autoref{table:del-Pezzo-conditions}, it suffices to show it passes through $11$ general points.
By ~\autoref{lem:genus3degree9}, there is a degree $9$ genus $3$
curve through $11$ general points in $\bp^6$. By \autoref{lem:sexticcontainingC},
there is a sextic del Pezzo containing a degree $3$ genus $9$ curve in
$\bp^6$.
\end{proof}

\section{The degree $8$, type 0 del Pezzos}
\label{sec:degree-8-type-0}
Next we consider $\bp^{1} \times \bp^{1} \subset \bp^{8}$ embedded by the linear system of $(2,2)$-curves. To prove weak interpolation, by \autoref{table:del-Pezzo-conditions}, we want to show there is such a surfaces passing through $12$ general points $\Gamma_{12} \subset \bp^{8}$.  As in the sextic del Pezzo case, we will again ``build'' a surface starting from a curve.

\begin{lemma}\label{lem:genus2degree10}
Through $12$ general points $\Gamma_{12} \subset \bp^{8}$ there passes a smooth genus $2$ curve of degree $10$.
\end{lemma}

\begin{proof}
This is a special case of \autoref{thm:NaskoEricDavid}.
\end{proof}

\begin{lemma}\label{lem:degree10divisors}
A general degree $5$ divisor class $D$ on a smooth genus $2$ curve may be written uniquely as $K_{C} + A$, where $A$ is a basepoint free degree $3$ divisor class.  A general degree $10$ divisor class $E$ can be expressed as $2(D)$ for $2^{4}$ distinct degree $5$ divisor classes $D$.
\end{lemma}

\begin{proof}
Similar to \autoref{lem:p+q+r}. We leave the details to the reader.
\end{proof}

\begin{lemma}\label{lem:buildP1P1}
The general genus $2$, degree $10$ curve $C \subset \bp^{8}$ is contained in a $\bp^{1} \times \bp^{1}$ embedded via the linear system of $(2,2)$ curves.
\end{lemma}

\begin{proof}
Let $H$ denote the degree $10$ hyperplane divisor class of $C \subset \bp^{8}$. Write $H = 2D$ for some degree five divisor class $D$, and write $D \sim K_{C} + A$ for a unique degree $3$ divisor class $A$. By generality assumptions, $A$ is basepoint free, and we obtain a map \[f \colon C \to \bp^{1} \times \bp^{1}\] given by the pair of series $(|K_{C}|, |A|)$. This map embeds $C$ as a $(2,3)$ curve.

The linear system $|(2,2)|$ on this $\bp^{1} \times \bp^{1}$ restricts to the complete linear system $2D$ on $C$, and therefore induces the original embedding $C \subset \bp^{8}$.  The image of $\bp^{1} \times \bp^{1}$ under the system $|(2,2)|$ is therefore the surface we desire.
\end{proof} 

\begin{theorem}\label{thm:weakinterpolationP1xP1}
$\bp^{1} \times \bp^{1} \subset \bp^{8}$ embedded via the linear system of $(2,2)$ curves satisfies weak interpolation.
\end{theorem}

\begin{proof}
	This follows by combining \autoref{lem:genus2degree10} and \autoref{lem:buildP1P1}.
\end{proof}

\begin{remark}
An interesting feature of this solution to our interpolation problem is that the surfaces we've constructed through the $12$ general points $\Gamma_{12}$ are in fact {\sl special} among the two dimensional family of surfaces passing through these points.  Indeed, the set $\Gamma_{12}$ is contained in a $(2,3)$ curve on the surfaces we've constructed, but a general set of twelve points on $\bp^{1}\times \bp^{1}$ does not lie on any $(2,3)$ curve!
\end{remark}

\section{Interlude: association} 
\label{section:association}

This section is meant to provide the reader with basic familiarity with {\sl association}, also known as the {\sl Gale transform}.  Association will be a recurring tool in the rest of the appendix. We closely follow the exposition in \cite{eisenbudP:the-projective-geometry-of-the-gale-transform}.

\subsection{Preliminaries} Throughout this section, we let $\Gamma$ be a Gorenstein scheme, finite over $\bk$ of length $\gamma = r + s + 2$, $L$ an invertible sheaf on $\Gamma$, and $V \subset H^{0}(\Gamma, L)$ a vector space of dimension $r+1$.  In practice, $\Gamma$ will be given as embedded in projective space $\bp^{r}$, $L$ will be $\so_{\Gamma}(1)$, and $V$ will be the image of the restriction map \[H^{0}(\bp^{r}, \so_{\bp^{r}}(1) ) \to H^{0}(\Gamma, \so_{\Gamma}(1)).\] 
For brevity, we will often refer to the data of the pair $(V, L)$ as a {\sl linear system} on $\Gamma$. For clarity, we will sometimes put subscripts on $\Gamma$ emphasizing the number of points.

The Gorenstein hypothesis on $\Gamma$ says that the dualizing sheaf $\omega_{\Gamma}$ is a line bundle, and furthermore Serre duality holds: There is a trace map $t \colon H^{0}(\Gamma, \omega_{\Gamma}) \to \bk$, and for any line bundle $L$ the {\sl trace pairing}  \[H^{0}(\Gamma, L) \otimes H^{0}(\Gamma, L^\vee \otimes \omega_{\Gamma}) \to \bk\] is nondegenerate.  

Therefore if $V$ is a $r+1$ dimensional subspace of $H^{0}(\Gamma, L)$, we obtain a natural $s+1$-dimensional subspace \[V^{\perp} \subset H^{0}(\Gamma,L^\vee \otimes \omega_{\Gamma}),\] namely the orthogonal complement of $V$ under the trace pairing.

\begin{definition}\label{def:associatedseries}
Let $\Gamma$ be a length $\gamma$ Gorenstein scheme over $\bk$, and let $(V, L)$ be a linear system on $\Gamma$. Then we say $(V^{\perp}, L^\vee \otimes \omega_{\Gamma})$ is the {\sl associated linear system} of $(V,L)$.
\end{definition} 

Notice that association provides a correspondence between {\sl vector spaces} $V \leftrightarrow V^{\perp}$, and not between vector spaces with chosen bases. Geometrically this means association gives a bijection between the $\pgl_{r+1}(\bk)$--orbits of Gorenstein $\Gamma \subset \bp^{r}$ and $\pgl_{s+1}(\bk)$--orbits of Gorenstein $\Gamma \subset \bp^{s}$.  Given this, in the future when we refer to ``the associated set,'' we really mean the $\pgl_{s}$--orbit.  Moreover, it is known that association provides an isomorphism of GIT quotients \[(\bp^{r})^{\gamma}//\pgl_{r+1}(\bk) \xrightarrow{\sim} (\bp^{s})^{\gamma}//\pgl_{s+1}(\bk),\] and therefore takes general subsets to general subsets.

\subsection{Inducing association from an ambient linear system} Association is a very algebraic construction. Therefore, it is interesting to find geometric constructions which ``induce'' association for a set $\Gamma \subset \bp^{r}$.  To see many examples of the geometry underlying association, we refer to \cite{eisenbudP:the-projective-geometry-of-the-gale-transform}. 

In \cite[p.\ 2]{coble:associated-sets-of-points}, Coble asks, in less modern language, whether there exists a linear system $W^{s+1} \subset H^{0}(\bp^{r}, \so(d))$ whose base locus is disjoint from $\Gamma$, and which restricts on $\Gamma$ to the associated linear system.

A linear system $W^{s+1} \subset H^{0}(\bp^{r}, \so(d))$ yields a rational map \[\phi_{W} \colon \bp^{r} \dashrightarrow \bp^{s}.\] 

\begin{definition}
Let $\Gamma \subset \bp^{r}$ be a Gorenstein scheme of degree $\gamma = r+s+2$. 
An {\sl ambient linear system} is any vector space $V \subset  H^{0}(\bp^{r}, \so(d))$.
An ambient linear system $W^{s+1} \subset  H^{0}(\bp^{r}, \so(d))$  {\sl induces association for} $\Gamma$ if its base locus is disjoint from $\Gamma$ and if the image $\phi_{W}(\Gamma) \subset \bp^{s}$ is the associated set of $\Gamma$.  
\end{definition}
  It is important to note, as Coble does, that an ambient system inducing association won't be unique in general.

  When association is induced from an ambient system, we automatically get a variety $\phi_{W}(\bp^{r}) \subset \bp^{s}$ containing $\Gamma$. Our task is ultimately to find an ambient linear system $W$ which induces association  for $\Gamma$, and such that $\phi_{W}(\bp^{r})$ is a prescribed type of variety, e.g. Veronese images, del Pezzo surfaces, etc. 

\subsection{Goppa's Theorem} Goppa's theorem is frequently useful when looking for ambient systems inducing association.  

\begin{theorem}[Goppa's Theorem]\label{thm:goppa}
Let $f \colon B \to \bp^{r}$ be a map from a smooth curve given by a nonspecial, complete linear system $|H|$. Let $\Gamma \subset B$ be a scheme of length $\gamma = r+s+2$.  Then association for $\Gamma$ is induced by the restriction of the linear system $|K_{B}+\Gamma-H|$ to $\Gamma$.
\end{theorem} 

In practice, we will typically find a curve $B \subset \bp^{r}$ passing through $\Gamma$, and will try to induce association by realizing the linear system $|K_{B} + \Gamma - H|$ on $B$ via an ambient system on $\bp^{r}$.

\section{Degree $9$ del Pezzos}
\label{sec:degree-7-8-9}
This section establishes weak interpolation for degree $9$ Del Pezzo surfaces, which are $3$-Veronese images of $\bp^{2}$ in $\bp^{9}$.  As we will see
in \autoref{ssec:degrees-7-and-8},
weak interpolation for degree $8$, type 1, and degree $7$ Del Pezzo surfaces immediately follow from the proof for degree $9$. We will also see in
\autoref{ssec:degrees-7-and-8} that tricanonical genus $3$ curves satisfy
interpolation.

\subsection{Results}
The main result of this section is:

\begin{theorem}[Existence]\label{thm:3veroneseexistence}
Let $\Gamma \subset \bp^{9}$ be thirteen general points. Then there exists a Veronese $3$-Veronese surface  containing $\Gamma$.
\end{theorem}
\begin{proof}[Proof assuming \autoref{thm:bijectionTriadsVeronese} and \autoref{theorem:DegenerateConfigurationInClosure}]
	By \autoref{thm:bijectionTriadsVeronese}, for a general
	$\Gamma_{13} \in \Hilb_{13} \bp^2$, there is a bijection
	between singular triads for $\Gamma_{13}$
	and $3$-Veronese surfaces containing the associated set $A(\Gamma_{13}) \subset \bp^{9}$. 
	By \autoref{theorem:DegenerateConfigurationInClosure},
	every such $\Gamma_{13}$ indeed possesses a singular triad.
\end{proof}

The essential tool used in proving \autoref{thm:3veroneseexistence} is association. 

Our next result relates the number of Veronese surfaces through $13$ general points to another, more tractable  enumerative problem.  Before stating it, we must make a definition. 

\begin{definition}\label{def:singulartriad}
Let $\Gamma \subset \bp^2$ be a general set of thirteen points in the plane. A subset $T = \{x,y,z\} \subset \bp^2$ of three distinct points is a {\sl singular triad} for $\Gamma$ if \[h^{0}(\bp^{2}, \so_{\bp^{2}}(5)\otimes \sci_{T}^{2}\sci_{\Gamma}) = 2.\]
\end{definition}

\begin{remark}
	In other words, $T = \{x,y,z\}$ is a singular triad for $\Gamma$ if there exists a pencil of quintic curves through $\Gamma$ and singular at $x, y,$ and $z$. A dimension count shows that we expect finitely many singular triads for a general set of thirteen points $\Gamma$, as is done in \autoref{lem:dimensionPhi}.
\end{remark}

Our second result is:

\begin{theorem}[Enumeration]\label{thm:3veronesenumber}
The number of $3$-Veronese surfaces through a general set of thirteen points in $\bp^{9}$ is equal to the number of singular triads for a general set of thirteen points in $\bp^{2}$.
\end{theorem}

\autoref{thm:3veronesenumber} points to an interesting enumerative problem on the Hilbert scheme ${\rm Hilb}_{3}(\bp^{2})$ of degree $3$, zero dimensional subschemes of the plane.  We discuss this problem at the end of the section, in 
\autoref{ssec:enumerating-singular-triads}.

\subsection{How singular triads arise}\label{sssec:howsingtriadsArise} For the benefit of the reader, we briefly explain how singular triads arise in the problem of enumerating $3$-Veronese surfaces through a general set $\Gamma_{13}$.  

Suppose $V_{3} \subset \bp^{9}$ is a $3$-Veronese containing $\Gamma_{13}$.   If we consider $V_{3}$ as isomorphic to $\bp^{2}$, it is tempting to think that the linear system $|3H|$ on $\bp^{2}$ would induce association for $\Gamma_{13} \subset \bp^{2}$.  However, this turns out not to be the case. 

In $\bp^{2}$, there is a unique pencil of quartic curves $Q_{t} \subset \bp^{2}$, $t \in \bp^{1}$, containing $\Gamma_{13}$.  Assuming the configuration $\Gamma_{13}$ is general, the pencil $Q_{t}$ will have three remaining, noncollinear basepoints $\{p,q,r\}$.  Let \[\alpha_{\{p,q,r\}} \colon \bp^{2} \dashrightarrow \bp^{2}\] be the Cremona transformation centered on the set $\{p,q,r\}$, and let $T = \{x,y,z\}$ be the exceptional set 
in the target $\bp^{2}$. Then $\alpha(Q_{t})$ is a pencil of quintic curves, singular at $T$ and containing $\alpha(\Gamma_{13})$ in its base locus.  In other words, $T$ forms a singular triad for $\alpha(\Gamma_{13})$. In the next section, we will show that the ambient system of sextics triple at $x,y$, and $z$ induces association for $\alpha(\Gamma_{13})$.  In other words, the ``naive'' system of cubics on the source $\bp^{2}$ induces association {\sl not} for $\Gamma_{13}$, but rather for $\alpha(\Gamma_{13})$.

\subsection{Inducing association from a singular triad} 

\begin{lemma}\label{lem:quinticbaselocus}
	Assume $\Gamma_{13} \subset \bp^2$ is a general set of $13$ points, and suppose $T = \{x,y,z\}$ is a singular triad for $\Gamma_{13}$, i.e. there exists a pencil of quintics $Q_{t}$ through $\Gamma_{13}$ and double at $x,y,z$. Furthermore, assume that the general element of the pencil has a smooth genus $3$ normalization, and has ordinary nodes at $x,y,z$. In particular, this implies $T$ is not contained in a line.  
Then, the scheme theoretic base locus of the pencil $Q_{t}$ consists of $\Gamma_{13}$ and three length four schemes supported on $x,y$ and $z$.
\end{lemma}

\begin{proof}
Clear.
\end{proof}

\begin{proposition}\label{prop:6h-3x-3y-3z}
	In the setting of \autoref{lem:quinticbaselocus}, the ten dimensional vector space \[W := H^{0}(\bp^{2}, \sci_{T}^{3}(6)) \subset H^{0}(\bp^{2}, \so(6))\] consisting of sextics triple at $x,y$ and $z$ induces association for $\Gamma_{13}$.
\end{proposition}

\begin{proof}
	We use Goppa's theorem, \autoref{thm:goppa}. Pick a general quintic $Q$ in the pencil $Q_{t}$, and let $\nu \colon \widetilde{Q} \to Q$ denote the smooth genus $3$ normalization.
	Let $H$ denote the hyperplane divisor class on $\bp^{2}$. 
	Note that by degree considerations and Riemann-Roch, the divisor class $H$ is nonspecial on $\widetilde{Q}$, and $\widetilde{Q}$ is mapped 
		via the complete linear series $|H|$. We claim that the linear system $|K_{\widetilde{Q}} + \Gamma_{13} - H|$ from Goppa's theorem is induced by sextic curves triple at $x,y$ and $z$.

Indeed, the canonical series $|K_{\widetilde{Q}}|$ is cut out by the adjoint series consisting of conics passing through the nodes $x,y,$ and $z$.  By \autoref{lem:quinticbaselocus}, the divisor $\Gamma_{13}$ is cut out by a quintic double at $x,y,z$.  Putting these together says that sextics triple at $x,y,$ and $z$ cut out divisors in the linear system $|K_{\widetilde{Q}} + \Gamma_{13} - H|$ on $\widetilde{Q}$. 

Finally, notice that there cannot be a sextic triple at $x,y$ and $z$ which also vanishes identically on $Q$ -- the residual curve would be a line containing $T$, but we are assuming $x,y,z$ are not collinear. Therefore, the system of sextics triple at $x,y,$ and $z$ cuts out the complete linear system $|K_{\widetilde{Q}} + \Gamma_{13} - H|$. 
\end{proof}

\subsection{The bijection between singular triads and Veroneses} Let $\Gamma_{13} \subset \bp^{9}$ be thirteen general points, and let $A(\Gamma_{13}) \subset \bp^{2}$ denote the associated set. 

We have already seen in ~\autoref{sssec:howsingtriadsArise} that a Veronese $3$-Veronese $V_{3}$ containing $\Gamma_{13}$ arises from a singular triad $T$ for $A(\Gamma_{13})$. Now show that distinct triads provide distinct Veroneses

\begin{proposition}\label{prop:distinctTriads}
Maintain the setting above. Distinct triads $T$ and $T'$ for $A(\Gamma_{13})$ give rise to distinct Veronese surfaces $V_{3}$ and $V'_{3}$ containing $\Gamma_{13}$. 
\end{proposition}

\begin{proof}
Let $W = H^{0}(\bp^{2}, \so_{\bp^{2}}(2)\otimes\sci_{T})$ and $W' = H^{0}(\bp^{2}, \so_{\bp^{2}}(2)\otimes\sci_{T'})$ be the vector spaces of conics passing through $T$ and $T'$ respectively.  

Denote by $\iota \colon \bp^{2} \dashrightarrow \bp(W)$ and $\iota' \colon \bp^{2} \dashrightarrow \bp(W')$ the Cremona maps associated to $W$ and $W'$. 

By \autoref{prop:6h-3x-3y-3z}, the vector spaces $\sym^{3}W$ and $\sym^{3}W'$ both induce association for $A(\Gamma_{13})$, so we identify them as the ten dimensional vector space $V$ giving the original embedding $\Gamma_{13} \subset \bp^{9}$. 

Let $\nu \colon \bp(W) \hookrightarrow \bp^{9}$ and $\nu' \colon \bp(W') \hookrightarrow \bp^{9}$ denote the respective Veronese maps. (Note that the target $\bp^{9}$ for the maps $\nu$ and $\nu'$ are the ``same'' given the previous paragraph.) 

The two Veronese surfaces $\bp(W)$ and $\bp(W')$ would be the same if and only if there existed an isomorphism $\alpha \colon \bp(W) \to \bp(W')$ such that $\nu' \circ \alpha \circ \iota = \nu' \circ \iota'$ as rational maps from $\bp^{2}$ to $\bp^{9}$. 

But the indeterminacy locus of a rational map is determined by the map, and the indeterminacy locus of $\nu' \circ \alpha \circ \iota$ is $T$, whereas the indeterminacy locus of $\nu' \circ \iota'$ is $T'$. This completes the proof.
\end{proof}

\begin{theorem}\label{thm:bijectionTriadsVeronese}
Let $\Gamma_{13} \subset \bp^{9}$ be a general set of thirteen points.  Then the Veronese $3$-Veronese surfaces containing $\Gamma_{13}$ are in bijection with the singular triads for $A(\Gamma_{13}) \subset \bp^{2}$.
\end{theorem}

\begin{proof}
This follows immediately from ~\autoref{sssec:howsingtriadsArise} and \autoref{prop:distinctTriads}.
\end{proof}

\subsection{Existence of singular triads}
\label{ssec:existence-of-singular-triads}

\begin{definition}
Define $\Phi \subset \Hilb_{3}\bp^{2} \times \Hilb_{13}\bp^{2}$ to be the closure of the set of 
pairs 
$(\{x,y,z\}, \Gamma_{13}) \subset \Hilb_3\bp^2 \times \Hilb_{13} \bp^2$ for which $\left\{ x, y, z \right\}$ is disjoint from the support of $\Gamma_{13}$, and for which there exists a pencil of quintics singular at $x,y,z$ whose base locus is precisely $\left\{ x,y,z \right\} \cup \Gamma_{13}$. 
Define the projections
\begin{equation}
	\nonumber
	\begin{tikzcd}
		\qquad & \Phi \ar {ld}{\pi_1} \ar {rd}{\pi_2} & \\
		\Hilb_3 \bp^2 && \Hilb_{13} \bp^2.
	 \end{tikzcd}\end{equation}
\end{definition}

\begin{theorem}\label{theorem:DegenerateConfigurationInClosure} There exists a point $(\{x,y,z\}, A_{13}) \in \Phi$ which is isolated in its fiber under the second projection $\pi_{2} \colon \Phi \to \Hilb_{13}\bp^{2}$. In particular, $\pi_{2}$ is dominant, and a general set $\Gamma_{13}$ possesses a singular triad.
\end{theorem}

The rest of the section is devoted to the proof of \autoref{theorem:DegenerateConfigurationInClosure}.  Before we proceed with the proof in \autoref{sssec:degeneration-proof}, we set some notation and outline the idea of the proof in
\ref{sssec:degeneration-proof-idea}

\begin{definition}
	\label{definition:}
	Let $x_0, x_1, x_2$ denote three fixed non-collinear points in $\bp^2$ and set $l_{i,j} :=\overline{x_{i}x_{j}}$ forming the coordinate triangle. 

Let $X := \blow_{\{x_{0},x_{1},x_{2}\}}\bp^{2}$, and let $E_{i}$ denote the exceptional divisor over $x_{i}$, $i = 1,2,3$. Set $L_{i,j}$ to be the proper transforms of the lines $l_{i,j} := \overline{x_i, x_j}$.   We let $H$ denote the hyperplane class on $\bp^{2}$ and its pullback on $X$.  By a {\sl line} in $X$, we mean an element of the linear system $|H|$ on $X$.
\end{definition}

\sssec{Idea of Proof of \autoref{theorem:DegenerateConfigurationInClosure}.}
\label{sssec:degeneration-proof-idea}

In order to prove \autoref{theorem:DegenerateConfigurationInClosure}, we will
construct a particular set $[\Gamma_{13}] \in \Hilb_{13} \bp^2$ which
we will be able to see is isolated in its fiber under the map $\pi_2$.
The construction is as follows. 
Start by choosing a general line $M$ and a general point $p_{7}$
not on $M$. Then, choose points 
\begin{align*}
	p_1, p_2 &\in \ell_{0,1} \\
	p_3, p_4 &\in \ell_{0,2} \\
	p_5, p_6 &\in \ell_{1,2} \\
	p_8, p_9, p_{10} &\in M
\end{align*}
all general with respect to the above conditions.
We will then see that there is an element $(\left( x_0, y_0, z_0 \right), \Gamma_{13}) \in \Phi$ so that $p_1 \cup \cdots \cup p_{10} \subset \Gamma_{13}$,
and further that the remaining degree three scheme of $\Gamma_{13}$ is
supported on $M$. The hard part of the proof will be seeing that this
configuration lies in $\Phi$. This is done in \autoref{cor:closureofresidual}.
Once we know this configuration does lie in $\Phi$, it is not difficult
to see it is isolated. Since $\Gamma_{13}$ intersects $M$ with degree $6$,
every quintic containing $\Gamma_{13}$ must contain $M$. We are then looking
for a pencil of quartics with base locus containing $p_1, \ldots, p_6, p_7$
and having three additional singular nodes.
If the three singular nodes
do not lie on $M$, then this can only happen if the pencil of quartics
contains curves in its base locus. A case by case analysis shows
that if the three nodes are not collinear, the only possibility, up to permutation of the points, is that the base locus of this pencil of quartics is
$\ell_{0,1} \cup \ell_{0,2} \cup \ell_{1,2} \cup p_{7}$ and the moving part
of this pencil is the pencil of lines containing $p_{7}$. This
will be isolated in its fiber. 
Then, this means $\pi_2$ is dominant
because both varieties are irreducible and
$\dim \Phi = 26 = \dim \Hilb_{13}\bp^2$, as shown in \autoref{lem:dimensionPhi}.

\begin{lemma}\label{lem:generalquinticpencil}
Let $\Gamma_{10} \subset X$ be ten general points. Then there is a unique pencil in the linear system $|5H - 2E_{1} - 2E_{2} - 2E_{3}|$ containing $\Gamma_{10}$  in its base locus.  Furthermore, the base locus of this pencil consists of the union of  $\Gamma_{10}$, and three residual points $\{a,b,c\} \subset S$ disjoint from $\Gamma_{10}$.

\end{lemma}

\begin{proof}
	Clear. 
\end{proof}
\begin{lemma}\label{lem:dimensionPhi}
$\Phi$ is $26$-dimensional.
\end{lemma} 
\begin{proof}
	First select three general points $\{x,y,z\}$ in $\bp^{2}$, giving $6$ dimensions. Using \autoref{lem:generalquinticpencil}, a general pencil of quintics singular at $\{x,y,z\}$ is determined by choosing ten general points to be in its base locus. The remaining three points of the base locus are determined
	by the initial choice of 10, by \autoref{lem:generalquinticpencil}. In total, we have that $\Phi$ is $26 = 6 + 2 \cdot 10$ dimensional.
\end{proof}

Let $\Gamma_{10}(t) = \{p_{1}(t), p_{2}(t), ... , p_{10}(t)\} \subset X \times \Delta$ be a family of ten points, parameterized by $\Delta := \spec \bk[[t]]$, general among those with the following properties:

\begin{enumerate}
\item Over the generic point $\eta \in \Delta$, the points $p_{i}(\eta)$ are general in the sense of \autoref{lem:generalquinticpencil}. 

\item Over the special point $t=0$, the ten points $p_{i}(0)$ are situated as follows: 

\begin{enumerate}
\item $p_{1}(0), p_{2}(0)$ are general in $L_{0,1}$.
\item $p_{3}(0), p_{4}(0)$ are general in $L_{0,2}$.
\item $p_{5}(0), p_{6}(0)$ are general in $L_{1,2}$.
\item $p_{7}(0)$ is general in $X$.
\item $p_{8}(0), p_{9}(0), p_{10}(0)$ are general on a general line $M \subset X$.

\end{enumerate}

\end{enumerate}

By \autoref{lem:generalquinticpencil}, there are three residual points $\{a(\eta), b(\eta), c(\eta)\}$ defined by the ten points $\{p_{i}(\eta)\}_{i=1, ... , 10}$. We let $\{a(t), b(t), c(t)\}$ denote the closures of these points.  (Note: a base change may be required to say the three residual basepoints $\{p_{i}(\eta)\}_{i=1, ... , 10}$ are defined over $\spec \bk((t))$. Performing such a base change does not affect the rest of the arguments.)

Now let $\scx$ be the threefold which is the blow up of $X \times \Delta$ at the union of the three curves $L_{i,j} \subset X \times \{0\}$, and let $b \colon \scx \to X \times \Delta$ be the blow up map.  Let $f \colon \scx \to \Delta$ denote the composition of $b$ with the projection onto the second factor of $X \times \Delta$.  Let $\scx_{\eta}$ and $\scx_{0}$ denote the general and special fibers of $f$.  Note that $\scx_{\eta} = X_{\eta}$.

There are three exceptional divisors $F_{i,j}$ lying over the corresponding curves $L_{i,j} \subset X \times \{0\}$.  The map $f$ is a flat family of surfaces, with generic fiber $\scx_{\eta} = X_{\eta}$ and with special fiber $\scx_{0}$ a simple normal crossing union of four surfaces: 
the exceptional divisors $F_{i,j}$, and $X$.  Their incidence is as follows: The surfaces $F_{i,j}$ are pairwise disjoint and $F_{i,j} \cap X = L_{i,j}$.  

Each exceptional divisor $F_{i,j}$ is isomorphic to the Hirzebruch surface $\bbf_{1}$. This is because each rational curve $L_{i,j} \subset X_{0}$ has self-intersection $(-1)$, and therefore has normal bundle $N_{L_{i,j}/X \times \Delta} \cong \so(-1)\oplus \so.$

On $\bbf_{1}$, we let $S$ denote the divisor class of a {\sl codirectrix}, a section class having self-intersection $+1$.
	We denote by $R$ the ruling line class.  We let $S_{i,j}$ and $R_{i,j}$ denote the corresponding divisor classes on $F_{i,j}$.  

Let $\scl$ be the line bundle $b^{*}(\so_{X \times \Delta}(5H - 2E_{1}-2E_{2}-2E_{3}))$, and let $p'_{i}(t) \subset \scx$ denote the lifts of $p_{i}(t)$ to $\scx$. In other words, $\{p'_{i}(t) \}_{i=1, ... 10}$ are the closures of the points $\{p_{i}(\eta)\} \in X_{\eta} = \scx_{\eta}$ in $\scx$. 

By the generality assumptions on the $1$-parameter family of points $\{p_{i}(t)\}_{i=1, ... , 10}$ in $X \times \Delta$, we may assume the following about the central configuration of points $p'_{i}(0)$ in $\scx_{0}$:
\begin{enumerate}
\item The points $p'_{1}(0), p'_{2}(0)$ are general in $F_{0,1}$.
\item The points $p'_{3}(0), p'_{4}(0)$ are general in $F_{0,2}$.
\item The points $p'_{5}(0), p'_{6}(0)$ are general in $F_{1,2}$.
\item The points $m_{i,j} := M \cap L_{i,j}$ are general in $F_{i,j}$ with respect to the other two points mentioned in each part above. 
\end{enumerate}

\begin{figure}
\begin{equation}
	\nonumber
  \begin{tikzcd} 
    F_{i,j} \ar{d} \arrow[hookrightarrow]{r} & \scx := Bl_{L_{i,j} \times \left\{ 0 \right\}} \left(  X \times \Delta\right) \ar {d}{b}  & p'_i(t) \arrow[hookrightarrow]{l} \\
     L_{i,j} \times \left\{ 0 \right\} \arrow[hookrightarrow]{r} & X \times \Delta 
     & p_i(t) \arrow[hookrightarrow]{l} \arrow[-,double equal sign distance]{u} 
  \end{tikzcd}\end{equation}
\caption{A pictorial summary of relevant schemes}
\end{figure}

Set \[\scl ' := \scl(-F_{0,1} - F_{0,2} - F_{1,2}).\]

\begin{lemma}\label{lem:restrictionsofLprime}
The line bundle $\scl '$ restricts to $\so_{F_{i,j}}(S_{i,j} + R_{i,j})$ on the exceptional divisors $F_{i,j} \subset \scx_{0}$ and restricts to $\so_{X}(2H)$ on $X \subset \scx_{0}$.
\end{lemma}

\begin{proof}
Straightforward.
\end{proof}

\begin{remark}
For the benefit of the reader, we give an alternate description of the linear system $|S + R|$ on $\bbf_{1}$ appearing in the above lemma.  If we view $\bbf_{1}$ as the blow up of $\bp^{2}$ at a point $q \in \bp^{2}$, then the linear system $|S+R|$ is the system of conics through the point $q$.  

In particular, if three more general points are chosen on $\bbf_{1}$, there will be a unique pencil of curves in $|S+R|$ containing them.
\end{remark}

Now consider the sheaf $\scf := \sci_{\{p'_{i}(t)\}_{i=1, ..., 10}} \otimes \scl '$ . 

\begin{lemma}\label{lem:cohomologybasechange}
The $\bk[[t]]$-module $H^{0}(\scx, \scf)$ is free of rank $2$.  Furthermore, the restriction map \[H^{0}(\scx, \scf) \to H^{0}(\scx_{0}, \scf|_{\scx_{0}})\] is surjective.
\end{lemma}

\begin{proof}
Since $\scf$ is a torsion free sheaf, $H^{0}(\scx, \scf)$ is a torsion free $\bk[[t]]$-module, i.e. it is free. \autoref{lem:generalquinticpencil} tells us that the rank must be $2$.

By Grauert's theorem, it suffices to show \[h^{0}(\scx_{0}, \scf|_{\scx_{0}}) = 2.\]

A section $s$ of $\scf|_{\scx_{0}}$ is a section of $\scl ' |_{\scx_{0}}$ vanishing at the ten points $p'_{i}(0)$. We will now analyze what the zero locus of $s$ must be on each of the four components of $\scx_{0}$, beginning with $X$. 

The restriction $s|_{X}$ vanishes on a conic containing $p'_{7}(0), p'_{8}(0), p'_{9}(0),$ and $p'_{10}(0)$. Since the latter three points are collinear lying on the line $M$, such a conic is degenerate, of the form $M \cup N$, where $N$ is any line containing $p'_{7}(0)$. 

The restriction $s|_{F_{0,1}}$ vanishes on a divisor of class $|S_{0,1} + R_{0,1}|$ containing the pair of points $p'_{1}(0), p'_{2}(0)$.  Similar descriptions hold for the remaining two components. 

A section $s$ of $\scf |_{\scx_{0}}$ consists of sections on each component which agree on the intersection curves $L_{i,j}$. 
We claim that such a global section is determined, up to scaling, by its restriction to the component $X$.  Indeed, by choosing a conic of the form $M \cup N$, we determine two points $m_{i,j}, n_{i,j}$ on each line $L_{i,j}$, namely the intersections $M \cap L_{i,j}$, $N \cap L_{i,j}$.  

From the generality assumptions we have imposed, we get that there is a unique curve in the class $|S_{0,1} + R_{0,1}|$ containing the four points  $p'_{1}(0), p'_{2}(0), m_{0,1},$ and $n_{0,1}$.  Similarly for the other components $F_{i,j}$. 
It follows that any global section of $\scf$ is determined, up to scaling, by its restriction to $X$. But the restriction to $X$ is a degenerate conic of the form $M \cup N$ as described above, and therefore $h^{0}(\scx_{0}, \scf|_{\scx_{0}}) = 2$, as we claimed. 
\end{proof}

\begin{lemma}\label{lem:centralbaselocus}
The common zero locus of all sections of $\scf |_{\scx_{0}}$ is  the scheme $M \cup \{p'_{1}(0) , ... , p'_{6}(0) , p'_{7}(0)\}$.
\end{lemma}

\begin{proof}
This follows from the description of the zero loci of sections of  $\scf |_{\scx_{0}}$ found in the proof of \autoref{lem:cohomologybasechange}.
\end{proof} 

Let $\langle f_{1}, f_{2} \rangle$ be a $\bk [[t]]$-basis for $H^{0}(\scx, \scf)$.  In particular, $\langle f_{1}, f_{2} \rangle$ restricts to a basis of $H^{0}(\scx|_{\scx}, \scf|_{\scx_{0}})$ by \autoref{lem:cohomologybasechange}.

\begin{lemma}\label{lem:familybaseloci}
Maintain the notation above, and let $\scy \subset \scx$ defined by $f_{1} = f_{2} = 0$ be the common zero scheme.  Then, as schemes, $\scy \cap \scx_{0} =  M \cup \{p'_{1}(0) , ... , p'_{6}(0) , p'_{7}(0)\}$ and $\scy \cap \scx_{\eta} = \{p_{1}(\eta), ... , p_{10}(\eta), a(\eta), b(\eta), c(\eta)\}$.
\end{lemma}

\begin{proof}
The generality assumptions on the original family of points $p_{i}(t)$ and \autoref{lem:generalquinticpencil} ensure the statement regarding $\scy \cap \scx_{\eta}$. Then, $\scy \cap \scx_{0} =  M \cup \{p'_{1}(0) , ... , p'_{6}(0) , p'_{7}(0)\}$ follows from \autoref{lem:centralbaselocus}.
\end{proof}

Now let $\{a'(t), b'(t), c'(t)\}$ denote the closures of $\{a(\eta), b(\eta), c(\eta)\}$ in $\scx$.

\begin{corollary}\label{cor:closureofresidual}
The scheme $\{p'_{8}(t), p'_{9}(t), p'_{10}(t), a'(t), b'(t), c'(t)\} \cap \scx_{0}$ is contained in the line $M \subset X \subset \scx_{0}$. 
\end{corollary}

\begin{proof}
This follows from \autoref{lem:familybaseloci}.  Indeed, \[\{p'_{8}(t), p'_{9}(t), p'_{10}(t), a'(t), b'(t), c'(t)\} \cap \scx_{0}\] must be a subscheme of $\scy \cap \scx_{0} =  M \cup \{p'_{1}(0) , ... , p'_{6}(0) , p'_{7}(0)\}$.  The sections $\{a'(t), b'(t), c'(t)\}$ cannot limit to any of the seven isolated points $\{p'_{1}, ... , p'_{7}\}$, since these seven points occur with multiplicity one in the scheme $\scy \cap \scx_{0}$. Therefore, the points $\{a'(0), b'(0), c'(0)\}$  must limit to $M$. 
\end{proof}

\ssec{Proof of \autoref{theorem:DegenerateConfigurationInClosure}.}
\label{sssec:degeneration-proof}

\begin{proof}	
	A one parameter family of thirteen points \[\{p_{1}(t), ... , p_{10}(t), a(t), b(t), c(t)\}\] discussed above limits, at $t=0$, to a configuration which we call $\Gamma_{13} \subset \bp^{2}$.  (Technically, $\Gamma_{13}$ is a set in $X$, but we view it as a set in $\bp^{2}$, since $\Gamma_{13}$ avoids the exceptional divisors in $X$.)

Now we argue that the pair $(\{x_{0},y_{0},z_{0}\}, \Gamma_{13})$ is isolated in its fiber under the projection $\pi_{2} \colon \Phi \to \Hilb_{13}\bp^2$.

It suffices to show that there are only finitely many noncollinear triads $T \subset \bp^{2}$ disjoint from $\Gamma_{13}$ for which there is a pencil of quintics $C_{t}$ all singular at $T$ and containing $\Gamma_{13}$.

Any pencil of quintics containing $\Gamma_{13}$ must contain the line $M$ in its base locus, since $6$ of the points of $\Gamma_{13}$, $\{p_{8}(0), p_{9}(0), p_{10}(0), a(0), b(0), c(0)\}$ lie on this line (\autoref{cor:closureofresidual}). Therefore, the residual quartic curves of the pencil, denoted $C'_{t}$, form a pencil of curves singular at $T$, and containing $\{p_{1}(0), ... , p_{6}(0), p_{7}(0)\}$ in its base locus.  Note that the set $\{p_{1}(0), ... , p_{6}(0), p_{7}(0)\}$ is a general set of seven points in the plane.

By degree considerations, a pencil of quartics $C'_{t}$ singular at $T$ and having $7$ remaining points in its base locus is forced to have an entire curve $B$ in its base locus.   The curve $B$ must have degree $1, 2,$ or $3$.  

A straightforward combinatorial check shows that if the three points of $T$ are not collinear, the curve $B$ must be the union of three lines joined by three pairs of points among the set $\{p_{1}(0) ,... ,p_{6}(0), p_{7}(0)\}$, and the triad $T$ is the vertices of the triangle $B$.  All told, there are only finitely many possibilities for $T$, which in turn implies that $(\{x_{0},y_{0},z_{0}\}, \Gamma_{13})$ is isolated in its fiber under projection $\pi_{2} \colon \Phi \to \Hilb_{13}\bp^2$. 
\end{proof}

\begin{remark}\label{remark:630}
The method of proof for \autoref{theorem:DegenerateConfigurationInClosure} actually shows that there are at least $630$  $3$-Veronese surfaces through thirteen general points in $\bp^{9}$. The reason for this is that we made several
choices in constructing an isolated point of the incidence correspondence.
We choose one of the points $p_1, \ldots, p_6, p_{10}$ to not lie
in the triangle containing the nodal base locus, and we then chose a division
of the remaining six points into three pairs of two points.
In total there are $7 \cdot \frac{6!}{2! \cdot 2! \cdot 2!} = 630$
such choices, and hence at least $630$ isolated points.
Then it follows that there are at least $630$ $3$-Veronese surfaces
through a general set of $13$ points by
\cite[II.6.3, Theorem 3]{shafarevich:basic-algebraic-geometry-1}.
\end{remark}

\section[The remaining del Pezzo surfaces]{The remaining del Pezzo surfaces and tricanonical genus 3 curves} 
\label{ssec:degrees-7-and-8}
\subsection{Degree $8$.} 
Weak interpolation for degree $8$, type $1$ del Pezzo surfaces asks
 whether such surfaces pass through $12$ general points $\Gamma_{12} \subset \bp^{8}$,
by \autoref{table:del-Pezzo-conditions}. In fact, weak interpolation
for degree $8$, type $1$ del Pezzo surfaces follows almost immediately from our knowledge of interpolation for degree $9$ del Pezzos. 

\begin{corollary}
  \label{corollary:degree-8-type-1-interpolation}
Degree $8$ del Pezzos isomorphic to the Hirzebruch surface $\bbf_{1}$ satisfy weak interpolation.
\end{corollary}

\begin{proof}

Indeed, Let $A(\Gamma_{12}) \subset \bp^{2}$ be the associated set. Now append a general thirteenth point $p \in \bp^{2}$ and let $B_{13} \subset \bp^{2}$ be the union. 

As follows from \autoref{prop:6h-3x-3y-3z}, association for $B_{13}$ is induced by the linear system of sextics triple at a singular triad for $B_{13}$.  Now we take the subsystem of such sextics with further basepoint at the chosen point $p$.  The resulting subsystem induces association for $A(\Gamma_{12})$, and maps $\bp^{2}$ birationally to a degree $8$ del Pezzo containing $\Gamma_{12}$, abstractly isomorphic to the Hirzebruch surface $\bbf_{1}$.
\end{proof} 

\begin{remark}
The parameter count suggests that there will be a two dimensional family of del Pezzo $8$'s through a general $\Gamma_{12}$. The argument above also has two dimensions of freedom in the choice of auxiliary point $p$.
\end{remark}

\subsection{Degree $7$ del Pezzo surfaces} 

\begin{corollary}
  \label{corollary:degree-7-interpolation}
Degree $7$ del Pezzo surfaces satisfy weak interpolation.
\end{corollary}

\begin{proof}
The parameter count says that weak interpolation for such surfaces is equivalent to asking them to pass through $11$ general points $\Gamma_{11} \subset \bp^{7}$.  We now proceed analogously to the previous case: We now append two general auxiliary points $p,q \in \bp^{2}$ to the associated set $A(\Gamma_{11}) \subset \bp^{2}$. 
\end{proof}

\begin{remark}
Paralleling the degree $8$ case, the dimension of del Pezzo 7's through eleven general points is four dimensional, as is the dimension of the space of auxiliary pairs $p,q \in \bp^{2}$. 
\end{remark}

\begin{remark}
The reason why this method fails for degree $6$ del Pezzos is that the number of points required by weak interpolation is not $10$, as the current pattern would suggest.  Rather, the required number of points is eleven, and therefore we needed a separate argument. 
\end{remark}

\ssec{Genus 3 tricanonical curves}

As a bonus, we show that the closed locus of the Hilbert scheme of
degree $12$ genus $3$ curves in $\bp^9$ which are tricanonically
embedded satisfy interpolation.

\begin{corollary}
	\label{corollary:tricanonical}
	The closed locus of the Hilbert scheme of degree $12$ genus $3$
	curves in $\bp^9$ which are tricanonically embedded satisfy interpolation.
\end{corollary}
\begin{proof}
	First, note that there is a $105 = 99 + 6$ dimensional space of
	tricanonically embedded curve, where $99 = \dim \pgl_{10}$ and
	$6 = \dim \mathscr M_3$. In this case, by \ref{interpolation-sweep},
	we have to show that there is a 1 dimensional
	family of such curves through $13$ points, sweeping out a surface.
	But, since we know a $3$-Veronese surface passes through these $13$
	points, we have a 1 dimensional family of tricanonical genus $3$ curves
	sweeping out this Veronese surface passing through $13$ points, as desired.
\end{proof}

\section[Enumerating singular triads]{Enumerating singular triads: observations and obstacles}
\label{ssec:enumerating-singular-triads}
We now discuss the obstacle we face in the computation of the number of singular triads for a general set $\Gamma_{13} \subset \bp^{2}$.  Set $S = \blow_{\Gamma_{13}}\bp^{2}$, and let $\scl = \so_{S}(5H - E_{1} - ... -E_{13})$.

We should set up the problem on a compact, smooth space.  A natural choice is the Hilbert scheme $\Hilb_{3}S$ parameterizing length three subschemes of $S$. 

The universal scheme $\scz \subset \Hilb_{3}S \times S$ 
has two obvious projections   $\pi_{1} \colon \scz \to \Hilb_{3}S$ and $\pi_{2} \colon \scz \to S$.  Next, we consider the sheaf \[\scf = \pi_{1*}(\pi_{2}^{*}\scl/(\sci^{2}_{\scz} \otimes \pi_{2}^{*}\scl)).\]

Unfortunately, the sheaf $\scf$, which has generic rank $9$, fails to be locally free precisely along the locus $F \subset \Hilb_{3}S$ parameterizing degree $3$ schemes of the form $\spec \bk[x,y]/(x^{2},xy,y^{2})$, also known as the ``fat points''. 

There is a natural restriction map \[\rho \colon \so_{\Hilb_{3}S}^{\oplus 8} \to \scf.\]  If $\scf$ were locally free of rank $9$, we could attempt to use Porteous' formula to find the locus where the rank of $\rho$ drops to $6$. Since $\scf$ is not locally free, this approach fails from the outset.

One fix would be to work on a blow up of $\Hilb_{3}S$ along the locus $F$, but then it's unclear what should replace the sheaf $\scf$.  What's more, we would need to identify the Chern classes of the replacement sheaf in the Chow ring of $\blow_{F}(\Hilb_{3}S)$, which is challenging in its own right. See \cite{elencwajg:l-anneau-de-chow-des-triangles-du-plan}.

Another potential fix would be to work in the {\sl nested} Hilbert scheme $\Hilb_{2,3}S$ parameterizing, $X_{2} \subset X_{3} \subset S$, pairs of length $2$ subschemes contained in length $3$ subschemes. It is known that $\Hilb_{2,3}S$ is smooth, and it has a generically finite, degree $3$ map to $\Hilb_{3}S$ given by forgetting $X_{2}$, which has one dimensional fibers (isomorphic to $\bp^{1}$) precisely over $F \subset \Hilb_{3}S$. This space $\Hilb_{2,3}S$ might be better suited for replacing the problematic sheaf $\scf$ above.  Finding a solution to these issues is the subject of ongoing work. 
As further references for enumerative geometry in the Hilbert scheme of three points, see \cite{russell:counting-singular-plane-curves-via-hilbert-schemes}, \cite{russell:degenerations-of-monomial-ideals}, and \cite{harrisP:severi-degrees-in-cogenus-3}.

\bibliographystyle{alpha}
\bibliography{master}

\end{document}